\documentclass[11pt,a4paper,reqno]{amsart}
\usepackage[a4paper,left=2.0cm,right=2.0cm,top=2.3cm,bottom=2.3cm]{geometry} 
\usepackage{mathrsfs,amssymb}
\usepackage{enumitem}
\usepackage{placeins}
\usepackage{tikz}
\usetikzlibrary{arrows.meta}
\definecolor{cbpInk}{RGB}{37,53,66}
\definecolor{cbpFluid}{RGB}{231,240,242}
\definecolor{cbpCore}{RGB}{151,61,45}
\definecolor{cbpFlow}{RGB}{35,102,117}
\tikzset{cbpwall/.style={draw=cbpInk,line width=.75pt},
cbpinner/.style={draw=cbpCore,line width=.9pt},
cbpleader/.style={draw=cbpInk!75,line width=.5pt},
cbpflow/.style={-{Stealth[length=2.1mm,width=1.4mm]},draw=cbpFlow,line width=.85pt},
cbpmeasure/.style={<->,>={Stealth[length=1.6mm,width=1.1mm]},line width=.55pt},
cbpradius/.style={-{Stealth[length=1.15mm,width=.9mm]},draw=cbpInk,line width=.6pt},
cbpfigure/.style={x=1cm,y=1cm,font=\small,line cap=round,
line join=round,line width=.55pt,>={Stealth[length=2.1mm,width=1.4mm]}}}
\usepackage[linkcolor=blue,citecolor=cyan,urlcolor=blue,colorlinks,hypertexnames=false]{hyperref}
\usepackage[T1]{fontenc}

\newtheorem{theorem}{Theorem}[section]
\newtheorem{lemma}[theorem]{Lemma}
\newtheorem{proposition}[theorem]{Proposition}

\theoremstyle{definition}
\newtheorem{definition}[theorem]{Definition}
\theoremstyle{remark}
\newtheorem{remark}[theorem]{Remark}

\newcommand{\R}{\mathbb R}
\newcommand{\T}{\mathbb T}

\numberwithin{equation}{section}
\allowdisplaybreaks
\begin{document}
\title[Cavity collapse and kinetic energy concentration in free-boundary Euler flow]{Cavity collapse and kinetic energy concentration in three-dimensional free-boundary Euler flow}
\author{Tao Luo and Siqi Yang}  
\address{Tao Luo\newline
Department of Mathematics, City University of Hong Kong, Hong Kong, China}
\email{taoluo@cityu.edu.hk}
\address{Siqi Yang\newline School of Mathematics, Statistics and Mechanics, Beijing University of Technology, Beijing 100124, China}
\email{siqiyang@bjut.edu.cn}

\begin{abstract}
We construct smooth axisymmetric solutions of the three-dimensional
incompressible free-boundary Euler equations in bounded domains for
which a tubular cavity collapses onto a circle at a finite time $T$.
The velocity has a single-valued potential, and the construction applies
without surface tension and with any fixed positive surface tension.
Both boundary components remain embedded and disjoint before collapse,
while the outer boundary converges smoothly. The terminal outer
meridional domain may be any smooth bounded simply connected domain
whose closure is separated from the rotation axis; the interior point
generating the collapse circle and the positive concentrated energy
may be prescribed independently. We prove a sharp collapse law with
radius asymptotic to an explicit multiple of
$(T-t)^{1/2}|\log(T-t)|^{-1/4}$, determine curvature blow-up asymptotics,
and show that kinetic energy concentrates uniformly on the circle.
For a canonical reflection symmetric family, we identify a common
cylindrical spacetime limit and sharp velocity and pressure profiles.
At every fixed spatial point in the region enclosed by the limiting
outer boundary and away from the collapse circle, the velocity tends
to zero while the pressure tends to positive infinity. For this family,
we characterize the topology of the limiting material map and prove
its optimal $1/2$-H\"older regularity and sharp $W^{1,p}$ threshold
$1\le p<2$, with first derivatives in weak $L^2$ but not in $L^2$.
These canonical solutions admit weak continuations that preserve mass
and momentum and whose total energy drops at collapse by exactly the
concentrated kinetic energy.
\end{abstract}

\maketitle
\enlargethispage{2pt}
\setcounter{tocdepth}{2}
\tableofcontents

\section{Introduction}

The geometry of a moving fluid domain can become singular even when
the velocity remains irrotational. We study a collapse in which an
entire free-boundary component contracts onto a curve: a tubular cavity
in a bounded three-dimensional ideal fluid shrinks to a circle in finite
time, while the outer free surface converges smoothly. The two boundary
components remain embedded and disjoint at every time before collapse.
The velocity has a single-valued potential throughout the evolution,
so both vorticity and circulation vanish.

Our main result constructs such flows with a prescribed terminal outer
geometry and a prescribed positive amount of kinetic energy concentrating
on the circle. The collapse rate is determined by the logarithmic
capacity of a thin tube and depends, to leading order, only on the
concentrated energy and the length of the circle. The construction works
both without surface tension and for every fixed positive surface
tension.

The logarithmically corrected scaling already appears in the
asymptotic potential-flow analysis of \cite{Gordillo2005}.
Here we realize it in the full three-dimensional Euler evolution of
two free boundaries and determine its leading constant from the
concentrated energy and the length of the collapse circle.
Related hollow-vortex implosions \cite{ChenWalshWheeler2026} occur in
planar exterior domains with nonzero circulation and a gas-pressure
law. Our bounded, globally potential flows allow the terminal outer
geometry and the concentrated energy to be prescribed independently
within the class specified below. For a canonical reflection symmetric
family, we further connect the concentration of energy to the
degeneration of the material map and to the exact energy loss in a
weak continuation on the filled domain.

For a fixed surface tension $\sigma\ge0$, we consider the three-dimensional incompressible
free-boundary Euler system
\begin{equation}\label{eq:euler-toroidal-front}
\begin{cases}
\partial_t v+(v\cdot\nabla)v+\nabla p=0,& x\in\Omega_t,\\
\nabla\cdot v=0,& x\in\Omega_t,\\
p=\sigma\mathcal H,& x\in\partial\Omega_t,\\
\mathcal V_{\partial\Omega_t}=v\cdot n,& x\in\partial\Omega_t,
\end{cases}
\end{equation}
Here $\Omega_t$ is the fluid domain, $v$ its velocity, $p$ its pressure,
and $\mathcal V_{\partial\Omega_t}$ the normal boundary velocity.
The unit normal $n$ points out of the fluid, and the mean curvature
$\mathcal H=\operatorname{div}_{\partial\Omega_t}n$.
The Taylor coefficient is $-\partial_n p$.
We seek solutions of the form $v=\nabla\Phi$ with a single-valued
potential $\Phi$.
We use cylindrical coordinates $(\rho,\phi,z)$, where
$(x_1,x_2,x_3)=(\rho\cos\phi,\rho\sin\phi,z)$ and
$\phi\in\T:=\R/(2\pi\mathbb Z)$.
The fluid occupies a bounded axisymmetric domain with two free-boundary components
\begin{equation*}
\partial\Omega_t=\Gamma_-(t)\cup\Gamma_+(t).
\end{equation*}
The inner component $\Gamma_-(t)$ bounds the tubular cavity, and
$\Gamma_+(t)$ is the outer free surface. We prescribe the limiting
outer meridional domain among smooth bounded simply connected domains
whose closures are separated from the rotation axis, and choose the
point generating the collapse circle in its interior.
The outer boundary evolves freely and converges smoothly to the
prescribed terminal shape.

\subsection{Main results}

Theorem~\ref{thm:prescribed-outer-geometry} gives the construction for
prescribed terminal geometry, together with the collapse law, curvature
asymptotics and energy concentration. The subsequent results concern
the canonical reflection symmetric trajectories of
Definition~\ref{def:canonical-terminal-objects}: they determine the
singular profiles, the terminal material map and the weak filling
continuations.

\subsubsection{Collapse with prescribed outer geometry}

Unless otherwise indicated, Sobolev norms of angular functions are taken on $\T$.
The notation $\mathcal H^d$ denotes $d$-dimensional Hausdorff measure,
and $\mathbf1_A$ denotes the indicator of a set $A$.
We write
$|\Sigma|=\mathcal H^2(\Sigma)$ for the area of a surface $\Sigma$.
For a smooth closed surface $\Sigma$, we use the spectral norm
\begin{equation}\label{eq:spectral-H32-definition}
\|f\|_{H^{3/2}_{\rm sp}(\Sigma)}
:=\|(1-\Delta_\Sigma)^{3/4}f\|_{L^2(\Sigma)},
\end{equation}
where $\Delta_\Sigma$ is the Laplace--Beltrami operator.
The notation $A\lesssim B$ means $A\le CB$ for a constant $C>0$. 
We write $A\asymp B$ when $A\lesssim B$ and $B\lesssim A$, and
$A\sim B$ when $A/B\to1$.

Extend the velocity by zero:
\begin{equation*}
\widetilde v(t,x)=\mathbf 1_{\Omega_t}(x)v(t,x),\qquad x\in\R^3.
\end{equation*}
%\FloatBarrier
\begin{theorem}\label{thm:prescribed-outer-geometry}
Let $D_*\Subset\{(\rho,z):\rho>0\}$ be a bounded smooth simply connected
domain, let $c_*=(R_*,Z_*)\in D_*$, and let $\Omega_*$ be its solid of
revolution about the $z$-axis.  Set
\begin{equation*}
\gamma_*=\{(R_*\cos\phi,R_*\sin\phi,Z_*):\phi\in\T\},
\qquad L_*=2\pi R_*.
\end{equation*}
For every fixed $\sigma\ge0$ and $E_c>0$, there is a smooth
axisymmetric globally potential solution of
\eqref{eq:euler-toroidal-front} on a finite interval $[0,T)$.
Its two free-boundary components remain embedded and disjoint for $t<T$.
The outer boundary converges smoothly to $\partial\Omega_*$, while
the inner boundary collapses to $\gamma_*$.

Let $a(t)$ be the angular mean radius of the inner meridian. 
For a sufficiently small fixed local reference length $\ell>0$, the scale satisfies
\begin{equation}\label{eq:final-collapse-law}
T-t
=\frac12\sqrt{\frac{\pi L_*}{E_c}}\,
a(t)^2\sqrt{\log\frac{\ell}{a(t)}}\,(1+o(1)),
\end{equation}
and hence
\begin{equation*}
a(t)=2^{3/4}\left(\frac{E_c}{\pi L_*}\right)^{1/4}
(T-t)^{1/2}|\log(T-t)|^{-1/4}(1+o(1))
\quad\text{as }t\uparrow T.
\end{equation*}

The mean curvature has the asymptotics
\begin{equation}\label{eq:L2-curvature-final}
\|\mathcal H(t)\|_{L^2(\Gamma_-(t))}
\sim(2\pi L_*)^{1/2}a(t)^{-1/2}\longrightarrow\infty,
\end{equation}
and 
\begin{equation}\label{eq:H32-curvature-final}
\|\mathcal H(t)\|_{H^{3/2}_{\rm sp}(\Gamma_-(t))}
\sim\sqrt{2\pi^2/R_*}\,a(t)^{-1}
\asymp (T-t)^{-1/2}|\log(T-t)|^{1/4}
\quad\text{as }t\uparrow T.
\end{equation}

The kinetic energy measures satisfy
\begin{equation}\label{eq:main-energy-concentration}
\frac12|\widetilde v(t,x)|^2\,dx
\overset{*}{\rightharpoonup}
\frac{E_c}{L_*}\,
\mathcal H^1\!\restriction_{\gamma_*}
\quad\text{as }t\uparrow T,
\end{equation}
in the space of finite Radon measures on $\R^3$.
\end{theorem}
\begin{remark}
Since $D_*$ is simply connected and separated from the rotation axis,
$\Omega_*$ is diffeomorphic to a solid torus.  The terminal outer section,
the interior point generating the collapse circle, and the concentrated
energy can be prescribed independently within the stated class.  The leading coefficient in the collapse law depends only on $E_c$ and $L_*$. At zero surface tension,
the pressure satisfies
\begin{equation*}
-\Delta p=\sum_{i,j=1}^3(\partial_i\partial_j\Phi)^2
\end{equation*} 
with zero
boundary values.  Since the velocity is nonconstant, Hopf's lemma gives $\partial_n p<0$ on both components
at every time before collapse.
\end{remark}
\begin{remark}
Section~\ref{sec:prescribed-outer-geometry} gives the centred
parametrizations and the quantitative geometric estimates
\eqref{eq:prescribed-geometry-tails} for this solution.
\end{remark}

Figures~\ref{fig:prescribed-geometry} and~\ref{fig:prescribed-evolution}
illustrate the prescribed terminal geometry and its curvature and energy
limits, and the boundary evolution, respectively.
In Figure~\ref{fig:prescribed-evolution}, $c(t)$ is the centre of the inner
meridian in the centred representation used to define $a(t)$ in
Theorem~\ref{thm:prescribed-outer-geometry}, and $c(t)\to c_*$ as
$t\uparrow T$.  The centred representation is specified in
Section~\ref{sec:prescribed-outer-geometry}.

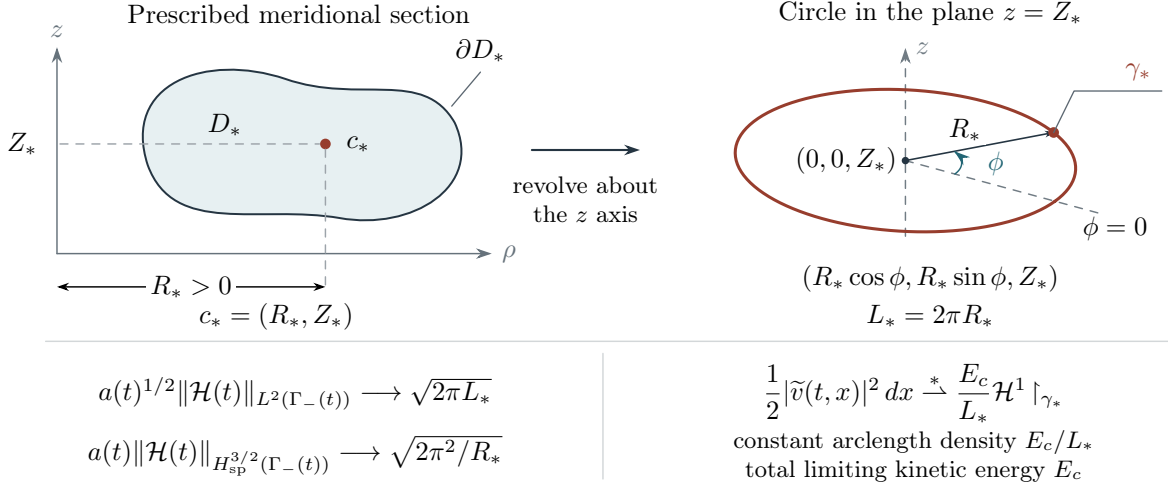
\begin{figure}[!ht]
\centering
\begin{tikzpicture}[cbpfigure]
\path[use as bounding box] (0,0) rectangle (15.7,6.65);
\begin{scope}[yshift=2.20cm]

\node at (3.75,4.12) {Prescribed meridional section};
\draw[->,cbpInk!70] (.55,.95)--(.55,3.68) node[above] {$z$};
\draw[->,cbpInk!70] (.55,.95)--(6.30,.95) node[right] {$\rho$};
\path[fill=cbpFluid,cbpwall] (1.75,2.12)
.. controls (1.40,3.20) and (2.52,3.61) .. (3.57,3.26)
.. controls (4.55,2.92) and (5.39,3.41) .. (5.81,2.68)
.. controls (6.20,1.83) and (5.24,1.25) .. (4.29,1.42)
.. controls (3.23,1.60) and (2.10,1.17) .. (1.75,2.12)--cycle;
\node at (2.77,2.65) {$D_*$};
\fill[cbpCore] (4.10,2.40) circle (2pt);
\node[anchor=west] at (4.24,2.40) {$c_*$};
\node at (3.45,.12) {$c_*=(R_*,Z_*)$};
\draw[dashed,cbpInk!45] (4.10,2.33)--(4.10,.51);
\draw[dashed,cbpInk!45] (.59,2.40)--(4.02,2.40);
\node[anchor=east] at (.43,2.40) {$Z_*$};
\draw[cbpmeasure] (.55,.51)--(4.10,.51)
node[midway,fill=white,inner sep=1.5pt] {$R_*>0$};
\draw[cbpleader] (5.77,2.90)--(6.15,3.35);
\node[anchor=south] at (6.15,3.35) {$\partial D_*$};
\draw[->,cbpInk,line width=.8pt] (6.83,2.32)--(8.26,2.32);
\node[align=center,font=\footnotesize] at (7.54,1.67) {revolve about\\the $z$ axis};
\node at (12.10,4.12) {Circle in the plane $z=Z_*$};
\begin{scope}[shift={(11.77,2.18)},x={(.92cm,-.25cm)},y={(.50cm,.36cm)},z={(0cm,1cm)}]
\draw[dashed,->,cbpInk!65] (0,0,-1.00)--(0,0,1.49) node[right] {$z$};
\draw[cbpCore,line width=1.15pt,domain=0:360,samples=120,variable=\ang]
plot ({2.15*cos(\ang)},{2.15*sin(\ang)},0);
\draw[dashed,cbpInk!65] (0,0,0)--(2.85,0,0);
\node[anchor=north west] at (2.40,0,0) {$\phi=0$};
\draw[cbpradius] (0,0,0)--({2.15*cos(58)},{2.15*sin(58)},0);
\node[fill=white,inner sep=1.5pt,anchor=south east] at ({1.17*cos(58)},{1.17*sin(58)},0) {$R_*$};
\draw[->,cbpFlow,domain=0:58,samples=35,variable=\ang]
plot ({.71*cos(\ang)},{.71*sin(\ang)},0);
\node[cbpFlow,anchor=west] at ({.91*cos(29)},{.91*sin(29)},0) {$\phi$};
\fill[cbpCore] ({2.15*cos(58)},{2.15*sin(58)},0) circle (2pt);
\fill[cbpInk] (0,0,0) circle (1.4pt);
\node[anchor=east,inner sep=3pt] at (0,0,0) {$(0,0,Z_*)$};
\end{scope}
\draw[cbpleader] (13.73,2.55)--(14.00,3.10)--(15.16,3.10);
\node[anchor=south east,cbpCore] at (15.16,3.10) {$\gamma_*$};
\node at (12.10,.62) {$(R_*\cos\phi,R_*\sin\phi,Z_*)$};
\node at (12.10,.12) {$L_*=2\pi R_*$};
\end{scope}
\draw[cbpInk!20] (.40,1.99)--(15.30,1.99);
\draw[cbpInk!20] (7.77,.17)--(7.77,1.83);
\node at (3.70,1.30)
{$a(t)^{1/2}\|\mathcal H(t)\|_{L^2(\Gamma_-(t))}
\longrightarrow\sqrt{2\pi L_*}$};
\node at (3.70,.48)
{$a(t)\|\mathcal H(t)\|_{H^{3/2}_{\rm sp}(\Gamma_-(t))}
\longrightarrow\sqrt{2\pi^2/R_*}$};
\node at (11.88,1.30)
{$\displaystyle\frac12|\widetilde v(t,x)|^2\,dx
\overset{*}{\rightharpoonup}
\frac{E_c}{L_*}\mathcal H^1\!\restriction_{\gamma_*}$};
\node[align=center,font=\footnotesize] at (11.88,.48)
{constant arclength density $E_c/L_*$\\total limiting kinetic energy $E_c$};
\end{tikzpicture}
\caption{Prescribed terminal geometry and singular limits in
Theorem~\ref{thm:prescribed-outer-geometry}.  Revolving $D_*$ and its
interior point $c_*=(R_*,Z_*)$ about the $z$ axis gives $\Omega_*$ and
$\gamma_*$, respectively.  The azimuthal angle $\phi$ is measured from
the positive $x_1$ direction in the plane $z=Z_*$.  The two curvature
norms diverge with the displayed leading constants, and the limiting
kinetic energy has constant arclength density $E_c/L_*$ on $\gamma_*$.}
\label{fig:prescribed-geometry}
\end{figure}

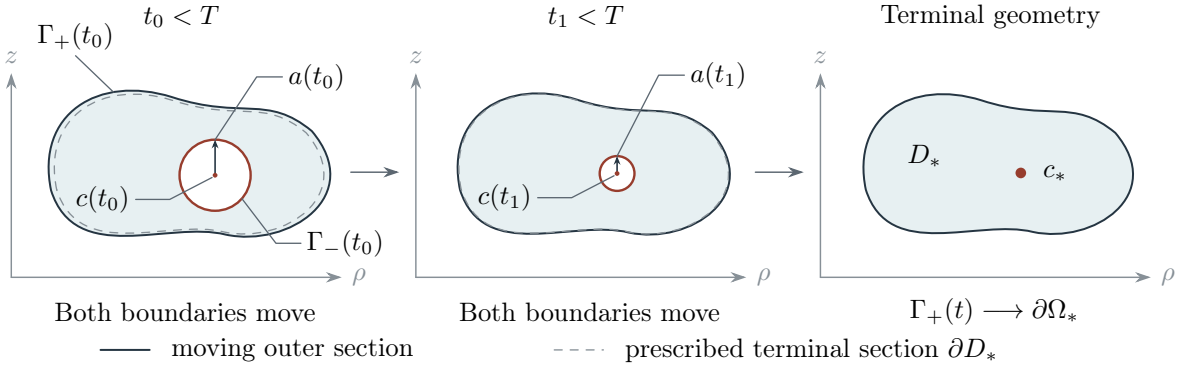
\begin{figure}[!ht]
\centering
\begin{tikzpicture}[cbpfigure]
\path[use as bounding box] (0,0) rectangle (15.7,4.95);
\def\cbpTerminalSection{(.85,2.28)
.. controls (.65,3.25) and (1.45,3.75) .. (2.38,3.46)
.. controls (3.28,3.17) and (3.58,3.50) .. (4.15,3.03)
.. controls (4.84,2.21) and (3.86,1.55) .. (3.06,1.72)
.. controls (2.34,1.90) and (1.16,1.31) .. (.85,2.28)--cycle}
\foreach \off/\ss/\aa/\dx/\dz/\tt in
{0/1.045/.47/.045/-.03/0,5.35/1.010/.23/.011/-.007/1}{
\begin{scope}[shift={(\off,0)}]
\node at (2.50,4.57) {$t_\tt<T$};
\draw[->,cbpInk!55] (.25,1.12)--(.25,3.82) node[above] {$z$};
\draw[->,cbpInk!55] (.25,1.12)--(4.62,1.12) node[right] {$\rho$};
\begin{scope}[shift={(2.50,2.50)},scale=\ss,shift={(-2.50,-2.50)}]
\path[fill=cbpFluid,cbpwall] \cbpTerminalSection;
\end{scope}
\draw[dashed,cbpInk!55,line width=.5pt] \cbpTerminalSection;
\begin{scope}[shift={({2.90+\dx},{2.50+\dz})}]
\filldraw[fill=white,cbpinner] (0,0) circle (\aa);
\draw[cbpradius] (0,0)--(0,\aa);
\fill[cbpCore] (0,0) circle (.9pt);
\draw[cbpleader] (0,\aa)--(.55,1.30)--(.89,1.30);
\node[anchor=west,inner sep=2pt] at (.89,1.30) {$a(t_\tt)$};
\draw[cbpleader] (-.04,-.03)--(-.60,-.32)--(-1.05,-.32);
\node[anchor=east,inner sep=2pt] at (-1.05,-.32) {$c(t_\tt)$};
\end{scope}
\node at (2.55,.67) {Both boundaries move};
\end{scope}}
\begin{scope}[shift={(10.70,0)}]
\node at (2.50,4.57) {Terminal geometry};
\draw[->,cbpInk!55] (.25,1.12)--(.25,3.82) node[above] {$z$};
\draw[->,cbpInk!55] (.25,1.12)--(4.62,1.12) node[right] {$\rho$};
\path[fill=cbpFluid,cbpwall] \cbpTerminalSection;
\node at (1.63,2.72) {$D_*$};
\fill[cbpCore] (2.90,2.50) circle (2pt);
\node[anchor=west] at (3.06,2.50) {$c_*$};
\node at (2.55,.67) {$\Gamma_+(t)\longrightarrow\partial\Omega_*$};
\end{scope}
\draw[cbpleader] (3.30,2.16)--(3.75,1.57)--(4.08,1.57);
\node[anchor=west,inner sep=2pt] at (4.08,1.57) {$\Gamma_-(t_0)$};
\draw[cbpleader] (1.60,3.57)--(1.10,4.05);
\node[anchor=south,inner sep=2pt] at (1.10,4.05) {$\Gamma_+(t_0)$};
\draw[->,cbpInk!75] (4.74,2.51)--(5.38,2.51);
\draw[->,cbpInk!75] (10.09,2.51)--(10.73,2.51);
\draw[cbpwall] (1.43,.16)--(2.08,.16);
\node[anchor=west] at (2.23,.16) {moving outer section};
\draw[dashed,cbpInk!55,line width=.5pt] (7.43,.16)--(8.08,.16);
\node[anchor=west] at (8.23,.16) {prescribed terminal section $\partial D_*$};
\end{tikzpicture}
\caption{Schematic meridional sections at $0\le t_0<t_1<T$ and as
$t\uparrow T$.  The labels $\Gamma_\pm(t_0)$ indicate the sections of
the two free boundaries.  Both boundaries move: the outer contours
converge smoothly to the fixed dashed curve $\partial D_*$, while the
cavity shrinks about $c(t)\to c_*$ and the boundaries remain disjoint.
The circular inner contours show the leading profiles.  Revolution
about the $z$-axis turns the small cavity into a tubular cavity and
$c_*$ into the nondegenerate circle $\gamma_*$.}
\label{fig:prescribed-evolution}
\end{figure}
%\FloatBarrier
\subsubsection{Canonical trajectories}

We now specialize to trajectories obtained by evolving concentric
circular terminal meridians backwards and passing to a subsequential
limit. This procedure selects the reflection symmetric family for
which we prove the finer asymptotic and continuation results below.
The cylindrical limit is common to this family and is independent of
the point and sequence used in its rescaling.

\begin{definition}
\label{def:canonical-terminal-objects}
\leavevmode
\begin{enumerate}[label=\textup{(\roman*)},leftmargin=*,itemsep=6pt]
\item \textbf{Terminal geometry and filled domain.}
Fix $\sigma\ge0$, $E_c>0$ and $R_*>0$, and choose $\ell>0$ with
$\ell/R_*$ sufficiently small. In the meridional $(\rho,z)$ half-plane, set
\begin{equation*}
c_*=(R_*,0),\qquad
D_*=\{(\rho,z):(\rho-R_*)^2+z^2<\ell^2\}.
\end{equation*}
In this family, $\ell$ is also the radius of the limiting outer meridian.
Revolving the point $c_*$ and the disk $D_*$ about the $z$-axis gives,
respectively, the \emph{collapse circle} and the \emph{filled domain}:
\begin{equation*} 
\gamma_*=\{(R_*\cos\phi,R_*\sin\phi,0):\phi\in\T\},\qquad
\Omega_*=\{x\in\R^3:(\sqrt{x_1^2+x_2^2}-R_*)^2+x_3^2<\ell^2\}. 
\end{equation*}
Thus $\Omega_*$ is an open solid torus, $\gamma_*\subset\Omega_*$,
and $\Gamma_+^*=\partial\Omega_*$ is the limiting outer boundary.
For a set $A\subset\R^3$ and $b>0$, write
$N_b(A)=\{x\in\R^3:\operatorname{dist}(x,A)<b\}$.
Thus $\Omega_*=N_\ell(\gamma_*)$.

\item \textbf{Canonical trajectories.}
Choose $0<a_N<\ell$ with $a_N\downarrow0$, and let $u_N$ be harmonic in
\begin{equation*}
\Omega_N^*=N_\ell(\gamma_*)\setminus
\overline{N_{a_N}(\gamma_*)},
\end{equation*}
with boundary values $1$ on $\partial N_{a_N}(\gamma_*)$ and $0$
on $\partial N_\ell(\gamma_*)$.  Prescribe the terminal potential
$\Phi_N^*=c_Nu_N$ and velocity $v_N^*=c_N\nabla u_N$, where,
for sufficiently large $N$, the positive constant $c_N$ satisfies
\begin{equation*}
\frac12c_N^2\int_{\Omega_N^*}|\nabla u_N|^2\,dx
+2\pi\sigma L_*a_N=E_c,\qquad L_*=2\pi R_*.
\end{equation*}
Evolve these states backwards by the Euler system and compare the
solutions in the common variable $Y=\log(\ell/a)$, where $a$ is
the mean inner meridional radius.  A \emph{canonical trajectory}
is a collapsing solution obtained by taking a diagonal subsequential
limit on compact $Y$ intervals as $a_N\downarrow0$, as constructed in
Sections~\ref{sec:terminal-existence} and~\ref{sec:zero-surface-tension}. 

\item \textbf{Terminal material map and inner meridians.}
For a canonical trajectory, let $\mathcal X_t$ denote the material
flow on the initial closed fluid domain:
\begin{equation*}
\partial_t\mathcal X_t(\xi)=v(t,\mathcal X_t(\xi)),\qquad
\mathcal X_0(\xi)=\xi,\qquad \xi\in\overline{\Omega_0}.
\end{equation*}
The \emph{terminal material map} is
\begin{equation*}
\mathcal X_T:\overline{\Omega_0}\longrightarrow\overline{\Omega_*},
\qquad \mathcal X_T(\xi)=\lim_{t\uparrow T}\mathcal X_t(\xi).
\end{equation*}
We shall prove that this limit exists uniformly in $\xi$
(see Theorem~\ref{thm:terminal-material-deformation}).
For $\phi\in\T$, define the initial inner boundary meridian and its
corresponding point on the collapse circle by
\begin{equation*}
M_\phi=\{\xi\in\Gamma_-(0):\text{the cylindrical azimuth of }\xi
\text{ is }\phi\},\qquad
x_\phi=(R_*\cos\phi,R_*\sin\phi,0).
\end{equation*}
Thus $M_\phi$ is the entire closed cross-sectional loop at fixed
azimuth. The meridional angle varies along this loop.
\end{enumerate}
\end{definition}

For canonical trajectories, the evolving domains are invariant under
$(\rho,\phi,z)\mapsto(\rho,\phi,-z)$, and the scalar potentials satisfy
$\Phi(t,\rho,-z)=\Phi(t,\rho,z)$.  Their velocities have zero azimuthal
component, so each fluid particle retains its initial cylindrical azimuth.

For a matrix $M$, we use the Euclidean operator norm
$|M|=\|M\|=\sup_{|e|=1}|Me|$.  The $L^p$ and weak $L^2$ norms of
matrix-valued fields use this pointwise norm.  We write $L^{2,\infty}(U)$
for the weak $L^2$ space, with quasi-norm
\begin{equation*}
\|F\|_{L^{2,\infty}(U)}
=\sup_{\lambda>0}\lambda
\bigl|\{x\in U:|F(x)|>\lambda\}\bigr|^{1/2},
\end{equation*}
where the set is measured by Lebesgue volume.
For the $L^\infty$ norm,
\begin{equation*}
\|M\|_{L^\infty(U)}
=\operatorname*{ess\,sup}_{x\in U}\|M(x)\|.
\end{equation*}
The symbols $\delta_y$ and $\delta_T$ denote the unit Dirac measures
at the spatial point $y$ and time $T$, respectively.  
For the energy balance, set
\begin{equation*}
\chi_t=\mathbf1_{\Omega_t},\qquad
A_t=\mathcal H^2\!\restriction_{\partial\Omega_t},
\end{equation*}
and write $n_t$ for the outward fluid normal at time $t$.
Products involving $\chi_t$ are extended by zero outside the fluid.
Define the spatial energy and flux measures by
\begin{equation}\label{eq:weak-capillary-energy-flux} 
\mathcal E_t=\frac12\chi_t|v(t)|^2\,dx+\sigma A_t,\qquad
\mathcal F_t=\chi_t\left(\frac12|v(t)|^2+p(t)\right)v(t)\,dx
+\sigma(v(t)\cdot n_t)n_t A_t. 
\end{equation}
We write $\mathcal E=\mathcal E_t\,dt$ and
$\mathcal F=\mathcal F_t\,dt$ for the corresponding spacetime measures.

\enlargethispage{\baselineskip}
\begin{theorem}\label{thm:toroidal-curvature-singularity}
Fix $\sigma>0$.  There exist constants $R_*,\ell,E_c>0$, with $\ell/R_*$ sufficiently
small, and smooth initial data consisting of a bounded fluid domain with one
tubular cavity around a circle and an axisymmetric, reflection symmetric, globally potential,
divergence-free velocity, for which \eqref{eq:euler-toroidal-front} has a
canonical trajectory in the sense of
Definition~\ref{def:canonical-terminal-objects} on $[0,T)$ for some
$T<\infty$. Set
\begin{equation*}
L_*=2\pi R_*,\qquad
\gamma_*=\{(R_*\cos\phi,R_*\sin\phi,0):\phi\in\T\}.
\end{equation*} The solution is smooth on every compact subinterval of
$[0,T)$ and satisfies the properties below, where all limits are taken
as $t\uparrow T$. 
\begin{enumerate}[label=\textup{(\roman*)},ref=\textup{(\roman*)},leftmargin=*]
\item\label{item:main-geometry} \textbf{Geometry and energy concentration.}
\begin{enumerate}[label=\textup{(\alph*)},leftmargin=*]
\item For every $t<T$, both free-boundary components are smooth, embedded, and
disjoint.

\item The inner component is a centred relative graph
\begin{equation}\label{eq:asymptotically-tubular-final}
r=a(t)(1+\eta(t,\theta)),\qquad
\|\eta(t)\|_{H^2}=o(1),\qquad
\|\eta(t)\|_{H^{7/2}}=o(a(t)),
\end{equation}
over a tube about an instantaneous centre circle $\gamma(t)$ of radius $R(t)$,
where the meridional centre is $c(t)=(R(t),0)$, $r$ is the distance
from $\gamma(t)$, and $\theta$ is measured about $c(t)$ from the
direction of increasing $\rho$.  The centre is chosen so that the first
angular harmonics of $r$ vanish, and $a(t)$ is its angular mean.
One has $a(t)\downarrow0$, $R(t)\to R_*>0$, and $\gamma(t)\to\gamma_*$.

\item With the fixed reference length $\ell$, $a(t)$ obeys
\eqref{eq:final-collapse-law} and its equivalent radius asymptotic.

\item The mean curvature satisfies
\eqref{eq:L2-curvature-final}--\eqref{eq:H32-curvature-final}.

\item The kinetic energy measure converges as in
\eqref{eq:main-energy-concentration}.
Here $\Gamma_+(t)$ converges smoothly to the terminal outer surface
$\Gamma_+^*$ and $E_c$ is
the limiting kinetic energy. 

\item Fix a sufficiently small tubular radius $r_0>0$.  For each fixed
$0<\lambda<1$, put $r_\lambda(t)=r_0^{1-\lambda}a(t)^\lambda$.  Then
\begin{equation}\label{eq:main-logarithmic-energy-law}
\frac12\int_{\Omega_t\cap N_{r_\lambda(t)}(\gamma(t))}|v(t,x)|^2\,dx
\longrightarrow(1-\lambda)E_c.
\end{equation}
For every fixed $C>1$, the same integral with $r_\lambda(t)$ replaced by
$Ca(t)$ tends to zero.
\end{enumerate}

\item\label{item:main-profiles} \textbf{Velocity and pressure profiles.}
\begin{enumerate}[label=\textup{(\alph*)},leftmargin=*]
\item As $t\uparrow T$,
\begin{equation}\label{eq:main-velocity-limits}
\begin{gathered}
\widetilde v(t)\longrightarrow0\quad\hbox{strongly in }L^p(\R^3),
\quad 1\le p<2,\\
\widetilde v(t)\rightharpoonup0\quad\hbox{weakly in }L^2(\R^3),
\qquad \|\widetilde v(t)\|_{L^2}^2\longrightarrow2E_c>0.
\end{gathered}
\end{equation}
In particular there is no continuation whose Eulerian velocity, extended by zero, is strongly continuous in $L^2(\R^3)$ through $T$.

\item Let $n_-=n|_{\Gamma_-(t)}$ and define the normalized inner flux by
\begin{equation*}
q(t)=\frac{1}{4\pi^2R(t)}
\int_{\Gamma_-(t)}v(t)\cdot n_-\,dS.
\end{equation*}
The inner normal velocity satisfies
\begin{equation}\label{eq:hly-normal-velocity-divergence}
\|v(t)\cdot n_-\|_{L^2(\Gamma_-(t))}
\sim (2\pi L_*)^{1/2}q(t)\,a(t)^{-1/2}\longrightarrow\infty.
\end{equation}
The velocity gradient satisfies
\begin{equation}\label{eq:main-gradient-blowup-rate}
\lim_{t\uparrow T}(T-t)
\|\nabla v(t)\|_{L^\infty(\Omega_t)}=\frac12.
\end{equation}

\item Let $U_*$ be the Dirichlet potential determined by
$-\Delta U_*=2\pi\mathcal H^1\!\restriction_{\gamma_*}$ in $\Omega_*$
in the sense of distributions, with zero boundary values.  Then $U_*>0$ in
$\Omega_*\setminus\gamma_*$.  For every compact
$K\Subset\Omega_*\setminus\gamma_*$ and fixed integer $k\ge0$,
$v(t)\to0$ in $C^k(K)$ and the pressure satisfies
\begin{equation}\label{eq:main-pressure-time-profile}
(T-t)|\log(T-t)|^{3/2}p(t)
\longrightarrow\sqrt{\frac{E_c}{2\pi L_*}}\,U_*
\quad\text{in }C^k(K),
\end{equation}
so $p(t,x)\to+\infty$ at every fixed
$x\in\Omega_*\setminus\gamma_*$.
\end{enumerate}

\item \textbf{Continuation through collapse.}
\begin{enumerate}[label=\textup{(\alph*)},leftmargin=*]
\item The solution admits no continuation through $T$ as a strong solution whose boundary remains nondegenerate and continuous in $H^4$.
Consequently $T$ is maximal in every such strong class.

\item There is $\delta>0$ and a continuation to $[0,T+\delta)$ in the
one-phase distributional formulation of
Subsection~\ref{subsec:weak-continuation}.  At $T$ its state is
$(\Omega_*,0)$, and for $t\ge T$ it is the classical capillary Euler
solution with this initial state.  With the energy and flux measures
defined in \eqref{eq:weak-capillary-energy-flux}, its exact local
total energy balance is
\begin{equation}\label{eq:main-weak-energy-defect}
\partial_t\mathcal E+\operatorname{div}_x\mathcal F
=-\delta_T\otimes\left(
\frac{E_c}{L_*}
\mathcal H^1\!\restriction_{\gamma_*}\right).
\end{equation}
In particular, the total energy decreases by exactly $E_c$
at $T$ and equals $\sigma|\Gamma_+^*|$ thereafter.
\end{enumerate}

\end{enumerate}
\end{theorem}

Figure~\ref{fig:logarithmic-concentration} illustrates the distribution
of the concentrating energy across logarithmic radial scales.

\begin{figure}[!htbp]
\centering
\begin{tikzpicture}[cbpfigure]
\path[use as bounding box] (0,0) rectangle (9.50,3.30);
\node at (4.75,3.05) {Equal logarithmic radial intervals};
\foreach \i in {0,1,2,3}{
\pgfmathsetmacro{\xx}{.40+2.15*\i}
\filldraw[fill=cbpFluid,draw=white,line width=1pt]
(\xx,1.72) rectangle ++(2.15,.65);
\node at (\xx+1.075,2.045) {$E_c/4$};}
\draw[->,thin] (.34,1.72)--(9.37,1.72);
\foreach \xx/\lab in {.40/{a(t)},2.55/{r_{3/4}(t)},4.70/{r_{1/2}(t)},6.85/{r_{1/4}(t)},9.00/{r_0}}{
\draw (\xx,1.67)--(\xx,1.80);
\node[below=4pt] at (\xx,1.67) {$\lab$};}
\node[cbpInk!80] at (4.70,2.64) {limiting kinetic energy in each band};
\node at (4.70,.70) {$r_\lambda(t)=r_0^{1-\lambda}a(t)^\lambda$};
\node[cbpInk!80] at (4.70,.17) {inside $Ca(t)$: energy $\longrightarrow0$ for fixed $C>1$};
\end{tikzpicture}
\caption{The radial distribution in
\eqref{eq:main-logarithmic-energy-law}.  Equal widths represent equal
increments of $\log(r/a(t))$.  Each of the four bands has limiting
energy $E_c/4$.  The fixed radius $r_0$ captures all limiting kinetic
energy.  The resulting uniform line measure is shown in
Figure~\ref{fig:prescribed-geometry}.}
\label{fig:logarithmic-concentration}
\end{figure}
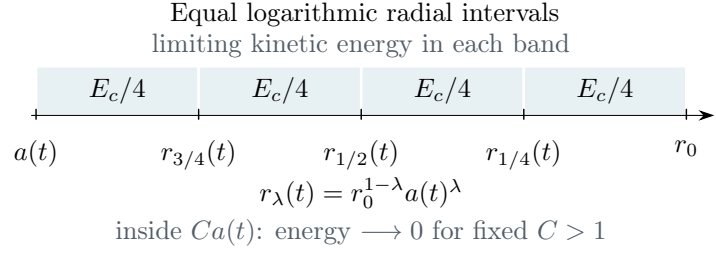
\begin{remark}
For $\sigma>0$, the canonical trajectories of
Theorem~\ref{thm:toroidal-curvature-singularity} evolve in non-simply
connected fluid domains and simultaneously realize alternatives (2)--(4) of
\cite[Theorem~1.1 and Remark~1.2]{HLY}.
The divergent physical $L^2$ norms in
\eqref{eq:L2-curvature-final} and
\eqref{eq:hly-normal-velocity-divergence} force divergence of the
inhomogeneous $H^{3/2}$ curvature and $H^{5/2}$ normal velocity norms,
respectively.  The rate \eqref{eq:main-gradient-blowup-rate} gives
\begin{equation*}
\int_0^T\|\nabla v(t)\|_{L^\infty(\Omega_t)}\,dt=\infty
\end{equation*}
and also determines the tangential gradient divergence.
The outer boundary has a smooth separated limit, whereas the terminal
inner material map has image $\gamma_*$ and ceases to be a
two-dimensional immersion.
\end{remark}
The same leading singularity occurs at zero surface tension.  

\begin{theorem} 
\label{thm:zero-surface-tension-construction}
For $\sigma=0$, there are smooth initial data in the same
axisymmetric, reflection symmetric, globally potential class
for which \eqref{eq:euler-toroidal-front} has a canonical trajectory on
$[0,T)$ satisfying parts~\ref{item:main-geometry}
and~\ref{item:main-profiles} of
Theorem~\ref{thm:toroidal-curvature-singularity} with $\sigma=0$.  
The filling continuation is explicitly
\begin{equation}\label{eq:zero-static-filling}
\Omega_t=\Omega_*,\qquad v(t)=0,\qquad p(t)=0,
\qquad t\ge T.
\end{equation}
Its exact local energy defect is
\eqref{eq:main-weak-energy-defect}, and its total energy is zero for
$t\ge T$ (the energy before collapse is $E_c>0$).  There is no strong free-boundary continuation whose boundary
remains nondegenerate and continuous in $H^4$ through $T$.
\end{theorem}

Figure~\ref{fig:filling-energy} summarizes the two filling continuations
and their energy loss.  

\begin{figure}[!htbp]
\centering
\begin{tikzpicture}[cbpfigure]
\path[use as bounding box] (0,0) rectangle (15.7,3.65);
\begin{scope}[shift={(.45,.53)}]
\draw[->,cbpInk] (0,0)--(7.25,0) node[right] {$t$};
\draw[->,cbpInk] (0,0)--(0,2.75) node[above] {$E(t)$};
\draw[cbpInk!40,dashed] (3.60,0)--(3.60,2.59);
\node[below] at (3.60,0) {$T$};
\draw[cbpFlow,very thick] (.10,2.15)--(3.60,2.15);
\filldraw[fill=white,draw=cbpFlow] (3.60,2.15) circle (2pt);
\draw[cbpFlow,very thick] (3.60,.60)--(6.93,.60);
\fill[cbpFlow] (3.60,.60) circle (2pt);
\node[above=4pt] at (1.83,2.15) {$E_c+\sigma|\Gamma_+^*|$};
\node[above=4pt] at (5.51,.60) {$\sigma|\Gamma_+^*|$};
\draw[->,cbpCore,thick] (4.21,2.04)--(4.21,.80);
\node[anchor=west] at (4.30,1.43) {$E_c$};
\end{scope}
\draw[cbpInk!20] (8.16,.26)--(8.16,3.42);
\node at (11.91,3.25) {Filled state at $T$: $(\Omega_*,0)$};
\node[anchor=west,align=left] at (8.68,2.35)
{$\sigma=0$: stationary fluid\\[4pt]
$\Omega_t=\Omega_*,\quad v=p=0$};
\draw[cbpInk!20] (8.57,1.62)--(15.38,1.62);
\node[anchor=west,align=left] at (8.68,.98)
{$\sigma>0$: classical capillary flow\\[4pt]
issued from the filled domain at rest};
\end{tikzpicture}
\caption{The specified filling continuations.  Here $E(t)$ denotes
the total kinetic plus surface energy.  The exact loss at $T$ is $E_c$,
supported on $\gamma_*$ in the local balance
\eqref{eq:main-weak-energy-defect}.  The lower energy level is zero
when $\sigma=0$.  For $\sigma>0$, the outgoing total energy stays
$\sigma|\Gamma_+^*|$ as the capillary flow evolves.}
\label{fig:filling-energy}
\end{figure}
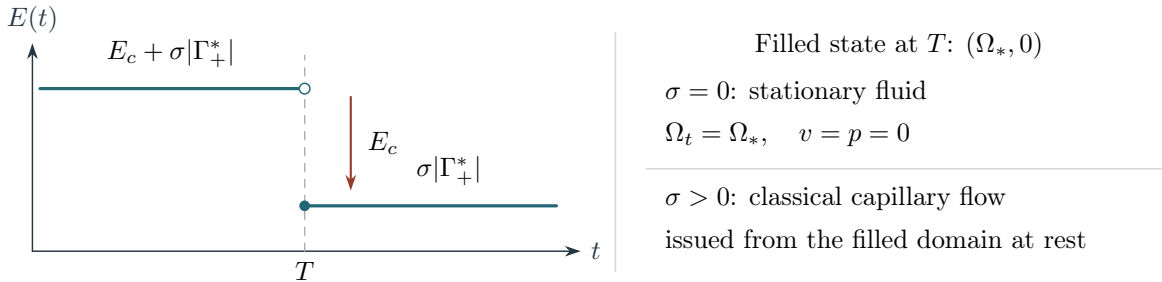

\begin{remark}
Cavity collapse for free-boundary Euler flows with nonzero vorticity
is the subject of a companion work by the authors
\cite{LuoYangVortical}, currently in preparation.
\end{remark}

\subsubsection{Local limits and material deformation}

The singularity has both a local Eulerian profile and a precise
Lagrangian description. Magnifying distances by the cavity radius
and time by the remaining lifetime gives a straight shrinking cylinder
with an explicit exterior flow. Proposition~\ref{prop:universal-cylindrical-tangent}
identifies this limit; Theorem~\ref{thm:terminal-material-deformation}
describes the corresponding degeneration of the material map.
The full coordinate formulation of the cylindrical limit is given in
Subsection~\ref{subsec:inner-pressure-tangent}.

\begin{proposition}
\label{prop:universal-cylindrical-tangent}
Consider a canonical trajectory from
Theorem~\ref{thm:toroidal-curvature-singularity} or
Theorem~\ref{thm:zero-surface-tension-construction}.
Fix any point $x_*\in\gamma_*$ and any sequence $t_j\uparrow T$.
Centre space at $x_*$, rotate the tangent to $\gamma_*$ onto the
third coordinate axis, and divide lengths by $a(t_j)$.
Write $z=(z_1,z_2,z_3)$ for the resulting coordinates and set
\begin{equation*}
\tau=\frac{t-t_j}{T-t_j},\qquad
z_\perp=(z_1,z_2,0),\qquad r=|z_\perp|.
\end{equation*}
Multiply the velocity in this rotated frame by $(T-t_j)/a(t_j)$
and the pressure by $((T-t_j)/a(t_j))^2$.

The rescaled inner boundary converges locally to the shrinking cylinder
$r=\sqrt{1-\tau}$ in every fixed spatial graph norm, uniformly
on compact time intervals in $(-\infty,1)$.
In its exterior, the rescaled velocity and pressure converge to
\begin{equation}\label{eq:universal-cylindrical-tangent}
u_\infty(\tau,z)=-\frac{z_\perp}{2r^2},\qquad
\pi_\infty(\tau,z)
=\frac18\left(\frac1{1-\tau}-\frac1{r^2}\right).
\end{equation}
Every compact subset of 
\begin{equation*}
\{(\tau,z):\tau<1,\ r>\sqrt{1-\tau}\}
\end{equation*}
lies in the rescaled fluid region for all sufficiently large $j$.
On each such compact set the convergence is uniform together with
all spatial derivatives of any fixed order.
The limiting flow is independent of the point and sequence chosen.
It solves the incompressible Euler equations, transports the cylinder
boundary, and has zero boundary pressure.
\end{proposition} 

Figure~\ref{fig:cylindrical-zoom} illustrates this magnification and
the limiting flow.

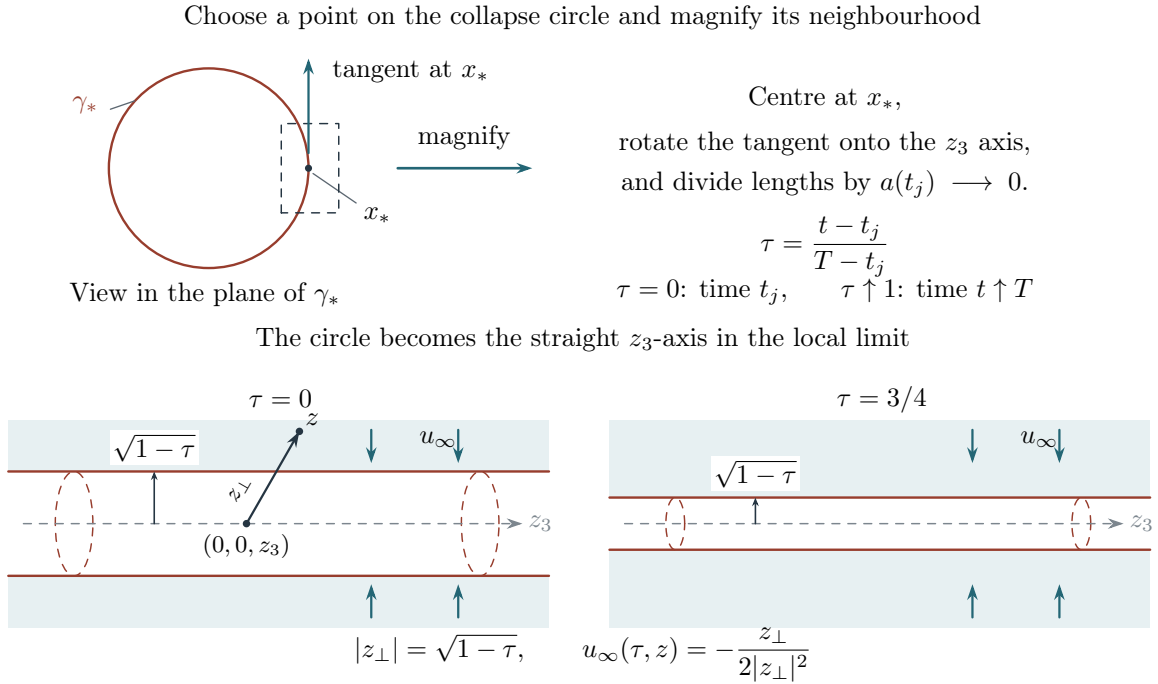
\begin{figure}[!htbp]
\centering
\begin{tikzpicture}[cbpfigure]
\path[use as bounding box] (0,0) rectangle (15.7,9.05);
\node at (7.85,8.80) {Choose a point on the collapse circle and magnify its neighbourhood};

\draw[cbpCore,line width=.9pt] (2.90,6.80) circle (1.32);
\node[anchor=east,text=cbpCore] at (1.58,7.64) {$\gamma_*$};
\draw[cbpleader] (1.60,7.58)--(1.92,7.69);
\fill[cbpInk] (4.22,6.80) circle (.047);
\node[anchor=west,fill=white,inner sep=2pt] at (4.87,6.18) {$x_*$};
\draw[cbpleader] (4.85,6.34)--(4.28,6.73);
\draw[cbpInk,dashed] (3.86,6.21) rectangle (4.62,7.39);
\draw[cbpflow] (4.22,7.00)--(4.22,8.25);
\node[anchor=west] at (4.40,8.08) {tangent at $x_*$};
\node at (2.85,5.13) {View in the plane of $\gamma_*$};

\draw[cbpflow,line width=1pt] (5.40,6.80)--(7.18,6.80);
\node at (6.28,7.21) {magnify};
\node[align=center,text width=6.8cm] at (11.05,7.15)
{Centre at $x_*$,\\[5pt]
rotate the tangent onto the $z_3$ axis,\\[3pt]
and divide lengths by $a(t_j)\longrightarrow0$.};
\node at (11.05,5.76) {$\displaystyle\tau=\frac{t-t_j}{T-t_j}$};
\node[align=center] at (11.05,5.17)
{$\tau=0$: time $t_j$,\qquad $\tau\uparrow1$: time $t\uparrow T$};

\node at (7.85,4.54) {The circle becomes the straight $z_3$-axis in the local limit};

\begin{scope}[yshift=.10cm]

\foreach \xx/\rr/\lab in {.10/.69/0,8.05/.345/{3/4}}{
\begin{scope}[shift={(\xx,0)}]
\node at (3.75,3.65) {$\tau=\lab$};
\fill[cbpFluid] (.15,.63) rectangle (7.30,3.37);
\fill[white] (.15,{2.0-\rr}) rectangle (7.30,{2.0+\rr});
\draw[cbpinner] (.15,{2.0-\rr})--(7.30,{2.0-\rr});
\draw[cbpinner] (.15,{2.0+\rr})--(7.30,{2.0+\rr});
\draw[cbpinner,dashed,line width=.55pt] (1.02,2.0) ellipse ({.36*\rr} and \rr);
\draw[cbpinner,dashed,line width=.55pt] (6.39,2.0) ellipse ({.36*\rr} and \rr);
\draw[->,dashed,cbpInk!60] (.35,2.0)--(6.95,2.0)
node[right=1pt,fill=white,inner sep=1pt] {$z_3$};
\draw[cbpradius,>={Stealth[length=1.35mm,width=1mm]}]
(2.08,2.0)--(2.08,{2.0+\rr});
\node[anchor=south,fill=white,inner sep=1.2pt]
at (2.08,{2.08+\rr}) {$\sqrt{1-\tau}$};
\foreach \xxv in {4.95,6.10}{
\draw[cbpflow] (\xxv,3.24)--(\xxv,2.80);
\draw[cbpflow] (\xxv,.76)--(\xxv,1.20);}
\node[anchor=west] at (5.45,3.12) {$u_\infty$};
\end{scope}}
\fill[cbpInk] (3.40,2.0) circle (1.4pt);
\node[anchor=north,inner sep=2pt,font=\footnotesize] at (3.40,1.94) {$(0,0,z_3)$};
\draw[cbpflow,draw=cbpInk] (3.40,2.0)--(4.10,3.22)
node[pos=.25,sloped,above=2pt,font=\footnotesize] {$z_\perp$};
\fill[cbpInk] (4.10,3.22) circle (1.4pt);
\node[anchor=south west,inner sep=2pt] at (4.10,3.22) {$z$};
\node at (7.85,.27)
{$|z_\perp|=\sqrt{1-\tau},\qquad
u_\infty(\tau,z)=-\dfrac{z_\perp}{2|z_\perp|^2}$};
\end{scope}
\end{tikzpicture}
\caption{The magnification and cylindrical spacetime limit in
Proposition~\ref{prop:universal-cylindrical-tangent}.  The upper drawing
selects $x_*$ on the collapse circle and its tangent direction.
The change of variables sends $x_*$ to the origin and the tangent to
the $z_3$-axis.  The lower drawings show two time slices of the infinite
straight cylinder: its radius $\sqrt{1-\tau}$ decreases from $1$ at
$\tau=0$ to $1/2$ at $\tau=3/4$.  Shading marks the fluid outside the
cylinder, and dashed ellipses mark transverse sections.
The vector $z_\perp=(z_1,z_2,0)$ joins $(0,0,z_3)$ to $z$.
The flow is radially inward with zero axial velocity.}
\label{fig:cylindrical-zoom}
\end{figure}
%\FloatBarrier

\begin{theorem} 
\label{thm:terminal-material-deformation}
For the canonical trajectories constructed in
Theorems~\ref{thm:toroidal-curvature-singularity}
and~\ref{thm:zero-surface-tension-construction}, let $\mathcal X_t$ be
the material flow defined in Definition~\ref{def:canonical-terminal-objects}.
The flow has a continuous terminal limit
$\mathcal X_T:\overline{\Omega_0}\to\overline{\Omega_*}$
with the following properties.
\begin{enumerate}
\item The convergence rate is
\begin{equation}\label{eq:main-lagrangian-terminal-rate}
\|\mathcal X_T-\mathcal X_t\|_{C^0(\overline{\Omega_0})}
=a(t)(1+o(1)).
\end{equation}
Writing
\begin{equation*}
\omega_T(\delta):=
\sup_{\xi,\zeta\in\overline{\Omega_0},\, |\xi-\zeta|\le\delta}
|\mathcal X_T(\xi)-\mathcal X_T(\zeta)|,
\end{equation*}
there are $c,C,\delta_0>0$ such that
\begin{equation}\label{eq:main-optimal-lagrangian-modulus}
c\sqrt\delta\le\omega_T(\delta)\le C\sqrt\delta,
\qquad 0<\delta<\delta_0.
\end{equation}
Thus the H\"older exponent $1/2$ is optimal.

\item The terminal map satisfies
\begin{equation}\label{eq:main-terminal-exact-fibers}
\begin{gathered}
\mathcal X_T^{-1}(\gamma_*)=\Gamma_-(0),\qquad
\det D_\xi\mathcal X_T=1\quad\hbox{in }\Omega_0,\\
\mathcal X_T:\Omega_0\longrightarrow\Omega_*\setminus\gamma_*
\quad\hbox{is a }C^\infty\hbox{ diffeomorphism}.
\end{gathered}
\end{equation}
Its restriction $\Gamma_+(0)\to\Gamma_+^*$ is also a
$C^\infty$ diffeomorphism.  Its nonsingleton fibers are exactly
the initial meridians of the inner boundary at fixed cylindrical azimuth.
The map $\mathcal X_T$ induces a homeomorphism
\begin{equation}\label{eq:main-terminal-quotient}
\overline{\Omega_0}/\!\sim\ \simeq\ \overline{\Omega_*}.
\end{equation}
\item The critical material regularity is
\begin{equation}\label{eq:main-critical-material-regularity} 
\mathcal X_t\longrightarrow\mathcal X_T
\quad\hbox{strongly in }W^{1,p}(\Omega_0),\quad 1\le p<2,\qquad
D\mathcal X_T\in L^{2,\infty}(\Omega_0)\setminus L^2(\Omega_0). 
\end{equation}
The inverse on $\Omega_*\setminus\gamma_*$, extended boundedly across
$\gamma_*$, belongs to $W^{1,p}(\Omega_*)$ for every $1\le p<2$, and
its derivative belongs to $L^{2,\infty}(\Omega_*)\setminus L^2(\Omega_*)$.
Put $d(\xi)=\operatorname{dist}(\xi,\Gamma_-(0))$, and
write $s_1(A)\ge s_2(A)\ge s_3(A)$ for the Euclidean singular values
of a $3\times3$ matrix $A$.  There is a fixed $d_0>0$ such that, for
$0<d=d(\xi)<d_0$,
\begin{equation}\label{eq:main-anisotropic-material-deformation} 
s_1(D\mathcal X_t)\asymp(a(t)^2+d)^{-1/2},\qquad
s_2(D\mathcal X_t)\asymp1,\qquad
s_3(D\mathcal X_t)\asymp(a(t)^2+d)^{1/2}, 
\end{equation}
uniformly at late times, including $t=T$ with $a(T)=0$. 
\end{enumerate}
\end{theorem} 

Figure~\ref{fig:terminal-material-map} depicts the terminal fibers
and the square-root normal deformation.

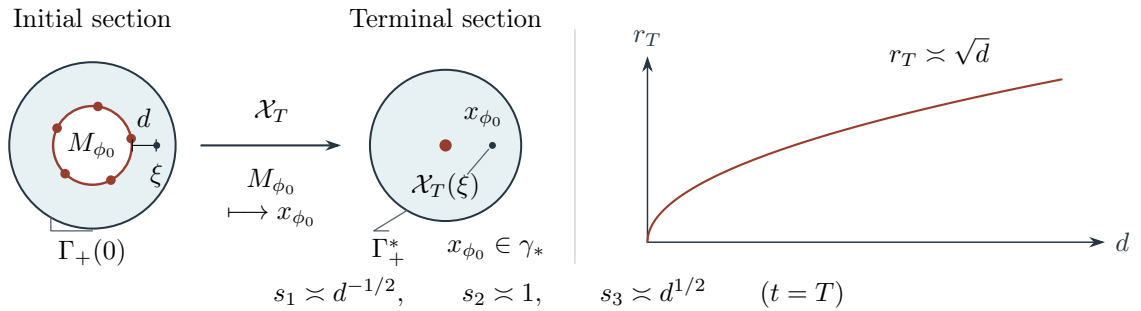
\begin{figure}[!htbp]
\centering
\begin{tikzpicture}[cbpfigure]
\path[use as bounding box] (0,0) rectangle (15.7,4.10);
\node at (1.68,3.78) {Initial section};
\filldraw[fill=cbpFluid,cbpwall] (1.68,2.09) circle (1.10);
\filldraw[fill=white,cbpinner] (1.68,2.09) circle (.52);
\foreach \ang in {10,82,154,226,298}{
\fill[cbpCore] (1.68,2.09)++(\ang:.52) circle (1.8pt);}
\node at (1.68,2.09) {$M_{\phi_0}$};
\draw[|-|,line width=.4pt] (2.20,2.09)--(2.53,2.09);
\node[above=3pt] at (2.36,2.09) {$d$};
\fill[cbpInk] (2.53,2.09) circle (1.3pt);
\node[below=3pt] at (2.53,2.09) {$\xi$};
\draw[->,cbpInk,thick] (3.12,2.09)--(4.97,2.09);
\node[above=5pt] at (4.05,2.09) {$\mathcal X_T$};
\node[below=5pt,align=center] at (4.05,2.09)
{$M_{\phi_0}$\\$\longmapsto x_{\phi_0}$};
\node at (6.35,3.78) {Terminal section};
\filldraw[fill=cbpFluid,cbpwall] (6.35,2.09) circle (1.00);
\fill[cbpCore] (6.35,2.09) circle (2.3pt);
\node[above right=3pt] at (6.35,2.09) {$x_{\phi_0}$};
\draw[cbpleader] (1.13,1.14)--(1.13,.96)--(1.68,.96);
\node[anchor=north,inner sep=2pt] at (1.68,.93) {$\Gamma_+(0)$};
\draw[cbpleader] (5.84,1.23)--(5.40,.96)--(5.60,.96);
\node[anchor=north,inner sep=2pt] at (5.60,.93) {$\Gamma_+^*$};
\node at (7.02,.72) {$x_{\phi_0}\in\gamma_*$};
\fill[cbpInk] (6.97,2.09) circle (1.3pt);
\draw[cbpleader] (6.92,2.04)--(6.64,1.71);
\node at (6.35,1.55) {$\mathcal X_T(\xi)$};
\draw[cbpInk!20] (8.06,.53)--(8.06,3.62);
\begin{scope}[shift={(9.02,.81)}]
\draw[->,cbpInk] (0,0)--(6.05,0) node[right] {$d$};
\draw[->,cbpInk] (0,0)--(0,2.46) node[above] {$r_T$};
\draw[cbpCore,thick,domain=0:2.34,samples=90,variable=\s]
plot ({\s*\s},{.92*\s});
\node at (3.86,2.52) {$r_T\asymp\sqrt d$};
\end{scope}
\node at (7.85,.15)
{$s_1\asymp d^{-1/2},\qquad s_2\asymp1,\qquad
s_3\asymp d^{1/2}\qquad(t=T)$};
\end{tikzpicture}
\caption{The terminal material map at a fixed azimuth $\phi_0$.
The initial inner meridian $M_{\phi_0}$ is the fiber over
$x_{\phi_0}=(R_*\cos\phi_0,R_*\sin\phi_0,0)$.
The fluid interior maps diffeomorphically onto
$\Omega_*\setminus\gamma_*$, and $\Gamma_+(0)$ maps to $\Gamma_+^*$.
The initial section and particle positions are schematic.
For a nearby initial point $\xi$,
$d=\operatorname{dist}(\xi,\Gamma_-(0))$ and
$r_T=\operatorname{dist}(\mathcal X_T(\xi),\gamma_*)$.
The curve illustrates the square-root law.  The singular values
$s_i=s_i(D_\xi\mathcal X_T(\xi))$, ordered as $s_1\ge s_2\ge s_3$,
describe the accompanying normal stretch and meridional compression.}
\label{fig:terminal-material-map}
\end{figure} 

\subsection{Relation to previous work}

The construction connects the potential-flow theory of cavity collapse
with the analysis of singularities in the free-boundary Euler equations.
We first discuss the closest collapse models, then place the result
among other mechanisms of breakdown and the local theories used here.

\subsubsection{Cavity collapse and potential flows}

Classical radial models describe spherical cavity collapse.
Rayleigh calculated the collapse of a spherical cavity in an infinite
incompressible liquid under constant pressure at infinity
\cite{Rayleigh1917}.  Plesset extended the radial equation to account for
time-dependent ambient pressure and a bubble boundary condition
incorporating vapor pressure and surface tension \cite{Plesset1949}.
For symmetric axisymmetric bubble pinch-off, the asymptotic
potential-flow analysis of Gordillo--Sevilla--Rodr\'iguez-Rodr\'iguez--Mart\'inez-Baz\'an
\cite{Gordillo2005} yields the logarithmically corrected relation
$T-t\asymp a^2\sqrt{\log(1/a)}$, where $a$ is the dimensionless
minimum neck radius and $T$ is the pinch-off time.
Eggers--Fontelos--Leppinen--Snoeijer's slender-body theory
\cite{Eggers2007} evolves both the radial and axial scales and predicts
a local radial exponent approaching $1/2$ at order
$|\log(T-t)|^{-1/2}$.
Best's boundary-integral computations describe toroidal bubbles formed
after jet penetration \cite{Best1993}.

Planar potential flows provide exact solutions and numerical models
for singular boundary motion.
Liu--Pego \cite{LiuPego2025} analyze explicit ideal potential flows and
investigate a corner formation scenario through scaling and time reversal
of numerical flows expanding from smooth approximations to a wedge.
Zhuravleva--Zubarev--Zubareva--Karabut \cite{ZhuravlevaEtAl2020}
construct exact planar cavity flows with zero gravity and zero surface tension
through a Hopf equation for the inverse complex velocity.
The ballistic cavity examples analyzed in \cite[Section~6]{LiuPego2025}
occupy unbounded fluid domains and satisfy $p<0$ in the fluid and
$\partial_n p=0$ on the free boundary.
Chen--Walsh--Wheeler \cite{ChenWalshWheeler2026} construct
self-similarly imploding hollow vortices by desingularizing
nondegenerate configurations of collapsing point vortices. Their flows
are planar exterior flows with nonzero circulation around the cavities
and a gas-pressure law on their boundaries.

\subsubsection{Free-boundary singularities}

Self-contact is a mechanism for free-boundary singularity formation.  Castro--C\'ordoba--Fefferman--Gancedo--G\'omez-Serrano constructed
two-dimensional irrotational flows forming splash and splat singularities
at zero surface tension \cite{CastroAnn} and at positive surface tension
\cite{CastroST}.  Coutand--Shkoller constructed three-dimensional splash
and splat singularities at zero surface tension, allowing vorticity
\cite{CSsplash}. The authors \cite{LuoYangMHD} construct finite-time self-intersection
singularities for the three-dimensional free-boundary ideal MHD
equations with positive surface tension, with curvature controlled up to
contact.  The construction also yields capillary Euler
self-intersection when the magnetic field vanishes.

For two-dimensional irrotational water waves at zero surface tension,
C\'ordoba--Enciso--Grubic \cite{CordobaEncisoGrubic2026} prove
finite-time breakdown from initially cornered interfaces, through
loss of weighted Sobolev regularity or self-intersection away from
the corner.  Agrawal \cite{AgrawalGravity} proves that a high-order energy becomes
unbounded in finite time for zero-gravity water waves in a
bounded planar domain with initially angle-crested interfaces.
Coutand \cite{Coutand2019} proves finite-time breakdown for a planar
two-phase flow in a fixed container, with gravity, surface tension and
a heavier inner fluid, under symmetry and smallness assumptions on the
initial geometry, velocities and outer fluid vorticity.

In the solutions constructed here, an entire boundary component
degenerates, with a quantitative curvature blow-up law, while the other
component has a smooth limit. The initial boundaries are smooth, and
both remain embedded and disjoint for every $t<T$.

\subsubsection{Singularities on fixed domains}

For Euler flows on fixed domains, Elgindi constructed finite-time
singularities on $\R^3$ for axisymmetric flows without swirl and with
$C^{1,\alpha}$ velocity, for a sufficiently small positive $\alpha$
\cite{Elgindi2021}. Elgindi--Pasqualotto construct finite-energy
axisymmetric flows with swirl whose initial velocity is $C^{1,\alpha}$
for some $\alpha>0$ and smooth away from a circle
\cite[Theorem~1.3]{ElgindiPasqualotto2023}. Their construction uses
amplification driven by swirl near that circle. Chen--Hou
\cite{ChenHouPNAS2025} give an overview of their computer-assisted
construction of nearly self-similar blow-up from smooth finite-energy
data for axisymmetric Euler flows with swirl in an axially periodic
cylinder with an impermeable wall. Chen constructs asymptotically
self-similar blow-up on $\R^3$ for axisymmetric flows without swirl,
with compactly supported $C^\alpha$ initial vorticity and
$C^{1,\alpha}\cap L^2$ initial velocity for every $0<\alpha<1/3$
\cite[Theorem~1.2]{ChenThreshold2026}.
These constructions develop singular vorticity in a fixed fluid
domain.

The recent manuscript of OpenAI \cite[Theorem~1.1]{OpenAIEuler2026}
states finite-time breakdown on $\R^3$ from smooth, compactly supported
initial velocity; Zheng \cite{ZhengEuler2026} gives an expanded
exposition of that construction. For smooth whole-space Euler flows,
the Beale--Kato--Majda criterion \cite{BKM} requires the time integral
of the vorticity maximum norm to diverge at breakdown. In the present
free-boundary problem, vorticity vanishes identically and the singularity
is carried by the evolving geometry and the concentrating potential flow.

\subsubsection{Local theory and continuation}

For free-boundary Euler flows, local existence and a priori estimates
have been developed in several settings.  Iguchi and Ogawa--Tani established early local
results at finite regularity for planar capillary--gravity flows
\cite{Iguchi,OgawaTani}.  Schweizer proved three-dimensional short-time
existence for small periodic boundary graphs with initial vorticity
vanishing on the free boundary \cite{Schweizer}.
Coutand--Shkoller established Sobolev existence theory at positive
surface tension, with uniqueness for sufficiently regular data, and
local well-posedness under the Taylor sign at zero surface tension
\cite{CS}.  Shatah--Zeng obtained geometric a priori estimates for
the one-phase problem \cite{SZ1} and for capillary interfaces between
two fluids \cite{SZ2}, and developed local well-posedness theory
allowing vorticity and multiply connected fluid domains \cite{SZ3}.
Disconzi--Kukavica \cite{DK} derive Lagrangian a priori
estimates at the $H^3$ bulk level with additional boundary trace
regularity.  Disconzi--Kukavica--Tuffaha \cite{DKT} obtain
$H^{5/2+\varepsilon}$ bulk estimates for $0<\varepsilon<1/2$,
including general bounded domains, with additional regularity of the
initial boundary and normal velocity.
Disconzi--Ebin \cite{DE} prove local well-posedness at positive
surface tension. 

For irrotational water waves, the Dirichlet--Neumann formulation and
paradifferential symmetrization reduce the free-surface equations to
a quasilinear dispersive system \cite{ABZ,dPN}. Lannes
\cite{Lannes} gives a systematic treatment of the Dirichlet--Neumann
operator and the water-wave equations.
Ming--Zhang establish local well-posedness for gravity--capillary
waves over a smooth bottom and prove the zero tension limit over a
flat bottom \cite{MingZhang}.  Their linearized estimates are uniform
in surface tension under a positive L\'evy coefficient.

Continuation criteria connect finite-time breakdown to the loss of
geometric or analytic control.  At zero surface tension, Ginsberg's criterion combines
boundary geometry and the Taylor sign with vorticity, boundary
velocity derivatives and a boundary pressure derivative norm
\cite{Ginsberg}.  Luo--Zhou obtain a criterion based on vorticity, coupled
to boundary control norms and graph nondegeneracy for a horizontally
periodic capillary surface above a fixed bottom \cite{LuoZhou}.
For general bounded capillary domains, Hao and the authors give four
breakdown alternatives involving boundary self-intersection, curvature
and boundary regularity, normal velocity, and the time integral of the
interior velocity gradient supremum
\cite[Theorem~1.1 and Remark~1.2]{HLY}.

\subsection{Mechanism and proof strategy}

The leading motion is governed by a logarithmic potential in the two
directions normal to the collapse circle. Its gradient is proportional
to the inverse distance from the circle, so its kinetic energy is
distributed over a logarithmic range of radii. Energy conservation
therefore fixes the amplitude of this potential and, through the
kinematic condition, the collapse rate. The leading singular term is
local; the outer geometry enters through the regular part of the Green
function. Making this separation quantitative is what permits the
prescription of an arbitrary outer meridional domain.

Two difficulties remain beyond this leading description. The inner
boundary shrinks while the outer boundary stays at a fixed scale, so
the boundary equations must be estimated in different coordinates.
Moreover, at positive surface tension the capillary coefficient in
the rescaled equations tends to zero. The proof requires coercive
estimates in the natural weighted norms, with a leading backward
growth rate and fixed derivative losses independent of the Sobolev
index. The constants in these estimates may depend on that index.

We start from smooth terminal states with circular cavities of radii
$a_N\downarrow0$ and solve backwards in $Y=\log(\ell/a)$ on
$[Y_0,Y_N]$, where $Y_N=\log(\ell/a_N)$.
The central task is to choose $Y_0$ independently of $N$ and of the
fixed regularity level at which compactness is taken.

\medskip
\noindent\emph{Step 1. Separate the flux and the geometric perturbations.}
The potential splits orthogonally in Dirichlet energy into a
flux-carrying condenser field and a zero-flux remainder.
For the normalized inner flux
$q=(2\pi L)^{-1}\int_{\Gamma_-}v\cdot n\,dS$, where $L=2\pi R$,
the energy identity determines $q$, and the capacity expansion gives
$q^2Y\to E_c/(\pi L_*)$. Differentiation yields the flux equation;
energy conservation also recovers the constant Bernoulli equation
between the boundaries.

The mean and first harmonic of the inner graph determine the radius
and centre. The first potential harmonic remains a separate dynamical
unknown. With this choice, the reference forcing of the scaled higher
modes starts at order $a^3$, up to polynomial factors in $Y$.
The zero-flux constraint removes the logarithmic monopole; its
finite-rank variations must still be retained in the nonlinear estimates.

\medskip
\noindent\emph{Step 2. Control the energy as the capillary coefficient vanishes.}
For fixed $\sigma>0$, the effective coefficient $B_q=\sigma a/q^2$
tends to zero. The principal inner energy is uniformly equivalent to
\begin{equation*}
\|\mathfrak h\|_{H^s}
+\sqrt{B_q}\,\|\mathfrak h\|_{H^{s+1}}
+\|\mathfrak\pi\|_{H^{s+1/2}},
\end{equation*}
where $(\mathfrak h,\mathfrak\pi)$ are the normalized inner variables
constructed in Section~\ref{sec:nonlinear-energy}. Thus the extra
derivative of $\mathfrak h$ carries a vanishing capillary weight.
Two cancellations are essential. The inner Taylor
coefficient, reconstructed from the Euler pressure, and the logarithmic
flux derivative contain the same capacity correction, which cancels
at leading logarithmic order in their coupling. In normal variables
and surface half-densities, the principal Dirichlet--Neumann and
curvature symbols are powers of the same arclength frequency.
Commuting with powers of that frequency cancels the leading capillary
commutator and gives a leading backward growth rate of $1$, independent
of the Sobolev index. Zero-flux localization and a finite symbol
expansion control the remainders. The fixed derivative losses and the
coefficient regularity needed for these estimates are specified in
Section~\ref{sec:nonlinear-energy}.
The zero-flux localization in
Proposition~\ref{prop:charge-free-localization} applies to planar
divergence-form equations with smooth uniformly positive scalar
weights, homogeneous outer Dirichlet data, and a shrinking nearly
circular hole uniformly separated from the outer boundary, under
the stated uniform coefficient and geometric bounds.

\medskip
\noindent\emph{Step 3. Close the Taylor energies at zero surface tension.}
At $\sigma=0$, both Taylor coefficients supply coercivity; the outer
one is controlled by the limiting Green potential and Hopf's lemma.
Minimizing the outer Dirichlet--Neumann form over the constant inner
trace enforces zero flux. We couple the inner energy at index $s$
to the outer energy at index $s+1$. The separated-boundary kernels
allow high output derivatives on the outer boundary while using only
a fixed lower norm of the inner input. The original outer Bernoulli
equation gives the additional decay of the outer potential trace
needed to close this coupling.

\medskip
\noindent\emph{Step 4. Obtain a common interval and pass to the collapsing limit.}
Coupling the inner and outer energies with suitable powers of the
radius controls their mutual forcing. The leading backward growth
rate is smaller than the decay exponent of the $a^3$ inner forcing.
This gives bounds independent of $Y_N$ with the same $Y_0$ at every
fixed regularity. Compactness on nondegenerate intervals yields a
smooth Euler trajectory, and the capacity law gives a finite collapse
time. For prescribed outer geometry, retaining both centre coordinates
and both first potential harmonics preserves these forcing orders
and the distribution of derivatives between the boundaries.

\medskip
\noindent\emph{Step 5. Identify the singular profiles and the terminal material map.}
For the canonical trajectories, comparison with a truncated logarithmic
potential determines the limiting line measure and its radial
distribution. In particular, a tube whose radius is a fixed multiple
of the cavity radius contains vanishing limiting energy: the
concentration occurs over a logarithmic range of scales. The Green
profile and the scaled collar estimates give the bulk pressure
asymptotics and the cylindrical tangent, respectively.

Subtracting boundary motion from the particle equation gives an
equation with integrable coefficients for the difference of squared
particle and boundary radii. The collar exit scale and volume
preservation then yield the square-root normal form and the sharp
derivative bounds for the terminal material map.

Finally, integrable momentum and energy fluxes permit concatenation
with a classical flow on the filled domain. Under the hypotheses of
Proposition~\ref{prop:weak-capillary-continuation}(i), the incoming
total-energy measure minus the outgoing one determines the local
total-energy defect at the joining time. For the canonical continuations, the capillary-energy traces agree
and the outgoing kinetic-energy trace vanishes. The defect is
therefore the tensor product of the unit Dirac mass at the collapse
time with the concentrated kinetic-energy measure on the collapse
circle.

\medskip
\noindent\textbf{Organization of the paper.}
Sections~\ref{sec:scaled-DN}--\ref{sec:nonlinear-energy} develop the
two-scale elliptic estimates, derive the exact boundary equations, and
close the nonlinear energy estimates at positive surface tension.
Section~\ref{sec:terminal-existence} constructs the canonical collapsing
trajectory and proves its singular profiles, material map properties and
filling continuation.  Section~\ref{sec:zero-surface-tension} supplies the
Taylor energies needed at zero surface tension and verifies the common
inputs to these limiting arguments.  Finally,
Section~\ref{sec:prescribed-outer-geometry} proves
Theorem~\ref{thm:prescribed-outer-geometry} for a prescribed outer section
and core point by retaining both centre coordinates and both first
potential harmonics.

\section{Two-scale Dirichlet--Neumann analysis}
\label{sec:scaled-DN}

For a single-valued axisymmetric velocity potential, the Laplace equation
reduces to a planar divergence-form equation with a smooth weight.
Using scaled coordinates near the inner boundary and physical coordinates
near the outer boundary, we obtain uniform estimates for the layer
potentials, the Dirichlet--Neumann operators and their first two shape
variations.  We then impose the zero-flux constraint by a Schur complement
and derive the capacity and flux expansions used in the boundary dynamics.

We write $\mathcal P_s(Y)$ for a quantity bounded in absolute
value by $C_s(1+Y)^{M_s}$, where $C_s,M_s\ge0$ are constants.  This
quantity may change from one occurrence to the next and is chosen
nonnegative in upper bounds.  We suppress the argument $Y$ when it is
clear and omit the subscript at fixed lower regularity indices.

\subsection{Weighted flux and layer-potential estimates}

We use $D_\theta=-i\partial_\theta$.

Let $\mathscr D\subset\{(\rho,z):\rho>0\}$ be the meridional section of an axisymmetric three-dimensional fluid domain.
An axisymmetric harmonic potential $\Phi=\Phi(\rho,z)$ satisfies
\begin{equation}\label{eq:axisym-elliptic}
\mathcal L\Phi:=\frac1\rho\partial_\rho(\rho\partial_\rho\Phi)+\partial_z^2\Phi=0
\qquad\text{in }\mathscr D.
\end{equation}
Suppose that $\partial\mathscr D=\mathscr C_-\cup\mathscr C_+$ has two connected components, parametrized by
\begin{equation*}
X_\pm(\theta)=(\rho_\pm(\theta),z_\pm(\theta)),\qquad \theta\in\mathbb T.
\end{equation*}
Set $j_\pm=|\partial_\theta X_\pm|$ and let $n_\pm$ denote the outer unit normal of the meridional fluid domain.  Given Dirichlet data $\boldsymbol\psi=(\psi_-,\psi_+)$, let $\Phi[\boldsymbol\psi]$ solve \eqref{eq:axisym-elliptic} with $\Phi|_{\mathscr C_\pm}=\psi_\pm$.  For a fixed reference major radius $R>0$, define the flux Dirichlet--Neumann (DN) matrix by
\begin{equation*}
\mathfrak G_\pm\boldsymbol\psi
:=\frac{\rho_\pm j_\pm}{R}\,\partial_{n_\pm}\Phi[\boldsymbol\psi].
\end{equation*}
The factor $\rho j$ is the surface Jacobian after integration in the azimuthal variable, apart from the common factor $2\pi$.
All boundary traces are pulled back by $X_\pm$.  We write
$\mathfrak G\boldsymbol\psi=(\mathfrak G_-\boldsymbol\psi,\mathfrak G_+\boldsymbol\psi)$
and use the unnormalized real pairing
\begin{equation*}
\langle\boldsymbol f,\boldsymbol g\rangle
=\sum_{\iota\in\{-,+\}}\int_\T f_\iota g_\iota\,d\theta.
\end{equation*}
For complex traces we use its Hermitian extension, conjugate-linear in
the first argument.

\begin{lemma}
For smooth data $\boldsymbol\psi,\boldsymbol\varphi$,
\begin{equation}\label{eq:flux-DN-green}
\sum_{\iota\in\{-,+\}}\int_{\mathbb T}
\varphi_\iota\,\mathfrak G_\iota\boldsymbol\psi\,d\theta
=\frac1R\int_{\mathscr D}\rho\,\nabla\Phi[\boldsymbol\varphi]\cdot\nabla\Phi[\boldsymbol\psi] \,d\rho dz.
\end{equation}
Consequently, $\mathfrak G$ is symmetric on smooth traces and its
Dirichlet form is nonnegative.

Let $\mathscr C_\pm(\tau)$ be a smooth normal deformation of $\mathscr C_\pm$ with deformation speeds $h_\pm$, so that
\begin{equation*}
\left.\partial_\tau X_\pm\right|_{\tau=0}=h_\pm n_\pm,
\end{equation*}
and transport the Dirichlet data $\psi_\pm$ with the parametrization,
keeping the normalization radius $R$ fixed during shape differentiation.  Set
$\boldsymbol h=(h_-,h_+)$ and
\begin{equation*}
B_\pm:=\partial_{n_\pm}\Phi[\boldsymbol\psi].
\end{equation*}
Then
\begin{equation}\label{eq:Hadamard-flux}
D\mathfrak G_\iota[\boldsymbol h]\boldsymbol\psi
=-\mathfrak G_\iota(\boldsymbol h\boldsymbol B)
-\frac1R\partial_\theta\left(\frac{\rho_\iota}{j_\iota}
h_\iota\partial_\theta\psi_\iota\right),
\qquad \iota\in\{-,+\}.
\end{equation}
Here $\boldsymbol h\boldsymbol B=(h_-B_-,h_+B_+)$, and the first term
includes the contributions from both boundary components.
\end{lemma}

\begin{proof}
Weighted integration by parts, with $ds=j_\iota d\theta$, gives
\eqref{eq:flux-DN-green}, symmetry, and nonnegativity.

To compute the shape derivative, differentiate the product of the surface
Jacobian and the normal derivative.
Let $u=\Phi[\boldsymbol\psi]$, and denote its Eulerian derivative by $u'$.
Transport of the Dirichlet data gives
\begin{equation*}
\mathcal Lu'=0,\qquad u'|_{\mathscr C_\jmath}=-h_\jmath B_\jmath.
\end{equation*}
For a normal deformation, the material derivative of the normal derivative is
\begin{equation*}
\dot B_\iota
=\partial_{n_\iota}u'-\nabla_{\mathscr C_\iota}h_\iota\cdot\nabla_{\mathscr C_\iota}\psi_\iota
+h_\iota\partial_{n_\iota n_\iota}u.
\end{equation*}
Decomposing the Euclidean Laplacian into normal and tangential parts gives
\begin{equation*}
\partial_{nn}u=-\mathcal H B-\Delta_{\Gamma}\psi,
\end{equation*}
where $\Gamma$ is the corresponding surface of revolution.  The surface Jacobian $\mathcal J_\iota=\rho_\iota j_\iota$ satisfies $\dot{\mathcal J}_\iota=\mathcal H_\iota h_\iota\mathcal J_\iota$.  Thus the $\mathcal H hB$ terms cancel when differentiating $\mathcal J B$.  For an axisymmetric scalar,
\begin{equation*}
\mathcal J_\iota\,\operatorname{div}_{\Gamma_\iota}
(h_\iota\nabla_{\Gamma_\iota}\psi_\iota)
=\partial_\theta\left(\frac{\rho_\iota}{j_\iota}
h_\iota\partial_\theta\psi_\iota\right),
\end{equation*}
while $u'$ has boundary data $-\boldsymbol h\boldsymbol B$.  Combining
these identities and dividing by $R$ gives \eqref{eq:Hadamard-flux}.
\end{proof}

For a meridional circle with major radius $R$ and minor radius $a$, use
\begin{equation*}
\rho=R+aX,\qquad z=aZ,\qquad \varepsilon=a/R.
\end{equation*}
Writing $u(X,Z)=\Phi(R+aX,aZ)$, equation
\eqref{eq:axisym-elliptic} becomes
\begin{equation*}
\Delta_{X,Z}u+\frac{\varepsilon}{1+\varepsilon X}\partial_Xu=0.
\end{equation*}
In polar coordinates $(X,Z)=r(\cos\theta,\sin\theta)$, this is
\begin{equation*}
\Delta u+\frac{\varepsilon}{1+\varepsilon r\cos\theta}
\left(\cos\theta\,\partial_ru-
\frac{\sin\theta}{r}\partial_\theta u\right)=0.
\end{equation*}
On the exact inner circle $r=1$, the normalized inner flux is
\begin{equation*}
(1+\varepsilon\cos\theta)(-\partial_ru)|_{r=1}.
\end{equation*}
The Dirichlet data and outer boundary condition are inherited from the full fluid domain.

For $j\ge0$, let $P_j$ denote the orthogonal projection onto the Fourier modes
$|m|=j$, and set $P_{\ge2}=I-P_0-P_1$ and $P_{\ne0}=I-P_0$.
We write $H^s_{\ge2}=P_{\ge2}H^s(\mathbb T)$ and
$\langle m\rangle=(1+m^2)^{1/2}$. The Fourier multiplier
$|D_\theta|$ acts by multiplication by $|m|$, and
$\langle D_\theta\rangle$ has symbol $\langle m\rangle$.

To estimate the outer boundary contribution for the admissible
macroscopic geometries, let $d_*>0$ be a uniform lower bound for the distance
from the collapsing centre point to the macroscopic boundary, and set
\begin{equation*}
L_a:=1+\left|\log\frac{a}{d_*}\right|.
\end{equation*}
The constants in the layer estimates below are uniform in $a$. They may depend on
$d_*$ and the fixed macroscopic geometry class.

The exact axisymmetric Green kernel separates the local logarithmic
singularity from a smooth contribution of the outer boundary.
Set $x_0=(R,0)$, $e_r(\theta)=(\cos\theta,\sin\theta)$, and let
$\mathscr D_0$ be the filled meridional domain.
The admissible geometries satisfy $\rho\ge\rho_*>0$ on $\mathscr D_0$,
with uniform bounds on the outer boundary and
$\operatorname{dist}(x_0,\partial\mathscr D_0)\ge d_*$.
Write $G_0$ for the Dirichlet Green kernel of
$-\operatorname{div}(\rho\nabla)$ in $\mathscr D_0$ and put $w=\rho$.

\begin{lemma} \label{lem:axisymmetric-green}
For $x=(\rho,z)$, $y=(\rho',z')$, and $d=|x-y|$, set
\begin{equation*}
s=(d^2+4\rho\rho')^{1/2},\qquad
k=\frac{2\sqrt{\rho\rho'}}s,\qquad
K(k)=\int_0^{\pi/2}(1-k^2\sin^2u)^{-1/2}\,du.
\end{equation*}
The symmetric kernel
\begin{equation*}
G_{\rm ax}(x,y)=\frac1{4\pi}\int_0^{2\pi}
\frac{d\varphi}{\sqrt{d^2+4\rho\rho'\sin^2(\varphi/2)}}
=\frac{K(k)}{\pi s}
\end{equation*}
satisfies $-\operatorname{div}_x(\rho\nabla_xG_{\rm ax})=\delta_y$.
On a fixed small product neighbourhood of $(x_0,x_0)$,
\begin{equation*}
\sqrt{\rho\rho'}G_0(x,y)
=\Gamma(x-y)+d^2A(x,y)\log d+B(x,y),
\qquad \Gamma(z)=-\frac1{2\pi}\log|z|,
\end{equation*}
where $A,B$ are smooth, with uniform bounds at each fixed order
for the admissible class of smooth geometries.
Let $B^{[1]}$ be the Taylor polynomial of $B$ at $(x_0,x_0)$
through total degree one.  Restricting the kernel
$d^2A\log d+B-B^{[1]}$ to
$x=x_0+ae_r(\theta)$, $y=x_0+ae_r(\varphi)$ gives the operator
$\mathcal R_a$, acting by integration against $f(\varphi)\,d\varphi$.
For every fixed integer $r\ge0$, it satisfies
\begin{align}
\|\mathcal R_af\|_{H^{r+2}}&\le C_ra^2L_a\|f\|_{H^r},
\label{eq:axisymmetric-circle-map}\\
\|a\partial_{n_x}\mathcal R_af\|_{H^r}
&\le C_ra^2L_a\|f\|_{H^r}.
\label{eq:axisymmetric-circle-normal-map}
\end{align}
Here the normal derivative is taken before restricting the kernel to the circle.
The same estimates hold after one $a\partial_a$ derivative and after
up to two smooth variations of the centre, outer geometry, or a small
relative inner graph, including mixed derivatives.
The constants depend only on finitely many smooth
norms of the geometry and variation directions.
\end{lemma}

\begin{proof}
We divide the proof into three steps.

\emph{Step 1. Expansion of the Green kernel.}
For a compactly supported meridional source $f$, its axisymmetric
three-dimensional extension with density $f(\rho,z)/\rho$ has Newton
potential $\int G_{\rm ax}(x,y)f(y)\,dy$: the cylindrical volume
Jacobian cancels the factor $1/\rho'$.  The identity
$\Delta_3u=\rho^{-1}\operatorname{div}(\rho\nabla u)$ proves the
normalization of the fundamental solution.  Changing variables in
its angular integral gives the expression in terms of $K$.
Put $k'=(1-k^2)^{1/2}=d/s$, $t=(k')^2$, and
$y(t)=K(\sqrt{1-t})$.
Differentiation of the defining integral and integration by parts give
\begin{equation*}
t(1-t)y''+(1-2t)y'-\tfrac14y=0.
\end{equation*}
Its analytic solution normalized to one at zero is
$F(t)=\sum_{n\ge0}c_nt^n$, where $c_0=1$ and
$c_{n+1}=((n+\tfrac12)/(n+1))^2c_n$.
The second Frobenius solution is $-\tfrac12F(t)\log t+H(t)$,
with $H$ analytic for $|t|<1$.  The series determined by this recurrence converges
on that disk.  In the defining integral, the neighborhood of
$u=\pi/2$ gives $y(t)=-\tfrac12\log t+O(1)$.
Consequently
\begin{equation*}
K(k)=F(t)\log(1/k')+H(t),\qquad F(0)=1.
\end{equation*}
Since
$\log(1/k')=\log s-\log d$ and
$F(d^2/s^2)/(\pi s)=(2\pi\sqrt{\rho\rho'})^{-1}+d^2C(x,y)$,
this gives the asserted smooth logarithmic decomposition for
$G_{\rm ax}$.  Write $G_0=G_{\rm ax}-H_0$, where $H_0(\cdot,y)$
is weighted harmonic with outer boundary value $G_{\rm ax}(\cdot,y)$.
The pole stays separated from that boundary.  The coercive form
$\int_{\mathscr D_0}\rho|\nabla u|^2$ and interior elliptic estimates
therefore give joint smooth bounds for $H_0$ near $(x_0,x_0)$.
Pulling the outer domain back by maps equal to the identity near the
pole and differentiating its coercive Dirichlet problem proves the
same bounds for two geometric variations.

\emph{Step 2. Estimates on the circle.}
On the circle, $d^2=a^2(2-2\cos(\theta-\varphi))$.
The Fourier coefficients of
$(2-2\cos\vartheta)\log|e^{i\vartheta}-1|$ are
$O(\langle m\rangle^{-3})$.  Multiplication by a smooth
$(\theta,\varphi)$ amplitude preserves this gain of three derivatives,
as its double Fourier expansion shows.
The term containing $\log a$ and the Taylor remainder of $B$
have uniformly smooth kernels of size $O(a^2L_a)$.
This proves \eqref{eq:axisymmetric-circle-map}.
A derivative in $x$ changes $d^2\log d$ to a combination of
$(x-y)\log d$ and $(x-y)$.
On the circle, $n_x\cdot(x-y)=\pm a(1-\cos(\theta-\varphi))$.
After multiplication by $a$, the differentiated singular kernel
therefore has the same Fourier bound with a smooth amplitude.  The differentiated Taylor remainder is $O(a^2)$.
This proves \eqref{eq:axisymmetric-circle-normal-map}.

\emph{Step 3. Geometric and scale variations.}
For a graph $X_\eta(\theta)=x_0+a r(\theta)e_r(\theta)$,
$r=1+\eta$, the identity
\begin{equation*}
|X_\eta(\theta)-X_\eta(\varphi)|^2
=4a^2\sin^2\frac{\theta-\varphi}{2}\,
\left(r(\theta)r(\varphi)
+\frac{(r(\theta)-r(\varphi))^2}
{4\sin^2((\theta-\varphi)/2)}\right)
\end{equation*}
expresses the quotient in parentheses as a smooth positive function
for graphs in a sufficiently small smooth neighborhood of zero.  The quotient extends across the
periodic diagonal by the formula for divided differences.  Its first two
graph variations have the same property.  Thus the singular angular
factor above is unchanged, and the remaining factors are smooth.
The unit normal is a smooth function of $(\eta,\eta_\theta)$, and
$n_\eta(\theta)\cdot(X_\eta(\theta)-X_\eta(\varphi))$ vanishes
to second order on the diagonal.  Division by
$\sin^2((\theta-\varphi)/2)$ is therefore smooth.
The same Fourier argument gives the estimates for variations of the normal trace.
Finally, $a\partial_a(a^2\log a)=a^2(2\log a+1)$,
while differentiating the smooth factors preserves their bounds.
\end{proof}

\begin{lemma}
\label{lem:weighted-layer-expansion}
Fix an integer $s\ge0$.  For $0<a<a_0$ define the normalized single-layer and
inner conormal operators on the small circle by
\begin{align}
(\mathcal S_a\mu)(\theta)
&:=\int_{\T}G_0(x_0+ae_r(\theta),x_0+ae_r(\varphi))\mu(\varphi)\,d\varphi,
\notag \\
(\mathcal T_a\mu)(\theta)
&:=a\,w(x_0+ae_r(\theta))\partial_{n_x}
\int_{\T}G_0(x,x_0+ae_r(\varphi))\mu(\varphi)\,d\varphi
\bigg|_{x=x_0+ae_r(\theta)} .
\label{eq:Ta-def}
\end{align}
The conormal trace in \eqref{eq:Ta-def} is taken from the exterior of
the small circle, with the fluid outward normal.
There are operators $S_*,T_*$, uniformly bounded
families $S_{1,a},T_{1,a}$, remainders $R^S_a,R^T_a$, and a fixed
exponent $M_s\ge0$ such that
\begin{align}
\mathcal S_a
&=c_{\rm lay}L_aP_0+S_*+aL_aS_{1,a}+a^2L_a^{M_s}R^S_a,
\label{eq:Sa-expansion}\\
\mathcal T_a
&=T_*+aL_aT_{1,a}+a^2L_a^{M_s}R^T_a,
\label{eq:Ta-expansion}
\end{align}
where $c_{\rm lay}>0$, $S_*$ is elliptic of order $-1$ on $P_{\ne0}$, and $T_*$ is of
order $0$.  The families $S_{1,a}$ and $T_{1,a}$ are uniformly bounded
from $H^r$ to $H^{r+1}$ and from $H^r$ to $H^r$, respectively, for
$0\le r\le s+1$.  The remainders satisfy
\begin{align}
\|R^S_a f\|_{H^{r+1}}&\le C_{r,s}\|f\|_{H^r},
\qquad 0\le r\le s+1,\label{eq:Sa-rem-map}\\
\|R^T_a f\|_{H^{r}}&\le C_{r,s}\|f\|_{H^r},
\qquad 0\le r\le s+1.\label{eq:Ta-rem-map}
\end{align}
The first-order correction satisfies the Fourier selection rule
\begin{equation}\label{eq:layer-first-selection}
P_{\ge2}S_{1,a}P_0=P_{\ge2}T_{1,a}P_0=0,
\qquad
(I-P_0-P_1-P_2)S_{1,a}P_1
=(I-P_0-P_1-P_2)T_{1,a}P_1=0.
\end{equation}
Moreover $\mathcal S_a$ is invertible for $a$ small and
\begin{align}
\|P_0\mathcal S_a^{-1}P_0\|_{\mathcal L(P_0L^2(\T))}&\le C L_a^{-1},
\label{eq:Sa-inverse-zero}\\
\|P_{\ne0}\mathcal S_a^{-1}P_{\ne0}f\|_{H^r}
&\le C_r\|f\|_{H^{r+1}},\qquad 0\le r\le s+1,
\label{eq:Sa-inverse-nonzero}\\
\|P_{\ge2}\mathcal T_a\mathcal S_a^{-1}P_0c\|_{H^s}
&\le C_sa^2L_a^{M_s}|c|,
\label{eq:layer-P0-leakage}\\
\|P_{\ge2}\mathcal T_a\mathcal S_a^{-1}P_1f_1\|_{H^s}
&\le C_saL_a^{M_s}\|f_1\|_{H^{s+1}}.
\label{eq:layer-P1-leakage}
\end{align}
All constants are uniform for the admissible macroscopic geometries.
For a small relative graph, replace the two circular arguments by
$X_\eta$ and the conormal prefactor $aw$ by
$w(X_\eta)|\partial_\theta X_\eta|$.  The layer system and its
inverse have two smooth derivatives with respect to the centre, radius,
outer boundary and inner graph at each fixed regularity level.
The selection rule \eqref{eq:layer-first-selection} holds at $\eta=0$.
\end{lemma}

\begin{proof}
We divide the proof into four steps.

\emph{Step 1. Expansions of the layer operators.}
Lemma~\ref{lem:axisymmetric-green} gives the exact factorization
\begin{equation}\label{eq:green-conjugation}
G_0(x,y)=w(x)^{-1/2}\mathcal K_0(x,y)w(y)^{-1/2}
\end{equation}
with $\mathcal K_0=\Gamma+d^2A\log d+B$.
After subtracting the first-order Taylor polynomial of $B$, the remainder
restricted to the circle and its normalized normal derivative satisfy
\eqref{eq:axisymmetric-circle-map}--
\eqref{eq:axisymmetric-circle-normal-map}.

Next expand the smooth conjugating factors:
\begin{equation}\label{eq:w-half-expansion}
w(x_0+ae_r(\theta))^{-1/2}
=w(x_0)^{-1/2}\bigl(1+a\alpha_1(\theta)+a^2\alpha_{2,a}(\theta)\bigr),
\end{equation}
where $\alpha_1\in P_1C^\infty(\T)$ and the family
$\alpha_{2,a}$ is uniformly bounded in every fixed $C^k$ norm.  The same Taylor expansion applies to the smooth boundary correction.  On the unit circle,
$\Gamma(a(e_r(\theta)-e_r(\varphi)))$ splits into the scalar
$-(2\pi)^{-1}\log a$ and the translation-invariant logarithmic kernel
$-(2\pi)^{-1}\log|e_r(\theta)-e_r(\varphi)|$.  Its multiplier on $e^{im\theta}$ is $(2|m|)^{-1}$ for
$m\ne0$, so $P_{\ne0}S_*P_{\ne0}=(2R)^{-1}|D_\theta|^{-1}P_{\ne0}$.
Combining this fact with
Lemma~\ref{lem:axisymmetric-green} and
\eqref{eq:w-half-expansion} gives
\eqref{eq:Sa-expansion}.  In the weighted remainder, the quadratic Taylor
terms of the conjugating factors multiply the logarithmic convolution,
which has order $-1$.  Smooth multiplication on either side preserves
its $H^r\to H^{r+1}$ bound, proving \eqref{eq:Sa-rem-map}.  On a fixed finite-dimensional smooth input space
the remainder is bounded into every fixed $H^k$, with the same explicit
$a^2L_a^{M_k}$ prefactor.  The family $S_{1,a}$ depends on $a$ because the logarithmic contribution
has size $a(\log a)$ and the smooth first Taylor term has size $O(a)$.
After division by $aL_a$, both contributions are uniformly bounded.
Every coefficient at first order in $a$ is a linear combination of $1$, $\cos\theta$ and $\sin\theta$ in either boundary variable.  Thus the first perturbation shifts the Fourier index by at most one.

For the reference logarithmic kernel, continuation of the nonzero
mode into the fluid as $r^{-|m|}e^{im\theta}$ gives
$T_*=P_0+\tfrac12P_{\ne0}$.  The constant mode follows by differentiating
$-R^{-1}\log r$.  More generally, integrating
$-\operatorname{div}(w\nabla u)=\mu\,d\theta$ across a thin strip around
the layer gives
\begin{equation*}
w|\partial_\theta X_\eta|
\bigl(\partial_nu\big|_{\rm fluid}-\partial_nu\big|_{\rm hole}\bigr)
=\mu,
\end{equation*}
where both derivatives use the fluid outward normal.  The logarithmic
potential is continuous across the layer.  Applying the same conjugation
and Taylor expansion to the normal derivative of
\eqref{eq:green-conjugation} proves
\eqref{eq:Ta-expansion}, \eqref{eq:Ta-rem-map}, and the same Fourier
selection rule for $T_{1,a}$.

\emph{Step 2. Invertibility and Sobolev bounds.}
Relative to $P_0\oplus P_1\oplus P_{\ge2}$, the $P_0P_0$ block of the single-layer operator has
size $c_{\rm lay}L_a+O(1)$.  The first-order off-diagonal blocks are $O(aL_a)$ and
shift the Fourier index by one. Couplings excluded by this selection rule are
$O(a^2L_a^{M_s})$.  We construct the inverse first at a fixed lower Sobolev index.
Put $Q=I-P_0$ and write, relative to $P_0L^2\oplus QL^2$,
\begin{equation*}
\mathcal S_a=
\begin{pmatrix}A_{00}&A_{0Q}\\ A_{Q0}&D_a\end{pmatrix},
\qquad D_a=Q\mathcal S_aQ.
\end{equation*}
As an operator from $QL^2$ to $QH^1$, the reference block on nonzero modes
has an inverse of order one.  Composing this inverse with the perturbation gives
an operator of norm $O(aL_a+a^2L_a^{M_0})$.  Choose $a_0$ so that the Neumann series
converges at this index.  The scalar Schur complement is
\begin{equation*}
s_a=A_{00}-A_{0Q}D_a^{-1}A_{Q0}
=c_{\rm lay}L_a+O(1)+O(a^2L_a^{M_0}),
\qquad s_a^{-1}=O(L_a^{-1}).
\end{equation*}
This proves invertibility at the lower Sobolev index.

Bounds at higher Sobolev indices follow from elliptic regularity for $D_a$
with the same upper bound on the radius.  The logarithmic
kernel and its exact positive conjugating factors give its principal
symbol
\begin{equation*}
\frac{1}{2w(x_0+ae_r(\theta))|\xi|},
\qquad |\xi|\ge1,
\end{equation*}
which is uniformly elliptic of order $-1$.  The removed zero mode
and the logarithmic rank-one terms are smoothing on $Q$. The latter
have size $O(a^2L_a)$ there, because $Q(w^{-1/2})=O(a)$.
The kernel expansions bound each required finite symbol seminorm
uniformly for $0<a<a_0$.  These bounds may depend on the higher geometric
norms and on $r$, while the threshold $a_0$ remains fixed.  An order-one elliptic
parametrix consequently gives, for each fixed $r\ge0$,
\begin{equation*}
\|\mu\|_{H^r}
\le C_r\bigl(\|D_a\mu\|_{H^{r+1}}+\|\mu\|_{L^2}\bigr),
\qquad Q\mu=\mu.
\end{equation*}
Combining this estimate with the inverse bound at the lower index proves
$\|D_a^{-1}f\|_{H^r}\le C_r\|f\|_{H^{r+1}}$.
The full block inverse has
\begin{equation*}
P_0\mathcal S_a^{-1}P_0=s_a^{-1},\qquad
Q\mathcal S_a^{-1}Q
=D_a^{-1}+D_a^{-1}A_{Q0}s_a^{-1}A_{0Q}D_a^{-1}.
\end{equation*}
The off-diagonal blocks have smooth range or are smooth scalar
functionals, with size $O(aL_a)$ in each fixed norm.
The block identities prove
\eqref{eq:Sa-inverse-zero}--\eqref{eq:Sa-inverse-nonzero}
for $0\le r\le s+1$, with $a_0$ independent of $r$.

\emph{Step 3. Estimates for the projected conormal map.}
To estimate the conormal map, combine these block identities with the
Fourier selection rule.  A first-order coupling maps $P_0$ only into
$P_1$.  A contribution in $P_{\ge2}$ requires two first-order couplings
or a second-order remainder.  A first-order coupling can map $P_1$ into $P_2$.
The inverse bounds at higher Sobolev indices preserve the explicit powers
of $a$ in these estimates.  The remaining powers of $L_a$ are
absorbed into $L_a^{M_s}$, proving
\eqref{eq:layer-P0-leakage}--\eqref{eq:layer-P1-leakage}.

\emph{Step 4. Graph variations.}
For a graph, the kernel factorization in
Lemma~\ref{lem:axisymmetric-green} preserves ellipticity of the
nonconstant block.  Its inverse is first obtained by perturbation
at the fixed lower index, then at higher indices by the same
elliptic estimate.  The coefficient of $-\log a$ is still the
rank-one kernel $w(X_\eta(\theta))^{-1/2}w(X_\eta(\varphi))^{-1/2}/(2\pi)$,
with $w(X_\eta)^{-1/2}=R^{-1/2}+O(a)$.  The other kernels are
uniformly bounded at the fixed lower index.  Thus
$A_{00}=L_a/R+O(1)$, $A_{0Q},A_{Q0}=O(1+aL_a)$, and
$s_a=L_a/R+O(1+a^2L_a^2)>0$ for small $a$.
The first state variation of the block inverse is given by
$D(S^{-1})[h]=-S^{-1}(DS[h])S^{-1}$.  Differentiating this identity
once more gives the second variation.  The differentiated kernels
are controlled by Lemma~\ref{lem:axisymmetric-green}, and the same
elliptic inverse estimate at the fixed index controls their compositions.
The conormal prefactor is the product of the arclength Jacobian and $w$,
so the density in the jump relation is still taken with respect to $d\theta$.
\end{proof}

The inner diagonal block of the global flux map, denoted by
$\mathfrak G_{--}^{\rm glob}$, satisfies
$\mathfrak G_{--}^{\rm glob}=R^{-1}\mathcal T_a\mathcal S_a^{-1}$.
Hence \eqref{eq:layer-P0-leakage}--\eqref{eq:layer-P1-leakage}
also hold for the global flux block, since $R$ is uniformly positive.

\subsection{Uniform elliptic estimates and state variations}

Throughout this subsection the macroscopic geometry satisfies
\begin{equation}\label{eq:macro-geometry-assumptions}
0<R_{\min}\le R\le R_{\max},\qquad \rho\ge\rho_*>0
\quad\hbox{in the fluid},
\end{equation}
where $R_{\min},R_{\max}$ are fixed positive constants, the outer
meridional boundary $\mathscr C_+$ belongs to a fixed bounded
smooth class, and
\begin{equation}\label{eq:outer-distance-assumption}
\operatorname{dist}(x_0,\mathscr C_+)\ge d_*>0.
\end{equation}
We use a fixed scaled collar at the inner boundary and a fixed physical
collar at the outer boundary.  For $\mathscr D_a=a^{-1}(\mathscr D-x_0)$,
$\mathscr C_+^{(a)}=a^{-1}(\mathscr C_+-x_0)$, and an outer trace $g_+$, put
\begin{equation*}
\|g_+\|_{H^m_{\rm ph}(\mathscr C_+^{(a)})}
:=\|g_+^{\rm ph}\|_{H^m(\mathscr C_+)},\qquad
g_+^{\rm ph}(x)=g_+((x-x_0)/a).
\end{equation*}

For prescribed global traces $(f,g_+)$, the zero-flux reduction
minimizes the weighted energy with outer trace $g_+$ and inner trace
$f+c_-$, where $c_-$ is free.  Its total inner flux is zero.
The Dirichlet solution with prescribed traces corresponds to $c_-=0$.

For the local problem, extend the normalized weight $\rho/R$ to equal one
outside a fixed scaled collar.  More precisely, choose radial cutoffs
$\chi=1$ near $r=1$, $\chi=0$ for $r\ge3/2$, and
$\chi_1=1$ for $r\le3$, $\chi_1=0$ for $r\ge4$.  On
$\mathcal A_*=\{r>1\}$ set
\begin{equation*}
\Theta_\eta(r,\theta)=(r+\chi(r)\eta(\theta))e_r(\theta),\qquad
e_r(\theta)=(\cos\theta,\sin\theta),
\end{equation*}
and, with $\varepsilon=a/R$,
\begin{equation}\label{eq:A-eps-eta}
A_{\varepsilon,\eta}
=\bigl(1+\varepsilon\chi_1(r)(\Theta_\eta)_1\bigr)
\det D\Theta_\eta\,(D\Theta_\eta)^{-1}(D\Theta_\eta)^{-T}.
\end{equation}
Thus $A_{\varepsilon,\eta}=I$ for $r\ge4$.  Let $u$ be the local solution with finite Dirichlet energy satisfying
$\operatorname{div}(A_{\varepsilon,\eta}\nabla u)=0$ in $\mathcal A_*$,
with $u|_{r=1}=f$.  Its constant limit at infinity is determined by the solution.
Equivalently, minimize the Dirichlet energy modulo constants with the
inner trace prescribed modulo constants, then choose the representative
with the prescribed inner trace.  The local flux map is
\begin{equation*}
\widehat{\mathfrak G}_{\varepsilon,\eta}f
=-e_r\cdot A_{\varepsilon,\eta}\nabla u\big|_{r=1}.
\end{equation*}
Global radius variations are taken in the relative direction
$a\partial_a$, and local parameter variations in the direction
$\partial_\varepsilon$.  At $\varepsilon=0$, $\eta=0$, the nonzero Fourier
mode $e^{im\theta}$ extends as $r^{-|m|}e^{im\theta}$, so the local
flux Dirichlet--Neumann map is $|D_\theta|$.

\begin{lemma}\label{lem:flattening-geometry}
Fix an integer $s_0\ge4$.  For $|\varepsilon|+\|\eta\|_{H^{s_0+2}}$
sufficiently small, $\Theta_\eta$ is a diffeomorphism onto the exterior
of the graph and
\begin{equation*}
\tfrac12|\xi|^2\le A_{\varepsilon,\eta}\xi\cdot\xi\le2|\xi|^2,
\qquad \xi\in\mathbb R^2.
\end{equation*}
Write $A=A_{\varepsilon,\eta}$.  Coefficient Sobolev norms below are
on $\mathcal A_*$ with planar Lebesgue measure. Graph and direction
norms are on $\T$.  For $s\ge s_0$ the coefficients satisfy
\begin{align*}
\|A-I\|_{H^{s+1}}
&\le C_s(|\varepsilon|+\|\eta\|_{H^{s+2}}),\\
\|A_h\|_{H^{s+1}}
&\le C_s\bigl(\|h\|_{H^{s+2}}
+(1+\|\eta\|_{H^{s+2}})\|h\|_{H^{s_0+2}}\bigr),
\\
\|A_{12}\|_{H^{s+1}}
&\le C_s\bigl(\|h_1\|_{H^{s+2}}\|h_2\|_{H^{s_0+2}}
+\|h_2\|_{H^{s+2}}\|h_1\|_{H^{s_0+2}}\\
&\hspace{19mm}+(1+\|\eta\|_{H^{s+2}})
\|h_1\|_{H^{s_0+2}}\|h_2\|_{H^{s_0+2}}\bigr),
\end{align*}
where $A_h=D_\eta A[h]$ and $A_{12}=D_\eta^2A[h_1,h_2]$.
The same bounds, with at most one factor at the higher Sobolev index,
hold for the first two variations in the physical outer charts.  Parameter derivatives of total order at
most two obey the corresponding estimates.
\end{lemma}
\begin{proof}
In orthonormal polar frames,
\begin{equation*}
D\Theta_\eta=\begin{pmatrix}
1+\chi'\eta&r^{-1}\chi\eta_\theta\\
0&1+r^{-1}\chi\eta
\end{pmatrix}.
\end{equation*}
Both diagonal entries stay positive, and the radial map is strictly
increasing for each angle and equals the identity for $r\ge3/2$.
This proves global invertibility and ellipticity.  The matrix entries
in \eqref{eq:A-eps-eta} are smooth rational functions of
$(\varepsilon,\eta,\eta_\theta)$ with uniformly positive denominators.
Differentiating once or twice and using the Sobolev product estimate,
with at most one factor at index $s+1$, proves the displayed bounds.
The physical outer charts have the same local form on a fixed collar.
The dependence on $\varepsilon$ is affine, and differentiated chart
coefficients have the same product structure.  Interpolation gives
noninteger indices with the same smallness threshold in the lower norm.
\end{proof}

\begin{lemma}
\label{lem:uniform-harmonic-extension}

Consider the global Dirichlet problem, its zero-flux reduction, or the
local exterior problem, after flattening.  Denote the flattened meridional domain by $\mathscr U$
and write $A$ for the uniformly elliptic matrix obtained from the
normalized weight $\rho/R$.  Bulk $L^2$ norms use planar Lebesgue
measure on $\mathscr U$.  For a vector field $F:\mathscr U\to\mathbb R^2$, the weak formulation of
$\operatorname{div}(A\nabla u)=\operatorname{div}F$ is
\begin{equation}\label{eq:large-domain-v}
\int_{\mathscr U} A\nabla u\cdot\nabla v\,dX\,dZ
=\int_{\mathscr U} F\cdot\nabla v\,dX\,dZ.
\end{equation}
The test functions have zero boundary traces in the Dirichlet case.
In the zero-flux case their inner trace is an arbitrary constant
and their outer trace is zero.  In the exterior case they belong to
the homogeneous energy space modulo constants and also have constant
inner trace.  In that case assume $F$ is supported in $r\le4$.  Then
\begin{equation}\label{eq:global-gradient-v}
\|\nabla u\|_2\le C\bigl(
\|F\|_2+\|f\|_{H^{1/2}}+
\|g_+\|_{H^1_{\rm ph}}\bigr),
\end{equation}
where the last term is omitted for the local problem.
For each integer $k\ge s_0$, if $A$ has uniformly controlled
$H^{k+1}$ coefficients on the inner collar and fixed lower ellipticity
and coefficient bounds, then
\begin{align}
&\|u-c_-\|_{H^{k+3/2}(1<r<3/2)}
+\|n\cdot A\nabla u|_{r=1}\|_{H^k}
\label{eq:large-domain-local-estimate}\\
&\quad\le C_k\bigl(
\|F\|_{H^{k+1/2}(1<r<3)}+\|F\|_2
+\|f\|_{H^{k+1}}+\|g_+\|_{H^1_{\rm ph}}\bigr).
\notag
\end{align}
Here $n=-e_r$ is the outward normal of the flattened fluid domain at
$r=1$, $c_-$ is the free inner constant, and $c_-=0$ when the
representative has prescribed inner trace $f$.  The corresponding
estimates hold in a fixed physical outer collar: the outer $H^k$
conormal norm is controlled by the physical $H^{k+1}$ norm of $g_+$,
the energy level $H^{1/2}$ norm of $f$, the global $L^2$ norm of $F$,
and its corresponding physical collar Sobolev norm, with constants
depending on the physical collar coefficient bounds.
The constants are independent of $a$.
\end{lemma}

\begin{proof}
Extend $f$ in a fixed scaled collar and $g_+$ in a fixed physical
outer collar, choosing disjoint supports.  Their gradient norms are
bounded by $C\|f\|_{H^{1/2}}$ and
$C\|g_+\|_{H^1_{\rm ph}}$: two-dimensional Dirichlet energy is
invariant under dilation.  Subtract these lifts in
\eqref{eq:large-domain-v} and test with the remainder.  In the
zero-flux problem, the inner constant of this admissible remainder may
be nonzero.  Uniform ellipticity proves
\eqref{eq:global-gradient-v} using only the lift bounds on the two fixed
collars.  The same argument constructs the solution in the
energy space.  For the exterior problem, the nonconstant trace is
continuous in the gradient norm by the collar trace and Poincar\'e
inequalities.  The minimizing gradient is unique.  Outside $r=4$
its harmonic continuation has finite energy, so its Fourier
expansion has a constant term and decaying nonzero modes, with no
logarithmic term.  Fixing the inner trace chooses the representative.
Subtracting its constant limit gives the equivalent representative that
vanishes at infinity.

Subtract $c_-$ and a collar lift of $f$ near the inner boundary.
The resulting function has zero inner trace, so the Poincar\'e
inequality on $1<r<3$ bounds its local $H^1$ norm by the global
gradient bound and the lift norm.  Nested cutoffs supported in this
fixed collar reduce the higher estimate to a uniformly elliptic
boundary problem.  Tangential differentiation gives estimates for the tangential
derivatives, and the divergence equation then recovers the normal derivatives.
The coefficient products require at most $H^{k+1/2}$ regularity
of $A$, supplied by its $H^{k+1}$ bound.  The fixed lower index
controls the Lipschitz norm of the coefficients.  Induction and interpolation
give the $H^{k+3/2}$ estimate in
\eqref{eq:large-domain-local-estimate}.  The flux trace follows from
$\nabla u\in H^{k+1/2}$, whose index is strictly above the trace
threshold.  On the outer collar the same argument is performed in
physical coordinates, subtracting a lift of $g_+$.  Its local
Poincar\'e constant is fixed.  The constants in these two collar estimates
are independent of the diameter of $\mathscr D_a$.
\end{proof}

\begin{proposition}\label{prop:full-C2-flux-DN}

Let $\boldsymbol\zeta=(\eta,z)$, where $\eta$ is the relative inner graph
and $z$ is the physical normal graph over a fixed smooth outer curve,
and let $\boldsymbol f=(f,g_+)$ be the trace pair.  For pairs, use the sum
of the scaled inner and physical outer Sobolev norms.  A graph
direction is $H=(h_-,h_+)$.  Omit the outer components for the local
map.  Let $\mathfrak G^\nu$ denote the global Dirichlet map, its
zero-flux reduction, or the local map $\widehat{\mathfrak G}$, with
$\nu$ indexing these three choices.  For $s\ge s_0$, put
$N_s=1+\|\boldsymbol\zeta\|_{H^{s+2}}$ and
$|H|_s=\|H\|_{H^{s+2}}$.
These maps are $C^2$ for graphs in $H^{s+2}$ and traces in $H^{s+1}$,
take values in $H^s$, and satisfy the uniform estimates
\begin{align}
\|\mathfrak G^\nu\boldsymbol f\|_{H^s}
&\le C_s\bigl(\|\boldsymbol f\|_{H^{s+1}}
+N_s\|\boldsymbol f\|_{H^{s_0+1}}\bigr),
\label{eq:G0-tame}\\
\|D\mathfrak G^\nu[H]\boldsymbol f\|_{H^s}
&\le C_s\bigl(|H|_{s_0}\|\boldsymbol f\|_{H^{s+1}}
+( |H|_s+N_s|H|_{s_0})\|\boldsymbol f\|_{H^{s_0+1}}\bigr),
\\
\|D^2\mathfrak G^\nu[H_1,H_2]\boldsymbol f\|_{H^s}
&\le C_s\bigl(|H_1|_{s_0}|H_2|_{s_0}\|\boldsymbol f\|_{H^{s+1}}
\notag\\
&\quad+(|H_1|_s|H_2|_{s_0}+|H_2|_s|H_1|_{s_0}
+N_s|H_1|_{s_0}|H_2|_{s_0})
\|\boldsymbol f\|_{H^{s_0+1}}\bigr).
\label{eq:D2G-final-tame}
\end{align}
The flux bounds are uniform in $a$.  The bounds for the free scalar trace may contain finite powers of $L_a$.  Mixed local $\varepsilon$ and graph
variations of total order at most two obey the same estimates, with
one direction norm omitted for each $\varepsilon$ derivative.
For simultaneous variations of the centre and relative radius, extend
the graph direction to $H=(h_-,h_+,\delta c,\delta\log a)$, where
$\delta c\in\mathbb R^2$ is the physical centre displacement, and use
\begin{equation*}
|H|_s=\|h_-\|_{H^{s+2}}+\|h_+\|_{H^{s+2}}
+|\delta c|+|\delta\log a|.
\end{equation*}
The global estimates hold for these extended directions.  Trace
variations are included by linearity.
\end{proposition}
\begin{proof}
We divide the proof into three steps.

\emph{Step 1. Differentiated equations and flux estimates.}
With fixed trace data in the reference coordinates, write $u_i=D u[H_i]$ and
$u_{12}=D^2u[H_1,H_2]$, and set $A_i=DA[H_i]$ and
$A_{12}=D^2A[H_1,H_2]$.  Coercivity in the fixed test space of
\eqref{eq:large-domain-v} allows differentiation and gives
\begin{align}
\operatorname{div}(A\nabla u_i)
&=-\operatorname{div}(A_i\nabla u),\label{eq:first-u-variation}\\
\operatorname{div}(A\nabla u_{12})
&=-\operatorname{div}(A_{12}\nabla u+A_1\nabla u_2+A_2\nabla u_1).
\label{eq:second-u-variation}
\end{align}
In the Dirichlet problem the differentiated traces vanish.  In the
zero-flux problem their inner traces are free constants, their outer
traces vanish, and the same equations are tested against functions
with arbitrary constant inner trace.  Thus the total conormal flux,
including the displayed coefficient sources, is zero at each order.
When the trace data vary, subtract their collar lifts.  Denote the inner
component by $\mathfrak G_-^\nu$, with
$\mathfrak G_-^\nu=\widehat{\mathfrak G}$ for the local map.
On the fixed inner boundary its differentiated normalized fluxes are
\begin{align*}
D\mathfrak G_-^\nu[H]\boldsymbol f
&=-e_r\cdot(A_H\nabla u+A\nabla u_H)|_{r=1},
\\
D^2\mathfrak G_-^\nu[H_1,H_2]\boldsymbol f
&=-e_r\cdot(A_{12}\nabla u+A_1\nabla u_2
+A_2\nabla u_1+A\nabla u_{12})|_{r=1}.
\end{align*}
Here $A_H=DA[H]$ and $u_H=Du[H]$.  Outer fluxes are evaluated in
physical charts.  Apply Lemma~\ref{lem:uniform-harmonic-extension}
first to $u$, then to
$u_i$ and $u_{12}$.  To retain tame dependence, commute the tangential
derivatives with the full operator $\operatorname{div}(A\nabla)$.
In the highest order commutator, the highest derivative of a coefficient
multiplies a lower derivative of $u$.  All other terms have
fewer derivatives of the solution.  The product bounds of
Lemma~\ref{lem:flattening-geometry}, induction, and recovery of normal
derivatives therefore require at most one factor at the higher Sobolev index.
The displayed flux identities give
\eqref{eq:G0-tame}--\eqref{eq:D2G-final-tame}.

\emph{Step 2. Variations of the global geometry.}
For the global problem use the inner scaled collar and a disjoint
physical outer collar.  A centre displacement is a translation near
the hole, cut off in a fixed neighborhood separated from the outer boundary.
A relative radius displacement has size $O(a)$ in an $O(a)$ collar.  Their coefficient
variations have uniform bounds in the respective charts.  Applying the same
energy tests and collar estimates, with constant trace test functions
for the zero-flux reduction, therefore gives the global formulas.
The smallness threshold in the lower norm remains fixed as $s$ increases.

\emph{Step 3. The free scalar trace.}
The free inner trace is recovered by
\begin{equation*}
c_-=-\frac{\langle e,\mathfrak G\boldsymbol f\rangle}
{\langle e,\mathfrak G e\rangle},\qquad e=(1,0).
\end{equation*}
The Schur complement estimate for the layer operator in
Lemma~\ref{lem:weighted-layer-expansion} gives $\langle e,\mathfrak G e\rangle\asymp L_a^{-1}$.
Differentiating this quotient twice introduces only finite powers
of $L_a$.  Finally, applying the same estimates to coefficient
differences proves continuity of both state differentials.
\end{proof}

\begin{lemma}
\label{lem:cross-block-small-hole}
Fix a Sobolev index $s\ge s_0$.  Assume that the outer meridional boundary remains in a uniformly bounded
$C^{s+4}$ class and stays a fixed positive distance from the centre point $x_0$.
For the inner circle $X_-(\theta)=x_0+ae_r(\theta)$, let $a\downarrow0$
and let $\mathfrak G_{-+}^{(a)}$ denote the block of the flux DN matrix
that maps the outer trace to the inner flux.  Then, after projecting away the
inner Fourier modes $0,\pm1$,
\begin{equation}\label{eq:cross-block-a2}
\|P_{\ge2}\mathfrak G_{-+}^{(a)}g\|_{H^s(\T)}
\le C_s a^2L_a^{M_s}\|g\|_{H^{s+5/2}(\mathscr C_+)}.
\end{equation}
For a centred perturbed circle, write $\mathfrak G_{-+}^{a,\eta}$ for the
corresponding block, so that $\mathfrak G_{-+}^{a,0}=\mathfrak G_{-+}^{(a)}$.
Then
\begin{align}
P_{\ge2}\mathfrak G_{-+}^{a,\eta}g
&=P_{\ge2}\mathfrak G_{-+}^{a,0}g+\mathcal E_{a,\eta}g,
\label{eq:cross-graph-corrected}\\
\mathcal E_{a,\eta}g
&:=\int_0^1P_{\ge2}D_\eta\mathfrak G_{-+}^{a,t\eta}[\eta]g\,dt.
\notag
\end{align}
The circular term satisfies \eqref{eq:cross-block-a2}, while the graph
term $\mathcal E_{a,\eta}$ is controlled by the tame estimates for shape derivatives.
\end{lemma}

\begin{proof}
Let $W$ solve the filled macroscopic Dirichlet problem with outer datum $g$.
Interior elliptic regularity and first-order Taylor expansion at $x_0$ give
\begin{align*}
W(x_0+ae_r(\theta))
&=W(x_0)+a\ell_1(\theta)+r_a(\theta),
\\
a\,w(x_0+ae_r(\theta))\partial_{e_r}W(x_0+ae_r(\theta))
&=a\widetilde\ell_1(\theta)+\widetilde r_a(\theta),
\end{align*}
where $\ell_1,\widetilde\ell_1\in P_1L^2(\T)$ and
\begin{equation*}
\|r_a\|_{H^{s+2}}+\|\widetilde r_a\|_{H^{s+1}}
\le C_sa^2\|g\|_{H^{s+5/2}(\mathscr C_+)}.
\end{equation*}
The perforated solution is $W+u$.  The single-layer correction is given by
\begin{equation*}
u(x)=\int_{\T}G_0(x,x_0+ae_r(\varphi))\mu(\varphi)\,d\varphi,
\qquad
\mathcal S_a\mu=-W\circ X_-.
\end{equation*}
The correction has zero outer trace.  Its normalized inner conormal trace is
$\mathcal T_a\mathcal S_a^{-1}(-W\circ X_-)$.
Apply \eqref{eq:layer-P0-leakage} to the constant term and
$a$ times \eqref{eq:layer-P1-leakage} to the first harmonic.
For $r_a$, separate its mean and apply the same constant mode bound
together with the uniform bound of order one on $P_{\ne0}$.  This gives
\begin{equation*}
\|P_{\ge2}\bigl((a w\partial_{n_-}u)\circ X_-\bigr)\|_{H^s}
\le C_sa^2L_a^{M_s}\|g\|_{H^{s+5/2}(\mathscr C_+)}.
\end{equation*}
The direct trace has the same bound by its Taylor expansion.
Adding both traces and dividing by $R$ proves \eqref{eq:cross-block-a2}.
Since $a=\ell e^{-Y}$,
$L_a^{M_s}\le C_s(1+Y)^{M_s}$.

Integrating the derivative from Proposition~\ref{prop:full-C2-flux-DN}
along $t\eta$ gives \eqref{eq:cross-graph-corrected} and its tame bound.
\end{proof}

Let
\begin{equation*}
r(\theta)=1+\eta(\theta),
\qquad
j(\theta)=\sqrt{r^2+r_\theta^2}.
\end{equation*}
For the normal pointing away from the centre circle, the dimensionless mean curvature of the surface of revolution is
\begin{equation}\label{eq:explicit-dimensionless-curvature}
\mathcal J_\varepsilon(\eta)
=\frac{r^2+2r_\theta^2-r r_{\theta\theta}}{j^3}
+\frac{\varepsilon\bigl(r\cos\theta+r_\theta\sin\theta\bigr)}
{j\bigl(1+\varepsilon r\cos\theta\bigr)}.
\end{equation}
Using the inner fluid normal reverses the sign and leaves the bounds in
Proposition~\ref{prop:full-C2-curvature} unchanged.

\begin{proposition}\label{prop:full-C2-curvature}
Under the hypotheses of Lemma~\ref{lem:flattening-geometry},
$\mathcal J_\varepsilon:H^{s+2}\to H^s$ is $C^\infty$ and
\begin{align}
\|D\mathcal J_\varepsilon(\eta)[h]\|_{H^s}
\le C_s\Bigl[
&\|h\|_{H^{s+2}}
+\|\eta\|_{H^{s+2}}\|h\|_{H^{s_0+2}}
\Bigr],\label{eq:DJ-final}\\
\|D^2\mathcal J_\varepsilon(\eta)[h_1,h_2]\|_{H^s}
\le C_s\Bigl[&
\|h_1\|_{H^{s+2}}\|h_2\|_{H^{s_0+2}}
+\|h_2\|_{H^{s+2}}\|h_1\|_{H^{s_0+2}}\notag \\
&+\|\eta\|_{H^{s+2}}
\|h_1\|_{H^{s_0+2}}\|h_2\|_{H^{s_0+2}}
\Bigr].\label{eq:D2J-final}
\end{align}
All mixed $\varepsilon$ derivatives through total order two obey the same tame bounds.  At the cylindrical state,
\begin{equation*}
D\mathcal J_0(0)[h]=-(\partial_\theta^2+1)h,
\end{equation*}
and
\begin{equation}\label{eq:curvature-eps-selection}
\left.\partial_\varepsilon\mathcal J_\varepsilon(0)\right|_{\varepsilon=0}
=\cos\theta.
\end{equation}
Thus the first correction in $\varepsilon$ at the circular state lies entirely in the first Fourier mode.
\end{proposition}

\begin{proof}
The expression \eqref{eq:explicit-dimensionless-curvature} is smooth
in $(\varepsilon,\eta,\eta_\theta,\eta_{\theta\theta})$ when the lower norm is sufficiently small, with denominators uniformly
bounded away from zero.
Tame composition and product estimates give
\eqref{eq:DJ-final} and \eqref{eq:D2J-final}.
Differentiating the planar polar curvature at $r=1$ gives
$D\mathcal J_0(0)[h]=-(h+h_{\theta\theta})$, while
\begin{equation*}
\mathcal J_\varepsilon(0)
=1+\frac{\varepsilon\cos\theta}{1+\varepsilon\cos\theta}
\end{equation*}
gives \eqref{eq:curvature-eps-selection}.
\end{proof}

For the outer energy estimates, we also need a separated boundary
estimate in which only the outer graph is measured at the output
regularity.

\begin{lemma} 
\label{lem:zero-separated-asymmetric}
Let $c$ be the inner centre in the meridional plane.  Suppose that
the inner boundary is $c+a(1+\eta(\theta))e_r(\theta)$ and the outer boundary is a
normal graph $z$ over a fixed smooth curve.  Assume that $c$ and both boundaries remain in the fixed separated
geometry class, that the cylindrical weight is bounded above and below by
positive constants, and that the fixed lower graph norms are small.
There are fixed integers $b,d$, independent of the output Sobolev index
$r\ge b$ and of $a$, such that the following estimates hold.

For an inner Dirichlet trace $f$, let $u_f$ solve the weighted harmonic
problem with inner trace $f$ and zero outer trace.  Define
\begin{equation*}
\mathsf G_{+-}(\eta,z)f
=\bigl(\partial_{n_+}u_f\bigr)\circ X_+ .
\end{equation*}
The same conclusions hold when the inner constant is chosen
by the zero-flux condition, and for the constant trace lifts
obtained by the scalar Schur complement.  For those lifts, the estimates may include a polynomial
factor $\mathcal{P}_r(Y)$, $Y=\log(\ell/a)$.  The normalization
of a prescribed flux is kept as an explicit scalar factor.

Put $A_r(z)=1+\|z\|_{H^{r+d}}$.  Then
\begin{equation*}
\|\mathsf G_{+-}(\eta,z)f\|_{H^r}
\le C_r\mathcal{P}_r(Y) A_r(z)\|f\|_{H^{1/2}}.
\end{equation*}
The inner graph enters only through its fixed lower norms.  To state
the variations, write $V_j=(\eta_j,z_j,f_j)$,
$\nu_j=\|\eta_j\|_{H^b}+\|z_j\|_{H^b}$, and
$F_j=\|f_j\|_{H^{1/2}}$, including $F_0=\|f\|_{H^{1/2}}$.
The first and second differentials satisfy
\begin{align}
\|D(\mathsf G_{+-}f)[V_1]\|_{H^r}
&\le C_r\mathcal{P}_r\bigl[
A_r(z)(F_1+F_0\nu_1)
+\|z_1\|_{H^{r+d}}F_0\bigr],
\label{eq:zero-separated-asymmetric-first}\\
\|D^2(\mathsf G_{+-}f)[V_1,V_2]\|_{H^r}
&\le C_r\mathcal{P}_r\bigl[
A_r(z)(F_0\nu_1\nu_2+F_1\nu_2+F_2\nu_1)\notag\\
&\hspace{35mm}
+\|z_1\|_{H^{r+d}}(F_0\nu_2+F_2)
+\|z_2\|_{H^{r+d}}(F_0\nu_1+F_1)\bigr].
\label{eq:zero-separated-asymmetric-second}
\end{align}
The constants depend on the fixed lower graph norms and the reference
geometry.  For centre and relative scale variations, extend
$V_j$ by $(\delta_jc,\delta_j\log a)$ and add
$|\delta_jc|+|\delta_j\log a|$ to $\nu_j$.
The same polynomial dependence on $Y$ holds.

The same assertions hold for the correction to the filled domain outer
Dirichlet--Neumann map caused by the hole.  Here the input may be measured
in a fixed $H^1$ outer trace norm.  Separation of the boundary components
gives estimates for the full kernel of this correction.

Here is a quantitative version for negative input indices. Fix an
integer $N_1\ge0$. Let $T_{+-}$ be the zero-flux inner-to-outer map,
and let $K_{++}$ be the zero-flux outer DN block minus the DN map of
the filled domain, both in the fixed boundary coordinates. For either
operator write $k_T(\boldsymbol\zeta;\theta,\varphi)$ for its kernel,
where $\boldsymbol\zeta=(\eta,z)$. There are fixed $b,d$, now also
depending on $N_1$ but not on $r$, such that, with
$\|\cdot\|_{r,N_1}=\|\cdot\|_{H^r_\theta H^{N_1}_\varphi}$,
\begin{align}
\|k_T(\boldsymbol\zeta)\|_{r,N_1}
&\le C_r\mathcal P_r(Y)(1+\|\boldsymbol\zeta\|_{H^{r+d}}),
\notag\\
\|Dk_T(\boldsymbol\zeta)[V]\|_{r,N_1}
&\le C_r\mathcal P_r(Y)\bigl[
\|V\|_{H^{r+d}}+
(1+\|\boldsymbol\zeta\|_{H^{r+d}})\|V\|_{H^b}\bigr],
\notag\\
\|D^2k_T(\boldsymbol\zeta)[V,W]\|_{r,N_1}
&\le C_r\mathcal P_r(Y)\bigl[
\|V\|_{H^{r+d}}\|W\|_{H^b}
+\|W\|_{H^{r+d}}\|V\|_{H^b}\notag\\
&\hspace{19mm}
+(1+\|\boldsymbol\zeta\|_{H^{r+d}})
\|V\|_{H^b}\|W\|_{H^b}\bigr].
\label{eq:outer-kernels-C2-negative-input}
\end{align}
These kernel estimates use the full graph norm on the right.
Increasing $d$ to include $b$, integration along $t\boldsymbol\zeta$
on the lower norm ball gives
\begin{equation}\label{eq:outer-kernel-negative-input-difference}
\|[T(\boldsymbol\zeta)-T(0)]v\|_{H^r}
\le C_r\mathcal P_r(Y)\|\boldsymbol\zeta\|_{H^{r+d}}
\|v\|_{H^{-N_1}}.
\end{equation}
The same bounds hold after the fixed normal and half-density factors
and finite projections are included.
\end{lemma}

\begin{proof}
We divide the proof into four steps.

\emph{Step 1. Energy estimates for the variations.}
Choose disjoint deformation collars: an inner collar of thickness
comparable to $a$, and an outer collar of fixed thickness.  The inner pullback is the identity outside its collar.  An inner graph
variation has physical size $a\eta_j$. The derivative of its cutoff
has size $a^{-1}$, and these factors cancel.  Consequently the
pulled-back elliptic matrices and their first two variations obey
uniform $L^\infty$ bounds in terms of the fixed $H^b$ norms in the
statement.  The outer matrix variations have the same bounds on a
fixed collar.  To vary the centre, use a translation near the hole, cut off in a
fixed neighborhood separated from the outer boundary.  Its coefficient variations are also
uniformly bounded.  A variation of $\log a$ has inner physical
size comparable to $a$ and is handled as an inner graph variation.

Apply the uniform gradient estimate \eqref{eq:global-gradient-v} to
$u_f$ and the material variations
$u_j=Du_f[V_j]$, $u_{12}=D^2u_f[V_1,V_2]$, which satisfy
\eqref{eq:first-u-variation}--\eqref{eq:second-u-variation}.
In the Dirichlet case, the first variation has inner trace $f_j$.
The mixed variation has zero trace, since the trace variable is linear.  The coefficient
bounds just established give
\begin{equation*} 
\|\nabla u_f\|_2\le CF_0,\qquad
\|\nabla u_j\|_2\le C(F_j+F_0\nu_j),\qquad
\|\nabla u_{12}\|_2\le
C(F_0\nu_1\nu_2+F_1\nu_2+F_2\nu_1). 
\end{equation*}
These estimates use only $H^{1/2}$ inner traces and fixed lower graph
norms.  The zero outer trace also gives a uniform collar Poincar\'e bound.

\emph{Step 2. Estimates for the outer conormal trace.}
For an inner variation the coefficient sources vanish on the
outer collar, whose coefficients depend only on $z$.  Boundary
regularity on this collar controls the output derivatives by the
energy bounds for $u_f,u_j,u_{12}$.  Tame estimates for derivatives of the collar coefficients and conormal
trace require only $z$ or $z_j$ to be measured in the higher norm $H^{r+d}$.
The inner factors remain in the fixed energy norms.  The fixed loss $d$ accounts for flattening, the
conormal trace, and Sobolev products.  For an outer variation the
sources lie in that collar, and the same tame estimate requires the higher norm only for
the outer coefficient or its variation.  These are
exactly the alternatives in
\eqref{eq:zero-separated-asymmetric-first}--
\eqref{eq:zero-separated-asymmetric-second}.

The same gradient estimates apply in the zero-flux variational
space.  For the constant trace lifts, the scalar inverse in the
Schur complement formula of Lemma~\ref{lem:weighted-layer-expansion}
and its first two variations contribute only polynomial factors
$\mathcal{P}_r(Y)$.

\emph{Step 3. The correction caused by the hole.}
For the last assertion, first extend the outer trace harmonically
in the filled domain.  Its energy is controlled by a fixed $H^1$
trace norm.  Its values and every fixed physical derivative in an
inner neighborhood are controlled by interior elliptic regularity,
using only the fixed lower outer geometry norms.  Subtract this
extension.  The correction has zero outer trace and an inner trace
controlled in $H^{1/2}$ by the same fixed input norm.  The estimates
already proved apply to this correction.  Differentiating the
problem in the filled domain and the correction once or twice gives the
same distribution of derivatives, with at most one factor at high regularity.  The derivatives of the output
are controlled by the outer geometry at the corresponding order and the
inner geometry at the fixed lower order.
Estimating the correction through its kernel on the separated boundaries
preserves these regularity requirements.

\emph{Step 4. Negative inputs and mixed kernel norms.}
Set $m=N_1+1$. First estimate the potential in a smaller output collar
in $H^1$ by an $H^{-m}$ norm of its input. Test its gradient against
an $L^2$ vector field supported in that collar and solve the adjoint
energy problem with the resulting divergence source. For $T_{+-}$,
the adjoint solution is homogeneous near the inner input boundary.
Fixed-order boundary regularity bounds its pulled-back conormal
density in $H^m$ by the $L^2$ norm of the test field. The inner density
includes its arclength factor, so the scaled estimate introduces no
factor $a^{-1}$. Green's identity then proves the desired negative-input
estimate. The zero outer trace gives the collar Poincar\'e estimate.

For $K_{++}$ use the difference of the perforated and filled adjoint
solutions, with the same source in their common outer collar. The
sources and boundary traces cancel there. Their difference is
homogeneous with zero outer trace, so the same argument gives its
$H^m$ conormal bound. Each adjoint gradient is bounded by the energy
test. The free inner constant is handled in the zero-flux test space;
its recovery contributes only the polynomial Schur factors.

Differentiate these adjoint problems once and twice. Their equations
are \eqref{eq:first-u-variation}--\eqref{eq:second-u-variation}.
In the common collar, the differentiated differences have sources
formed from coefficient variations and the already smooth differences.
The energy estimate followed by boundary regularity of order $m$
therefore uses only a fixed graph index $b_m$ and the direction norms
at that index. This proves the $H^{-m}\to H^1$ output-collar bounds
also for the first two variations, with the corresponding products
of lower direction norms.

Let $t_T$ denote this potential kernel. Applying these bounds to
$e^{in\varphi}$ gives, at the lower output index,
\begin{equation}\label{eq:negative-input-kernel-summation}
\sum_{n\in\mathbb Z}\langle n\rangle^{2N_1}
\|T_{\rm pot}e^{in\varphi}\|_{H^1({\rm collar})}^2
\le C\mathcal P(Y)^2
\sum_{n\in\mathbb Z}\langle n\rangle^{-2}<\infty.
\end{equation}
The same summation applies to both variations. Thus the input
derivatives of $t_T$ through order $N_1$ are controlled in
$H^1({\rm collar};L^2_\varphi)$.
Apply the output boundary estimates to these Hilbert-space-valued
equations. At zero, one and two variations the coefficient sources are,
respectively, zero, $-\operatorname{div}(A_V\nabla t_T)$, and
\[
-\operatorname{div}\bigl(A_{VW}\nabla t_T
+A_V\nabla Dt_T[W]+A_W\nabla Dt_T[V]\bigr).
\]
The product bounds in Lemma~\ref{lem:flattening-geometry} put at most
one graph or direction at the higher index. Taking the conormal trace,
including its differentiated coefficient, proves
\eqref{eq:outer-kernels-C2-negative-input}. The loss $d$ includes
$b_m$, the conormal trace, and the fixed product losses, and is
independent of $r$. Integration along $t\boldsymbol\zeta$ and pairing
in the input variable prove
\eqref{eq:outer-kernel-negative-input-difference}.
\end{proof}

\subsection{Zero-flux reduction and condenser capacity}

The scalar flux constraint is enforced by a constant
lift depending on the geometry.  Differentiating its Schur complement preserves the symmetry
of the energy form.

\begin{proposition}
\label{prop:schur-differentiation}
Use the flux DN matrix $\mathfrak G$ on the direct sum of the two boundary
trace spaces, with the pairing defined above in fixed charts. Adjoints
are taken in this pairing.  Let $e=(1,0)$,
$c=\langle e,\mathfrak G e\rangle>0$, and define
\begin{equation*}
\Pi=I-e c^{-1}\langle e,\mathfrak G\,\cdot\rangle,
\qquad K=\mathfrak G-\mathfrak G e c^{-1}\langle e,\mathfrak G\,\cdot\rangle.
\end{equation*}
For a $C^2$ family of geometries, take the first and second derivatives
in the fixed charts along directions $H_1,H_2$, and denote them by
$\mathfrak G_i,\mathfrak G_{12}$ and $K_i,K_{12}$.
Then, on smooth trace pairs,
\begin{align}
K&=\Pi^*\mathfrak G\Pi=\mathfrak G\Pi=\Pi^*\mathfrak G,\qquad K\ge0,
\label{eq:schur-positive-factorization}\\
K_i&=\Pi^*\mathfrak G_i\Pi,
\label{eq:schur-first-differential}\\
K_{12}&=\Pi^*\mathfrak G_{12}\Pi
-\Pi^*\mathfrak G_1e c^{-1}\langle e,\mathfrak G_2\Pi\,\cdot\rangle
-\Pi^*\mathfrak G_2e c^{-1}\langle e,\mathfrak G_1\Pi\,\cdot\rangle.
\label{eq:schur-second-differential}
\end{align}
Thus the first derivative is the unreduced Hadamard form evaluated on
the two zero-flux lifts.  The second derivative also contains the two
rank-one corrections in \eqref{eq:schur-second-differential}.

For a normal geometric variation $\boldsymbol H=(H_-,H_+)$, keep the
normalization $R$ fixed when taking this directional derivative.
Let $f,g$ be smooth real trace pairs and let $u_f,u_g$ be the harmonic
extensions of $\Pi f,\Pi g$.  For $\iota\in\{-,+\}$, write
$B_{f,\iota}=(\partial_{n_\iota}u_f)\circ X_\iota$ and
$V_{f,\iota}=j_\iota^{-1}\partial_\theta(\Pi f)_\iota$,
with the analogous notation for $g$.  Then
\begin{equation}\label{eq:schur-Hadamard}
\langle f,DK[\boldsymbol H]g\rangle
=\sum_{\iota\in\{-,+\}}\int_\T
H_\iota\frac{\rho_\iota}{R}j_\iota
(V_{f,\iota}V_{g,\iota}-B_{f,\iota}B_{g,\iota})
\,d\theta.
\end{equation}
If the scalar normalization $R$ also varies, its explicit derivative in
$\mathfrak G$ is included in $\mathfrak G_i$ in \eqref{eq:schur-first-differential}.
\end{proposition}
\begin{proof}
The identities $\Pi e=0$ and $\langle e,\mathfrak G\Pi f\rangle=0$ give
\eqref{eq:schur-positive-factorization}. Positivity also follows from
\begin{equation*}
\langle f,Kf\rangle
=\min_{z\in\mathbb R}\langle f+ze,\mathfrak G(f+ze)\rangle,
\qquad z=-c^{-1}\langle e,\mathfrak G f\rangle.
\end{equation*}
For complex traces, use complex scalars and the Hermitian pairing.
Differentiation gives
\begin{equation*}
\Pi_i=-e c^{-1}\langle e,\mathfrak G_i\Pi\,\cdot\rangle,
\qquad \Pi_i^*\mathfrak G\Pi=\Pi^*\mathfrak G\Pi_i=0.
\end{equation*}
Thus differentiating $K=\Pi^*\mathfrak G\Pi$ proves
\eqref{eq:schur-first-differential}. Differentiating once more and
substituting $\Pi_2$ and its adjoint gives
\eqref{eq:schur-second-differential} with the order of factors unchanged.

Apply the flux Hadamard identity to the fixed traces $\Pi f,\Pi g$.
Symmetry of $\mathfrak G$ gives the term $-\int H(\rho/R)jB_fB_g$.
Tangential integration by parts gives $\int H(\rho/R)jV_fV_g$.
This proves \eqref{eq:schur-Hadamard}.  The derivatives of the lifts cancel in the operator identity.
\end{proof}

We next apply the layer estimates to the condenser capacity and its
normalized flux.

\begin{lemma} 
\label{lem:capacity-single-layer-schur}
Let $\Omega$ be the three-dimensional domain obtained by revolving
$\mathscr D$ about the axis, and let $u_c$ be harmonic in $\Omega$
with traces $1,0$ on its inner and outer boundaries.  Set
$\operatorname{Cap}=\int_\Omega|\nabla u_c|^2\,dx$, with three-dimensional
Lebesgue measure $dx$.  Write
$j_\eta=((1+\eta)^2+\eta_\theta^2)^{1/2}$.
Let the macroscopic geometry satisfy
\eqref{eq:macro-geometry-assumptions}--\eqref{eq:outer-distance-assumption},
and assume that the centred relative graph satisfies $P_0\eta=P_1\eta=0$
and belongs to a fixed sufficiently small ball in $H^{s_0+2}$.
For $Y=\log(\ell/a)$ sufficiently large,
\begin{equation}\label{eq:collar-zero-schur}
\frac{2\pi L}{\operatorname{Cap}}
=Y+c_\infty(R,\Gamma_+)+\mathfrak c_{\rm rem},
\qquad L=2\pi R.
\end{equation}
Here $c_\infty(R,\Gamma_+)$ is the limit of
$2\pi L/\operatorname{Cap}-\log(\ell/a)$ as $a\downarrow0$ with
$\eta=0$ and $(R,\Gamma_+)$ fixed.  This is the smooth Robin functional
of the weighted Green function in the filled domain. Here $\ell$ is fixed.
For an angular function $f$, write
$f_0=(2\pi)^{-1}\int_\T f(\theta)\,d\theta$.
If the geometry is invariant under $(\rho,z)\mapsto(\rho,-z)$, put
$b_{\rm cap}=j_\eta(1+\eta)^{-1}\partial_{n_-}u_c$.  Then
\begin{equation}\label{eq:condenser-flux-expansion-v2}
\frac{b_{\rm cap}}{(b_{\rm cap})_0}
=1+\beta_a\cos\theta+\mathfrak b_{\ge2},
\qquad P_0\mathfrak b_{\ge2}=P_1\mathfrak b_{\ge2}=0.
\end{equation}
Fix a spatial index $s\ge s_0$.  Along a smooth path with $a_Y=-a$, set
\begin{equation*}
\mathcal J=\|\eta_Y\|_{H^{s_0+2}}+|R_Y|
+\|X_{+,Y}\|_{H^{s_0+2}},\qquad
\mathcal M_s=|R_Y|+\|X_{+,Y}\|_{H^{s+2}}.
\end{equation*}
On a fixed bounded geometric class, for a polynomial $\mathcal P_s(Y)$,
\begin{align}
|\mathfrak c_{\rm rem}|
&\le \mathcal P_s(Y)\bigl(a+\|\eta\|_{H^{s+2}}\bigr),\notag\\
|(\mathfrak c_{\rm rem})_Y|
&\le \mathcal P_s(Y)(1+\mathcal J)
\bigl(a+\|\eta\|_{H^{s+2}}+\|\eta_Y\|_{H^{s+2}}
+\mathcal M_s\bigr),
\label{eq:collar-c-rem-jet}\\
\|\mathfrak b_{\ge2}\|_{H^s}
&\le \mathcal P_s(Y)\bigl(a^2+\|\eta\|_{H^{s+2}}\bigr),\notag\\
\|(\mathfrak b_{\ge2})_Y\|_{H^s}
&\le \mathcal P_s(Y)(1+\mathcal J)
\bigl(a^2+\|\eta\|_{H^{s+2}}+\|\eta_Y\|_{H^{s+2}}
+a\mathcal M_s\bigr).
\label{eq:collar-flux-rem-jet}
\end{align}
The coefficient $\beta_a$ and its derivative $(\beta_a)_Y$ satisfy the
bounds in \eqref{eq:collar-flux-rem-jet}, with the $H^s$ norms replaced by
absolute values and $a^2$ replaced by $a$.  The operator families and the
scalar resistance $2\pi L/\operatorname{Cap}$ are $C^2$ in the geometric
variables on the Sobolev scales used above.
In particular,
$\mathfrak c=2\pi L/\operatorname{Cap}-Y$ is uniformly bounded when
$\|\eta\|_{H^{s_0+2}}\le K(1+Y)^{M_0}a^2$ for fixed $K,M_0\ge0$.
\end{lemma}

\begin{proof}
We divide the proof into three steps.

\emph{Step 1. Capacity and the Schur complement.}
Let $X_-(\theta)=x_0+a(1+\eta(\theta))e_r(\theta)$ and define
\begin{equation*}
(\mathcal S_{a,\eta}\mu)(\theta)
=\int_\T G_0(X_-(\theta),X_-(\varphi))\mu(\varphi)\,d\varphi.
\end{equation*}
The density is taken with respect to $d\theta$ and absorbs the arclength
factor.  In this normalization, the jump of the normal derivative is
weighted by $\rho|\partial_\theta X_-|$.
Use the block notation of Lemma~\ref{lem:weighted-layer-expansion},
with $Q=I-P_0$, $D=Q\mathcal S_{a,\eta}Q$,
and $s_a=A_{00}-A_{0Q}D^{-1}A_{Q0}$.
That lemma gives the graph inverse, its first two state variations,
and $s_a^{-1}=O(Y^{-1})$.
Proposition~\ref{prop:full-C2-flux-DN} supplies the tame flux
and capacity variations at the stated Sobolev indices.

Solve $\mathcal S_{a,\eta}\mu=1$.  The layer potential equals one in
the hole by uniqueness for the interior Dirichlet problem.  Its interior
conormal trace vanishes, so the jump relation gives
\begin{equation*}
\mu=\rho a j_\eta\partial_{n_-}u_c,
\qquad \operatorname{Cap}=2\pi\int_\T\mu\,d\theta=4\pi^2\mu_0.
\end{equation*}
Block elimination gives $\mu_0=s_a^{-1}$ and hence the exact identity
\begin{equation*}
\frac{2\pi L}{\operatorname{Cap}}=Rs_a.
\end{equation*}
For $\eta=0$, the mean block is $R^{-1}Y$ plus a smooth
macroscopic term.  The off-diagonal blocks are $O(aY)$.  Their composition
through $D^{-1}$ is $O(a^2\mathcal P_s(Y))$.
Thus the layer expansion gives \eqref{eq:collar-zero-schur} at a circle.
Integrating the first shape derivative along $t\eta$, $0\le t\le1$,
adds at most $\mathcal P_s(Y)\|\eta\|_{H^{s+2}}$.

\emph{Step 2. Expansion of the normalized flux.}
At a circle, apply \eqref{eq:layer-P0-leakage} to the constant
condenser datum, together with the first-order block expansion of
Lemma~\ref{lem:weighted-layer-expansion}.
Dividing the jump identity by $\rho a j_\eta$ and multiplying by the
kinematic factor $j_\eta/(1+\eta)$ gives
\eqref{eq:condenser-flux-expansion-v2} and the stated bound for the higher Fourier modes.
The Hopf lemma gives $\mu>0$, and the jump identity implies
$(b_{\rm cap})_0\asymp (aY)^{-1}$.  The common factor $a^{-1}$
cancels in the normalized ratio.  The graph contribution follows by
integration of the shape variation.  Reflection symmetry removes
the sine coefficient.

\emph{Step 3. State and scale derivatives.}
The block bounds and inverse variations in
Lemma~\ref{lem:weighted-layer-expansion} apply to $s_a$.
The $a\partial_a$ kernel estimates in Lemma~\ref{lem:axisymmetric-green}
and the same inverse formulas control the relative radius derivative.
For the normalized flux, differentiate $\mathcal S_{a,\eta}\mu=1$:
\begin{equation*}
\dot\mu=-\mathcal S_{a,\eta}^{-1}
\dot{\mathcal S}_{a,\eta}\mu.
\end{equation*}
Differentiation with respect to the relative radius preserves each power
of $a$, because
$a\partial_a[a^\nu P(Y)]=a^\nu(\nu P-P')$.
At $\eta=0$, macroscopic differentiation preserves the first-order
selection rule and the factor $a^2$ in the higher Fourier modes.
For the graph difference, differentiate its integral along $t\eta$.
The resulting terms contain $\eta_Y$, or $\eta$ times a first
geometric variation, and are controlled by the second state
differential.  The tame product estimates require at most one factor at index $s+2$.
This proves \eqref{eq:collar-c-rem-jet}--\eqref{eq:collar-flux-rem-jet}
and the bounds for the first coefficient.  The $C^2$ layer bounds
and nonvanishing mean give the asserted state dependence.  The estimate
for $\mathfrak c$ follows from \eqref{eq:collar-zero-schur}.
\end{proof}

\section{Boundary dynamics and linearized energy estimates}
\label{sec:boundary-dynamics}

We write the Euler equations in boundary variables and recover the flux
from energy conservation.  In centred polar coordinates, we estimate
the first potential harmonic and determine the Taylor coefficients by
reconstructing the pressure.  We then derive the linearized system and
prove the energy and terminal estimates used in the nonlinear construction.

\subsection{Boundary dynamics and energy reconstruction}

Throughout this section, $v=\nabla\Phi$ for a single-valued
axisymmetric potential $\Phi$.

Let $X_\iota(t,\theta)$, $\iota\in\{-,+\}$, parametrize the two meridional boundary curves.  Using the arclength density $j_\iota$ introduced in Section~\ref{sec:scaled-DN}, set
\begin{equation*}
\tau_\iota=j_\iota^{-1}\partial_\theta X_\iota.
\end{equation*}
Let $n_\iota$ be the outer unit normal of the fluid meridional domain.  Write the boundary traces as $\psi_\iota=\Phi|_{\Gamma_\iota}$ and $v_\iota=v|_{\Gamma_\iota}$, and set
\begin{equation*}
B_\iota=\partial_{n_\iota}\Phi,\qquad
V_\iota=j_\iota^{-1}\partial_\theta\psi_\iota.
\end{equation*}
We choose material parametrizations on the free boundaries,
\begin{equation*}
\partial_tX_\iota=B_\iota n_\iota+V_\iota\tau_\iota.
\end{equation*}
Introduce the derivative along the normal motion
\begin{equation*}
\mathcal D_\iota:=\partial_t-\frac{V_\iota}{j_\iota}\partial_\theta.
\end{equation*}
Then $\mathcal D_\iota X_\iota=B_\iota n_\iota$.  With the pressure convention $p=\sigma\mathcal H$ on each free boundary, Bernoulli's equation gives
\begin{equation}\label{eq:exact-boundary-system}  
\mathcal D_\iota X_\iota=B_\iota n_\iota,\qquad
\mathcal D_\iota\psi_\iota
=\frac12(B_\iota^2-V_\iota^2)-\sigma\mathcal H_\iota+b(t),\qquad
B_\iota=\frac{R}{\rho_\iota j_\iota}\,\mathfrak G_\iota\boldsymbol\psi. 
\end{equation}
The scalar $b(t)$ reflects the freedom to add a function of time to the
potential.  We fix this freedom by requiring the outer trace to have
zero angular mean, $\int_\T\psi_+\,d\theta=0$.

\begin{lemma}\label{lem:exact-boundary-reduction}
For smooth axisymmetric globally potential flows on a connected fluid
domain, \eqref{eq:exact-boundary-system}, together with the harmonic
extension and material boundary parametrization defined above, is equivalent to the
bulk Euler system and its kinematic and capillary boundary conditions.
The potential is determined up to one common function of time.
On a connected fluid domain, a single-valued potential exists precisely when
the curl and all circulation periods vanish.
\end{lemma}

\begin{proof}
Irrotationality, vanishing circulation periods, and incompressibility give a
single-valued harmonic potential $\Phi$.  The additive constants induced by rotations
form a continuous homomorphism from the rotation circle to $\R$ and hence vanish.
The first equation is the kinematic condition.  Bernoulli's equation
and the chain rule for the trace give
\begin{equation*}
\partial_t\psi_\iota
=\partial_t\Phi+B_\iota^2+V_\iota^2
=\tfrac12(B_\iota^2+V_\iota^2)-\sigma\mathcal H_\iota+b(t).
\end{equation*}
Subtracting $(V_\iota/j_\iota)\partial_\theta\psi_\iota=V_\iota^2$
gives the second equation. The last equation is the definition of the flux DN map.
Conversely, extend $\boldsymbol\psi$ harmonically and set
$v=\nabla\Phi$, $p=b(t)-\Phi_t-|\nabla\Phi|^2/2$.
Spatial differentiation gives the Euler equation, and harmonicity gives
incompressibility. The same trace identities give the kinematic condition
and $p=\sigma\mathcal H$ on both components, with one common $b(t)$.
\end{proof}

Write the inner interface in the two equivalent forms
\begin{equation*}
X_-(Y,\theta)=(R(Y),0)+(a(Y)+h(Y,\theta))e_r
=(R(Y),0)+a(Y)(1+\eta(Y,\theta))e_r,
\qquad \eta=\frac h a .
\end{equation*}
Here $e_r$ points radially out of the cavity and is opposite to the fluid
normal at the circle.  The physical radial displacement is $h$. The
uniform inner elliptic estimates are expressed in terms of the relative graph $\eta$.
We use the dimensionless arclength factor
\begin{equation*}
j_\eta:=\bigl((1+\eta)^2+\eta_\theta^2\bigr)^{1/2},
\end{equation*}
so that the inner arclength density is $a j_\eta$.

On a collapsing trajectory put
\begin{equation*}
a=\ell e^{-Y},\qquad
-aa_t=\frac{a^2}{t_Y}>0.
\end{equation*}

For the canonical trajectories, the domains and potentials are invariant
under $z\mapsto-z$.  The scalar graph and potential traces are therefore
even in $\theta$, and $P_1$ acts on them as projection onto $\cos\theta$.
The elliptic estimates of Section~\ref{sec:scaled-DN} retain both first
harmonics unless reflection symmetry is explicitly assumed.

Set $L(Y)=2\pi R(Y)$ and define the inward flux
\begin{equation*}
q(Y):=\frac{1}{2\pi L(Y)}\int_{\Gamma_-(Y)} B_-\,dS.
\end{equation*}
On the collapsing branch, $q>0$. For a round tube, the fluid normal points
into the cavity.  In the meridional plane, write
\begin{equation*}
X_-=(R,0)+a(1+\eta)e_r,\quad
e_r=(\cos\theta,\sin\theta),\quad
e_\theta=(-\sin\theta,\cos\theta).
\end{equation*}
Set
\begin{equation*}
C_\eta=\cos\theta+\frac{\eta_\theta}{1+\eta}\sin\theta,
\qquad b_\eta=\frac{j_\eta}{1+\eta}B_-.
\end{equation*}
For an even scalar function $f$ on $\T$, write
$f_0=(2\pi)^{-1}\int_\T f\,d\theta$ and
$f_1=\pi^{-1}\int_\T f\cos\theta\,d\theta$.
The centering conditions are $\eta_0=\eta_1=0$.

\begin{proposition}
\label{prop:exact-centred-modulation}
On any smooth centred graph trajectory for which
$d_\eta=(b_\eta)_0(C_\eta)_1-(b_\eta)_1(C_\eta)_0\ne0$,
\begin{equation}\label{eq:exact-centred-modulation}
t_Y=\frac{a(C_\eta)_1}{d_\eta},\qquad
R_Y=-\frac{a(b_\eta)_1}{d_\eta},
\end{equation}
and the kinematic equation for the higher modes is
\begin{equation}\label{eq:exact-eta-Y}
\eta_Y=1+\eta+
\frac{(b_\eta)_1}{d_\eta}C_\eta
-\frac{(C_\eta)_1}{d_\eta}b_\eta.
\end{equation}
Define the time rescaling factor by
\begin{equation*}
\Theta:=\frac{q\,t_Y}{a^2},\qquad -aa_t=\frac q\Theta.
\end{equation*}
For a round graph $\eta=0$ with condenser velocity, $B_->0$, and
\begin{equation}\label{eq:round-modulation-exact}
t_Y=\frac a{(B_-)_0}>0,\qquad
R_Y=-a\frac{(B_-)_1}{(B_-)_0}.
\end{equation}
\end{proposition}

\begin{proof}
The fluid normal is
$n_-=(-(1+\eta)e_r+\eta_\theta e_\theta)/j_\eta$.
Taking the normal component of $X_{-,Y}$ gives
\begin{equation*}
a(\eta_Y-1-\eta)+R_Y C_\eta=-t_Yb_\eta.
\end{equation*}
Taking its mean and first cosine coefficient and using the differentiated
centering conditions gives
\begin{equation*}
t_Y(b_\eta)_0+R_Y(C_\eta)_0=a,\qquad
t_Y(b_\eta)_1+R_Y(C_\eta)_1=0.
\end{equation*}
Solving this system gives \eqref{eq:exact-centred-modulation}. Substitution
into the normal equation gives \eqref{eq:exact-eta-Y}.
The identity for $\Theta$ follows from $a_Y=-a$.
At a round graph, $(C_\eta)_0=0$ and $(C_\eta)_1=1$.
The Hopf lemma gives $B_->0$ for a positive condenser flow,
which proves \eqref{eq:round-modulation-exact}.
\end{proof}

We recover the flux in the modulation equations from the conserved energy.

Let $u_c$ be the condenser potential with traces $1,0$ and set
$\operatorname{Cap}=\int_\Omega|\nabla u_c|^2\,dx$.
Since $\int_{\Gamma_-}\partial_n\Phi\,dS=2\pi Lq$, write
\begin{equation}\label{eq:exact-energy-components}
\begin{gathered}
\Phi_0=\Phi-\frac{2\pi Lq}{\operatorname{Cap}}u_c=q\varphi_z,
\qquad A=|\Gamma_-|+|\Gamma_+|,\qquad
\mathfrak c=\frac{2\pi L}{\operatorname{Cap}}-Y,\\
D_b=\pi L(Y+\mathfrak c),\qquad
D_0=\frac12\int_\Omega|\nabla\varphi_z|^2\,dx,\qquad
K_0=q^2D_0,\qquad E_0=\sigma A+K_0.
\end{gathered}
\end{equation}
Here
$z=(q^{-1}P_{\ne0}\Phi_{0,-},q^{-1}\Phi_{0,+})$
is the pair of normalized traces. The subscripts on $\Phi_0$ denote its
boundary traces, and the outer trace has zero angular mean.
The zero-flux condition determines the inner constant, so $\varphi_z$
depends linearly on $z$ at fixed geometry.  Write the remaining
geometric coordinates, apart from $a=\ell e^{-Y}$, as
$\mathcal G=(R,\eta,h_+)$. 
Let $X_+^*:\T\to\partial D_*$ parametrize the terminal outer meridian
of $\Gamma_+^*$, with $D_*$ as in
Definition~\ref{def:canonical-terminal-objects}, and let $n_+^*$ be
its outward unit normal.  In the fixed normal graph chart, write
\begin{equation*}
X_+(Y,\theta)=X_+^*(\theta)+h_+(Y,\theta)n_+^*(\theta),
\end{equation*}
so $h_+$ is the signed outer normal displacement.

The exact inner area is
\begin{equation}\label{eq:inner-area-exact}
A_-=2\pi a\int_0^{2\pi}
\big(R+a(1+\eta)\cos\theta\big)j_\eta\,d\theta.
\end{equation}

\begin{proposition}
\label{prop:exact-energy-and-capacity}
For a globally potential solution with $q>0$ and conserved energy $E_*$,
the flux is reconstructed in normalized variables by
\begin{equation}\label{eq:q-normalized-explicit}
E_*=\sigma A+q^2(D_b+D_0),\qquad
q^2=\frac{E_*-\sigma A}{D_b+D_0}.
\end{equation}
In physical trace coordinates,
\begin{align}
q^2&=\frac{E_*-E_0}{\pi L(Y+\mathfrak c)},\qquad E_*-E_0>0,
\label{eq:corrected-q-reconstruction}\\
\frac{q_Y}{q}
&=-\frac{(E_0)_Y}{2(E_*-E_0)}-\frac{L_Y}{2L}
-\frac{1+\mathfrak c_Y}{2(Y+\mathfrak c)}.
\label{eq:corrected-q-log-derivative}
\end{align}
If the normalized equations are written as
\begin{equation*}
\mathcal G_Y=\mathcal V_0(Y,\mathcal G,z,q),\qquad
z_Y=\mathcal B_0(Y,\mathcal G,z,q)-p_qz,\qquad p_q=q_Y/q,
\end{equation*}
define $\mathcal D_0^g=\partial_Y+\mathcal V_0\cdot D_{\mathcal G}$,
where $D_{\mathcal G}$ denotes differentiation in the geometric variables
and $\partial_Y$ includes the explicit dependence $a=\ell e^{-Y}$.
We write $D_z$ for differentiation in the normalized trace variables.
Then
\begin{equation}\label{eq:p-normalized-explicit}
p_q=-\frac{\mathcal D_0^g(D_b+D_0)+D_zD_0[\mathcal B_0]
+\sigma q^{-2}\mathcal D_0^gA}{2D_b}.
\end{equation}
Put $\nu_\iota=X_{\iota,Y}\cdot n_\iota$ and $\boldsymbol\nu=(\nu_-,\nu_+)$.  The capacity satisfies
\begin{equation}\label{eq:capacity-Y-exact}
\operatorname{Cap}_Y=-\sum_{\iota\in\{-,+\}}
\int_{\Gamma_\iota}\nu_\iota
(\partial_{n_\iota}u_c)^2\,dS.
\end{equation}
\end{proposition}
\begin{proof}
Zero inner flux and Green's identity give
$\int_\Omega\nabla u_c\cdot\nabla\varphi_z=0$.
The condenser contribution to the kinetic energy is
$q^2(2\pi L)^2/(2\operatorname{Cap})=q^2D_b$, proving
\eqref{eq:q-normalized-explicit}.  Substituting $K_0=q^2D_0$
and differentiating gives
\eqref{eq:corrected-q-reconstruction}--\eqref{eq:corrected-q-log-derivative}.
The quadratic form $D_0$ satisfies
\begin{equation*}
D_zD_0[z]=2D_0.
\end{equation*}
Differentiating the normalized energy and dividing by $q^2$ gives
\begin{equation*}
0=2p_q(D_b+D_0)+\mathcal D_0^g(D_b+D_0)
+D_zD_0[\mathcal B_0-p_qz]+\sigma q^{-2}\mathcal D_0^g A.
\end{equation*}
The quadratic identity cancels $2p_qD_0$ and leaves the coefficient $2D_b$,
which gives \eqref{eq:p-normalized-explicit}.

To differentiate the capacity, apply \eqref{eq:Hadamard-flux}
with the normalization $R$ held fixed and trace $e_-=(1,0)$.
Both traces are constant, so the tangential term vanishes.  Symmetry gives
\begin{equation*}
\operatorname{Cap}_Y
=2\pi R\langle e_-,D\mathfrak G[\boldsymbol\nu]e_-\rangle
=-\sum_\iota\int_{\Gamma_\iota}\nu_\iota
(\partial_{n_\iota}u_c)^2\,dS,
\end{equation*}
which is \eqref{eq:capacity-Y-exact}.
\end{proof}

\begin{lemma} 
\label{lem:energy-recovers-bernoulli-mean}
Let $t\mapsto(\Omega_t,\Phi)$ be a smooth family with two boundary components,
with $\Delta\Phi=0$, $v=\nabla\Phi$ and boundary normal velocity $v\cdot n$.
Let $\sigma\ge0$ be constant.  Suppose the nonconstant Bernoulli
projections hold on both connected boundary components, so that
\begin{equation*}
\left.\left(\partial_t\Phi+\tfrac12|v|^2
+\sigma\mathcal H\right)\right|_{\Gamma_\iota}
=b_\iota(t),\qquad \iota\in\{-,+\},
\end{equation*}
for scalar functions $b_\iota$.  Write
\begin{equation*}
2\pi L(t)q(t)=\int_{\Gamma_-}v\cdot n\,dS,
\qquad
E_{\rm tot}(t)=\frac12\int_{\Omega_t}|v|^2\,dx
+\sigma\bigl(|\Gamma_-|+|\Gamma_+|\bigr).
\end{equation*}
Then the identity
\begin{equation}\label{eq:energy-bernoulli-mean-identity}
\frac{d}{dt}E_{\rm tot}=2\pi Lq(b_--b_+)
\end{equation}
holds.  If $q\ne0$, conservation of $E_{\rm tot}$ is equivalent to
$b_-=b_+$.  In that case
\begin{equation*}
p=b_+(t)-\partial_t\Phi-\tfrac12|v|^2
\end{equation*}
recovers the full Euler equations and $p=\sigma\mathcal H$ on both
components.  Thus, for positive flux and a fixed common potential gauge,
the energy reconstruction determines the remaining constant mode.
\end{lemma}
\begin{proof}
With this pressure, the velocity satisfies the Euler equation, and
$p-\sigma\mathcal H=b_+-b_\iota$ on $\Gamma_\iota$.
Reynolds' formula, incompressibility, and the kinematic condition give
\begin{equation*}
\frac{d}{dt}\frac12\int_{\Omega_t}|v|^2\,dx
=-\sum_\iota\int_{\Gamma_\iota}p\,v\cdot n\,dS,
\qquad
\frac{d}{dt}|\Gamma_\iota|
=\int_{\Gamma_\iota}\mathcal H\,v\cdot n\,dS,
\end{equation*}
using the fluid normal on both components.  Adding and using
$\int_{\Gamma_+}v\cdot n=-2\pi Lq$ proves
\eqref{eq:energy-bernoulli-mean-identity} and the equivalence for $q\ne0$.
The difference $b_--b_+$ is unchanged by the common potential gauge.

\end{proof}

\subsection{Terminal data and polar dynamics}
\label{subsec:canonical-terminal-data}

Fix $\sigma\ge0$, $E_c>0$ and a circle $\gamma_*$ of radius $R_*>0$.
Choose $\ell/R_*$ sufficiently small and set
$\Gamma_+^*=\partial N_\ell(\gamma_*)$.
The canonical domains satisfy $\rho\ge R_*-\ell>0$.  Choose the
neighbourhood of outer graphs so that $\rho\ge(R_*-\ell)/2$.
For $Y_N\nearrow\infty$, put
\begin{equation}\label{eq:aN-def}
a_N=\ell e^{-Y_N},\qquad
\Gamma_{-,N}^*=\partial N_{a_N}(\gamma_*),\qquad
\Omega_N^*=N_\ell(\gamma_*)\setminus\overline{N_{a_N}(\gamma_*)}.
\end{equation}
Let $u_N$ be the condenser potential
\begin{equation}\label{eq:cap-terminal}
\Delta u_N=0\quad\hbox{in }\Omega_N^*,\qquad
u_N|_{\Gamma_{-,N}^*}=1,\qquad u_N|_{\Gamma_+^*}=0,
\qquad \operatorname{Cap}_N=\int_{\Omega_N^*}|\nabla u_N|^2dx.
\end{equation}
Since $|\Gamma_{-,N}^*|=2\pi L_*a_N$, where $L_*=2\pi R_*$,
for large $N$ there is a unique $c_N>0$ such that
\begin{equation}
\frac12c_N^2\operatorname{Cap}_N+\sigma|\Gamma_{-,N}^*|=E_c.
\end{equation}
Prescribe the potential $\Phi_N^*=c_Nu_N$ and the velocity
\begin{equation}\label{eq:terminal-velocity}
v_N^*=c_N\nabla u_N.
\end{equation}

Figure~\ref{fig:canonical-terminal-objects} shows the prescribed
approximating terminal geometries as $a_N\downarrow0$.
The terminal material map of the resulting trajectory, obtained as
$t\uparrow T$, is depicted in Figure~\ref{fig:terminal-material-map}.

\begin{figure}[!hbp]
\centering
\begin{tikzpicture}[cbpfigure]
\path[use as bounding box] (0,0) rectangle (15.7,4.48);
\begin{scope}[yshift=-4.17cm]

\node at (3.65,8.35) {Approximating terminal state};
\node at (12.00,8.35) {Filled terminal geometry};
\begin{scope}[shift={(3.65,6.35)}]
\filldraw[fill=cbpFluid,cbpwall] (0,0) circle (1.38);
\filldraw[fill=white,cbpinner] (0,0) circle (.50);
\draw[dashed,cbpInk!45] (-1.60,0)--(1.60,0);
\fill[cbpCore] (0,0) circle (1.8pt);
\node[anchor=north,inner sep=3pt] at (0,-.02) {$c_*$};
\draw[cbpradius] (0,0)--({1.38*cos(128)},{1.38*sin(128)});
\node[anchor=south east,fill=cbpFluid,inner sep=1.5pt]
at (-.42,.57) {$\ell$};
\draw[cbpradius] (0,0)--(.50,0);
\draw[cbpleader] (.25,0)--(.83,-.67)--(1.66,-.67);
\node[anchor=west,inner sep=2pt] at (1.66,-.67) {$a_N>0$};
\draw[cbpleader] (.40,.30)--(1.00,.78)--(1.69,.78);
\node[anchor=west,inner sep=2pt] at (1.69,.78) {$\Phi_N^*=c_N$};
\draw[cbpleader] (-1.29,-.49)--(-1.86,-.88);
\node[anchor=east,inner sep=2pt] at (-1.86,-.88) {$\Phi_N^*=0$};
\node[anchor=east,inner sep=2pt] at (-1.65,0) {$z=0$};
\end{scope}
\draw[->,cbpInk!75,line width=.8pt] (6.46,6.35)--(9.39,6.35);
\node[above=5pt] at (7.92,6.35) {$a_N\downarrow0$};
\node[below=5pt,font=\footnotesize,align=center] at (7.92,6.35)
{limit of prescribed\\terminal geometry};
\begin{scope}[shift={(12.00,6.35)}]
\filldraw[fill=cbpFluid,cbpwall] (0,0) circle (1.38);
\draw[dashed,cbpInk!45] (-1.60,0)--(1.60,0);
\node at (-.70,.47) {$D_*$};
\fill[cbpCore] (0,0) circle (2.2pt);
\draw[cbpleader] (.06,-.03)--(.56,-.59)--(1.54,-.59);
\node[anchor=west,inner sep=2pt] at (1.54,-.59) {$c_*=(R_*,0)$};
\draw[cbpleader] (.83,1.10)--(1.36,1.44)--(1.77,1.44);
\node[anchor=west,inner sep=2pt] at (1.77,1.44) {$\partial D_*$};
\end{scope}
\node at (7.85,4.52)
{Revolve about the $z$-axis: $c_*\longmapsto\gamma_*$,
\quad $D_*\longmapsto\Omega_*$ (solid torus).};
\end{scope}
\end{tikzpicture}
\caption{The canonical terminal geometries of
Definition~\ref{def:canonical-terminal-objects}.
Both panels are meridional sections.  The left panel shows the exact
concentric circles and constant boundary potentials imposed at a finite
approximating terminal state.  The right panel shows the limiting outer
disk and its core point.  The horizontal dashed lines pass through the
disk centre at $z=0$.}
\label{fig:canonical-terminal-objects}
\end{figure}

We use the local existence and uniqueness theory on smooth bounded
fluid domains.  Positive surface tension permits arbitrary smooth
divergence-free data.  At $\sigma=0$, one assumes $-\partial_n p>0$
on the entire boundary \cite[Theorems~1.1 and~1.2]{CS}.
The same local theory gives continuation through smooth, nondegenerate
states when the continuation norms remain bounded and, at zero surface
tension, the Taylor coefficient has a positive lower bound.

\begin{lemma}\label{lem:terminal-data-admissible}
For large $N$, the data \eqref{eq:aN-def}--\eqref{eq:terminal-velocity}
generate a smooth Euler solution evolving backwards from the prescribed
terminal state, with axisymmetry, reflection symmetry and a single-valued
potential.
The terminal graph perturbations and zero-flux remainder vanish.
The inner flux is positive, and at $\sigma=0$ both boundaries satisfy
the strict Taylor sign.
\end{lemma}
\begin{proof}
The condenser is smooth and has the stated symmetries.
The velocity is the gradient of a single-valued harmonic potential and
is therefore divergence free.
Green's identity and the Hopf lemma give
\begin{equation*}
\int_{\Gamma_{-,N}^*}v_N^*\cdot n\,dS
=c_N\operatorname{Cap}_N>0,\qquad
v_N^*\cdot n>0\quad\hbox{on }\Gamma_{-,N}^*.
\end{equation*}
The total boundary flux is zero.  At $\sigma=0$, the pressure solves
$-\Delta p_N=\sum_{i,j=1}^3(\partial_i v_{N,j}^*)^2$ with zero boundary trace.
A constant velocity has zero flux across a closed surface, so the
positive inner flux implies $Dv_N^*\not\equiv0$.
The strong maximum principle gives $p_N>0$, and Hopf's lemma gives
$-\partial_n p_N>0$ on both compact boundary components.
The local theory therefore applies in both cases.
Solve forward with velocity $-v_N^*$ and reverse time to obtain the
backward branch.

Uniqueness preserves the symmetries.  For $\omega=\nabla\times v$, the equation
$(\partial_t+v\cdot\nabla)\omega=(\omega\cdot\nabla)v$ preserves $\omega=0$.
For every material closed curve $\mathcal C_t$,
\begin{equation*}
\frac d{dt}\oint_{\mathcal C_t}v\cdot dx
=\oint_{\mathcal C_t}\nabla\bigl(|v|^2/2-p\bigr)\cdot dx=0.
\end{equation*}
Since all terminal periods vanish and the evolving domains are
diffeomorphic, the potential remains single valued.
The terminal geometry is circular and the terminal velocity is the
condenser field.  Hence all terminal perturbations vanish.
\end{proof}

Use the centred radius of each local branch to define
\begin{equation}\label{eq:common-Y-gauge}
Y=\log\frac{\ell}{a^{(N)}(t)},\qquad a(Y)=\ell e^{-Y}.
\end{equation}
At the terminal state, \eqref{eq:round-modulation-exact} gives
$t_Y=a_N/(B_-)_0>0$, so this is initially a valid common time coordinate.

The centering conditions on the graph do not constrain the first
potential harmonic.  In the centred polar chart, define
\begin{equation*}
\alpha(Y):=\frac1\pi\int_0^{2\pi}
\Phi_0(X_-(Y,\theta),Y)\cos\theta\,d\theta.
\end{equation*}

At fixed chart coordinates, write $v_{\tau,\iota}=v\cdot\tau_\iota=V_\iota$.
The trace chain rule and $X_{\iota,Y}\cdot n_\iota=t_YB_\iota$ give
\begin{equation}\label{eq:charted-Bernoulli}
(\Phi|_{X_\iota})_Y
=t_Y\left(\tfrac12B_\iota^2-\tfrac12v_{\tau,\iota}^2
-\sigma\mathcal H_\iota\right)
+v_{\tau,\iota}X_{\iota,Y}\cdot\tau_\iota+\widetilde b(Y).
\end{equation}
Here $\widetilde b(Y)=t_Y b(t(Y))$ is the common Bernoulli gauge in the $Y$ coordinate.
The condenser trace is constant on each boundary.  With the outer trace
normalized to have zero angular mean in this chart, set
$f=P_{\ge2}\Phi_{0,-}$ and $g=\Phi_{0,+}$.
We evolve $\alpha$ from $\alpha(Y_N)=0$ and use
$(\eta,f,\alpha,R,h_+,g)$ as the state variables.  The energy identity reconstructs
the flux and, by Lemma~\ref{lem:energy-recovers-bernoulli-mean}, recovers
the constant Bernoulli equation.

\begin{proposition}
\label{prop:exact-polar-raw-system}
On a smooth centred trajectory put
\begin{equation*}
d=1+\eta,\quad e=\eta_\theta,\quad j=(d^2+e^2)^{1/2},\quad
u=e/d,\quad \varsigma=j/d=(1+u^2)^{1/2},
\end{equation*}
\begin{equation*}
\begin{gathered}
n=\frac{aB_-}{q},\quad \upsilon=\frac{av_{\tau,-}}q,\quad
\kappa=\frac{R_Y}{a},\quad \beta=\frac\alpha q,\quad F=\frac f q,\\
p_q=\frac{q_Y}q,\quad B_q=\frac{\sigma a}{q^2},\quad \mathscr C=a\mathcal H_-.
\end{gathered}
\end{equation*}
The curvature is computed with the fluid normal from
Proposition~\ref{prop:exact-centred-modulation}. Thus $\mathscr C=-1$
for a circular cylindrical cavity.  Then
\begin{align}
\eta_Y&=1+\eta-\kappa(\cos\theta+u\sin\theta)
-\Theta\varsigma n,
\label{eq:polar-eta-exact}\\
\frac{X_{-,Y}\cdot\tau_-}{a}
&=-\kappa\varsigma\sin\theta-\Theta u n,
\label{eq:polar-chart-eliminated}\\
\upsilon&=j^{-1}(-\beta\sin\theta+F_\theta).
\label{eq:polar-tangent-trace}
\end{align}
Define the kinetic function
\begin{equation}\label{eq:polar-kinetic-function}
\mathcal K(\eta,e,n,\upsilon,\theta)
=\frac\Theta2(n^2-\upsilon^2)-\Theta\frac e{1+\eta}n\upsilon
-\kappa\left(1+\frac{e^2}{(1+\eta)^2}\right)^{1/2}
\upsilon\sin\theta.
\end{equation}
The nonconstant inner Bernoulli equations are
\begin{align}
\frac{\alpha_Y}{q}&=(\mathcal K-\Theta B_q\mathscr C)_1,
\label{eq:polar-alpha-exact}\\
F_Y+p_qF&=P_{\ge2}(\mathcal K-\Theta B_q\mathscr C).
\label{eq:polar-F-exact}
\end{align}
Moreover, for $h=a\eta$,
\begin{equation}\label{eq:polar-h-exact}
h_Y=aP_{\ge2}\mathscr H(\eta,e,n,\theta),\qquad
\mathscr H=1-\kappa(\cos\theta+u\sin\theta)-\Theta\varsigma n.
\end{equation}

For reflection symmetric data, decompose
\begin{equation*}
n=n_0+n_1\cos\theta+N,\quad N=P_{\ge2}n,\qquad
\upsilon=\upsilon_1\sin\theta+V,
\end{equation*}
where $\upsilon_1=\pi^{-1}\int_0^{2\pi}\upsilon\sin\theta\,d\theta$ and $V$ has zero first sine coefficient.  Then
\begin{equation*}
P_{\ge2}\mathcal K=\mathfrak f_2\cos2\theta+\mathcal R,
\end{equation*}
where
\begin{equation}\label{eq:exact-kinetic-low-source}
\mathfrak f_2=\frac\Theta4(n_1^2+\upsilon_1^2)+\frac\kappa2\upsilon_1.
\end{equation}
Every remaining term in $\mathcal R$ contains $N$, $V$, $u$, or
$\varsigma-1$.  In particular $\mathcal R=0$ when $N=V=u=0$.
The term $-\Theta P_{\ge2}(un\upsilon)$ contains the graph derivative
$u=\eta_\theta/(1+\eta)$.
If $|n_1|+|\upsilon_1|+|\kappa|\le C_0(1+Y)^Ma$ and $|\Theta|\le2$, then
$|\mathfrak f_2|\le C(1+Y)^{2M}a^2$.
After multiplication by $a$, this contribution from the first harmonics
is $O((1+Y)^{2M}a^3)$ in the scaled momentum equation.
\end{proposition}
\begin{proof}
The normal component of $X_{-,Y}$ gives \eqref{eq:polar-eta-exact}.
Its tangential component is
\begin{equation*}
\frac{X_{-,Y}\cdot\tau_-}{a}
=\kappa\frac{-d\sin\theta+e\cos\theta}{j}
+\frac{(\eta_Y-d)e}{j}.
\end{equation*}
Substituting the first equation and using
$-d\sin\theta+e\cos\theta-e(\cos\theta+(e/d)\sin\theta)
=-j^2\sin\theta/d$ gives \eqref{eq:polar-chart-eliminated}.
Angular differentiation of the trace gives \eqref{eq:polar-tangent-trace},
since the condenser contribution and the eliminated mean are constant
on the inner boundary.
Divide \eqref{eq:charted-Bernoulli} by $q$ and substitute this chart velocity and
$t_Y=\Theta a^2/q$.  The kinetic term is \eqref{eq:polar-kinetic-function}
and the curvature coefficient is $\sigma t_Y/(aq)=\Theta B_q$.
Projecting onto the nonconstant modes and using
$f_Y/q=F_Y+p_qF$ gives the trace equations.
$h_Y=a(\eta_Y-\eta)$ gives the last equation.

Expand \eqref{eq:polar-kinetic-function} and use
$P_{\ge2}\cos^2\theta=\tfrac12\cos2\theta$ and
$P_{\ge2}\sin^2\theta=-\tfrac12\cos2\theta$.
Every term remaining after subtraction of $\mathfrak f_2\cos2\theta$
contains one of the stated factors.  The estimate follows from the bounds
assumed on the three scalar coefficients.
\end{proof}

\begin{proposition}
\label{prop:clock-centre-finite-rank}
Put $C=\cos\theta+u\sin\theta$, $N_* =\varsigma n$, and
$\Delta=(N_*)_0C_1-(N_*)_1C_0$.  When $\Delta\ne0$,
\begin{equation}\label{eq:clock-centre-normalized}
\Theta=C_1/\Delta,\qquad \kappa=-(N_*)_1/\Delta.
\end{equation}
For an arbitrary variation of the normalized state at fixed $a,R$,
\begin{align}
\dot\Delta&=(\dot N_*)_0C_1+(N_*)_0\dot C_1
-(\dot N_*)_1C_0-(N_*)_1\dot C_0,\notag\\
\dot\Theta&=\dot C_1/\Delta-C_1\dot\Delta/\Delta^2,
\qquad
\dot\kappa=-(\dot N_*)_1/\Delta+(N_*)_1\dot\Delta/\Delta^2,
\label{eq:clock-centre-differential}\\
\dot C&=\dot u\sin\theta,\qquad
\dot N_* =\varsigma\dot n+(u/\varsigma)n\dot u,
\quad \dot u=\dot\eta_\theta/d-e\dot\eta/d^2.\notag
\end{align}
In the higher mode equations, the subscript $\mathrm m$ denotes the
terms produced by varying $\Theta$ and $\kappa$ alone:
\begin{align*}
\dot h_Y\big|_{\rm m}
&=-aP_{\ge2}(N_*\dot\Theta+C\dot\kappa),
\\
\dot{(f_Y/q)}\big|_{\rm m}
&=P_{\ge2}\left[
\left(\tfrac12(n^2-\upsilon^2)-un\upsilon-B_q\mathscr C\right)\dot\Theta
-\varsigma \upsilon\sin\theta\dot\kappa\right].
\end{align*}
These terms have finite rank in the linearization at fixed $Y$.
\end{proposition}
\begin{proof}
Since $b_\eta=(q/a)N_*$ and $C_\eta=C$, normalizing
\eqref{eq:exact-centred-modulation} gives \eqref{eq:clock-centre-normalized}.
The quotient rule gives \eqref{eq:clock-centre-differential}.
Differentiating \eqref{eq:polar-h-exact} and
$f_Y/q=P_{\ge2}(\mathcal K-\Theta B_q\mathscr C)$ with respect to
$\Theta,\kappa$ alone gives the last two equations.  Variations of $q,B_q$
are treated separately.
\end{proof}

\begin{lemma}
\label{lem:zero-flux-dipole-reconstruction}
Fix an integer $k\ge4$ and an admissible smooth macroscopic geometry.
Use the trace coordinates $(\alpha,f,g)$ just defined.
The mean inner trace is determined by zero total inner flux.  There are constants
$C,M$, uniform in the macroscopic geometry, such that, with
$B_0=\partial_{n_-}\Phi_0$ and $\|\eta\|_{H^{k+2}}$ small,
\begin{align}
\|aB_0-\alpha\cos\theta\|_{H^k}
&\le C(1+Y)^M\bigl((a+\|\eta\|_{H^{k+2}})|\alpha|
+\|f\|_{H^{k+1}}+\|g\|_{H^{k+4}}\bigr),
\label{eq:dipole-normal-reconstruction}\\
\|a v_{\tau,-}+\alpha\sin\theta\|_{H^k}
&\le C\bigl(\|\eta\|_{H^{k+2}}|\alpha|+\|f\|_{H^{k+1}}\bigr).
\label{eq:dipole-tangent-reconstruction}
\end{align}
In particular, the right side of \eqref{eq:dipole-normal-reconstruction}
also bounds $\|P_{\ge2}aB_0\|_{H^k}$.
\end{lemma}

\begin{proof}
Write $\langle u\rangle=(2\pi)^{-1}\int_\T u$, let
$\mathcal D\boldsymbol\varphi=a\partial_{n_-}\Phi[\boldsymbol\varphi]$,
and put $w_\eta=(\rho/R)j_\eta$.
The zero-flux constraint eliminates the inner mean by
\begin{equation*}
\mathcal D^{\rm red}\boldsymbol\varphi
=\mathcal D\boldsymbol\varphi-\mathcal De_-\,
\frac{\langle w_\eta\mathcal D\boldsymbol\varphi\rangle}
{\langle w_\eta\mathcal De_-\rangle},
\qquad e_-=(1,0).
\end{equation*}
Thus $aB_0=\mathcal D^{\rm red}(\alpha\cos\theta+f,g)$.

At $\eta=0$, the planar nonzero block is $|D_\theta|$.
Lemma~\ref{lem:weighted-layer-expansion}, including conversion
from weighted flux to normal velocity, gives
\begin{equation*}
\|\mathcal D^{\rm red}_{a,0}(\cos\theta,0)-\cos\theta\|_{H^k}
\le C(1+Y)^Ma.
\end{equation*}
The scalar denominator is comparable to $(1+Y)^{-1}$, so its
inverse introduces only a polynomial factor.
Propositions~\ref{prop:full-C2-flux-DN}
and~\ref{prop:schur-differentiation}, with the smooth normal trace
factors, give
\begin{align*}
\|\mathcal D^{\rm red}_{a,\eta}(f,g)\|_{H^k}
&\le C(1+Y)^M
\bigl(\|f\|_{H^{k+1}}+\|g\|_{H^{k+4}}\bigr),\\
\|D_\eta\mathcal D^{\rm red}_{a,\eta}[\zeta]
(\cos\theta,0)\|_{H^k}
&\le C(1+Y)^M\|\zeta\|_{H^{k+2}}.
\end{align*}
The outer trace estimate also follows from
Lemma~\ref{lem:cross-block-small-hole}.
Integrating the second bound along $t\eta$, $0\le t\le1$,
and using linearity proves
\eqref{eq:dipole-normal-reconstruction}.

Finally,
$a v_{\tau,-}=j_\eta^{-1}(-\alpha\sin\theta+f_\theta)$.
The Sobolev product estimate and
$\|j_\eta^{-1}-1\|_{H^k}\le C\|\eta\|_{H^{k+2}}$
prove \eqref{eq:dipole-tangent-reconstruction}.
\end{proof}

\begin{proposition}
\label{prop:first-mode-backward-amplitude}
Fix an integer $k\ge4$.  Consider a smooth centred Euler trajectory on $[Y_0,Y_N]$ with
$a=\ell e^{-Y}$, $q\asymp(1+Y)^{-1/2}$ and uniformly admissible
macroscopic geometry.  Suppose, in the trace coordinates of
Lemma~\ref{lem:zero-flux-dipole-reconstruction}, that
\begin{equation}\label{eq:first-mode-physical-bootstrap}
\|\eta\|_{H^{k+2}}+
q^{-1}\bigl(\|f\|_{H^{k+1}}+\|g\|_{H^{k+4}}\bigr)
\le C_0(1+Y)^{M_0}a^2,\qquad
|\alpha|\le\delta_1q,
\end{equation}
where $0<\delta_1\le1/8$ is fixed, and $\alpha(Y_N)=0$.
For sufficiently large $Y_0$, depending on the constants in the
hypotheses but independent of $N$, the Euler equation implies
\begin{align}
\alpha_Y-\alpha&=\mathfrak f_1,\qquad
|\mathfrak f_1(Y)|\le Cq(Y)(1+Y)^Ma(Y),
\label{eq:first-mode-actual-source}\\
|\alpha(Y)|+|\alpha_Y(Y)|
&\le Cq(Y)(1+Y)^Ma(Y).
\label{eq:first-mode-amplitude-closed}
\end{align}
This improves the bootstrap bound $|\alpha|\le\delta_1q$.
Moreover,
\begin{equation}\label{eq:first-mode-high-feedback}
q^{-1}\|P_{\ge2}aB_0\|_{H^k}
+|R_Y|+\|\eta_Y\|_{H^k}
\le C(1+Y)^{M'}a^2.
\end{equation}
\end{proposition}

\begin{proof}
We divide the proof into three steps.

\emph{Step 1. Equation for the first harmonic.}
Write $\beta=\alpha/q$ and put $\epsilon(Y)=C(1+Y)^Ma$, increasing
$C,M$ in intermediate bounds if necessary.  The condenser profile in
Lemma~\ref{lem:capacity-single-layer-schur}, its exact flux normalization, and
Lemma~\ref{lem:zero-flux-dipole-reconstruction} give
\begin{equation*}
\frac{aB_-}{q}=1+\beta\cos\theta+O_{H^k}(\epsilon),\qquad
\frac{a v_{\tau,-}}q=-\beta\sin\theta+O_{H^k}(\epsilon).
\end{equation*}
The graph contribution to the condenser profile is bounded by a
polynomial in $1+Y$ times $\|\eta\|_{H^{k+2}}$, and its first circular
harmonic is $O((1+Y)^Ma)$.  The zero-flux reduction includes the scalar correction determined by
the weighted mean of $B_0$.
Substitution in the exact centred modulation identities gives
\begin{equation*}
\Theta=\frac{qt_Y}{a^2}=1+O(\epsilon),\qquad
\frac{R_Y}{a}=-\Theta\beta+O(\epsilon),\qquad
\|\eta_Y\|_{H^k}\le C\epsilon.
\end{equation*}
The normalized modulation determinant is $1+O(\epsilon)$ and is
therefore positive for large $Y_0$.
Insert these bounds, $u=O_{H^k}((1+Y)^Ma^2)$ and
$\varsigma=1+O_{H^k}((1+Y)^Ma^4)$ into
\eqref{eq:polar-kinetic-function} and \eqref{eq:polar-alpha-exact}.
The leading kinetic contribution is
\begin{equation*}
q\Theta\left[
\tfrac12(1+\beta\cos\theta)^2
-\tfrac12\beta^2\sin^2\theta-\beta^2\sin^2\theta
\right]_1=q\Theta\beta=\Theta\alpha.
\end{equation*}
Both $\cos^2\theta$ and $\sin^2\theta$ have zero first cosine coefficient.
Since $|\beta|\le\delta_1$, each substitution error is bounded by
$Cq\epsilon$.  The curvature formula gives
$|(\mathcal H_-)_1|\le C(1+a^{-1}\|\eta\|_{H^{k+2}})$, so
\begin{equation*}
\sigma t_Y|(\mathcal H_-)_1|
\le C\frac{a^2}{q}(1+(1+Y)^{M_0}a)
\le Cq(1+Y)^{M'}a.
\end{equation*}
This proves \eqref{eq:first-mode-actual-source}.

\emph{Step 2. The terminal amplitude estimate.}
Backward variation of constants gives
\begin{equation*}
\alpha(Y)=-\int_Y^{Y_N}e^{Y-s}\mathfrak f_1(s)\,ds.
\end{equation*}
Since $q(s)/q(Y)\le C\sqrt{(1+Y)/(1+s)}$, $a(s)=a(Y)e^{Y-s}$, and
$(1+s)^M\le(1+Y)^M(1+s-Y)^M$, the integral is bounded by
$Cq(Y)(1+Y)^Ma(Y)$ independently of $Y_N$.
Equation~\eqref{eq:first-mode-actual-source} then gives the bound on $\alpha_Y$.

\emph{Step 3. Estimates for the higher modes.}
Using \eqref{eq:dipole-normal-reconstruction} again gives
$q^{-1}\|P_{\ge2}aB_0\|_{H^k}\le C(1+Y)^{M'}a^2$.
The condenser has a first harmonic $O((1+Y)^{M'}a)$ and high harmonics
$O_{H^k}((1+Y)^{M'}a^2)$.  Multiplication by $j_\eta/(1+\eta)$
preserves these bounds, and $P_{\ge2}C_\eta=O_{H^k}(\eta)$.
The exact equations for $R_Y$ and $\eta_Y$ therefore prove
\eqref{eq:first-mode-high-feedback}.
\end{proof}

\begin{proposition} 
\label{prop:raw-state-first-jet}
Consider a smooth Euler trajectory on $[Y_0,Y_N]$ with
$\alpha(Y_N)=0$, and fix $k\ge s_0$.
Assume the physical trace bootstrap bound
\eqref{eq:first-mode-physical-bootstrap} at index $k+4$ in place of $k$,
with the macroscopic charts in the corresponding bounded smooth class.
Assume also the uniform positive energy gap in \eqref{eq:corrected-q-reconstruction}.
Then Proposition~\ref{prop:first-mode-backward-amplitude} implies,
for $j=0,1$,
\begin{align}
q^{-1}|\partial_Y^j\alpha|&\le C(1+Y)^Ma,\notag\\
\|\partial_Y^j\eta\|_{H^{k+2}}
+q^{-1}\bigl(\|\partial_Y^jf\|_{H^{k+1}}
+\|\partial_Y^jg\|_{H^{k+4}}\bigr)
&\le C(1+Y)^Ma^2,\notag\\
|R_Y|+\|X_{+,Y}\|_{H^{k+4}}
&\le C(1+Y)^Ma^2,\notag \\
\frac{q_Y}{q}
&=-\frac1{2(Y+\mathfrak c)}+O((1+Y)^Ma).
\label{eq:raw-q-first-jet}
\end{align}
Moreover,
\begin{equation*}
|\mathfrak c_Y|+|(E_0)_Y|
\le C(1+Y)^Ma.
\end{equation*}
\end{proposition}

\begin{proof}
We divide the proof into three steps.

\emph{Step 1. Inner derivatives.}
Use the outer gauge with zero mean and eliminate the inner mean by
Lemma~\ref{lem:zero-flux-dipole-reconstruction}.
The nonconstant projections of \eqref{eq:charted-Bernoulli}
remove both $\widetilde b(Y)$ and the inner condenser trace.
These equations determine $(\alpha_Y,f_Y,g_Y)$ from the instantaneous
state and $q$, while centred modulation determines
$(R_Y,\eta_Y,X_{+,Y})$.

Write $n=aB_-/q$ and $\upsilon=av_{\tau,-}/q$, and set $r=k+2$ for the following estimates.
The condenser and zero-flux estimates, together with the bound on the
first harmonic, give
\begin{equation*}
n=n_0+n_1\cos\theta+n_{\ge2},\qquad
n_0=1+O((1+Y)^Ma^2),\quad n_1=O((1+Y)^Ma),\quad
n_{\ge2}=O_{H^r}((1+Y)^Ma^2).
\end{equation*}
The estimate for the mean follows from
$1=\langle(\rho/R)j_\eta n\rangle$ and the centering conditions.
Also $\upsilon=-(\alpha/q)\sin\theta+O_{H^r}((1+Y)^Ma^2)$.
The dimensionless curvature has constant component
$-1+O((1+Y)^Ma^2)$, first harmonic $O((1+Y)^Ma)$, and higher harmonics
$O_{H^r}((1+Y)^Ma^2)$.
These estimates require two more graph derivatives and one more trace
derivative.  Since $\kappa=O((1+Y)^Ma)$ and $u=O_{H^r}((1+Y)^Ma^2)$,
\eqref{eq:exact-kinetic-low-source} and the explicit kinetic remainder
give $P_{\ge2}\mathcal K=O_{H^r}((1+Y)^Ma^2)$.
Projecting the Bernoulli equation onto the higher modes and using
$\sigma a/q^2=O((1+Y)^Ma)$
gives $f_Y/q=O_{H^r}((1+Y)^Ma^2)$.
Projection onto the first harmonic and the centred kinematic equations give the stated bounds for
$\alpha_Y$, $\eta_Y$, and $R_Y$.

\emph{Step 2. Outer derivatives.}
At the outer boundary the geometry is nondegenerate and $B_+=O(q)$
up to the regular zero-flux extension.  The factors $t_Y$ and $t_Y/q$
are $O((1+Y)^Ma^2)$.  The trace and product estimates at fixed scale therefore
give the same bound for $X_{+,Y}$ and $g_Y/q$.
The four additional derivatives control the normal trace, curvature,
and products of chart coefficients at all the displayed indices.

\emph{Step 3. The flux derivative.}
To estimate the scalar derivative, regard $K_0$ as the reduced quadratic
Dirichlet energy in the physical traces $(\alpha,f,g)$.
Its coefficients depend on the geometry, while its definition at fixed
physical traces is independent of $q$.
The first shape estimates and \eqref{eq:schur-first-differential},
together with the trace derivative bounds above, give
$|K_0|+|(K_0)_Y|\le C(1+Y)^Ma^2$.
The inner area formula \eqref{eq:inner-area-exact} and the regular outer area variation then
give $|(E_0)_Y|\le C(1+Y)^Ma$.
Lemma~\ref{lem:capacity-single-layer-schur}, now applied with the known
geometric derivatives, gives $|\mathfrak c_Y|\le C(1+Y)^Ma$.
Inserting these bounds and $L_Y=2\pi R_Y$ into
\eqref{eq:corrected-q-log-derivative} proves
\eqref{eq:raw-q-first-jet}.
Finally, differentiating the harmonic extension and its zero-flux lift
once with respect to the shape gives the first derivatives of the
normalized condenser and velocity coefficients.  The required first and
second state variations are estimated in Section~\ref{sec:scaled-DN} and
Proposition~\ref{prop:schur-differentiation}.
\end{proof}

\subsection{Pressure reconstruction and the Taylor coefficient}

We reconstruct the pressure from its harmonic part and estimate its
normal derivative on both boundaries.  The inner estimate supplies the
positive Taylor coefficient in the energy estimate below.

\begin{proposition}
\label{prop:actual-taylor-coefficient}
Fix an integer $k\ge0$.  Let a smooth Euler state satisfy
\eqref{eq:first-mode-physical-bootstrap} and
\eqref{eq:first-mode-amplitude-closed} at index $k+8$.
Suppose the outer geometry belongs to the fixed smooth class and remains
separated from the inner boundary.  Assume also
$R\asymp1$, $q\asymp(1+Y)^{-1/2}$ and $|\mathfrak c|\le C$.
Set
\begin{equation*}
\mathfrak a_\iota=-\partial_{n_\iota}p,
\qquad \mathfrak t_-:=\frac{a^3}{q^2}\mathfrak a_-.
\end{equation*}
Then
\begin{equation}\label{eq:actual-inner-taylor}
\left\|\mathfrak t_--\left(1-\frac1{2(Y+\mathfrak c)}\right)\right\|_{H^k}
\le C_k(1+Y)^{M_k}a.
\end{equation}
In particular, $\mathfrak t_-\ge1/2$ for sufficiently large $Y_0$.
On the outer boundary, put
$\kappa_+=-(Y+\mathfrak c)\partial_{n_+}u_c$.
Then $\|\kappa_+\|_{C^k}\le C_k$ and
\begin{equation}\label{eq:actual-outer-taylor-Y}
t_Y\mathfrak a_+
=\frac{q\Theta}{2(Y+\mathfrak c)}\kappa_+
+O_{C^k}\bigl(q(1+Y)^{M_k}a\bigr).
\end{equation}
Consequently, the bound
$\|t_Y\mathfrak a_+\|_{C^k}
\le C_k((1+Y)^{-3/2}+(1+Y)^{M_k}a)$
gives integrability in $Y$.
\end{proposition}

\begin{proof}
We divide the proof into three steps.

\emph{Step 1. Harmonic pressure and remainder estimates.}
In Cartesian coordinates, incompressibility and irrotationality give
$-\Delta p=\operatorname{tr}((\nabla v)^2)$ and
$\Delta(|v|^2/2)=\operatorname{tr}((\nabla v)^2)$.  Hence
\begin{equation}\label{eq:harmonic-pressure-reconstruction}
\mathsf H:=p+\tfrac12|v|^2\quad\hbox{is harmonic},\qquad
\mathsf H|_{\Gamma_\iota}
=\sigma\mathcal H_\iota+\tfrac12|v_\iota|^2.
\end{equation}
In meridional coordinates, $\mathsf H$ satisfies
$\mathcal L\mathsf H=0$, with $\mathcal L$ defined in
\eqref{eq:axisym-elliptic}.

In a fixed scaled collar outside the inner graph, write $s=r/a$.
The DN calculus gives the scaled estimates
\begin{equation}\label{eq:scaled-inner-collar-velocity}
\frac a q v=-s^{-1}e_r+O_{C^{k+1}}((1+Y)^Ma),
\qquad
\frac{a^3}{2q^2}\partial_{n_-}|v|^2
=1+O_{H^k}((1+Y)^Ma).
\end{equation}
For the first estimate, differentiate the local layer representation
and the zero-flux reconstruction.  The condenser contributes
$-s^{-1}e_r$, and the correction due to the curvature of the core circle is
$O((1+Y)^Ma)$.  The dipole has size $\alpha/q=O((1+Y)^Ma)$, while the
remaining traces and graph have size $O((1+Y)^Ma^2)$.
Boundary elliptic regularity in the fixed scaled collar gives the extra
derivative.  The azimuthal derivatives in the three-dimensional gradient contain the
additional factor $a/\rho$ and satisfy the same error bound.  The normal
derivative estimate follows by differentiating
$|s^{-1}e_r|^2=s^{-2}$ in the direction of the inward fluid normal and using
$n_-=-e_r+O((1+Y)^Ma^2)$.

Set $D=q^2/(2a^2)$ and decompose the harmonic pressure as
\begin{equation*}
\mathsf H=D u_c+\mathsf R.
\end{equation*}
The traces of the remainder satisfy
\begin{equation*}
\|\mathsf R_-\|_{H^{k+1}}
\le C\frac{q^2}{a^2}(1+Y)^Ma,
\qquad
\|\mathsf R_+\|_{H^{k+4}}\le C(1+q^2).
\end{equation*}
Indeed, $a v_-/q=-e_r+O((1+Y)^Ma)$ and
$\mathcal H_-=-a^{-1}+O((1+Y)^M)$, while the outer velocity is
$O_{C^{k+2}}(q)$ and the outer curvature is uniformly bounded.
The normalized inner DN estimate and the outer trace estimate at fixed scale give
\begin{align*}
\|a\partial_{n_-}\mathsf R\|_{H^k}
&\le C(1+Y)^M(\|\mathsf R_-\|_{H^{k+1}}
+\|\mathsf R_+\|_{H^{k+4}}),\\
\|\partial_{n_+}\mathsf R\|_{C^k}
&\le C(1+Y)^M(\|\mathsf R_-\|_{H^{k+1}}
+\|\mathsf R_+\|_{H^{k+4}}).
\end{align*}
On the separated macroscopic boundary, the second estimate follows
from the filled domain layer representation and scalar Schur complement
used for the cross blocks.  Its constants grow at most polynomially in $1+Y$.

\emph{Step 2. The inner Taylor coefficient.}
The condenser flux in the normalization of $q$ is
$1/(Y+\mathfrak c)$.  Its profile and the weighted mean identity give
\begin{equation*}
a\partial_{n_-}u_c
=\frac1{Y+\mathfrak c}\bigl(1+O_{H^k}((1+Y)^Ma)\bigr).
\end{equation*}
Combining this with
$\mathfrak a_-=-\partial_{n_-}\mathsf H
+\tfrac12\partial_{n_-}|v|^2$
proves \eqref{eq:actual-inner-taylor}. The outer trace remainder contributes
only $O((1+Y)^Ma^2)$ after multiplication by $a^3/q^2$.

\emph{Step 3. The outer Taylor coefficient.}
The fixed outer Green representation gives
$\|(Y+\mathfrak c)\partial_{n_+}u_c\|_{C^k}\le C_k$.
Also $\partial_{n_+}|v|^2=O_{C^k}(q^2)$.
Apply the same pressure identity on the outer boundary, using $t_Y=\Theta a^2/q$.
The leading term is $-t_YD\partial_{n_+}u_c$, and the displayed
remainder estimates give the error in \eqref{eq:actual-outer-taylor-Y}.
The resulting bound is integrable in $Y$.
\end{proof}

\begin{lemma}
\label{lem:normalized-taylor-variation}
Fix a lower elliptic index $b_{\rm ell}$ and a common integer loss
$d_{\rm ell}\ge4$ for the collar extension, normalized normal trace,
condenser reconstruction and their first two material variations in
Lemmas~\ref{lem:uniform-harmonic-extension},
\ref{lem:capacity-single-layer-schur} and
Proposition~\ref{prop:full-C2-flux-DN}. Here a loss $d_{\rm ell}$ means
that an output norm of index $r$ is controlled by graph, trace and
direction norms of index at most $r+d_{\rm ell}$, with lower norms
at the fixed index $b_{\rm ell}$. Include the physical outer collar
estimates and two derivatives for Sobolev embeddings in this choice.
The displayed tame estimates in those results permit these choices
independently of $r$ and $a$. Set
\begin{equation}\label{eq:taylor-fixed-spatial-loss}
L_{\rm T}=12+4d_{\rm ell}.
\end{equation}
For $k\ge b_{\rm ell}$, assume the hypotheses of
Proposition~\ref{prop:actual-taylor-coefficient}, strengthening
\eqref{eq:first-mode-physical-bootstrap} to index $k+L_{\rm T}+4$.
Require the outer graph in $H^{k+L_{\rm T}+10}$ in the fixed admissible
chart, with its perturbation bounded by $\mathcal P(Y)a^2$.
Put $\beta=\alpha/q$, $F=f/q$ and $G=g/q$.
Suppose
\begin{equation*}
\begin{gathered}
|R_Y|+\|\eta_Y\|_{H^{k+L_{\rm T}+2}}+
\|F_Y\|_{H^{k+L_{\rm T}+1}}
+\|G_Y\|_{H^{k+L_{\rm T}+4}}+
\|X_{+,Y}\|_{H^{k+L_{\rm T}+4}}\le\mathcal P(Y)a^2,
\\ |\beta_Y|\le\mathcal P(Y)a,
\end{gathered}
\end{equation*}
\begin{equation*}
|\mathfrak c_Y|\le\mathcal P(Y)a,\qquad
p_q:=\frac{q_Y}{q}
=-\frac1{2(Y+\mathfrak c)}+O(\mathcal P(Y)a).
\end{equation*}
Then
\begin{equation}\label{eq:normalized-taylor-variation}
\|\partial_Y\mathfrak t_-\|_{H^k}
\le C_k\bigl((1+Y)^{-2}+(1+Y)^{M_k}a\bigr).
\end{equation}
\end{lemma}
\begin{proof}
We first record the derivative budget. The normalized velocity uses
one harmonic reconstruction and a first spatial derivative. Forming
the pressure traces adds products and two graph derivatives for
curvature. The harmonic pressure reconstruction and its conormal
trace form a second elliptic operation. On differentiating in $Y$,
each of these two reconstructions is replaced by its first material
variation, whose defining source is a coefficient variation times
the preceding gradient. Thus at most four successive applications
of the elliptic estimates with loss $d_{\rm ell}$ suffice. The
additional twelve derivatives cover the explicit curvature and
gradient derivatives, the trace indices $k+1,k+4$, and the fixed
embeddings. The outer coefficient matrices involve at most one
derivative of the normal graph, and the imposed outer index
$k+L_{\rm T}+10$ covers the same operations there. The stated
spatial and first $Y$-derivative bounds control these operations.

Normalize the harmonic remainder by
$\widetilde{\mathsf R}=(a^2/q^2)\mathsf R$.
Its inner trace and its first $Y$ derivative are
$O_{H^{k+1}}((1+Y)^Ma)$. The outer trace and its first derivative are
$O_{H^{k+4}}((1+Y)^Ma^2)$.
The normalized collar velocity and its derivative equal
$-s^{-1}e_r+O((1+Y)^Ma)$ and $O((1+Y)^Ma)$, respectively.
These bounds follow from the layer formula for the normalized traces
and its derivative.  The curvature contribution is
$\sigma a^2\mathcal H_-/q^2=O((1+Y)^Ma)$, with the same derivative bound.
Pull the normalized inner DN operator back to the scaled collar and
differentiate its defining equation.  At the circle, the principal part
is independent of $a$.  Differentiating the axisymmetric weight and the
graph perturbations gives coefficient derivatives of size
$O((1+Y)^Ma)$, while the differentiated logarithmic scalar inverse has
at most polynomial growth.  Write the pulled-back harmonic equation as
$\operatorname{div}(A_{\rm ell}\nabla u)=0$, where $A_{\rm ell}$ is its coefficient matrix.
The derivative of $u$ solves this elliptic equation with source
$-\operatorname{div}((A_{\rm ell})_Y\nabla u)$ and the differentiated trace.
The inverse estimate in the collar and the matching estimate on the filled domain therefore give
\begin{equation*}
\|\partial_Y(a\partial_{n_-}\widetilde{\mathsf R})\|_{H^k}
\le C(1+Y)^Ma.
\end{equation*}
The differentiated condenser expansion is
\begin{equation*}
\partial_Y(a\partial_{n_-}u_c)
=-\frac{1+\mathfrak c_Y}{(Y+\mathfrak c)^2}
+O_{H^k}((1+Y)^Ma).
\end{equation*}
Finally use the exact normalized pressure identity
\begin{equation*}
\mathfrak t_-=-\tfrac12a\partial_{n_-}u_c
-a\partial_{n_-}\widetilde{\mathsf R}
+\tfrac12a\partial_{n_-}|av/q|^2.
\end{equation*}
The last term is $1+O((1+Y)^Ma)$ with derivative
$O((1+Y)^Ma)$ in the scaled collar.  Differentiating this identity and
using $\mathfrak c_Y=O((1+Y)^Ma)$ proves
\eqref{eq:normalized-taylor-variation}.
\end{proof}

\subsection{Linearized energy and terminal estimates}

The estimates below control the backward evolution as the effective
capillary coefficient tends to zero.  The principal growth rate in
Theorem~\ref{thm:natural-variable-principal-energy} is independent of
the Sobolev index, which permits a common terminal interval at all fixed regularities.

Choose a smooth even function $\varrho(\xi)$ equal to $|\xi|$ for
$|\xi|\ge 3/2$ and bounded below by $1$.  On the projected space
$P_{\ge2}$, the multiplier $\varrho(D_\theta)$ equals $|D_\theta|$.  We write $T_a$ for a Bony paraproduct and
$\operatorname{Op}^{\rm w}(a)$ for full periodic Weyl quantization.
We next specify the cutoff in coefficient frequencies used in the energy.

Restricting coefficient frequencies relative to the input frequency
keeps the required coefficient regularity independent of the energy index.
The composition formulas below follow directly from the Fourier matrix.

\begin{lemma}
\label{lem:finite-bony-weyl}
Choose a real smooth even function $\chi_0$, equal to one on
$[-\epsilon_1,\epsilon_1]$ and zero outside
$(-\epsilon_2,\epsilon_2)$, with fixed sufficiently small
$0<\epsilon_1<\epsilon_2<1/10$.  Put
$\langle\xi\rangle=(1+\xi^2)^{1/2}$ and
$\chi(\ell,\xi)=\chi_0(\ell/\langle\xi\rangle)$.
Define
\begin{equation*}
[\mathcal W(a)]_{jk}
=\chi(j-k,(j+k)/2)\widehat a(j-k,(j+k)/2).
\end{equation*}
Here $j,k\in\mathbb Z$ are Fourier indices, and $\widehat a(\ell,\xi)$
is the $\ell$th Fourier coefficient of $a(\theta,\xi)$ in $\theta$.
Our conventions for the two terms in the composition formula are
\begin{equation*}
\{a,b\}=a_\xi b_\theta-a_\theta b_\xi,\qquad
\mathfrak q_2(a,b)=-\frac18\left(
a_{\xi\xi}b_{\theta\theta}-2a_{\theta\xi}b_{\theta\xi}
+a_{\theta\theta}b_{\xi\xi}\right).
\end{equation*}
Let $a=a_0(\xi,B)+a_1$ and $b=b_0(\xi,B)+b_1$ be scalar symbols
of orders $m,m'$.
Assume their symbol bounds through eight angular and twenty frequency
derivatives are uniform for $0<B\le1$, with size one for the multipliers
and $\epsilon$ for the perturbations.  Then, for every fixed $r$,
\begin{align}
\mathcal W(a)\mathcal W(b)
&=\mathcal W\left(ab+\frac1{2i}\{a,b\}+\mathfrak q_2(a,b)\right)
+\mathcal R_3,\label{eq:bw-moyal}\\
\|\mathcal R_3u\|_{H^r}
&\le C_{r,m,m'}(\epsilon+\epsilon^2)
\|u\|_{H^{r+m+m'-3}}.\label{eq:bw-remainder}
\end{align}
The number of angular derivatives is independent of $r$.
For real $a$, $\mathcal W(a)$ is symmetric on smooth periodic functions.
If $\{a,b\}=0$, the commutator also gains three orders.
The estimates remain valid after one $Y$ derivative when the differentiated
symbols satisfy the same bounds.  The constants are independent of $B$.
The product formula also holds for finite matrix-valued symbols, with
the matrix factors in the displayed order.  In this case, symmetry
requires a Hermitian symbol, and the commutator follows from the full
ordered composition formula.
\end{lemma}
\begin{proof}
Interchanging $j,k$ and taking complex conjugates proves symmetry,
since $a$ is real and $\chi$ is even in its first argument.
In every interaction retained by the cutoff, the input, output, and intermediate
frequencies are comparable.  An $H^r$ weight can therefore be transferred
between them with a constant depending on $r$, using the same fixed
number of angular derivatives of the symbols.
Taylor expansion of the two frequency arguments in the matrix product
through total degree two gives the Moyal polynomial.  Indeed, the shifts
are one half of the opposite coefficient frequency, so the linear terms
give $(2i)^{-1}\{a,b\}$ and the three quadratic terms give
$\mathfrak q_2(a,b)$.  Every omitted term has total differential degree
at least three.  Three frequency derivatives lower the symbol order by three.  The eight angular derivatives control three powers of the coefficient Fourier indices
and leave summable convolution kernels.
The composed cutoff differs from the cutoff of the product only when
a coefficient frequency is comparable to the input frequency.  In this
region, the same angular derivatives give at least three orders of gain.  The weighted Schur estimate proves \eqref{eq:bw-remainder}.
Fixed low-frequency completions satisfy the same estimates.
The remainder vanishes when both symbols are Fourier multipliers.
Bilinearity leaves two terms of size $\epsilon$ and one of size $\epsilon^2$.
The even terms cancel in the scalar commutator.  Applying the product
rule to this bilinear expansion gives the differentiated estimate with
constants uniform in $B$.
\end{proof}

We next derive the normal variational equations for the boundary
displacement and the good unknown.

\begin{proposition}
\label{prop:exact-normal-variational-system}
Linearize the globally potential Euler equations with fixed $\sigma\ge0$
at a smooth trajectory and at fixed physical time.  On either boundary
write the normal displacement as $H_\iota$ and set
\begin{equation*}
U_\iota=\delta\psi_\iota-B_\iota H_\iota
=\delta\Phi|_{\Gamma_\iota},\qquad \mathbf U=(U_-,U_+).
\end{equation*}
Parametrize the variation so that
$\delta X_\iota=H_\iota n_\iota$ at the base state.
If $X_{\iota,t}=v_\iota+w_\iota$, with $w_\iota$ tangential, put
$D_\iota=\partial_t-w_\iota\cdot\nabla_{\Gamma_\iota}$,
$c_\iota=n_\iota\cdot(\nabla v)n_\iota$, and
$J_\iota=-\Delta_{\Gamma_\iota}-|A_{\Gamma_\iota}|_{\mathrm{HS}}^2$,
where $A_{\Gamma_\iota}$ is the second fundamental form of $\Gamma_\iota$,
and $|A_{\Gamma_\iota}|_{\mathrm{HS}}^2$ denotes the sum of the squares
of its two principal curvatures.
Then the system on the two boundary components is
\begin{equation}\label{eq:exact-normal-variational-system}
D_\iota H_\iota=(\mathsf G\mathbf U)_\iota+c_\iota H_\iota,\qquad
D_\iota U_\iota=-(\mathfrak a_\iota+\sigma J_\iota)H_\iota+\delta b(t).
\end{equation}
Here $\mathsf G$ is the ordinary normal DN matrix, related to the
weighted flux matrix by $(\mathsf G\boldsymbol\psi)_\iota
=R(\rho_\iota j_\iota)^{-1}\mathfrak G_\iota\boldsymbol\psi$.
The scalar $\delta b(t)$ is the variation of the common Bernoulli constant.

\end{proposition}
\begin{proof}
Differentiate $(X_{\iota,t}-v)\cdot n_\iota=0$ using
$\delta X_\iota=H_\iota n_\iota$ and
$\delta n_\iota=-\nabla_{\Gamma_\iota}H_\iota$ to obtain
\begin{equation*}
\partial_tH_\iota=(\mathsf G\mathbf U)_\iota+c_\iota H_\iota
+w_\iota\cdot\nabla_{\Gamma_\iota}H_\iota.
\end{equation*}
Here $\delta\Phi$ is harmonic and has trace $U_\iota$ on each boundary.
The linearized Bernoulli equation is
$\partial_t\delta\Phi+v\cdot\nabla\delta\Phi+\delta p=\delta b(t)$,
and the boundary condition gives
\begin{equation*}
\delta p+H_\iota\partial_{n_\iota}p
=\sigma J_\iota H_\iota.
\end{equation*}
The trace chain rule along $X_{\iota,t}=v_\iota+w_\iota$ now gives
the second equation.

\end{proof}

\noindent\textbf{The physical subprincipal symbols.}
The principal symbols of the weighted axisymmetric DN map and the normal
Jacobi operator are powers of the same arclength frequency.  In the
scaled evolution, the second-order curvature term carries the factor
$B_q$.

\begin{proposition}
\label{prop:physical-local-symbols}
Put $j=j_\eta$, $w=(\rho/R)j$, $A=j^{-2}$ and
$\mathcal J_\eta=a^2J_-=a^2(-\Delta_{\Gamma_-}-|A_{\Gamma_-}|_{\mathrm{HS}}^2)$.
Its half-density conjugate is
\begin{equation}\label{eq:Jacobi-half-density-symbol}
w^{1/2}\mathcal J_\eta w^{-1/2}
=\operatorname{Op}^{\rm w}(A\xi^2-1+c_0),\qquad
c_0=w^{-1/2}\partial_\theta(A\partial_\theta w^{1/2})
+\tfrac14A_{\theta\theta}+1-a^2|A_{\Gamma_-}|_{\mathrm{HS}}^2.
\end{equation}
At each fixed smooth geometry, the local Weyl symbol of the inner
diagonal DN block after conjugation by $w^{1/2}$ is
\begin{equation}\label{eq:DN-half-density-symbol}
w^{1/2}a\mathsf G_{--}w^{-1/2}
=\operatorname{Op}^{\rm w}(x+\ell_0+r_{-1}),\qquad
x=|\xi|/j,\qquad
\ell_0=-\frac a2\,\partial_{n_-}\log\rho ,
\end{equation}
modulo a smoothing operator, with $r_{-1}$ of order $-1$.
\end{proposition}
\begin{proof}
On axisymmetric scalars,
$\mathcal J_\eta=-w^{-1}\partial_\theta(wA\partial_\theta)-a^2|A_{\Gamma_-}|_{\mathrm{HS}}^2$.
Conjugation by $w^{1/2}$ gives
$-\partial_\theta A\partial_\theta
+w^{-1/2}\partial_\theta(A\partial_\theta w^{1/2})
-a^2|A_{\Gamma_-}|_{\mathrm{HS}}^2$.
Since
$\operatorname{Op}^{\rm w}(A\xi^2)
=-\partial_\theta A\partial_\theta-A_{\theta\theta}/4$,
this proves \eqref{eq:Jacobi-half-density-symbol}.

For the DN map, set $\widetilde u=\sqrt\rho\,u$.  Direct differentiation gives
\begin{equation*}
(\partial_\rho^2+\rho^{-1}\partial_\rho+\partial_z^2)u
=\rho^{-1/2}\left(\Delta_{\rho,z}+\frac1{4\rho^2}\right)\widetilde u,
\qquad
\sqrt\rho\,\partial_{n_-}u=\partial_{n_-}\widetilde u-\tfrac12(\partial_{n_-}\log\rho)\widetilde u.
\end{equation*}
For the planar operator, use boundary arclength $s$ and signed distance
$r$, positive in the outward normal direction.  The metric is
$dr^2+(1+\kappa_{\rm pl}(s)r)^2ds^2$, with signed curvature $\kappa_{\rm pl}$.
Put $h_r=1+\kappa_{\rm pl}r$.  The leading normal factor is
$\lambda_1=|\xi|/h_r$, and
$\partial_r\lambda_1+(\partial_r\log h_r)\lambda_1=0$.
At $r=0$, the tangential derivatives of $h_r$ vanish, so the order-one
equation in the normal factorization gives
$2\lambda_1\lambda_0=0$.  Thus the planar DN subprincipal symbol
on arclength half-densities is zero.  The interior potential enters the
next equation, at DN order $-1$.
Pulling arclength half-densities back to $\theta$ proves
\eqref{eq:DN-half-density-symbol}.  At a fixed separated geometry, the
other boundary changes the diagonal block only by a smoothing operator.
\end{proof}

\begin{theorem}
\label{thm:natural-variable-principal-energy}
Let $P=P_{\ge2}$, let $\mathcal W$ be the quantization defined with a fixed cutoff in
Lemma~\ref{lem:finite-bony-weyl}, and put
$x(Y,\theta,\xi)=\varrho(\xi)/j(Y,\theta)$.
Consider the symmetric operators on $P L^2$
\begin{equation*}
\mathsf L=\Theta P\mathcal W(x+\ell_0)P,\qquad
J=P\mathcal W(x^2-1)P,\qquad
\mathsf M=\Theta(\tau_0P+BJ),
\end{equation*}
where $\Theta=\Theta(Y)$, $3/4\le\tau_0(Y)\le5/4$, and $\ell_0$ is
real of order zero.  Suppose $j-1$, $\Theta-1$ and $\ell_0$ have
sufficiently small fixed symbol seminorms, bounded together with their
first $Y$ derivatives and $\tau_{0,Y}$ by $\epsilon(Y)$.
The required coefficient regularity is independent of the energy index.
Assume
\begin{equation*}
0<B\le1,\qquad B_Y/B=-1-2p_q,\qquad
|p_q|\le1/4,\quad p_q\le\epsilon(Y),\qquad
\epsilon\ge0,\quad\epsilon\in L^1_Y.
\end{equation*}
Let $T=iP\mathcal W(V\xi)P$, where $V$ is real with fixed symbol seminorms bounded by $\epsilon$.  For a smooth solution of
\begin{equation*}
(\partial_Y+T)h=h+\mathsf L\pi+f_h,\qquad
(\partial_Y+T)\pi=-\mathsf M h-(1+p_q)\pi+f_\pi
\end{equation*}
and every fixed $s\ge0$, put $P_s=P\mathcal W(x^s)P$ with the fixed low-frequency completion, and set $h_s=P_sh$, $\pi_s=P_s\pi$.
Define
\begin{equation*}
E_s^2=\langle\mathsf M h_s,h_s\rangle+\langle\mathsf L\pi_s,\pi_s\rangle
+\langle\mathsf M h,h\rangle+\langle\mathsf L\pi,\pi\rangle.
\end{equation*}
Then
\begin{equation*}
-\partial_Y E_s\le(1+C_s\epsilon)E_s+C_s\mathcal N_s(f_h,f_\pi),
\qquad
\mathcal N_s(f_h,f_\pi)=\|f_h\|_{H^s}+\sqrt B\|f_h\|_{H^{s+1}}
+\|f_\pi\|_{H^{s+1/2}}.
\end{equation*}
The energy is uniformly equivalent to
$\|h\|_{H^s}+\sqrt B\|h\|_{H^{s+1}}+\|\pi\|_{H^{s+1/2}}$.
The leading backward rate $1$ is independent of $s$ and $B$.
\end{theorem}
\begin{proof}
We divide the proof into three steps.

\emph{Step 1. Coercivity and coefficient derivatives.}
Small relative perturbations of the circular forms give
\begin{equation*}
\langle \mathsf L u,u\rangle\asymp\|u\|_{H^{1/2}}^2,\qquad
\langle Ju,u\rangle\asymp\|u\|_{H^1}^2,\qquad
\langle \mathsf M u,u\rangle\asymp\|u\|_2^2+B\|u\|_{H^1}^2.
\end{equation*}
The reference Jacobi form is $|D_\theta|^2-1$ on $P L^2$.
Boundedness with a fixed cutoff controls the perturbations as forms of their
respective orders.  The coefficient assumptions imply
\begin{equation*}
|\langle \mathsf L_Yu,u\rangle|\le C\epsilon\langle \mathsf L u,u\rangle,\qquad
-\langle \mathsf M_Yu,u\rangle\le(1+C\epsilon)\langle \mathsf M u,u\rangle .
\end{equation*}
For the second bound, use
$-B_Y/B=1+2p_q\le1+2\epsilon$ and the positivity of $J$.

\emph{Step 2. Commutator estimates.}
Elliptic parametrices for $P_s$, together with the added lower energy,
give the asserted equivalence uniformly in $B$.
The capillary commutators preserve the derivative count because
\begin{equation*}
\{x^s,x\}=\{x^s,x^2\}=0.
\end{equation*}
The finite Weyl calculus therefore gives orders $s-2$ and $s-1$
for $[P_s,\mathcal W(x)]$ and
$[P_s,\mathcal W(x^2)]$, respectively.  The $\ell_0$ commutator has order $s-1$.  The scalar term
$\Theta\tau_0P$ commutes with $P_s$.  Consequently
\begin{equation*}
\|[P_s,\mathsf L]\pi\|_2+\sqrt B\|[P_s,\mathsf L]\pi\|_{H^1}
+\|[P_s,\mathsf M]h\|_{H^{1/2}}\le C_s\epsilon E_s.
\end{equation*}
The finite projection errors satisfy the same bound because the fixed
cutoff controls all interactions between coefficient and input
frequencies.

The symbol identity
$(x^s)_Y=-s(j_Y/j)x^s$ controls the terms arising from $Y$ differentiation
of $P_s$ in the same norms.
Since $T$ is skew-symmetric, its contribution to a natural quadratic form
is a commutator with $\mathsf L$ or $\mathsf M$.  These commutators have
orders one and two, respectively, with the factor $B$ in the order-two
part.  The corresponding forms are bounded by $C\epsilon$ times the
natural energy.
Also, $[P_s,T]$ has order $s$ and coefficient size $C_s\epsilon$,
so its contribution is bounded by the same mixed energy.

\emph{Step 3. The energy inequality.}
The cross terms containing $\mathsf L$ and $\mathsf M$ cancel exactly.
The remaining dilation and coefficient terms in the negative
derivative of each quadratic form are
\begin{equation*}
\langle(-\mathsf M_Y-2\mathsf M)h_s,h_s\rangle+
\langle(-\mathsf L_Y+2(1+p_q)\mathsf L)\pi_s,\pi_s\rangle
\le (2+C_s\epsilon)E_s^2 .
\end{equation*}
The first term is bounded above by
$(-1+C\epsilon)\langle \mathsf M h_s,h_s\rangle$.
The commutator estimates and the pairings with the source terms give
$-\partial_Y(E_s^2)\le(2+C_s\epsilon)E_s^2+
C_sE_s\mathcal N_s(f_h,f_\pi)$.  Regularization at zero completes the proof.
\end{proof}

\begin{lemma}
\label{lem:zeroth-mixed-cost}
For a fixed Sobolev index $s$ and $0<B\le1$, set
\begin{equation*}
\mathfrak E_s(h,\pi)^2
=\|h\|_{H^s}^2+B\|h\|_{H^{s+1}}^2
+\|\pi\|_{H^{s+1/2}}^2.
\end{equation*}
Then
\begin{equation}\label{eq:weighted-half-derivative}
\|h\|_{H^{s+1/2}}
\le B^{-1/4}\bigl(\|h\|_{H^s}^2+B\|h\|_{H^{s+1}}^2\bigr)^{1/2}.
\end{equation}
For an operator $C$ of order zero with a fixed cutoff and boundedness
seminorm $c_C$ at the fixed coefficient regularity,
\begin{equation}\label{eq:zeroth-mixed-cost}
\|Ch\|_{H^{s+1/2}}\le C_s c_C B^{-1/4}\mathfrak E_s(h,\pi).
\end{equation}
Thus, in this mixed norm, the operator norm distance from
$(h,\pi)\mapsto(h,\pi+\mathcal W(c)h)$ to the identity is at most
$C_sB^{-1/4}\|c\|_{C^r}$, for a fixed sufficiently large coefficient index $r$.
For fixed $\sigma>0$, $B=B_q=\sigma a/q^2$ and
$q\asymp(1+Y)^{-1/2}$, an order-zero operator with seminorm
$c_C\le\mathcal P(Y)a^\nu$, $\nu>1/4$, contributes
an integrable homogeneous energy error in \eqref{eq:zeroth-mixed-cost},
with $\mathcal P$ chosen as a nonnegative polynomial majorant.
\end{lemma}
\begin{proof}
The Fourier inequality
$\sqrt B\langle\xi\rangle\le1+B\langle\xi\rangle^2$
proves \eqref{eq:weighted-half-derivative}. Order-zero boundedness
then gives \eqref{eq:zeroth-mixed-cost} and the estimate for the triangular change of variables.
Finally, $B_q^{-1/4}=\sigma^{-1/4}a^{-1/4}q^{1/2}$ gives
$c_CB_q^{-1/4}\le C\mathcal P(Y)a^{\nu-1/4}q^{1/2}$, which is
integrable for $\nu>1/4$.
\end{proof}

The following lemma combines the inner and outer energy inequalities
and gives bounds independent of the terminal coordinate $Y_N$.

\begin{lemma} 

\label{lem:zero-coupled-terminal-budget}
Let $a(Y)=\ell e^{-Y}$, $Y_0\ge0$, and let $Z,E,O\ge0$ be absolutely
continuous on $[Y_0,Y_N]$, with zero terminal values.  Let $k,\epsilon\ge0$
belong to $L^1([Y_0,\infty))$ and $0\le P(Y)\le C(1+Y)^M$, $M\ge0$.
If $-Z_Y\le(\lambda+k)Z+Pa^3$ with $\lambda<3$, then
\begin{equation}\label{eq:backward-polynomial}
Z(Y)\le e^{\|k\|_{L^1}}
\int_Y^{Y_N}e^{\lambda(S-Y)}P(S)a(S)^3\,dS
\le C'(1+Y)^Ma(Y)^3.
\end{equation}
For the coupled estimates, enlarge $P$ to $C(1+Y)^M\ge1$.
The system
\begin{equation}\label{eq:coupled-terminal-input} 
-E_Y\le(1+\epsilon)E+P(a^3+a^2O),\qquad
|O_Y|\le\epsilon O+P(a^2+aE) 
\end{equation}
implies
\begin{equation*}
E(Y)\le C'(1+Y)^Ma(Y)^3,\qquad
O(Y)\le C'(1+Y)^Ma(Y)^2.
\end{equation*}
The system arising from the Taylor energy estimate
\begin{equation}\label{eq:zero-coupled-terminal-input} 
-E_Y\le(1+\epsilon)E+P(a^3+a^2O),\qquad
-O_Y\le\epsilon O+P(a+E) 
\end{equation}
implies
\begin{equation}\label{eq:zero-coupled-terminal-output}
E(Y)\le C'(1+Y)^{2M}a(Y)^3,\qquad
O(Y)\le C'(1+Y)^Ma(Y).
\end{equation}
All constants are independent of $Y_N$.  The starting point $Y_0$ can
be kept fixed when the polynomial and the finite coefficient integrals
depend on a fixed higher Sobolev index.
\end{lemma}
\begin{proof}
The integrating factor gives the first bound in
\eqref{eq:backward-polynomial}.  Substitute $S=Y+r$ and use
\begin{equation*}
a(Y+r)=a(Y)e^{-r},\qquad
(1+Y+r)^M\le(1+Y)^M(1+r)^M,
\qquad \int_0^\infty e^{-(3-\lambda)r}(1+r)^M\,dr<\infty
\end{equation*}
to obtain its second bound.
For \eqref{eq:coupled-terminal-input}, set $Z=E+aO$.  Since $a_Y=-a$,
\begin{equation*}
-Z_Y\le(1+k)Z+2Pa^3,\qquad
k=\epsilon+P(a+a^2)\in L^1.
\end{equation*}
Apply the scalar estimate with $\lambda=1$ and use $E\le Z$, $aO\le Z$.

For \eqref{eq:zero-coupled-terminal-input}, set $Z=E+4a^2PO$.
Since $-\partial_Y\log(a^2P)=2-M/(1+Y)\le2$,
\begin{equation*}
-Z_Y\le(9/4+k)Z+(P+4P^2)a^3,\qquad
k=\epsilon+4a^2P^2\in L^1.
\end{equation*}
Here $Pa^2O=\tfrac14(4a^2PO)$, while $4a^2P^2E\le kZ$.
Apply the scalar estimate with $\lambda=9/4$ and polynomial degree $2M$.
The bounds $E\le Z$ and $O\le Z/(4a^2P)$ give
\eqref{eq:zero-coupled-terminal-output}.
\end{proof}

\section{Nonlinear energy estimates}
\label{sec:nonlinear-energy}

We prove Theorem~\ref{thm:one-state-closed}, which continues the
canonical terminal solutions to a common initial value of $Y$ for each
fixed $\sigma>0$.  We retain the notation and reflection symmetry of
Section~\ref{sec:boundary-dynamics}.  Localization of the zero-flux
Dirichlet--Neumann operator and paralinearization in normal variables
reduce the problem to coupled inner and outer energy estimates.

\subsection{Normal variables and the Dirichlet--Neumann operator}

For the local variational identities, we hold
$a,R,q,\Theta,\kappa,p_q,\beta$ fixed at their instantaneous values,
where $\beta=\alpha/q$.  In the nonlinear paralinearization we hold only
$a,R,\beta$ fixed and reconstruct $q,\Theta,\kappa,p_q$ from the state.
The resulting scalar differentials are included in the smooth kernel
part of the differential. Differentiation in $Y$ acts on all these
parameters, including $R$ and $\beta$.
We use the polynomial bounds defined in Section~\ref{sec:scaled-DN}.
The required coefficient regularity is independent of the energy index.

\begin{proposition} 
\label{prop:polar-full-cancellation}
Let $\sigma\ge0$ and use the variables of
Proposition~\ref{prop:exact-polar-raw-system}.  Put
\begin{equation*}
b=\frac{dn-e\upsilon}{j},\qquad z=\frac{en+d\upsilon}{j},\qquad
\mathcal V=\frac{\Theta z+\kappa\sin\theta}{d},\qquad
\mathcal D=\partial_Y+\mathcal V\partial_\theta .
\end{equation*}
For spatial variations at fixed $a,R,q,p_q,\Theta,\kappa,\beta$, set
\begin{equation*}
H=a\,\delta\eta,\qquad Z=a\,\delta F,\qquad
h_0=(d/j)H,\qquad \pi_0=-(Z+bH),\qquad
d_-=t_Yn_-\cdot(\nabla v)n_-.
\end{equation*}
The local inner system is
\begin{equation}\label{eq:polar-complete-normal-block}
\mathcal D h_0=\Theta a\mathsf G_{--}\pi_0+d_-h_0,\qquad
\mathcal D\pi_0=-\Theta(\mathfrak t_-+B_q\mathcal J_\eta)h_0
-(1+p_q)\pi_0,
\end{equation}
with $\mathcal J_\eta$ as in Proposition~\ref{prop:physical-local-symbols}.
The source terms contain the contributions from the other boundary,
the constant lift, scalar modulation, the common gauge and finite
projections.  Coefficients introduced by the change of unknown are
differentiated along the base trajectory.  Under the physical trace bounds
used in Proposition~\ref{prop:raw-state-first-jet},
$d_-=1+O_{C^k}(\mathcal P a)$ and
$\mathcal V=O_{C^k}(\mathcal P a)$, with a fixed number of additional
derivatives in the trace bounds.
\end{proposition}
\begin{proof}
The geometric and velocity identities are
\begin{equation*}
N:=n_-=-\frac d j e_r+\frac e j e_\theta,\qquad
\tau_-=\frac e j e_r+\frac d j e_\theta,\qquad
\frac a qv=-b e_r+z e_\theta,\qquad \mathcal D X_-=t_Yv.
\end{equation*}
The last equality follows from
\eqref{eq:polar-eta-exact}--\eqref{eq:polar-chart-eliminated}.
Decompose the radial variation as
$He_r=-h_0N+(e/j)H\tau_-$.  Removing its tangential part changes the
potential trace variation by $-v_{\tau,-}(e/j)H$ and leaves its Eulerian
variation unchanged.  Consequently the normal variables of
Proposition~\ref{prop:exact-normal-variational-system} satisfy
\begin{equation*}
H_-=-h_0,\qquad
U_-=\delta\psi_--v\cdot(He_r)=\frac qa(Z+bH)=-\frac qa\pi_0,
\end{equation*}
up to the constant lift included in the source terms.
In this chart $t_YD_-=\mathcal D$.  Substitution into
\eqref{eq:exact-normal-variational-system}, using
$t_Y=\Theta a^2/q$ and $(a/q)_Y=-(1+p_q)a/q$, gives
\begin{equation*}
\mathcal D h_0=\frac{t_Yq}{a}\mathsf G_{--}\pi_0+d_-h_0,
\qquad
\mathcal D\pi_0=-(1+p_q)\pi_0
-\frac{at_Y}{q}(\mathfrak a_-+\sigma J_-)h_0.
\end{equation*}
The identities $a^3\mathfrak a_-/q^2=\mathfrak t_-$ and
$\sigma a^3J_-/q^2=B_q\mathcal J_\eta$ prove
\eqref{eq:polar-complete-normal-block}.

To compute the coefficient derivative in this change of unknown, differentiate
$b=-(a/q)v\cdot e_r$ along $\mathcal D X_-=t_Yv$.
Euler's equation and $\mathcal D e_r=\mathcal V e_\theta$ give
$\mathcal D b+(1+p_q)b=(a/q)t_Y\nabla p\cdot e_r-\mathcal V z$.
Since $\nabla p=-\mathfrak a_-N+
\sigma\tau_-(aj)^{-1}\partial_\theta\mathcal H_-$ on the boundary,
\begin{equation}\label{eq:zero-inner-pressure-cancellation}
\mathcal D b+(1+p_q)b
=\Theta\mathfrak t_-d/j
+\Theta B_q\frac e{j^2}\partial_\theta\mathscr C-\mathcal Vz.
\end{equation}
Thus the time derivative of the shear in the change of unknown contains
the Taylor coefficient of the Euler solution.  The scaled collar velocity estimate in
Proposition~\ref{prop:actual-taylor-coefficient} gives
$a^2N\cdot(\nabla v)N/q=1+O_{C^k}(\mathcal P a)$.
Together with $\Theta=1+O(\mathcal P a^2)$, this bounds $d_-$.
The formula for $\mathcal V$, the bounds $|\alpha|/q+|R_Y|/a\le
\mathcal P a$, and the high trace bounds give its stated estimate.
\end{proof}

\noindent\textbf{A localization estimate for zero flux.}
We obtain localization estimates uniform in the hole radius by comparing
the energies of the global and local problems.

\begin{proposition}
\label{prop:charge-free-localization}
Let $D$ be a smooth bounded planar domain, let $c$ remain a fixed positive
distance from $\partial D$, and let $w$ be smooth and uniformly strictly
positive on $D$.  Assume that the required finite smooth norms are uniformly bounded
and that the geometry remains nondegenerate.  Let
\begin{equation*}
K_a=c+aK_\eta,\qquad
\partial K_\eta=\{(1+\eta(\theta))e_r(\theta):\theta\in\T\},
\end{equation*}
where $\eta$ is in a fixed small smooth neighborhood of zero.
For an inner trace $f$ modulo constants, let $u_g$ minimize
$\int_{D\setminus K_a}w|\nabla u|^2$ among functions with outer trace
zero and inner trace $f$ plus an arbitrary constant.  Let $\mathfrak K_g$
be the map sending $f$ to its inner conormal density, pulled back to $d\theta$.

Let $w_l$ agree with $w$ for $|x-c|<3a$ and with $w(c)$ for
$|x-c|>4a$, using a fixed scaled cutoff.  Let $u_l$ be the corresponding
exterior energy minimizer with inner trace $f$ modulo constants, normalized
to tend to zero at infinity.  Denote its inner conormal density by
$\mathfrak K_l f$, pulled back to $d\theta$.  Then, uniformly for small $a$,
\begin{equation}\label{eq:weak-small-localization}
\|\mathfrak K_g-\mathfrak K_l\|_{H^{1/2}/\R\to H^{-1/2}}
\le C a^{1/2}.
\end{equation}
For every fixed integer $N\ge0$, the difference has a smooth kernel
$k_a(\theta,\varphi)$ and
\begin{equation}\label{eq:kernel-small-localization}
\|k_a\|_{C^N(\T^2)}\le C_N a^{1/4}.
\end{equation}
The constant in \eqref{eq:kernel-small-localization} depends on finitely
many geometric derivatives.  Their number depends on $N$ and is
independent of the energy index.  Polynomial bounds on parameter dependent geometries
give polynomial factors in these estimates.
The quotient removes constants and retains the inner first harmonic.
\end{proposition}
\begin{proof}
We divide the proof into four steps.

\emph{Step 1. Uniform energy and flux bounds.}
The collar extension and energy argument of
Lemma~\ref{lem:uniform-harmonic-extension} apply with the normalized
coefficient $w/w(c)$.  It gives both minimizing gradients, fixes the
exterior representative by its limit at infinity, and yields
\begin{equation}\label{eq:zero-charge-gradient-bound}
\int w|\nabla u_g|^2+\int w_l|\nabla u_l|^2
\le C\|P_{\ne0} f\|_{H^{1/2}}^2.
\end{equation}
Variation of the free inner constant shows that each minimizer has
zero total inner conormal flux.  Fill the hole with the solution of
$\operatorname{div}(w\nabla u)=0$ having the same boundary trace.
For the local problem, use the corresponding local coefficient.
To bound the filling, write its trace as $f_{\perp}+c_0$, where
$f_{\perp}=P_{\ne0}f$.  Choose a uniformly bi-Lipschitz map
$\Psi_\eta:B_1\to K_\eta$ whose boundary value is
$\Psi_\eta(e_r(\theta))=(1+\eta(\theta))e_r(\theta)$.
Such a map is obtained by extending the radial perturbation in
$1/2<r<1$ and keeping the map equal to the identity near the origin.
If $E f_{\perp}$ is the harmonic extension to $B_1$, then
\begin{equation*}
\int_{B_1}|\nabla E f_{\perp}|^2
=2\pi\sum_{n\ne0}|n|\,|\widehat f_{\perp}(n)|^2
\le C\|f_{\perp}\|_{H^{1/2}}^2.
\end{equation*}
Use $c_0+(E f_{\perp})\circ\Psi_\eta^{-1}$ as a competitor after the
change of variables $x=c+ay$.  The factors $a^2$ from area and $a^{-2}$
from the squared gradient cancel.  Uniform ellipticity, the uniform
bounds for $D\Psi_\eta$ and its inverse, and minimization therefore give
\begin{equation*}
\int_{K_a}w|\nabla u_{\rm fill}|^2
\le C\|P_{\ne0}f\|_{H^{1/2}}^2,
\end{equation*}
with $C$ independent of $a$ and of the free trace constant.
The same argument applies to $w_l$.
The jump density $\mu$, including the scaled arclength
factor, satisfies
\begin{equation}\label{eq:jump-bound-localization}
\int\mu=0,\qquad
\|\mu\|_{H^{-1/2}}\le C\|P_{\ne0} f\|_{H^{1/2}}.
\end{equation}
For the second bound, test the distributional equation of the
filled solution against a two-sided collar extension of an arbitrary
$H^{1/2}$ boundary function and apply Cauchy--Schwarz together with
\eqref{eq:zero-charge-gradient-bound}.

\emph{Step 2. Decay away from the hole.}
Let $G_D$ be the Dirichlet Green kernel of $-\operatorname{div}(w\nabla)$.
The filled global solution equals
\begin{equation*}
u_g(x)=\int_\T
[G_D(x,c+a\zeta(\varphi))-G_D(x,c)]\mu(\varphi)\,d\varphi,
\qquad \zeta=(1+\eta)e_r.
\end{equation*}
The zero-mean identity in \eqref{eq:jump-bound-localization}
justifies the subtraction.
Local Green estimates, obtained from the logarithmic parametrix and
interior elliptic estimates, give for $Ca\le r=|x-c|\le r_0$
\begin{equation}\label{eq:dipole-annulus-estimate}
|u_g(x)|\le C(a/r)\|P_{\ne0} f\|_{H^{1/2}},\qquad
|\nabla u_g(x)|\le C(a/r^2)\|P_{\ne0} f\|_{H^{1/2}}.
\end{equation}
The same bounds hold for the $H^{1/2}_\varphi$ norms of the
subtracted kernel and its $x$ gradient, since every derivative of
$a\zeta$ contributes a factor $a$.  For the local solution, the same
estimates follow by expanding its zero-flux harmonic extension outside
$4a$ in decaying nonzero Fourier modes and using the collar energy bound
for the trace.

\emph{Step 3. Comparison of the two energies.}
Set $r_a=\sqrt a$ after a fixed nondimensionalization of the macroscopic
lengths.  Cut either solution off between $r_a$ and $2r_a$ and use it as
a competitor for the other problem.  Within $2r_a$ the relative
coefficient difference is $O(r_a)$.  The extra cutoff energy and the
energy in the cutoff annulus are $O(a^2/r_a^2)\|P_{\ne0} f\|_{H^{1/2}}^2$
by \eqref{eq:dipole-annulus-estimate}.  Applying minimality to both competitors gives
\begin{equation*}
|\langle(\mathfrak K_g-\mathfrak K_l)f,f\rangle|
\le C(r_a+a^2/r_a^2)\|P_{\ne0} f\|_{H^{1/2}}^2
\le C a^{1/2}\|P_{\ne0} f\|_{H^{1/2}}^2.
\end{equation*}
Both operators are self-adjoint.  Polarization proves
\eqref{eq:weak-small-localization}.

\emph{Step 4. Estimates for the smooth kernel.}
For \eqref{eq:kernel-small-localization}, the two solutions satisfy
the same equation near the inner boundary, where their traces differ
by a constant.  Subtract this constant from their difference.
The resulting function has zero inner Dirichlet trace, and its scaled
collar $H^1$ norm is controlled by
\eqref{eq:zero-charge-gradient-bound}.  Boundary regularity in a
smaller scaled collar gives, for every fixed $m$,
\begin{equation*}
\mathfrak K_g-\mathfrak K_l:H^{1/2}\longrightarrow H^m
\quad\hbox{with norm at most }C_m.
\end{equation*}
Adding a constant leaves the conormal trace unchanged.  By
self-adjointness, the dual estimate is
$H^{-m}\to H^{-1/2}$.  Interpolating these two estimates, with $m$
sufficiently large, gives a bounded map $H^{-K}\to H^K$ for any fixed
$K$.  Interpolating this bound with
\eqref{eq:weak-small-localization} at parameter $1/2$ gives
\begin{equation*}
\|\mathfrak K_g-\mathfrak K_l\|_{H^{-n}\to H^n}
\le C_n a^{1/4}
\end{equation*}
for any prescribed fixed $n$, by taking $K>2n+1$.
If $k_{jk}$ are its Fourier matrix coefficients, then
$|k_{jk}|\le C_n a^{1/4}\langle j\rangle^{-n}\langle k\rangle^{-n}$.
Taking $n>N+2$ and summing the differentiated Fourier series proves
the kernel estimate with the stated finite regularity dependence.
\end{proof}

\begin{lemma}
\label{lem:finite-tame-dn-symbol}
Fix nonnegative integers $N,N_1$.  The lower regularity indices and
numbers of additional derivatives required below may depend on these fixed integers and are
independent of the output index $r$.  In the centred polar chart, let
$\mathcal L_{\rm in}$ be the normalized zero-flux inner diagonal DN block
conjugated by $w^{1/2}$, where $w=(\rho/R)j$.
On $P_{\ge2}$, it admits the decomposition
\begin{equation*}
\mathcal L_{\rm in}
=\operatorname{Op}^{\rm w}(x+\ell_0+r_l)+K_g,\qquad
x=|\xi|/j,\qquad \ell_0=-\tfrac a2\partial_{n_-}\log\rho,
\end{equation*}
where $r_l$ is a symbol of order $-1$, with a smooth completion at
bounded frequencies.  There is a fixed integer $d_N$.  Write $\boldsymbol\zeta=(\eta,h_+)$ for the two graph
perturbations and define the symbol seminorm by
\begin{equation*}
|\mathfrak a|_{r,m,N}
:=\max_{0\le\beta\le N}\sup_{\xi\in\R}
\langle\xi\rangle^{-m+\beta}
\|\partial_\xi^\beta\mathfrak a(\cdot,\xi)\|_{H^r(\T)}.
\end{equation*}
Then
\begin{align}
|r_l(\boldsymbol\zeta)|_{r,-1,N}
&\le C_r\mathcal P_r(Y)
\bigl(a+\|\eta\|_{H^{r+d_N}}\bigr),
\label{eq:dn-high-tame-symbol}\\
|r_l(\boldsymbol\zeta)-r_l(\widetilde{\boldsymbol\zeta})|_{r_0,-1,N}
&\le C\mathcal P(Y)\|\boldsymbol\zeta-\widetilde{\boldsymbol\zeta}\|_{H^{r_0+d_N}}
\label{eq:dn-low-lipschitz-symbol}
\end{align}
when the fixed lower norms of the graphs are in a fixed small ball.
Here $r\ge r_0$, with $r_0$ fixed.  Increasing $r_0$ and $d_N$ if necessary,
on the fixed ball $\|\boldsymbol\zeta\|_{H^{r_0+d_N}}\le1$ we also have
\begin{equation}\label{eq:dn-high-tame-reference-difference}
|r_l(\boldsymbol\zeta)-r_l(0)|_{r,-1,N}
\le C_r\mathcal P_r(Y)\|\boldsymbol\zeta\|_{H^{r+d_N}}.
\end{equation}
The other retained local symbols satisfy the same estimates at their
respective orders, with a number of additional derivatives independent of $r$.
For the fixed input index $N_1$, the global part satisfies
\begin{align}
\|K_g(0)\|_{H^{-N_1}\to H^r}
&\le C_r\mathcal P_r(Y)a^{1/4},\notag\\
\|K_g(\boldsymbol\zeta)-K_g(0)\|_{H^{-N_1}\to H^r}
&\le C_r\mathcal P_r(Y)\|\boldsymbol\zeta\|_{H^{r+d_{N_1}}}.
\label{eq:global-kernel-tame-operator}
\end{align}
The constants are uniform in $B_q$ and polynomial in $1+Y$.
The finite constant trace and centering corrections satisfy the same
kernel estimates, and their reference coefficients between inner high
modes vanish as $a\to0$.
\end{lemma}
\begin{proof}
We divide the proof into four steps.

\emph{Step 1. Construction of a finite Poisson parametrix.}
All elliptic estimates in this step are on the fixed scaled exterior
chart of Lemma~\ref{lem:flattening-geometry}. Its coefficients equal the
identity outside a fixed compact set. In the boundary collar put
$r_{\rm polar}=ae^t$, $V=\sqrt\rho\,\Phi$,
$u=\log(1+\eta)$, and $\zeta=t-u(\theta)$. Then
\begin{equation*}
LV=(1+u_\theta^2)V_{\zeta\zeta}
-2u_\theta V_{\zeta\theta}+V_{\theta\theta}
-u_{\theta\theta}V_\zeta+W V=0,
\qquad W=\frac{a^2e^{2(\zeta+u)}}{4\rho^2}.
\end{equation*}
The potential is $O(a^2)$ in every required fixed scaled norm.
Extend it smoothly beyond the collar, with a fixed cutoff. Set
$b=u_\theta$, $A=1+b^2$, and, for $|\xi|\ge1$, write
\begin{align*}
e^{-i\theta\xi}Le^{i\theta\xi}
&=L_{0,\xi}+Q_\xi,\\
L_{0,\xi}&=A\partial_\zeta^2-2ib\xi\partial_\zeta-\xi^2,\\
Q_\xi&=-2b\partial_\zeta\partial_\theta
+2i\xi\partial_\theta+\partial_\theta^2
-b_\theta\partial_\zeta+W.
\end{align*}
The roots of $L_{0,\xi}$ are
$\lambda_\pm=(ib\xi\pm|\xi|)/A$.  Their real parts have opposite
signs and absolute values comparable to $|\xi|$, uniformly in
$a$.  Let $J=N+3$ and define
\begin{equation*}
p_0=e^{\lambda_-\zeta},\qquad
p_{k+1}=-L_{0,\xi}^{-1}(Q_\xi p_k),\quad 0\le k<J,
\end{equation*}
where the inverse is taken with zero trace at $\zeta=0$ and decay at infinity.
The explicit Green kernel for this one-dimensional Dirichlet problem
has scale $|\xi|^{-1}$ and exponential decay on the scale
$|\xi|^{-1}$. Consequently $p_k$ has Poisson order $-k$: each normal
derivative raises the order by one, each frequency derivative lowers
it by one, and angular derivatives change only the coefficient norm.
For example, for fixed nonnegative $h,\beta$ and every nonnegative integer $j$,
\begin{equation*}
\|\partial_\theta^j\partial_\xi^\beta
\partial_\zeta^h p_k(\zeta,\cdot,\xi)\|_{L^2_\theta}
\le C_{j,k,h,\beta}
(1+\|\eta\|_{H^{j+d_J}})
\langle\xi\rangle^{h-k-\beta}
(1+\zeta\langle\xi\rangle)^{M}
e^{-c\zeta\langle\xi\rangle}.
\end{equation*}
Here $d_J$ can be chosen to depend only on $J,h,\beta$ and fixed lower
coefficient norms, independently of $j$.  The polynomial exponent $M$ may
depend on $j$. Its contribution to the normal integral is absorbed in
$C_j$.  At each recursive step, $Q_\xi$ uses at most two angular
derivatives. Differentiating
its Green kernel produces derivatives of smooth functions of $b$ with
uniformly nonzero denominators. The product estimate therefore places at most one factor at the
highest angular regularity.  The same argument applies to the first
and second coefficient variations.

Multiply the finite sum $\sum_{k=0}^Jp_k$ by a cutoff equal to one
near $\zeta=0$ and supported entirely inside this boundary chart.
The resulting sum has boundary trace one. Before the cutoff is applied,
its residual is exactly $Q_\xi p_J$. The $L^2_\zeta$ norm of the residual
has tangential symbol order $1/2-J$, since integration of the
exponential profile contributes $\langle\xi\rangle^{-1/2}$.
The cutoff errors, supported at a positive distance from the boundary,
are of arbitrary negative order. Complete the finitely many low
frequencies with a fixed smooth collar lift. Thus the residual source
operator $F_J$ satisfies, after increasing the fixed derivative loss $d_J$,
\begin{equation*}
\|\operatorname{ad}_D^j F_J\|
_{\mathcal L(H^{-N-1},L^2)}
\le C_j(1+\|\eta\|_{H^{j+d_J}}),
\qquad D=-i\partial_\theta.
\end{equation*}
Here $\operatorname{ad}_A B=[A,B]=AB-BA$.
Indeed $1/2-J=-N-5/2<-N-1$. Two additional angular derivatives make
the Fourier convolution summable. All source supports lie in a fixed
scaled compact set. Multiplication by the smooth factors that return
from $V$ to $\Phi$ preserves these estimates.

Return to the weighted divergence equation and let $w$ be the residual
correction in the local exterior problem, with zero Dirichlet trace.
The energy estimate and the local zero-trace $H^2$ estimate give
\begin{equation*}
\|w\|_{H^2(\text{inner collar})}+\|\nabla w\|_{L^2}
\le C\|F_J f\|_{L^2}.
\end{equation*}
Here the compact support of $F_J f$ and the zero boundary trace give
the Poincar\'e estimate needed for the energy test. In particular, the conormal trace lies in $H^{1/2}$, hence in $L^2$.
For the commutators, use the same $H^2$ bound on a fixed annulus
containing the support of all coefficient variations.  Interior
regularity gives the estimate outside the boundary collar.
Write $\mathscr L$ for this flattened weighted operator and $w_j$
for the $j$th commutator of its residual solution operator. Then
\begin{equation*}
\mathscr Lw_j=\operatorname{ad}_D^jF_J
-\sum_{i=1}^j\binom ji(\operatorname{ad}_D^i\mathscr L)w_{j-i},
\qquad w_j|_{\zeta=0}=0.
\end{equation*}
Each $\operatorname{ad}_D^i\mathscr L$ differentiates only coefficients and
is bounded from local $H^2$ to $L^2$ by a coefficient norm with
$i+d_J$ derivatives. Expand this recursion and interpolate the coefficient factors between
their fixed lower norm and $H^{j+d_J}$.  Since the derivative indices
in each product sum to $j$, at most one coefficient is estimated at
high regularity.  Applying the same argument to the conormal coefficients
gives the bound for the boundary correction. The number $d_J$ remains independent of $j$.

\emph{Step 2. Angular and frequency norms of the exact symbol.}
By Proposition~\ref{prop:physical-local-symbols}, the half-density DN
terms of orders one and zero are $x$ and $\ell_0$.
Let $R_l$ be the exact local DN map minus
$\operatorname{Op}^{\rm w}(x+\ell_0)$. The finite boundary symbol
constructed in Step~1 has orders $-1,-2,\ldots$ after these two
subtractions. Its frequency derivatives have the corresponding
lowered orders. Its residual correction maps $H^{-N-1}$ to $L^2$,
and so do its multiplication commutators of orders at most $N$.
The preceding angular estimates and the identity
$[D,e^{i\theta}]=e^{i\theta}$ therefore imply, for $0\le\beta\le N$,
\begin{equation*}
\|\operatorname{ad}_D^j\operatorname{ad}_M^\beta R_l\|
_{\mathcal L(H^{-1-\beta},L^2)}
\le C_j\mathcal P_j(Y)
\bigl(a+\|\eta\|_{H^{j+d_J}}\bigr),
\qquad Mu=e^{i\theta}u.
\end{equation*}
To obtain the small factor on the right, perform the construction and its
zero-trace correction along $(ta,t\eta)$, $0\le t\le1$, and integrate
the coefficient derivative. At $(a,\eta)=(0,0)$ the exact local map
is $|D_\theta|$, so $R_l=0$ on the retained modes. The recursion for the first variation has the same orders and requires
the same fixed number of additional derivatives as above.

Let $r^{\rm L}(\theta,n)=e^{-in\theta}R_l(e^{in\theta})$ and
$\Delta_n r(n)=r(n+1)-r(n)$. Direct calculation gives
\begin{equation*}
\operatorname{ad}_M^\beta R_l(e^{in\theta})
=(-1)^\beta e^{i(n+\beta)\theta}\Delta_n^\beta r^{\rm L}(\theta,n).
\end{equation*}
Together with the angular commutators, this yields
\begin{equation*}
\sup_{n\in\mathbb Z}\langle n\rangle^{1+\beta}
\|\Delta_n^\beta r^{\rm L}(\cdot,n)\|_{H^j}
\le C_j\mathcal P_j(Y)
\bigl(a+\|\eta\|_{H^{j+d_J}}\bigr).
\end{equation*}
Use the fixed Schwartz interpolation kernel in the proof of
\cite[Theorem~5.2]{RuzhanskyTurunen} to extend $r^{\rm L}$ to real
frequencies.  For $0\le\beta\le N$, discrete summation by parts expresses
$\partial_\xi^\beta r^{\rm L}(\cdot,\xi)$ as
$\sum_n K_\beta(\xi-n)\Delta_n^\beta r^{\rm L}(\cdot,n)$,
with a fixed Schwartz function $K_\beta$.
The triangle inequality in $H^j$ and
$\langle\xi\rangle^{1+\beta}
\le C_\beta\langle\xi-n\rangle^{1+\beta}
\langle n\rangle^{1+\beta}$ preserve the preceding bounds. For the Weyl symbol use
$\widehat r^{\rm w}(\ell,\xi)
=\widehat r^{\rm L}(\ell,\xi-\ell/2)$.
The inequality
$\langle\xi\rangle^{1+\beta}
\le C_N\langle\xi-\ell/2\rangle^{1+\beta}
\langle\ell\rangle^{1+\beta}$
shows that $N+3$ additional angular derivatives suffice: the
remaining squared Fourier weights are bounded by
$C\langle\ell\rangle^{-4}$ and are summable. Thus a fixed number of additional derivatives larger than $d_J+N+4$ suffices,
including rounding up to the next integer when $r$ is noninteger. This proves
\eqref{eq:dn-high-tame-symbol} with a derivative loss independent of $r$.

\emph{Step 3. Dependence on the boundary graphs.}
Differentiate the finite factorization and its residual collar
problem along $t\boldsymbol\zeta$.  The residual obeys
\eqref{eq:first-u-variation}, with its differentiated source and boundary
data.  In each product the high derivative falls on either the
variation $\boldsymbol\zeta$ or the base coefficient at $t\boldsymbol\zeta$.
All other factors have bounded lower norms.  The uniform scaled
elliptic estimate therefore gives
\begin{equation*}
|D r_l(t\boldsymbol\zeta)[\boldsymbol\zeta]|_{r,-1,N}
\le C_r\mathcal P_r(Y)\|\boldsymbol\zeta\|_{H^{r+d_N}}
(1+\|\boldsymbol\zeta\|_{H^{r_0+d_N}}).
\end{equation*}
Integration over $0\le t\le1$ proves
\eqref{eq:dn-high-tame-reference-difference}, with a fixed derivative loss and constants uniform in $B_q$.
The same argument applies to the retained local symbols.
Differentiating the elliptic reconstruction gives the same estimates
for smooth functions and products of the normalized velocity
coefficients, with a fixed number of additional spatial derivatives.

\emph{Step 4. The global kernel and finite dimensional corrections.}
Work first with the flux maps before multiplication by half-densities.
For an inner trace $f$, let $T_{\boldsymbol\zeta}f$ be the difference
of the global and local potentials in a common scaled inner collar,
with its constant inner trace subtracted. The two equations and their
coefficients agree there. Thus $T_{\boldsymbol\zeta}f$ solves a
homogeneous elliptic equation with zero inner Dirichlet trace.
Its conormal trace is the difference of the two flux maps.

We first estimate the operator on inputs of negative Sobolev regularity,
using only fixed lower geometric norms.  Fix $m=N_1+1$ and test
$\nabla T_{\boldsymbol\zeta}f$ against a vector field $Z\in L^2$
supported in the common collar. Denote by $v_g$ and $v_l$ the solutions
of the global and local adjoint energy problems with divergence source
$Z$, zero outer trace where applicable, and free constant inner trace.
Their gradients have norm at most $C\|Z\|_2$.  In the common collar,
their difference solves the homogeneous equation and has constant inner
trace.  After subtracting that constant, boundary regularity of fixed
order bounds its conormal trace in $H^m$ by $C_m\|Z\|_2$. Green's identity expresses the original pairing as
the pairing of $f$ with this conormal difference. The direct boundary
contributions of $Z$ cancel. Hence, on a slightly smaller fixed collar,
\begin{equation*}
\|T_{\boldsymbol\zeta}f\|_{H^1}
\le C_m\mathcal P(Y)\|f\|_{H^{-m}}.
\end{equation*}
The potential estimate follows from its zero trace and the collar
Poincar\'e inequality. The constant depends on a fixed graph norm,
say $H^{b_m}$, with $b_m$ independent of the later output index.

The same bound holds for the first graph variation, with an additional
factor $\|H\|_{H^{b_m}}$. Indeed, differentiating the two adjoint
problems gives divergence sources consisting of a first coefficient
variation times the original adjoint gradient. Their energy bounds
use only $\|H\|_{H^{b_m}}$. In the common collar the difference of
these differentiated equations has source
$-\operatorname{div}(A_H\nabla(v_g-v_l))$, where $v_g-v_l$ is the
smooth homogeneous adjoint difference just estimated. Boundary regularity of fixed order again gives the required $H^m$ conormal bound.
The estimate for $DT[H]$ on inputs of negative Sobolev regularity
uses only these fixed lower geometric and variation norms.

Let $t_{\boldsymbol\zeta}(\theta,\zeta',\varphi)$ denote the potential
kernel in this collar. Here $\zeta'$ is its normal coordinate. Applying
the preceding bound to the Fourier inputs $e^{in\varphi}$ and summing
shows
\begin{equation*}
\|\partial_\varphi^{\le N_1}t_{\boldsymbol\zeta}\|
_{H^1_{\theta,\zeta'}L^2_\varphi}
\le C\mathcal P(Y),\qquad
\|\partial_\varphi^{\le N_1}Dt_{\boldsymbol\zeta}[H]\|
_{H^1_{\theta,\zeta'}L^2_\varphi}
\le C\mathcal P(Y)\|H\|_{H^{b_m}}.
\end{equation*}
Here $\partial_\varphi^{\le N_1}$ denotes the collection of derivatives
of order at most $N_1$.  The choice $m=N_1+1$ gives
$\sum_n\langle n\rangle^{2N_1-2m}<\infty$.

Now apply tame boundary regularity in the output variables, treating
$L^2_\varphi$ as the coefficient Hilbert space. Write $\mathscr L=\operatorname{div}(A\nabla)$.  On nested fixed
collars the kernel and its variation satisfy
\begin{equation*}
\mathscr L t_{\boldsymbol\zeta}=0,\qquad
\mathscr L(Dt_{\boldsymbol\zeta}[H])
=-\operatorname{div}(A_H\nabla t_{\boldsymbol\zeta}),
\qquad t_{\boldsymbol\zeta}=Dt_{\boldsymbol\zeta}[H]=0
\quad\hbox{on the inner boundary}.
\end{equation*}
Write $k(\boldsymbol\zeta)(\theta,\varphi)$ for the corresponding
conormal boundary kernel.  The lower kernel bounds established above
provide the initial estimate for this regularity argument.
Commuting angular derivatives with $\mathscr L$ and recovering normal
derivatives gives a tame estimate in which the highest derivative
falls either on the coefficient or on the kernel. For the variation,
the source estimate has the two contributions
$\|A_H\|_{\rm high}\|\nabla t\|_{\rm low}$ and
$\|A_H\|_{\rm low}\|\nabla t\|_{\rm high}$.
Consequently, for a fixed derivative loss $d_m$ independent of $r$,
\begin{align*}
\|\partial_\varphi^{\le N_1}k(\boldsymbol\zeta)\|
_{H^r_\theta L^2_\varphi}
&\le C_r\mathcal P_r(Y)
(1+\|\boldsymbol\zeta\|_{H^{r+d_m}}),\\
\|\partial_\varphi^{\le N_1}Dk(\boldsymbol\zeta)[H]\|
_{H^r_\theta L^2_\varphi}
&\le C_r\mathcal P_r(Y)\left(
\|H\|_{H^{r+d_m}}
+(1+\|\boldsymbol\zeta\|_{H^{r+d_m}})
\|H\|_{H^{b_m}}\right).
\end{align*}
Here $d_m$ includes the fixed input regularity $b_m$, coefficient
derivatives in the conormal trace, and the finite loss in the tame
elliptic estimates. The number of input derivatives is independent of $r$.
Half-density factors and finite projections obey the same product
bounds. Constant lifts satisfy the Schur differentiation formulas.
The inverse capacitance contributes only polynomial factors in $Y$.

At the reference graphs,
Proposition~\ref{prop:charge-free-localization} bounds the kernel
in every fixed required norm by $C_r\mathcal P_r(Y)a^{1/4}$.
Integrate the variation estimate along $t\boldsymbol\zeta$,
$0\le t\le1$, with $H=\boldsymbol\zeta$.
After increasing the fixed lower index to include $b_m$, the mixed term
$\|\boldsymbol\zeta\|_{H^{r+d_m}}
\|\boldsymbol\zeta\|_{H^{b_m}}$
is absorbed by $\|\boldsymbol\zeta\|_{H^{r+d_m}}$.
The reference difference therefore has norm at most
$C_r\mathcal P_r(Y)\|\boldsymbol\zeta\|_{H^{r+d_m}}$
in $H^r_\theta H^{N_1}_\varphi$.
Pairing this estimate and the reference kernel bound with
$H^{-N_1}$ inputs proves
\eqref{eq:global-kernel-tame-operator}. The corresponding first
variation also gives the fixed lower Lipschitz estimates. At the
circular cylindrical reference geometry the high block is $|D_\theta|$, and the finite
lifts preserve orthogonality to the constant mode. This proves the
assertion about their reference coefficients.
\end{proof}

\subsection{Paralinearization and changes of variables}
We use a dyadic paralinearization with fixed derivative loss to
transfer the variational cancellations to the nonlinear state.
The uniform estimates for the shrinking weighted domain with two
boundary components follow from
Proposition~\ref{prop:charge-free-localization} and
Lemma~\ref{lem:finite-tame-dn-symbol}.

\begin{lemma}
\label{lem:finite-loss-dyadic}
Let $\mathcal F$ be a smooth map acting on a finite vector of functions on $\T$.
Assume that its first two differentials satisfy tame estimates with a fixed
loss $d$.  In particular, for $r\ge r_0$ with $r_0$ fixed,
\begin{align}
\|D^2\mathcal F(U)[V,W]\|_{H^{r-d}}
\le C_r\mathcal P_r(Y)\bigl(&
\|V\|_{H^r}\|W\|_{H^{r_0}}
+\|W\|_{H^r}\|V\|_{H^{r_0}}\notag\\[-2pt]
&+(1+\|U\|_{H^r})
\|V\|_{H^{r_0}}\|W\|_{H^{r_0}}\bigr).
\label{eq:coarse-second-tame}
\end{align}
Suppose the differential is the sum of a local pseudodifferential
operator with components of orders $2,1,0,-1$ and a smooth kernel operator.
Assume the local symbols have tame bounds with at most one factor at high
regularity and Lipschitz bounds in fixed lower norms as in
\eqref{eq:dn-high-tame-symbol}--
\eqref{eq:dn-low-lipschitz-symbol}, with fixed finite losses.
For each retained local symbol component $b_m(U)$ of order
$m\in\{2,1,0\}$ and the local remainder $b_{-1}(U)$, also assume
the following estimate for the difference from the reference symbol:
\begin{equation}\label{eq:dyadic-high-symbol-difference}
|b_m(U)-b_m(0)|_{r,m,N}
\le C_r\mathcal P_r(Y)\|U\|_{H^{r+d_{\mathrm{sym}}}},
\qquad r\ge r_0,\quad m\in\{2,1,0,-1\},
\end{equation}
on the admissible domain in the lower Sobolev norm, with a fixed finite
$d_{\mathrm{sym}}$, as supplied for the DN symbols by
\eqref{eq:dn-high-tame-reference-difference}.
Assume also that, for some fixed input index $N_1$, the
smooth kernel part $\mathsf K(U)$ satisfies
\begin{equation*}
\|[\mathsf K(U)-\mathsf K(0)]\varphi\|_{H^r}
\le C_r\mathcal P_r(Y)\|U\|_{H^{r+d_{N_1}}}
\|\varphi\|_{H^{-N_1}},
\end{equation*}
with $d_{N_1}$ fixed, and assume the reference kernel is bounded in
each required fixed smooth norm.  For the DN map, these bounds are
\eqref{eq:global-kernel-tame-operator}.
Then there is a fixed finite $s_{\mathrm{dy}}$, independent of $s$, such that, for
$\|U\|_{H^{s_{\mathrm{dy}}}}\le1$ and $s\ge s_{\mathrm{dy}}$, spatial paralinearization gives
\begin{equation}\label{eq:finite-loss-nonlinear-reduction}
\mathcal F(U)=\mathcal F(0)+D\mathcal F(0)U
+\mathcal W\bigl(\mathfrak a(U)-\mathfrak a(0)\bigr)U
+\mathcal R_{-1}(U)U+\mathcal R_{\rm hh}(U).
\end{equation}
Here $\mathfrak a$ is the finite symbol of orders at least zero.
The fixed cutoff operator $\mathcal R_{-1}(U)$ has order $-1$,
vanishes at $U=0$, and satisfies the corresponding lower seminorm bounds.
The remaining term satisfies
\begin{equation}\label{eq:finite-loss-hh-bound}
\|\mathcal R_{\rm hh}(U)\|_{H^{s+2}}
\le C_s\mathcal P_s(Y)\|U\|_{H^{s_{\mathrm{dy}}}}\|U\|_{H^s}.
\end{equation}
We retain the exact reference linear operator $D\mathcal F(0)$
and estimate its linear remainder separately.
The admissible domain may be a convex neighbourhood of zero depending
on $Y$. In that case require that it contain every frozen state and
Taylor segment used in the proof, and that the stated tame bounds
hold there with the same polynomial weights. No positive lower bound
on the radius of that neighbourhood is required.
\end{lemma}
\begin{proof}
We divide the proof into four steps.

\emph{Step 1. Dyadic expansion and the quadratic remainder.}
Set $S_0=0$ and, for $j\ge1$, choose smooth inhomogeneous Fourier
cutoffs $S_j$ with fixed finite spectral overlap and $S_j\to I$.
Put $U_j=S_{j+1}U-S_jU$, and interpret $S_j=0$ when $j<0$.
The first block $U_0$ contains all bounded frequencies.  Taylor's formula gives
\begin{equation*}
\mathcal F(U)=\mathcal F(0)+
\sum_{j\ge0}D\mathcal F(S_jU)U_j+
\sum_{j\ge0}\int_0^1(1-t)
D^2\mathcal F(S_jU+tU_j)[U_j,U_j]\,dt,
\end{equation*}
with the bounded frequency term included in $j=0$.  Replace $S_j$ in
the linear term by $S_{j-K}$, where $K$ is a fixed integer large enough
for the cutoff separating coefficient and input frequencies.
The resulting error has the form of the quadratic Taylor remainder,
with one increment replaced by a fixed sum of neighboring blocks.

Apply \eqref{eq:coarse-second-tame} at $r=s+2+d$.
For each block of the remainder, use
\begin{equation*}
\|U_j\|_{H^{r_0}}\le C2^{-(s_{\mathrm{dy}}-r_0)j}\|U\|_{H^{s_{\mathrm{dy}}}},\quad
\|U_j\|_{H^{s+2+d}}\le C2^{(2+d)j}\|U_j\|_{H^s},
\end{equation*}
\begin{equation*}
\|S_{j+1}U\|_{H^{s+2+d}}
\le C2^{(2+d)j}\|U\|_{H^s}.
\end{equation*}
To sum these terms, set
\begin{equation*}
L=\|U\|_{H^{s_{\mathrm{dy}}}},\qquad H=\|U\|_{H^s},\qquad
b_j=\|U_j\|_{H^s},\qquad \delta=s_{\mathrm{dy}}-r_0.
\end{equation*}
The two mixed increment terms, the constant part of the last term in
\eqref{eq:coarse-second-tame}, and the part containing the frozen state
are bounded, respectively, by
\begin{equation*}
C_s\mathcal P_s(Y)\left(
2^{(d+2-\delta)j}Lb_j
+2^{-2\delta j}L^2
+2^{(d+2-2\delta)j}L^2H\right).
\end{equation*}
Since $\sum_j b_j^2\le C_sH^2$, $L\le H$, and $L\le1$,
the condition $\delta>d+4$ makes their sum at most
$C_s\mathcal P_s(Y)LH$.  The same estimates apply to the fixed
number of neighboring increments introduced by replacing $S_j$ with
$S_{j-K}$.  Thus the series converges absolutely in $H^{s+2}$ and
has the bound \eqref{eq:finite-loss-hh-bound}.  Approximation extends
the conclusion to functions with the stated Sobolev regularity.

\emph{Step 2. Truncation of the coefficient frequencies.}
Insert the full symbol of the differential into the remaining linear
sum with frozen coefficients.  Truncate the coefficient frequencies at
$S_{j-K}U$ before replacing the frozen coefficient by its value at $U$.
In this order, the low symbol seminorms control the output frequencies.

The terms discarded by the cutoff in coefficient frequencies satisfy the
same quadratic bound.  Their matrix entries have coefficient frequency
at least a fixed multiple of $2^j$.  Cauchy--Schwarz in Fourier variables
bounds their $H^{s+2}$ output by a constant times the product of a
coefficient $H^{s+d_2}$ seminorm and a fixed lower norm of $U_j$,
where $d_2$ is fixed.
By \eqref{eq:dyadic-high-symbol-difference} and Bernstein's inequality,
\begin{equation*}
\|\mathfrak a(S_{j-K}U)-\mathfrak a(0)\|_{H^{s+d_2}\text{ symbol}}
\le C_s\mathcal P_s(Y)2^{d_3j}\|U\|_{H^s}
\end{equation*}
for fixed $d_3$.  Use
$\|U_j\|_{H^{r_0}}\le C2^{-(s_{\mathrm{dy}}-r_0)j}\|U\|_{H^{s_{\mathrm{dy}}}}$
and choose $s_{\mathrm{dy}}>r_0+d_3+2$.  This proves absolute summability.
The same argument applies to the local order $-1$ symbol,
using its fixed finite frequency seminorms.

\emph{Step 3. Replacing the frozen coefficients.}
After this truncation, the input and output frequencies of each block
are comparable.  The difference between a retained coefficient at
$S_{j-K}U$ and at $U$, in each fixed boundedness seminorm, is at most
\begin{equation*}
C\mathcal P(Y)2^{-(s_{\mathrm{dy}}-r_0-d_1)j}\|U\|_{H^{s_{\mathrm{dy}}}}
\end{equation*}
for a fixed coefficient loss $d_1$.  Since the frequencies are
comparable, the operator has order at most two, and a gain of two
derivatives is required, its $H^{s+2}$ output is bounded by this expression times
$C_s2^{4j}\|U_j\|_{H^s}$.  Choose
$s_{\mathrm{dy}}>r_0+d_1+6$ and sum.  This summation justifies replacing the coefficients within the
fixed cutoff on coefficient frequencies.

\emph{Step 4. Smooth kernels and the choice of the lower index.}
For a smooth kernel term, apply the operator estimate directly to
its difference from the reference kernel on a frozen block:
\begin{equation*}
\|[\mathsf K(S_{j-K}U)-\mathsf K(0)]U_j\|_{H^{s+2}}
\le C_s\mathcal P_s(Y)2^{(2+d_{N_1})j}\|U\|_{H^s}
2^{-(s_{\mathrm{dy}}+N_1)j}\|U\|_{H^{s_{\mathrm{dy}}}}.
\end{equation*}
Choose $s_{\mathrm{dy}}+N_1>2+d_{N_1}$.  The resulting series is summable and satisfies the stated
quadratic bound.  The reference sum is exactly $\mathsf K(0)U$ and remains in
$D\mathcal F(0)U$.  This argument uses only the kernel bound for $\mathsf K$.
The untruncated negative-order reference terms also remain in
$D\mathcal F(0)U$.  These estimates prove
\eqref{eq:finite-loss-nonlinear-reduction}.

Choose $s_{\mathrm{dy}}$ larger than the regularity index required to bound the coefficients and
\begin{equation*}
\max\{r_0+d+4,\ r_0+d_3+2,\ r_0+d_1+6,\ 2+d_{N_1}-N_1\}.
\end{equation*}
All thresholds are fixed before $s$ is chosen.  Bernstein's inequality
controls the higher norms of the frozen states.
Every use of the differential estimates is on the listed frozen
states or the line segments joining them. Thus the same proof applies
to the admissible domains specified in the statement.
\end{proof}

\begin{lemma} 
\label{lem:common-conjugacy} 
Use the fixed cutoff calculus of Lemma~\ref{lem:finite-bony-weyl}.
Let $\mathsf S(Y,\theta)$ be a uniformly invertible triangular matrix
of order zero, with the fixed lower coefficient bounds and one
$Y$ derivative along the trajectory.  Suppose the complete original and transformed differential
generators satisfy, before finite projections,
\begin{equation}\label{eq:common-exact-conjugacy}
M_{\mathsf S}\mathscr A+M_{\mathsf S_Y}=\mathscr N M_{\mathsf S},
\end{equation}
with the contributions from scalar modulation, the other boundary
and smooth kernels treated separately.  Here $M_{\mathsf S}$ denotes multiplication and
$\mathsf S_Y$ is evaluated along the actual trajectory.  Assume the
tame hypotheses of Lemma~\ref{lem:finite-loss-dyadic}.  If $A,N$ are the local matrix
symbols through order zero, of maximal order two, then
\begin{equation}\label{eq:common-BW-conjugacy}
\mathcal W(\mathsf S)\mathcal W(A)+\mathcal W(\mathsf S_Y)
=\mathcal W(N)\mathcal W(\mathsf S)+\mathcal R_{-1}.
\end{equation}
Here $\mathcal R_{-1}$ has order at most $-1$, with the bounds for the fixed cutoff.
Finite projections add smoothing operators of finite rank.
An explicit scalar factor in an order-two entry is retained in each
corresponding composition remainder, including $B_q$ in the inner
capillary block.  For the nonlinear state, the same conjugacy holds after adding the
reference and high--high terms in
\eqref{eq:finite-loss-nonlinear-reduction}, with the bound
\eqref{eq:finite-loss-hh-bound}.

Let $P$ be a fixed componentwise Fourier projection of finite
codimension, including $\operatorname{diag}(I,P_+)$ for the outer
variables with $P_+=I-P_0$, and put
$\mathcal S=P\mathcal W(\mathsf S)P$ on the retained modes.
If $\|V\|_2\le C\|\mathcal S V\|_2$, then, for every real $s\ge0$,
\begin{equation}\label{eq:common-elliptic-reconstruction}
\|V\|_{H^s}\le C_s\|\mathcal S V\|_{H^s}.
\end{equation}
The lower smallness threshold is independent of $s$.
The reconstruction uses the same Sobolev index for all components.  Additional
weighted half-derivative bounds are treated in the corresponding energy.
These statements apply for $\sigma\ge0$, with the source norms chosen
for the capillary or Taylor energy.
\end{lemma}
\begin{proof}
We divide the proof into two steps.

\emph{Step 1. Conjugacy of the linear and nonlinear systems.}
For a zero-order factor composed with an operator of order at most two,
retain the product, Poisson and second Moyal terms in
\eqref{eq:bw-moyal}.  The remainder has order at most $-1$.
Apply this expansion to both sides of
\eqref{eq:common-exact-conjugacy}.  Equality of the complete generators
identifies every nonnegative-order term, including $\mathsf S_Y$,
and proves \eqref{eq:common-BW-conjugacy}.  The full DN cancellation
is performed before this truncation.  Bilinearity of composition preserves any explicit scalar factor in
an order-two entry.  We retain the exact finite-rank operators arising
from finite projections and constant lifts.  Their kernels have the bounds established
by the spatial reconstruction.

Write $V_Y$ in the form supplied by
Lemma~\ref{lem:finite-loss-dyadic} and differentiate
$P\mathcal W(\mathsf S)PV$ exactly:
\begin{equation*}
(\mathcal S V)_Y=\mathcal S V_Y+
P\mathcal W(\mathsf S_Y)PV.
\end{equation*}
Substitution of \eqref{eq:common-BW-conjugacy} proves the nonlinear
statement.  The bounded zero-order factors preserve the dyadic
remainder estimate, and the exact reference operator remains separate.
The term $P\mathcal W(\mathsf S_Y)PV$ therefore retains all scalar
derivatives along the trajectory.

\emph{Step 2. Sobolev estimates for reconstruction.}
For the Sobolev estimate, the inverse symbol and the composition rule give
\begin{equation*}
\mathcal Q\mathcal S=I+\mathcal R_{-1},\qquad
\mathcal Q=P\mathcal W(\mathsf S^{-1})P,
\qquad
\|\mathcal R_{-1}V\|_{H^s}\le C_s\|V\|_{H^{s-1}}.
\end{equation*}
Errors from the finite projections satisfy the same bound.  Consequently
$\|V\|_{H^s}\le C_s(\|\mathcal S V\|_{H^s}+\|V\|_{H^{s-1}})$.
The assumed $L^2$ inverse estimate starts the induction in steps of one.
For $0<s\le1$, the last norm is bounded by $\|V\|_2$.
This proves \eqref{eq:common-elliptic-reconstruction} with the same
smallness threshold in the lower norm at every higher index.
\end{proof}

For a fixed lower index $b$, the physical bootstrap bounds used in the
remaining estimates of this section are
\begin{equation}\label{eq:physical-bootstrap-neighbourhood}
\begin{gathered}
\|\eta\|_{H^b}+q^{-1}\|f\|_{H^b}
+\|h_+\|_{H^{b+1}}+\|g\|_{H^{b+1/2}}+|R-R_*|
\le K(1+Y)^{M_b}a^2,\\
|\alpha|\le\delta_1q.
\end{gathered}
\end{equation}
Here $h_+$ is the outer normal graph in the fixed chart about
$\Gamma_+^*$, and $K,M_b$ are fixed.  The index $b$ is large enough to
supply the additional derivatives required in
\eqref{eq:first-mode-physical-bootstrap} and
Proposition~\ref{prop:raw-state-first-jet}.

\begin{lemma}
\label{lem:zero-charge-outer-derivative}
In the setting of Proposition~\ref{prop:charge-free-localization},
let the geometry depend on $Y$, with $a_Y=-a$.  Use material pullbacks
with the initial boundary coordinates held fixed, and let $K$ bound their
scaled first coefficient derivatives and the required fixed
macroscopic derivatives.  Let $u$ be the zero-flux solution with
inner trace $f$ modulo constants and zero outer trace.  Write $\mathcal N_+u$
for its outer conormal trace in the material boundary coordinate.  Then
\begin{align}
\|\mathcal N_+u\|_{H^k}
&\le C_k a\|P_{\ne0} f\|_{H^{1/2}},\label{eq:outer-dipole-size}\\
\|\partial_Y(\mathcal N_+u)\|_{H^k}
&\le C_k a\bigl(\|P_{\ne0} f_Y\|_{H^{1/2}}
+(1+K)\|P_{\ne0} f\|_{H^{1/2}}\bigr).
\label{eq:outer-dipole-derivative}
\end{align}
The constants require only the fixed finite geometric regularity
associated with $k$.  Under \eqref{eq:physical-bootstrap-neighbourhood}
and the hypotheses of Proposition~\ref{prop:raw-state-first-jet}, after subtracting the
flux-carrying Green function and the filled domain extension of the outer
trace, these estimates imply
\begin{equation}\label{eq:outer-velocity-first-jet}
\|B_+\|_{C^k}\le C_k q+\mathcal P(Y)a^2,
\qquad
\|(B_+)_Y\|_{C^k}
\le C_k(1+Y)^{-3/2}+\mathcal P(Y)a.
\end{equation}
\end{lemma}
\begin{proof}
Differentiate the equation in the fixed pullback coordinates by
\eqref{eq:first-u-variation}, with test functions having zero outer
trace and free constant inner trace.  The differentiated solution has
nonconstant inner trace $P_{\ne0}f_Y$, and its coefficient source is
bounded in $L^2$ by $CK\|\nabla u\|_2$.
The uniform collar estimate and \eqref{eq:zero-charge-gradient-bound}
give the same bound for $\nabla u_Y$ and the differentiated interior
filling.  Applying the argument for the jump density in
\eqref{eq:jump-bound-localization} to their difference yields
\begin{equation*}
\|\nabla u_Y\|_2+\|\mu_Y\|_{H^{-1/2}}
\le C\bigl(\|P_{\ne0} f_Y\|_{H^{1/2}}
+K\|P_{\ne0} f\|_{H^{1/2}}\bigr),
\qquad \int\mu_Y=0.
\end{equation*}
The density includes its pulled-back arclength factor, so the last
identity is the derivative of the zero total charge.

At the separated outer boundary, write the exact representation using
$\mathcal N_{+,x}G_D(x,c+a\zeta)-\mathcal N_{+,x}G_D(x,c)$.
Its $H^k_xH^{1/2}_\varphi$ norm is $O(a)$, and its first $Y$ derivative
is $O(a(1+K))$ in the same norm.  Here
$(a\zeta)_Y=a(-\zeta+\zeta_Y)$.  Pairing with $\mu$ and $\mu_Y$
proves both estimates.

For the application, decompose the potential as
\begin{equation*}
\Phi=H_g+QG_D(\cdot,c)+u,
\end{equation*}
where $Q$ is its total inner weighted flux and $H_g$ is the harmonic
extension in the filled domain of its outer trace.  Thus $u$ has zero
inner flux and zero outer trace.  The nonconstant inner trace of $u$ and its
first derivative have size $q\mathcal P(Y)a$.  The nonconstant
part of the weighted logarithmic Green trace has this size, and
Proposition~\ref{prop:raw-state-first-jet} controls the first mode,
high traces and their first derivatives.
The filled domain terms satisfy fixed elliptic estimates.
The weighted flux $Q$ is a fixed constant times $Rq$, so
\eqref{eq:raw-q-first-jet} gives
$Q_Y=O(q/(1+Y))+q\mathcal P(Y)a$.
The derivative of the macroscopic Green profile involves only $R_Y$
and the outer geometric velocity, both of size
$\mathcal P(Y)a^2$.  Apply
\eqref{eq:outer-dipole-size}--\eqref{eq:outer-dipole-derivative}
and $q\asymp(1+Y)^{-1/2}$ to obtain
\eqref{eq:outer-velocity-first-jet}.
\end{proof}

\noindent\textbf{Scalar reconstruction.}
In the normalized coordinates of Proposition~\ref{prop:exact-energy-and-capacity},
write $z=(\beta,F,G)=(\alpha,f,g)/q$ and
$\boldsymbol\zeta=(\eta,h_+)$, keeping $R$ separate.

\begin{lemma}
\label{lem:scalar-noncircular}
On the neighbourhood \eqref{eq:physical-bootstrap-neighbourhood}, with
$q\asymp(1+Y)^{-1/2}$, one has
$E_*-\sigma A=E_c+O(\mathcal P a)>0$ and $D_b\asymp1+Y$.
The original equations have the normalized form in
Proposition~\ref{prop:exact-energy-and-capacity}.  Thus
\eqref{eq:q-normalized-explicit}--\eqref{eq:p-normalized-explicit}
express $q,p_q,\Theta,\kappa$ and the coefficients in the equations for the natural variables
as functions of the complete instantaneous state.
Their first two state differentials, measured in any prescribed fixed
lower scaled seminorm, have finite spatial losses and polynomial bounds
in $Y$.  The same holds when the physical outer trace $g$ is treated as an
independent variable, with the explicit polynomial factor $q^{-1}$.
\end{lemma}
\begin{proof}
Since $E_*=E_c+\sigma|\Gamma_+^*|$ and
$A=|\Gamma_+^*|+O(\mathcal P a)$, the numerator is positive for large
$Y_0$.  The capacity estimate gives $D_b\asymp1+Y$.
The exact normalized equations in
Proposition~\ref{prop:exact-energy-and-capacity} therefore give the
instantaneous reconstruction, including the common gauge and inner
constant lift.

Lemma~\ref{lem:finite-tame-dn-symbol} and the Schur formulas of
Proposition~\ref{prop:schur-differentiation} give the required finite-loss
bounds for the normalized traces, capacity and their variations.
For the differentiated Dirichlet forms in
\eqref{eq:p-normalized-explicit}, first use
\eqref{eq:schur-Hadamard} and \eqref{eq:capacity-Y-exact} to express
their geometric derivatives as boundary quadratic forms.  Differentiating these explicit forms twice in the state uses at most
the first two variations of the normalized velocity and condenser flux
established above.  Area, curvature and velocity require one graph
derivative, two graph derivatives and one elliptic trace operator,
respectively.  The resulting bounds have fixed spatial losses and
at most one factor at high regularity.
The scalar denominators are the positive energy numerator,
$D_b+D_0$, $D_b$ and the nonzero kinematic determinant.
Their inverses are bounded or polynomial in $Y$.
The factors $a\delta\eta$ and $a\partial_n$ remove the inner scale,
while centre variations use a translation cut off at a fixed macroscopic
scale.  The resulting constants depend polynomially on $Y$.

To use the physical trace $g$, split $\varphi_z$ into its inner and
outer trace parts.  This gives the scalar equation
\begin{equation*}
(D_b+D_{ii})q^2+D_{ig}q+D_{gg}
=E_*-\sigma A,
\end{equation*}
where $D_{ii},D_{gg}\ge0$, and $D_{ig}$ is the cross form.
On the stated neighbourhood, $|D_{ig}|=o(qD_b)$ and
$D_{gg}=o(1)$ as $Y\to\infty$.  Implicit differentiation at the positive root gives the denominator
$2q(D_b+D_{ii})+D_{ig}\ge c q(1+Y)$.
Since $q^{-1}\asymp\sqrt{1+Y}$, two implicit differentiations and
\eqref{eq:p-normalized-explicit} give the stated bounds with $g$
as an independent variable.

\end{proof}

\begin{lemma}[Admissible states for the dyadic expansion]
\label{lem:frozen-admissibility}
Fix $\sigma\ge0$, $E_c>0$, and
$E_*=E_c+\sigma|\Gamma_+^*|$. Fix $a=\ell e^{-Y}$, let $R$ range
in a compact admissible interval, and assume
$|\beta|\le\mathcal P_{\rm ref}(Y)a$ for a prescribed polynomial
$\mathcal P_{\rm ref}$. Write $U=(\eta,F,h_+,g)$, with the original
centering and mean conditions and with $g$ the physical outer trace.
There is a fixed index $b_{\rm ad}$, independent of the later energy
index, for which the following conclusions hold. They remain valid
if $b_{\rm ad}$ is increased.

Let $C_{\rm cut}\ge1$ bound the dyadic cutoffs and their convex
interpolants on every Sobolev space. Define
\begin{equation}\label{eq:frozen-admissible-domain}
\mathcal A_Y=\{U:\ \|U\|_{H^{b_{\rm ad}}}
<2C_{\rm cut}a^{3/2}\}.
\end{equation}
For all sufficiently large $Y$, throughout $\mathcal A_Y$ one has
\begin{equation}\label{eq:frozen-admissibility-bounds}
\begin{gathered}
E_*-\sigma A-D_{gg}\ge E_c/2,\qquad
c(1+Y)\le D_b+D_{ii}\le C(1+Y),\\
c(1+Y)^{-1/2}\le q\le C(1+Y)^{-1/2},\\
2q(D_b+D_{ii})+D_{ig}\ge c\sqrt{1+Y},
\qquad |\Delta|\ge\tfrac12 .
\end{gathered}
\end{equation}
Here $q$ is the positive energy root, and $\Delta$ is defined in
Proposition~\ref{prop:clock-centre-finite-rank}. The reconstructed
scalars $q,\Theta,\kappa,p_q$ have first and second state differentials
with fixed spatial losses and polynomial bounds in $Y$. The constants
depend on the reference data and $\mathcal P_{\rm ref}$, independently
of the radius of $\mathcal A_Y$ and of any bootstrap constant.
If $\|U\|_{H^{b_{\rm ad}}}\le a^{3/2}$, all frozen states and Taylor
segments used in Lemma~\ref{lem:finite-loss-dyadic} belong to
$\mathcal A_Y$.
\end{lemma}
\begin{proof}
Choose $b_{\rm ad}$ to control the energy forms, the normalized flux
in $H^1$, and the two modulation moments and their first two state
variations. The finite losses in Section~\ref{sec:scaled-DN} make
this choice independent of the later energy index.

The geometry and boundary traces first determine the zero-flux harmonic
fields. Uniform energy
bounds and Cauchy--Schwarz give
\begin{equation*}
\begin{gathered}
0\le D_{ii}\le C(|\beta|+\|F\|_{H^1})^2,\qquad
0\le D_{gg}\le C\|g\|_{H^1}^2,\\
|D_{ig}|\le2\sqrt{D_{ii}D_{gg}}.
\end{gathered}
\end{equation*}
The area formula gives
$E_*-\sigma A=E_c+O(a+\|h_+\|_{H^1}+a\|\eta\|_{H^1})$.
The capacity expansion gives $D_b\asymp1+Y$ uniformly on
\eqref{eq:frozen-admissible-domain}: every polynomial error multiplied
by $a$ or $a^{3/2}$ tends to zero. Put
$A_q=D_b+D_{ii}$ and $N_q=E_*-\sigma A-D_{gg}$. Then
\begin{equation}\label{eq:frozen-positive-root}
q=\frac{-D_{ig}+\sqrt{D_{ig}^2+4A_qN_q}}{2A_q},
\qquad
2A_qq+D_{ig}=\sqrt{D_{ig}^2+4A_qN_q}.
\end{equation}
The first four bounds in \eqref{eq:frozen-admissibility-bounds}
follow, including positivity of the root, before any modulation
scalar has been defined.

The normalized harmonic field now defines $n$, and hence $\Delta$.
At $U=0$, the circular flux expansion gives
$\Delta(0)=1+O(\mathcal P(Y)a)$. The finite-loss trace estimates,
the root formula, and the moment formulas give
\begin{equation*}
|\Delta(U)-\Delta(0)|
\le\mathcal P(Y)\|U\|_{H^{b_{\rm ad}}}
\le2C_{\rm cut}\mathcal P(Y)a^{3/2}.
\end{equation*}
This proves the last bound for large $Y$. Define $\Theta,\kappa$
by \eqref{eq:clock-centre-normalized}, and $p_q$ by
\eqref{eq:p-normalized-explicit}. Differentiating these formulas and
\eqref{eq:frozen-positive-root} twice proves the scalar tame estimates
on $\mathcal A_Y$. The denominators have the explicit lower bounds
above; the elliptic estimates use a fixed ambient geometric ball.
Thus all constants grow at most polynomially in $Y$ and do not depend
on a bootstrap constant or on the shrinking radius of $\mathcal A_Y$.

If a state satisfies $\|U\|_{H^{b_{\rm ad}}}\le a^{3/2}$,
every $S_jU$, $S_{j-K}U$, and
$S_jU+t(S_{j+1}-S_j)U$, $0\le t\le1$, lies strictly in
$\mathcal A_Y$. So do the segments joining frozen states used in
Lemma~\ref{lem:finite-loss-dyadic}. Fourier cutoffs preserve the
centering and mean conditions. This verifies admissibility for the
whole expansion. The threshold for the static construction depends
only on the reference data and $\mathcal P_{\rm ref}$; increasing
$Y_0$ to place an actual bootstrap trajectory in this domain may
depend on its bootstrap constants.
\end{proof}

\begin{lemma}[Kernels of the reconstructed scalar and finite-rank terms]
\label{lem:reconstructed-kernel-bounds}
Fix an integer $N_1\ge0$ and use the static reconstruction of
Lemma~\ref{lem:frozen-admissibility}, with fixed $\sigma\ge0$.
After a fixed enlargement of $b_{\rm ad}$, there are indices
$b_0\le b_{\rm ad}$ and $d_{\rm ker}$, independent of the output
index $r$, with the following property.

Keep $a,R,\beta$ fixed, put
$\lambda=\log q$ and $\Lambda=(\Theta,\kappa,p_q,\lambda)$, and let
$\mathscr F(U,\Lambda)$ denote the boundary vector field before
substituting these reconstructed scalars. Its inner rows evolve
$(a\eta,aF)$ and its outer rows evolve $(h_+,g)$. The scalar chain
rule terms in its differential are
\begin{equation}\label{eq:complete-scalar-chain-rule}
\sum_{\mu}\partial_{\Lambda_\mu}\mathscr F(U,\Lambda(U))
D_U\Lambda_\mu(U)[V].
\end{equation}
After the cancellation in Proposition~\ref{prop:polar-full-cancellation},
separate the full DN shape terms from the scalar, constant-lift and
fixed-projection terms. After division by the inner row factor $a$
and the outer row factor $a^2$, the latter terms have kernels that
are finite sums of
\begin{equation}\label{eq:finite-rank-kernel-representation}
k(U;\theta,\varphi)=\phi(U;\theta)\psi(U;\varphi),\qquad
\ell_U(V)=\sum_\nu\int_\T\psi_\nu(U;\varphi)V_\nu(\varphi)\,d\varphi.
\end{equation}
The kernels may be rectangular between the two boundary trace spaces.
For smooth $U\in\mathcal A_Y$, smooth state variations $V$, and
$r\ge b_0$, their first state variations
and reference differences satisfy
\begin{align}
\|Dk(U)[V]\|_{H^r_\theta H^{N_1}_\varphi}
&\le C_r\mathcal P_r(Y)\bigl(
\|V\|_{H^{r+d_{\rm ker}}}
+(1+\|U\|_{H^{r+d_{\rm ker}}})\|V\|_{H^{b_0}}\bigr),
\label{eq:finite-rank-kernel-tame}\\
\|k(U)-k(0)\|_{H^r_\theta H^{N_1}_\varphi}
&\le C_r\mathcal P_r(Y)\|U\|_{H^{r+d_{\rm ker}}}.
\label{eq:finite-rank-kernel-difference}
\end{align}
The polynomial weights are independent of the radius of
$\mathcal A_Y$. In extracting the outer factor $a^2$, the factor
$t_Y/a^2=\Theta/q$ remains part of the reconstructed vector field.
\end{lemma}
\begin{proof}
To obtain the functions $\psi_\nu$, use
\eqref{eq:clock-centre-differential} for the two moments in the
modulation equations, and differentiate
\eqref{eq:q-normalized-explicit} and
\eqref{eq:p-normalized-explicit} for the energy scalars. Express the
geometric energy derivatives by \eqref{eq:schur-Hadamard} and
\eqref{eq:capacity-Y-exact}. Tangential integration by parts transfers
derivatives of $V$ onto the coefficients; Green's identity transfers a
DN operator acting on a trace variation onto its smooth coefficient.
For the lifts use, in the notation of
Proposition~\ref{prop:schur-differentiation},
\begin{equation*}
D\Pi[H]=-e c^{-1}\langle e,D\mathfrak G[H]\Pi\,\cdot\rangle.
\end{equation*}
Differentiating this identity gives the second differential of $\Pi$.
For the reduced flux operator $K$, the second differential is given by
\eqref{eq:schur-second-differential}, including both rank-one terms.
These formulas express all input functions
$\psi_\nu$ using finitely many products, elliptic traces, and the
scalar inverses listed in Lemma~\ref{lem:scalar-noncircular}.

The fixed $H^{N_1}$ input norms of $\psi_\nu$ and their first state
variations require only a fixed number of state derivatives. The
first two elliptic variations and the tame product estimates therefore
give the first bound in \eqref{eq:finite-rank-kernel-tame}, after a
fixed choice of $b_0$ and $d_{\rm ker}$.
The second estimate follows by integrating the first along $tU$ on
the convex domain $\mathcal A_Y$; increase $d_{\rm ker}$ to include $b_0$.
For example, differentiating \eqref{eq:finite-rank-kernel-representation}
gives $(D\phi[V])\psi+\phi D\psi[V]$. The input factor is measured
only in $H^{N_1}$, so its required state norm is fixed, whereas the
output factor is measured in $H^r$. This gives the displayed
distribution of derivatives. Pairing with an $H^{-N_1}$ input gives exactly the kernel hypothesis of
Lemma~\ref{lem:finite-loss-dyadic}. These estimates use the first two
state differentials of the scalar reconstruction and Schur operators.
\end{proof}

\begin{lemma}[Zero-flux cross-boundary kernels]
\label{lem:cross-kernel-bounds}
Fix an integer $N_1\ge0$ and use the smooth admissible geometries of
Lemma~\ref{lem:frozen-admissibility}. Let
$\boldsymbol\zeta=(\eta,h_+)$, and let $\mathsf G^0_{-+}$ send an
outer trace $\gamma$ to the physical inner normal derivative of its
harmonic extension with constant inner trace and zero total inner
flux. In flux coordinates it is the corresponding block of
\begin{equation}\label{eq:cross-exact-schur}
K=\mathfrak G-\mathfrak G e c^{-1}
\langle e,\mathfrak G\,\cdot\rangle.
\end{equation}
Here $e=(1,0)$ and $c=\langle e,\mathfrak G e\rangle$.
Normal traces are recovered by division by $(\rho/R)aj$.
After a fixed enlargement of $b_{\rm ad}$, there are fixed indices
$b_0\le b_{\rm ad}$ and $d_{\rm ker}$ such that the kernel
$k_{-+}$ of $a\mathsf G^0_{-+}$ satisfies, for $r\ge b_0$ and
smooth graph variations $V,W$,
\begin{align}
\|k_{-+}(\boldsymbol\zeta)\|_{H^r_\theta H^{N_1}_\varphi}
&\le C_r\mathcal P_r(Y)a
(1+\|\boldsymbol\zeta\|_{H^{r+d_{\rm ker}}}),
\label{eq:cross-kernel-a-factor}\\
\|Dk_{-+}(\boldsymbol\zeta)[V]\|_{H^r_\theta H^{N_1}_\varphi}
&\le C_r\mathcal P_r(Y)a\bigl(
\|V\|_{H^{r+d_{\rm ker}}}
+(1+\|\boldsymbol\zeta\|_{H^{r+d_{\rm ker}}})
\|V\|_{H^{b_0}}\bigr).
\label{eq:cross-kernel-a-variation}
\end{align}
The second variation satisfies
\begin{equation}\label{eq:cross-kernel-a-second-variation}
\begin{split}
\|D^2k_{-+}(\boldsymbol\zeta)[V,W]\|_{H^r_\theta H^{N_1}_\varphi}
\le C_r\mathcal P_r(Y)a\bigl[
&\|V\|_{H^{r+d_{\rm ker}}}\|W\|_{H^{b_0}}\\
{}+&\|W\|_{H^{r+d_{\rm ker}}}\|V\|_{H^{b_0}}\\
{}+&(1+\|\boldsymbol\zeta\|_{H^{r+d_{\rm ker}}})
\|V\|_{H^{b_0}}\|W\|_{H^{b_0}}\bigr].
\end{split}
\end{equation}

The indices depend on $N_1$ and the fixed elliptic estimates, and
are independent of $r$ and $a$. For the reconstructed clock and flux
of Lemma~\ref{lem:frozen-admissibility}, at the reference graphs,
\begin{equation}\label{eq:cross-reference-scaled-bound}
\|t_Y\mathsf G^0_{-+}(0)\gamma\|_{H^r}
\le C_r\mathcal P_r(Y)a^2\|\gamma\|_{H^{-N_1}}.
\end{equation}
These estimates hold for each fixed $\sigma\ge0$.
\end{lemma}
\begin{proof}
Let $H_\gamma$ be the harmonic extension of
$\gamma$ to the filled domain and write
\begin{equation}\label{eq:cross-filled-representation}
\begin{gathered}
u_\gamma=H_\gamma+v_\gamma,
\qquad
v_\gamma|_{\Gamma_+}=0,\\
v_\gamma|_{\Gamma_-}=-H_\gamma|_{\Gamma_-}\quad\bmod\R,
\end{gathered}
\end{equation}
where $v_\gamma$ is the zero-flux extension. Since $H_\gamma$ extends
through the hole, its total inner flux is zero by the divergence
theorem. Thus \eqref{eq:cross-filled-representation} has precisely
the flux and trace conditions defining $\mathsf G^0_{-+}$.
If $c_{\rm hole}=(R,0)$ denotes the hole centre and
$\zeta(\theta)=(1+\eta(\theta))e_r(\theta)$, then the exact identity
\begin{equation}\label{eq:cross-regular-trace-factor}
H_\gamma(c_{\rm hole}+a\zeta)-H_\gamma(c_{\rm hole})
=a\int_0^1\nabla H_\gamma(c_{\rm hole}+ta\zeta)\cdot\zeta\,dt
\end{equation}
supplies a factor $a$ in the nonconstant inner datum. Interior
regularity in the filled domain bounds every fixed derivative on
this neighbourhood by a fixed outer trace norm. The scaled inner DN
operator $a\partial_{n_-}$ is uniformly bounded on the corresponding
trace spaces, with the fixed tame losses already proved. Hence both
terms in $a\partial_{n_-}u_\gamma$ have a factor $a$.

We prove the mixed kernel bounds, including the two graph variations.
Set $m=N_1+1$ and denote the integral on the right of
\eqref{eq:cross-regular-trace-factor} by $J_\gamma$, so that the
nonconstant datum is $aJ_\gamma$. For $\gamma\in H^{-m}$,
Green's identity with a source supported in the fixed interior
neighbourhood reduces the interior estimate for $H_\gamma$ to the
$H^m$ conormal norm of an adjoint solution at the outer boundary.
The adjoint solution is homogeneous in a separated outer collar.
Its energy bound and fixed-order boundary regularity give this norm
using only a fixed outer graph index. Interior elliptic regularity
then bounds every required physical derivative near the hole.
The same argument differentiated twice uses
\eqref{eq:first-u-variation}--\eqref{eq:second-u-variation};
the graph and direction norms needed for the input estimate remain
at a fixed index depending on $m$.

Write $v_\gamma=a w_\gamma$, choosing the free constant consistently.
The function $w_\gamma$ has zero outer trace, inner trace
$-J_\gamma$ modulo constants, and zero total inner flux.
In the fixed pullback coordinates, write $A$ for the normalized
elliptic coefficient matrix and put
$\mathscr L=\operatorname{div}(A\nabla)$,
$A_V=DA[V]$, and $A_{VW}=D^2A[V,W]$.
For graph directions $V,W$, its variations satisfy
\begin{equation}\label{eq:cross-normalized-variation-problems}
\begin{aligned}
\mathscr L w_V&=-\operatorname{div}(A_V\nabla w_\gamma),\\
\mathscr L w_{VW}
&=-\operatorname{div}\bigl(A_{VW}\nabla w_\gamma
+A_V\nabla w_W+A_W\nabla w_V\bigr).
\end{aligned}
\end{equation}
Their outer traces are zero and their inner traces are
$-DJ_\gamma[V]$ and $-D^2J_\gamma[V,W]$ modulo constants.
The equations are tested in the zero-flux space, so the differentiated
conormal flux includes the displayed coefficient sources.
The factor $a$ has been removed before differentiating: it is a fixed
parameter for these variations. Inner displacements are $a\delta\eta$,
and the derivatives of the scaled pullback coefficients are uniformly
bounded by Lemma~\ref{lem:flattening-geometry}. Thus the coefficients
in both variation equations have uniform bounds at the scaled inner boundary.

The preceding negative-input bounds for $J_\gamma$ and its variations
give the lower collar $H^1$ bounds for $w_\gamma,w_V,w_{VW}$ by the
energy estimate. The direct term from $H_\gamma$ satisfies the same
bounds. Apply these estimates to $\gamma=e^{in\varphi}$ and sum with
weight $\langle n\rangle^{2N_1}$, exactly as in
\eqref{eq:negative-input-kernel-summation}. The choice $m=N_1+1$
makes the sum converge. Thus the potential kernels and their first
two variations have the required $H^1$ output-collar norm with
$N_1$ derivatives in the input variable.

Finally apply tame boundary regularity to
\eqref{eq:cross-normalized-variation-problems}, treating
$L^2_\varphi$ as a coefficient Hilbert space. The conormal variation
contains $A_V\nabla w_\gamma+A\nabla w_V$ at first order, and the
three source terms in the second equation together with
$A\nabla w_{VW}$ at second order. Each product has at most one
factor at high regularity. Restoring $v_\gamma=a w_\gamma$ and the
direct term $a\partial_{n_-}H_\gamma$ restores a factor $a$ in every
term. This proves the three kernel estimates, with the fixed input
regularity and tame elliptic loss included in $b_0,d_{\rm ker}$.
At the reference graphs these bounds hold at every fixed output
index. Pairing the kernel with an $H^{-N_1}$ input and multiplying
by $t_Y/a=\Theta a/q$ gives
\eqref{eq:cross-reference-scaled-bound}, since $\Theta/q$ has
polynomial bounds by Lemma~\ref{lem:frozen-admissibility}.
\end{proof}

\subsection{Energy estimates at positive surface tension}

For the outer graph and physical potential trace, set
\begin{equation*}
\mathcal X_s^+=\|h_+\|_{H^{s+1}}+\|g\|_{H^{s+1/2}}.
\end{equation*}
This norm measures the outer boundary's contribution to the inner
system in physical outer trace variables.

\noindent\textbf{Choice of the fixed regularity indices.}
We make the order of choices precise. Fix $N=24$ frequency derivatives
in Lemma~\ref{lem:finite-tame-dn-symbol} and fix the kernel input index
$N_1=4$, using $m=N_1+1=5$ in the adjoint and Fourier summation
arguments. The extra frequency derivatives cover the second composition
terms while leaving the twenty derivatives used in
Lemma~\ref{lem:finite-bony-weyl}. Enlarge the fixed losses, when
necessary, to cover the first two state differentials, products,
Schur lifts, and the kernels of
Lemmas~\ref{lem:reconstructed-kernel-bounds} and~\ref{lem:cross-kernel-bounds}
at $N_1=4$, together with the outer kernels in
\eqref{eq:outer-kernels-C2-negative-input}. Take $b_0,d_{\rm ker}$ to
dominate their fixed lower indices and kernel losses, and enlarge
$b_{\rm ad}$ to cover these lower indices. This only shrinks
$\mathcal A_Y$, so Lemma~\ref{lem:frozen-admissibility} still applies.
Denote the common lower index for the differential estimates by $r_0$
and their coarse loss by $d$. Let $d_1,d_3$ be the losses
in Steps~3 and~2, respectively, of
Lemma~\ref{lem:finite-loss-dyadic} for this collection of operators.
Choose an integer $k_0$ controlling all these fixed lower seminorms,
including $b_{\rm ad}$, $b_0$, $b_{\rm ell}$, the lower indices for the
$m=5$ adjoint estimates and both kernel variations, the eight angular
derivatives in Lemma~\ref{lem:finite-bony-weyl}, and the coefficient
norms in the first mode, outer velocity and pressure estimates.
The elliptic loss $d_{\rm ell}$ and $L_{\rm T}$ in
\eqref{eq:taylor-fixed-spatial-loss} have already been fixed. Set
$L_{\rm coeff}=L_{\rm T}+12$. Then choose
integers satisfying
\begin{equation}\label{eq:positive-regularity-hierarchy}
\begin{split}
s_{\rm dy}&>\max\{k_0,\ r_0+d+4,\ r_0+d_3+2,
\ r_0+d_1+6,\ 2+d_{\rm ker}-N_1\},\\
s_*&\ge\max\{s_{\rm dy}+1,\ k_0+L_{\rm coeff}\}.
\end{split}
\end{equation}
All losses in this prescription are fixed before the energy index $s$,
the terminal value $Y_N$, and the bootstrap constants are chosen.
For clarity, the sufficient spatial indices for
Lemma~\ref{lem:normalized-taylor-variation} at output index $k$ are
\begin{equation}\label{eq:coefficient-regularity-budget}
\begin{array}{c|cccc}
&\eta&F&g/q&h_+\\ \hline
\hbox{spatial input}
&k+L_{\rm T}+6&k+L_{\rm T}+5&k+L_{\rm T}+8&k+L_{\rm T}+10\\
\hbox{first }Y\hbox{ derivative}
&k+L_{\rm T}+2&k+L_{\rm T}+1&k+L_{\rm T}+4&k+L_{\rm T}+4 .
\end{array}
\end{equation}
The first line lists Sobolev indices. The second line is supplied by
Proposition~\ref{prop:raw-state-first-jet} at index $k+L_{\rm T}$:
its trace inputs are exactly the first three entries of the first line.
It also supplies $R_Y,\beta_Y,\mathfrak c_Y$ and $p_q$.
The outer chart requirements are included explicitly in the last
column. The margin of twelve in $L_{\rm coeff}$ includes this column
and two further derivatives for embeddings. Hence the bootstrap at
$b=s_*$ supplies every coefficient estimate at $k=k_0$.
Factors $q^{\pm1}$ change only the polynomial weights. The constants $d,d_1,d_3,d_{\rm ker}$ include
the outer kernel estimates \eqref{eq:outer-kernels-C2-negative-input}
as well as the inner estimates.

\begin{theorem}
\label{thm:actual-principal-reduction}
Fix $\sigma>0$.  There exist fixed integers $s_{\mathrm{dy}}<s_*$, independent of the
Sobolev index $s$ and terminal value $Y_N$, chosen by
\eqref{eq:positive-regularity-hierarchy}, with the following property.
Suppose a smooth Euler solution on $[Y_0,Y_N]$ satisfies
\eqref{eq:physical-bootstrap-neighbourhood} at $b=s_*$, with
$q\asymp(1+Y)^{-1/2}$, $|\alpha|\le\delta_1q$ and the prescribed
terminal first mode equal to zero.  Choose $Y_0$ sufficiently large using the bounds at this fixed lower
regularity.  Then there are inner variables
$(\mathfrak h,\mathfrak\pi)$ and an energy $E_s$, constructed in Step~1
of the proof using the fixed coefficient cutoff.
The variables vanish at canonical terminal data and reconstruct the physical
high traces through
\begin{equation}\label{eq:new-trace-reconstruction}
\|\eta\|_{H^s}+q^{-1}\|f\|_{H^s}
\le C_s a^{-1}E_s,\qquad s\ge s_*.
\end{equation}
They satisfy the nonlinear equations
\begin{align}
(\partial_Y+T)\mathfrak h
&=\mathfrak h+\mathsf L_0\mathfrak\pi+F_*,\notag\\
(\partial_Y+T)\mathfrak\pi
&=-\mathsf M_0\mathfrak h-(1+p_q)\mathfrak\pi+G_*,
\label{eq:actual-principal-equations}
\end{align}
where, with $P=P_{\ge2}$,
\begin{equation}\label{eq:retained-principal-block}
\mathsf L_0=\Theta P\mathcal W(x+\ell_0)P,\qquad
\mathsf M_0=\Theta P\mathcal W(\tau_0(Y)+B_q(x^2-1))P,
\qquad \tau_0=P_0\mathfrak t_-.
\end{equation}
The common transport $T=iP\mathcal W(\mathcal V\xi)P$ is
skew-symmetric on smooth functions.
With the natural source norm of
Theorem~\ref{thm:natural-variable-principal-energy}, evaluated at $B=B_q$,
\begin{equation}\label{eq:actual-full-source-bound}
\mathcal N_s(F_*,G_*)
\le C_s\varepsilon_*(Y)E_s+
C_s\mathcal P_s(Y)(a^3+a^2\mathcal X_s^+),
\qquad
\varepsilon_*\in L^1([Y_0,\infty)).
\end{equation}
The function $\varepsilon_*$ depends only on the lower norms and may
be chosen as a constant multiple of $(1+Y)^{-3/2}+a^{1/8}$.
The energy satisfies the backward differential inequality
\begin{equation}\label{eq:actual-inner-energy}
-\partial_Y E_s\le(1+C_s\varepsilon_*)E_s+
C_s\mathcal P_s(Y)(a^3+a^2\mathcal X_s^+).
\end{equation}
\end{theorem}
\begin{proof}
We divide the proof into six steps.

\emph{Step 1. The change of unknown and the energy.}
Proposition~\ref{prop:first-mode-backward-amplitude} gives the decay estimates
for the first potential harmonic and the centre.  Proposition~\ref{prop:raw-state-first-jet}
and the Taylor coefficient estimates give the fixed lower coefficient
bounds and their first $Y$ derivatives used below.

Let $w=(\rho/R)j$, and define the triangular symbol
\begin{equation*}
\mathsf S=\begin{pmatrix}
\sqrt w\,d/j&0\\
-\sqrt w\,b&-\sqrt w
\end{pmatrix}.
\end{equation*}
With the fixed cutoff quantization of Lemma~\ref{lem:finite-bony-weyl}, put
\begin{equation*}
\binom{\mathfrak h}{\mathfrak\pi}
=P\mathcal W(\mathsf S)P\binom{a\eta}{aF}.
\end{equation*}
The projection $P$ uses the original polar angle.  At the circular cylindrical
reference state with zero nonconstant traces,
$\mathsf S_0=\left(\begin{smallmatrix}1&0\\-1&-1\end{smallmatrix}\right)$.
The fixed lower coefficient bounds make
$\|P\mathcal W(\mathsf S)P-\mathsf S_0\|_{PL^2\to PL^2}$ small.
Use the retained operators in \eqref{eq:retained-principal-block} and
put $P_s=P\mathcal W(x^s)P$, with the fixed low-frequency completion
of Theorem~\ref{thm:natural-variable-principal-energy}.  Define
\begin{equation*}
E_s^2=\sum_{A\in\{I,P_s\}}\left(
\langle\mathsf M_0 A\mathfrak h,A\mathfrak h\rangle
+\langle\mathsf L_0 A\mathfrak\pi,A\mathfrak\pi\rangle\right),
\end{equation*}
where $I$ is the identity on $PL^2$ and the brackets denote the
$L^2(d\theta)$ pairing.  The fixed lower coefficient bounds and
the added lower form give positivity and, for each fixed $s$,
\begin{equation*}
E_s\asymp\|\mathfrak h\|_{H^s}
+\sqrt{B_q}\|\mathfrak h\|_{H^{s+1}}
+\|\mathfrak\pi\|_{H^{s+1/2}},
\end{equation*}
uniformly in $Y$ and the finite terminal value.  The quadratic form
and elliptic reconstruction are those of
Theorem~\ref{thm:natural-variable-principal-energy}.
Lemma~\ref{lem:common-conjugacy} then gives
\eqref{eq:new-trace-reconstruction} at every higher index with the
same threshold.
Lemma~\ref{lem:zeroth-mixed-cost} controls the factor
$O(\mathcal P a)B_q^{-1/4}$ arising from the off-diagonal terms in the mixed norms.

\emph{Step 2. The nonlinear map in the original variables.}
Use Proposition~\ref{prop:exact-polar-raw-system} and the outer
normal graph equations with
$U=(\eta,F,h_+,g)$, keeping $a,R,\beta$ as parameters. Throughout this
proof the reference $U=0$ retains these parameters; in particular
$\beta$ need not vanish. Lemma~\ref{lem:scalar-noncircular} reconstructs
$q,\Theta,\kappa,p_q$ as functions of $U$, using the physical outer
trace $g$. We apply Lemma~\ref{lem:finite-loss-dyadic} to this complete
vector field on $\mathcal A_Y$ in \eqref{eq:frozen-admissible-domain}.
The first mode estimate supplies $\mathcal P_{\rm ref}$, and the
bootstrap at $s_*$ places the actual state in its smaller
$a^{3/2}$ ball after increasing $Y_0$.
Lemma~\ref{lem:frozen-admissibility} gives the energy root and
modulation bounds throughout the frozen states and Taylor segments.
The domain version of Lemma~\ref{lem:finite-loss-dyadic} therefore
applies. Its local differential is computed with the reconstructed
scalars frozen, and its scalar chain rule terms are included as kernel
operators. The DN reconstruction includes the exact constant lift.
Lemma~\ref{lem:finite-tame-dn-symbol}, the Schur differentials and
product estimates give its first two differentials with fixed spatial
loss and polynomial constants.  Inner variations are $a\delta\eta$
and $q\delta F$, and the normalized normal derivative is $a\partial_n$,
so the elliptic constants depend polynomially on $Y$.  The curvature
$B_q\mathscr C$, with $\mathscr C=a\mathcal H_-$ and $B_q\le1$, is a
uniformly tame graph map of order two by Proposition~\ref{prop:full-C2-curvature}.  The scalar inverses contributed by the constant lift have polynomial
bounds.  The outer rows use the same elliptic variation equations on
a fixed collar and carry $t_Y$.
Together with the Bernoulli products, these bounds verify
\eqref{eq:coarse-second-tame} for each original equation.  The remainder carries a factor
$a$ in the scaled inner rows and $t_Y$ in the outer rows.

Perform the full DN shape cancellation in
Proposition~\ref{prop:polar-full-cancellation} before separating the
local and kernel terms; this includes the DN map applied to $bH$
when $b$ contains the constant lift. Lemma~\ref{lem:reconstructed-kernel-bounds}
controls the scalar chain rule, lift and fixed-projection kernels,
with the outer factor $a^2$ extracted before differentiating.
Lemma~\ref{lem:cross-kernel-bounds} supplies the zero-flux cross kernel
and its first two graph variations, including the factor $a$ in
$a\mathsf G^0_{-+}$ and the factor $a^2$ in the reference scaled row.

The normal variational equations use
$\gamma=\delta g-B_+H_+$, up to the fixed graph factors.
The required lower norms of $B_+$ are controlled by
Lemma~\ref{lem:zero-charge-outer-derivative}. Multiplying the input
kernel by these factors preserves the estimates. The reference factors
are smooth at every fixed index. Thus
\eqref{eq:cross-reference-scaled-bound} controls both outer trace and
outer shape inputs. The variation bounds give the kernel hypotheses
for the complete nonlinear map, including these inputs.

\emph{Step 3. Paralinearization and the high frequency remainder.}
The DN differential has the finite tame symbol of
Lemma~\ref{lem:finite-tame-dn-symbol}.  The local transport, kinetic
and curvature differentials have tame coefficients depending smoothly
on finitely many geometric derivatives and the reconstructed velocity.  In the full shape formula of
Proposition~\ref{prop:polar-full-cancellation}, separate the local and
smooth kernel parts of the inner DN contribution acting on $bH$.
Use the decomposition of Lemma~\ref{lem:finite-tame-dn-symbol} in its
normalized half-density variables, with the normalization and
half-density factors restored in the physical variables.
This decomposition accounts for the action of high shape frequencies
on a low-frequency potential before the dyadic estimate is applied.
Apply Lemma~\ref{lem:finite-loss-dyadic}, choosing $s_{\mathrm{dy}}$
large enough for its fixed derivative losses and symbol seminorms.
For the inner rows, this gives
\begin{equation*}
a\|\mathcal R_{\rm hh}\|_{H^{s+2}}
\le C_s\mathcal P_s(Y)\|U\|_{H^{s_{\mathrm{dy}}}}
(E_s+a\mathcal X_s^+).
\end{equation*}
Here $U=(\eta,F,h_+,g)$ consists of the unscaled inner high mode
variables and the physical outer variables, with $F=f/q$.
In particular, this application includes the scalar chain rule kernels
of \eqref{eq:complete-scalar-chain-rule}.
The algebraic change of unknown reconstructs the inner variables.
By the lower bootstrap bound
$\|U\|_{H^{s_{\mathrm{dy}}}}\le\mathcal P(Y)a^2$, the remainder is bounded
in the required source norms by an integrable multiple of $E_s$ plus
$C_s\mathcal P_s(Y)a^2\mathcal X_s^+$.  The exact reference linear operator remains
as in \eqref{eq:finite-loss-nonlinear-reduction}.  The errors from the coefficient cutoff
are smooth, vanish at the limiting constant-coefficient state,
and have seminorms at the reference circle bounded by $C_s\mathcal P_s(Y)a^{1/4}$.
They are therefore integrable linear errors.

\emph{Step 4. Conjugacy and the finite dimensional corrections.}
Proposition~\ref{prop:polar-full-cancellation}, followed by the normal
and half-density factors $d/j$ and $\sqrt w$, verifies
\eqref{eq:common-exact-conjugacy} for this matrix.  The normal
good unknown identity cancels the full DN contribution in the
momentum row before truncation, including its nonlocal part.
The conversion from radial to normal variations incorporates the tangential curvature
contribution into \eqref{eq:polar-complete-normal-block}.
Equation~\eqref{eq:Jacobi-half-density-symbol} contains the complete
first-order Jacobi term in $\operatorname{Op}^{\rm w}(x^2)$.
The nonlinear conclusion of
Lemma~\ref{lem:common-conjugacy} therefore gives the common
transport and the retained block \eqref{eq:retained-principal-block},
with $\mathsf S_Y$ evaluated along the trajectory.

The local conjugacy freezes the reconstructed scalars only when
identifying the local symbol. Their derivatives along the trajectory
remain in $(\mathsf S)_Y$ and the scaling terms. The complete nonlinear
reduction in Steps~2--3 has already included their state differentials.
We now estimate the exact reference kernels left by that reduction.

For clarity, the scalar columns in the two scaled inner rows are
$a$ times
\begin{equation}\label{eq:inner-scalar-columns}
\begin{array}{c|c|c}
&a^{-1}\partial_{\Lambda_\mu}(a\eta)_Y
&a^{-1}\partial_{\Lambda_\mu}(aF)_Y\\ \hline
\Theta&-P(\varsigma n)
&P\bigl((n^2-\upsilon^2)/2-un\upsilon-B_q\mathscr C\bigr)\\
\kappa&-PC&-P(\varsigma\upsilon\sin\theta)\\
p_q&0&-F\\
\lambda&-\Theta P(\varsigma n_\lambda)
&\Theta P\bigl((n-u\upsilon)n_\lambda+2B_q\mathscr C\bigr).
\end{array}
\end{equation}
Here the partial derivatives freeze the other scalar parameters and
$\beta,F,g$; linearity of the harmonic reconstruction gives
$n_\lambda=-a\mathsf G^0_{-+}g/q$. In particular $n_\lambda=0$ at
$U=0$. At this reference state the condenser and first mode expansions
give
\begin{gather*}
n=n_0+n_1\cos\theta+O_{H^r}(\mathcal P_r a^2),\qquad
n_0=1+O(\mathcal P a^2),\\
n_1=O(\mathcal P a),\qquad
\upsilon=-\beta\sin\theta,\qquad |\beta|+|\kappa|\le\mathcal P a.
\end{gather*}
and $P\mathscr C=O_{H^r}(a^2)$. The latter follows directly by
expanding the circular curvature $-a^{-1}-\cos\theta/(R+a\cos\theta)$
after multiplication by $a$. Thus the reference column norms in
\eqref{eq:inner-scalar-columns}, at every fixed $H^r$ index, are
bounded by $C_r\mathcal P_r a^2$ for $\Theta$, by
$C_r\mathcal P_r a$ for $\kappa$, vanish for $p_q$, and are bounded
by $C_r\mathcal P_r a^2$ for $\lambda$; here $B_q\le1$ suffices.
In the outer rows every scalar column has size
$C_r\mathcal P_r a^2$, since the physical equations carry
$t_Y=\Theta a^2/q$ and the reference velocity is smooth at the
outer scale.

The moment formulas and Lemma~\ref{lem:scalar-noncircular} give, for
the fixed index $b_0\le k_0$ chosen above,
\begin{equation}\label{eq:scalar-functional-scaled-bound}
|D\Lambda(0)[V]|
\le\mathcal P(Y)\left[
a^{-1}(\|H\|_{H^{b_0}}+\|Z\|_{H^{b_0}})
+\|H_+\|_{H^{b_0}}+\|\delta g\|_{H^{b_0}}\right],
\end{equation}
where $H=a\delta\eta$ and $Z=a\delta F$.
Multiplying this bound by the preceding column bounds proves an
$O(\mathcal P a)$ inner-to-inner reference kernel, an
$O(\mathcal P a^2)$ outer-to-inner reference kernel, and an
$O(\mathcal P a)$ inner-to-outer reference kernel in the scaled
equations. These are bounds for the full smooth reference kernels at
every fixed output index, with a fixed input index.

For the Schur terms the same powers can be checked before estimating
the scalar inverse. Write $g_e=\mathfrak G e$. At the circular
reference graphs the layer expansion gives
\begin{equation}\label{eq:schur-reference-profile}
\begin{gathered}
c^{-1}=O(\mathcal P(Y)),\qquad
(g_e)_-=c/(2\pi)+O_{H^r}(\mathcal P_r a),\\
P(g_e)_-=O_{H^r}(\mathcal P_r a^2).
\end{gathered}
\end{equation}
The formula is in the fixed flux pairing; the constant term is fixed
by $c=\int_\T(g_e)_-\,d\theta$.
Division by the inner metric factor and multiplication by the normal
and half-density factors preserve the $O(a)$ bound after removal of
the constant mode, and the high reference profiles have polynomial
bounds with this factor. The nonconstant part of each input moment
has the same bounds by adjointness. At the cylindrical limit these
moments annihilate the retained high modes. A first-order perturbation
of a profile, a moment, or a chart factor therefore supplies at least
one factor $a$, and $c^{-1}$ supplies only a polynomial in $Y$.
The normalized inner reference Schur and finite-projection kernels
are consequently $O_{H^r_\theta H^{N_1}_\varphi}(\mathcal P_r a)$.
For an outer input the possibly flux-carrying terms must be combined
as in \eqref{eq:cross-exact-schur}; their scaled bound is
\eqref{eq:cross-reference-scaled-bound}.
For an inner input and outer output the physical potential factor is
$q/a$, so that $t_Y(q/a)=\Theta a$; the separated trace estimates of
Lemma~\ref{lem:zero-separated-asymmetric} give the stated
$O(\mathcal P_r a)$ bound. The adjoint collar argument in
Lemma~\ref{lem:cross-kernel-bounds} allows the same fixed negative input index for these reference kernels.

In the original chart equations, the projections $P$ on the inner
trace and $P_+=I-P_0$ on the outer trace annihilate the common
Bernoulli scalar. We apply these fixed projections before changing
the unknown. In particular
$P\mathcal W(\mathsf S)P$ applied to the inner potential component of
a constant gauge is zero. On the outer boundary the right-hand
projection is $\operatorname{diag}(I,P_+)$ and gives the same
cancellation. If a projection is commuted through a chart factor in
the normal variational identity, retain the resulting finite-rank
commutator. Its input and output profiles are that factor applied to
a fixed low Fourier mode, so it satisfies
\eqref{eq:finite-rank-kernel-tame}--\eqref{eq:finite-rank-kernel-difference}.
The inner chart factors are constant at the cylindrical limit; their
reference nonconstant parts are $O(\mathcal P_r a)$. The outer
commutators carry $t_Y$.

To record the nonlinear conclusion, let $\delta_b$ bound the fixed
lower norms of $U$, and let $\mathcal R^{\rm sg}_-$ and
$\mathcal R^{\rm sg}_+$ denote the scalar, lift, and projection
contributions after the local conjugacy, including their shares of
the dyadic remainder. The reference kernel estimates above and
\eqref{eq:finite-loss-hh-bound} give
\begin{align}
\mathcal N_s(\mathcal R^{\rm sg}_-)
&\le C_s\mathcal P_s\bigl[
aE_s+a^2\mathcal X_s^++\delta_b(E_s+a\mathcal X_s^+)\bigr],
\label{eq:scalar-gauge-schur-inner-bound}\\
\|\mathcal R^{\rm sg}_+\|_{H^{s+2}}
&\le C_s\mathcal P_s\bigl[
aE_s+a^2\mathcal X_s^++a^2\delta_b(a^{-1}E_s+\mathcal X_s^+)\bigr].
\label{eq:scalar-gauge-schur-outer-bound}
\end{align}
Here the outer norm denotes the sum of the two component norms;
using $H^{s+2}$ is permissible for these smooth reference kernels and
dyadic remainders. The homogeneous terms have zero value at $U=0$;
the value of the full vector field at $U=0$ is estimated in Step~6.
To see explicitly why the nonlinear kernel differences have no
derivative loss, their frozen dyadic blocks, after division by the
row factor, are bounded by
\begin{equation*}
C_s\mathcal P_s\,
2^{(2+d_{\rm ker}-s_{\rm dy}-N_1)j}
\|U\|_{H^{s_{\rm dy}}}\|U\|_{H^s}.
\end{equation*}
This is Step~4 of Lemma~\ref{lem:finite-loss-dyadic} applied to
\eqref{eq:finite-rank-kernel-difference}; the exponent is strictly
negative by \eqref{eq:positive-regularity-hierarchy}. The quadratic
Taylor terms use its Step~1. Multiplication by $a$ or $a^2$ and
\eqref{eq:new-trace-reconstruction} gives precisely the last terms
in \eqref{eq:scalar-gauge-schur-inner-bound}--
\eqref{eq:scalar-gauge-schur-outer-bound}. The bounded fixed-cutoff
changes of unknown preserve these bounds. In particular, taking
$\delta_b\le\mathcal P a^2$, or the weaker $\delta_b\le a^{3/2}$
used for bootstrap closure, gives an integrable diagonal error and
the required $\mathcal P_s a^2\mathcal X_s^+$ and
$\mathcal P_s aE_s$ cross errors.

\emph{Step 5. Estimates for the coefficients.}
Proposition~\ref{prop:raw-state-first-jet} supplies the physical trace,
geometric and scalar derivatives at the indices in
\eqref{eq:coefficient-regularity-budget}. Its estimate \eqref{eq:raw-q-first-jet} uses
only the physical bootstrap and the first mode bound.  Differentiating
the normalized harmonic reconstruction once therefore gives
$n_Y,\upsilon_Y=O(\mathcal P a)$.  The flux identity and the centred
modulation system retain the stronger mean bounds and give
$\Theta_Y=O(\mathcal P a^2)$.
These first state derivatives satisfy the hypotheses of
Lemma~\ref{lem:normalized-taylor-variation}, which gives
$(\mathfrak t_-)_Y=O((1+Y)^{-2}+\mathcal P a)$.
The lower bootstrap supplies all these spatial and time-direction
norms by \eqref{eq:positive-regularity-hierarchy}.
All norms in the next four displays are the fixed symbol or coefficient
norms required by the preceding lemmas, and $c_0$ is defined in
\eqref{eq:Jacobi-half-density-symbol}.  The resulting estimates are
\begin{equation*}
|\beta|+|\kappa|+\|n-1\|+\|\upsilon\|+\|b-1\|+\|b_Y\|
+\|\mathcal V\|\le\mathcal P(Y)a,
\end{equation*}
\begin{equation*}
\|j-1\|+\|j_Y\|+|\Theta-1|+|\Theta_Y|
\le\mathcal P(Y)a^2,
\end{equation*}
\begin{equation*}
\tau_0\in[3/4,5/4],\quad
|(\tau_0)_Y|\le C(1+Y)^{-2}+\mathcal P(Y)a,\quad
\|\mathfrak t_--\tau_0\|+\|c_0\|\le\mathcal P(Y)a,
\end{equation*}
\begin{equation*}
|\ell_0|+|\partial_Y\ell_0|\le\mathcal P(Y)a,
\qquad |p_q|\le1/4,\quad p_q\le\mathcal P(Y)a,
\qquad (B_q)_Y/B_q=-1-2p_q.
\end{equation*}
The identity relating the two means is
\begin{equation*}
1+p_q-\tau_0=O(\mathcal P(Y)a).
\end{equation*}
It follows from \eqref{eq:raw-q-first-jet} and
\eqref{eq:actual-inner-taylor} with the same capacity correction
$\mathfrak c$.  Using this common correction gives the
$O(\mathcal P(Y)a)$ remainder required for integrability after the
$B_q^{-1/4}$ loss.
At the circular cylindrical reference state, the transformed matrix of the
original system is exactly
$\left(\begin{smallmatrix}1&|D_\theta|\\
-(1+p_q)-B_q(|D_\theta|^2-1)&-(1+p_q)\end{smallmatrix}\right)$.
This identity controls the scalar coefficient that is not small
in the comparison with \eqref{eq:retained-principal-block}.
The derivatives of the original variables are given by
Proposition~\ref{prop:raw-state-first-jet}.

After subtracting the retained block, the remaining linearized
coefficients have the following orders and natural source bounds,
listed by matrix entry:
\begin{equation*}
\begin{array}{c|c|c}
\hbox{entry}&\hbox{order and coefficient size}&\hbox{natural source bound}\\ \hline
\mathfrak h\to F_*\,,\ \mathfrak\pi\to G_* &
0,\ \mathcal P a&C_s\mathcal P a E_s\\
\mathfrak h\to G_*&0,\ \mathcal P a&
C_s\mathcal P a B_q^{-1/4}E_s\\
\mathfrak\pi\to F_*&-1,\ \mathcal P a^{1/4}&
C_s\mathcal P a^{1/4}E_s\\
\hbox{remaining entries}&-1\ \hbox{or lower},\ \mathcal P a^{1/4}&
C_s\mathcal P a^{1/4}E_s.
\end{array}
\end{equation*}
The first row includes the half-density time coefficient, the skew
transport completion, and the perturbation of $d_-$ from one.  The
second includes $(\mathsf S)_Y$, the angular Taylor coefficient
$-\Theta(\mathfrak t_--\tau_0)$ and $-\Theta B_qc_0$.
The order-zero DN term $\ell_0$ remains in $\mathsf L_0$ to preserve the
source bound of order minus one in the third row.  The fixed cutoff estimates control all negative-order remainders
through fixed lower symbol seminorms.
Every capillary composition remainder retains its factor $B_q$.

\emph{Step 6. Forcing and coupling to the outer boundary.}
When the perturbation variables vanish, the forcing in the inner high modes has size
$\mathcal P_s(Y)a^3$.  The contributions are the normalized high modes of the condenser,
of size $O(\mathcal P_s a^2)$; the DN coupling from the first to the
high modes, of size $O(\mathcal P_s a\beta)$ by
Lemma~\ref{lem:weighted-layer-expansion}; the square of the first mode
and its centre term in \eqref{eq:exact-kinetic-low-source}; and the
high modes of the circular curvature.  Each is multiplied by $a$ in the scaled
inner rows. Projection removes the zeroth and first harmonics from this forcing.
In the normal variational system, the outer good unknown appears only
in the kinematic row. Lemma~\ref{lem:cross-kernel-bounds} gives its exact zero-flux
operator \eqref{eq:cross-exact-schur} and the reference bound
\eqref{eq:cross-reference-scaled-bound}. The normal and half-density
factors and finite projections preserve this bound. With $r=s+2$ this bounds
its natural source norm by $C_s\mathcal P_s a^2\mathcal X_s^+$.
For the state-dependent cross kernels,
\eqref{eq:cross-kernel-a-factor}--\eqref{eq:cross-kernel-a-variation}
give the kernel hypothesis used in Step~3. Their dyadic and quadratic
remainders are part of the already estimated complete remainder, with
bound $C_s\mathcal P_s\delta_b(E_s+a\mathcal X_s^+)$ in the inner
source norm. The scalar modifications of the cross operator are
included in \eqref{eq:scalar-gauge-schur-inner-bound}. Thus the full
cross contribution obeys
\begin{equation}\label{eq:complete-cross-source-bound}
\begin{split}
\mathcal N_s(\mathcal R_{-+})
&\le C_s\mathcal P_s\bigl[
a^2\mathcal X_s^++\delta_b(E_s+a\mathcal X_s^+)\bigr]\\
&\le C_s\varepsilon_*E_s+C_s\mathcal P_s a^2\mathcal X_s^+.
\end{split}
\end{equation}
This estimate includes the outer shape input through its good unknown.

Each homogeneous term in the table and each high--high remainder is
bounded by $C_s\varepsilon_*E_s$.  This follows from $B_q=\sigma a/q^2$
and $\mathcal P_s(Y)a^\delta\le C_{s,\delta}a^{1/8}$ for
$\delta\ge1/4$, after increasing a fixed constant.  The terms with polynomial
decay are bounded by $C_s(1+Y)^{-3/2}$.
This proves \eqref{eq:actual-full-source-bound} and the system
\eqref{eq:actual-principal-equations}.
The retained coefficients satisfy the assumptions of
Theorem~\ref{thm:natural-variable-principal-energy}, which gives
\eqref{eq:actual-inner-energy}.
\end{proof}

\begin{proposition}
\label{prop:actual-outer-principal-energy}
Fix $\sigma>0$.  Under the assumptions of
Theorem~\ref{thm:actual-principal-reduction}, there is an outer energy
$\mathcal O_s\asymp\mathcal X_s^+$ such that
\begin{equation}\label{eq:actual-outer-energy}
|\partial_Y\mathcal O_s|
\le C_s\varepsilon_*(Y)\mathcal O_s+
C_s\mathcal P_s(Y)(a^2+aE_s).
\end{equation}
The fixed lower norm is the one in
Theorem~\ref{thm:actual-principal-reduction}.
\end{proposition}
\begin{proof}
We divide the proof into four steps.

\emph{Step 1. The coercive outer forms.}
Use a fixed macroscopic normal graph and, for its variations, the good unknown
$U_+=\delta\psi_+-B_+H_+$.  After the surface half-density change, the local DN symbol is
$x_++\ell_{+,0}$ modulo symbols of order $-1$, and the capillary symbol
is $\sigma x_+^2$ modulo symbols of order zero.  Here $j_+$ is the outer arclength density,
$x_+=|\xi|/j_+$, and $\ell_{+,0}$ is the order-zero DN symbol.
The chart bounds give $0<c\le j_+\le C$.  Retain
\begin{equation*}
\mathsf L_+=\mathcal W(x_++\ell_{+,0}),\qquad
\mathsf M_+=\sigma\mathcal W(x_+^2)+C_\mathrm{low}.
\end{equation*}
Here $C_{\mathrm{low}}$ is a fixed self-adjoint correction on finitely
many low modes.  After removing the potential gauge, choose this correction and,
if necessary, the analogous correction to $\mathsf L_+$ so that both
forms are coercive.  Their contributions to the source contain $t_Y$ and are integrable.

\emph{Step 2. Local coefficients and the effect of the hole.}
At the fixed outer scale, apply the local symbol estimates and the energy
arguments of Lemma~\ref{lem:finite-tame-dn-symbol}, followed by
Lemma~\ref{lem:finite-loss-dyadic} on the domains $\mathcal A_Y$.
For the hole correction $K_{++}$ to the zero-flux outer diagonal DN
map, \eqref{eq:outer-kernels-C2-negative-input} supplies the kernel
and its first two graph variations with fixed loss. In particular,
with $r=s+2$,
\begin{align}
\|K_{++}(0)v\|_{H^{s+2}}
&\le C_s\mathcal P_s(Y)\|v\|_{H^{-N_1}},\notag\\
\|[K_{++}(\boldsymbol\zeta)-K_{++}(0)]v\|_{H^{s+2}}
&\le C_s\mathcal P_s(Y)
\|\boldsymbol\zeta\|_{H^{s+2+d_{\rm ker}}}
\|v\|_{H^{-N_1}}.
\label{eq:outer-diagonal-kernel-application}
\end{align}
The same bounds hold with the normal good-unknown and half-density
factors included: their input profiles require only fixed lower
norms, and their output profiles obey the same tame product bounds.
The smooth part of the filled-domain DN map is treated by the
fixed-collar version of the argument in
Lemma~\ref{lem:finite-tame-dn-symbol}. Its first two variations obey
the same estimates.

Keep the reference operators exactly. For the differences, on a
frozen dyadic block the second estimate in
\eqref{eq:outer-diagonal-kernel-application} gives
\begin{equation*}
C_s\mathcal P_s(Y)
2^{(2+d_{\rm ker}-s_{\rm dy}-N_1)j}
\|U\|_{H^{s_{\rm dy}}}\|U\|_{H^s}.
\end{equation*}
This is summable by \eqref{eq:positive-regularity-hierarchy}, and the
second-variation bound controls the Taylor terms. Extract $a^2$
from the outer vector field before applying this argument; the
remaining factor $\Theta/q$ has polynomial state derivatives on
$\mathcal A_Y$. Thus the nonlinear kernel terms contribute at most
$C_s\mathcal P_s a^2\delta_b(a^{-1}E_s+\mathcal O_s)$.
The exact reference kernels contribute
$C_s\mathcal P_s a^2\mathcal O_s$ after the change of unknown.

The first derivatives of $j_+$ and of the local coefficients involve
only the outer geometric velocity and have size $\mathcal P(Y)a^2$ in the
required fixed norms.  Lemma~\ref{lem:zero-charge-outer-derivative} controls the derivative
of the outer good unknown coefficient.
For the remaining zeroth-order Taylor term,
Proposition~\ref{prop:actual-taylor-coefficient} gives
\begin{equation*}
\|t_Y\mathfrak a_+\|_{C^k}
\le C_k(1+Y)^{-3/2}+\mathcal P(Y)a.
\end{equation*}
It maps an outer height in $H^{s+1}$ to a momentum source in
$H^{s+1/2}$ with this integrable coefficient.  The physical dilation
and transport terms, after multiplication by $t_Y$, have the same
integrable bounds.  Thus all remaining diagonal and triangular coefficients have integrable
bounds in the fixed capillary energy norm.

\emph{Step 3. Coupling and forcing.}
For the coupling between the two boundaries,
\eqref{eq:exact-normal-variational-system} contains the inner potential
only in its kinematic row.  The inner good unknown
is $-(q/a)\mathfrak\pi$ up to bounded half-density factors.  Consequently
its outer coefficient is
\begin{equation*}
t_Y(q/a)\mathsf G_{+-}=\Theta a\mathsf G_{+-}.
\end{equation*}
Lemma~\ref{lem:zero-separated-asymmetric} gives polynomial bounds in
$Y$ for this separated map and the constant trace lifts, requiring high
regularity only of the outer geometry.  The first two variation
estimates from that lemma, together with the same dyadic reduction,
bound the reference coupling by $C_s\mathcal P_s(Y)aE_s$ in the outer
source norm. The scalar and Schur corrections are covered by
\eqref{eq:scalar-gauge-schur-outer-bound}; the state-dependent kernel
differences belong to the complete dyadic remainder below.

At zero perturbation, the outer kinematic and Bernoulli fields have
size $\mathcal P_s(Y)a^2$ in $Y$ time: the fixed scale velocity and
curvature fields are multiplied by $t_Y$.  The outer Bernoulli equation uses the pressure trace prescribed by
curvature.  The high--high remainders have the extra factor $t_Y$.
Their coarse bound is
\begin{equation*}
C_s\mathcal P_s(Y)t_Y\|U\|_{H^{s_{\mathrm{dy}}}}
(a^{-1}E_s+\mathcal O_s),
\end{equation*}
which is bounded by the right side of
\eqref{eq:actual-outer-energy}.

\emph{Step 4. The energy inequality.}
Use $P_s^+=\mathcal W(x_+^s)$ with the fixed low-frequency completion.
The principal energy argument applies since
$\{x_+^s,x_+\}=\{x_+^s,x_+^2\}=0$, while the order-zero DN commutator
has order $s-1$.  The coefficient derivatives are controlled by
\eqref{eq:outer-velocity-first-jet}, and the lower commutators by
Lemma~\ref{lem:finite-bony-weyl}.  The common factor $t_Y$ cancels in the Hamiltonian pairing.
The energy weight is independent of $t_Y$.
The resulting fixed capillary energy has the stated norm equivalence
and satisfies \eqref{eq:actual-outer-energy}.
\end{proof}

Combining the inner and outer inequalities closes the bootstrap and
continues the canonical solutions to a common initial value of $Y$.

\begin{theorem}
\label{thm:one-state-closed}
Fix $\sigma>0$.  For the canonical data of Lemma~\ref{lem:terminal-data-admissible}, there is
a common $Y_0$, independent of $Y_N$, to which every corresponding solution
continues backwards.  Each solution satisfies, for every fixed $s$,
\begin{equation}\label{eq:one-state-closed-estimates}
E_s(Y)\le C_s\mathcal P_s(Y)a(Y)^3,
\qquad \mathcal O_s(Y)\le C_s\mathcal P_s(Y)a(Y)^2,
\qquad Y_0\le Y\le Y_N.
\end{equation}
In addition,
\begin{equation}\label{eq:complete-state-closed-estimates}
|\alpha|+|\alpha_Y|\le q\mathcal P(Y)a,
\quad |R-R_*|+|R_Y|\le\mathcal P(Y)a^2,
\quad
\|\eta\|_{H^s}+q^{-1}\|f\|_{H^s}
\le C_s\mathcal P_s(Y)a^2.
\end{equation}
Proposition~\ref{prop:raw-state-first-jet} bounds the first $Y$ derivatives
using a fixed number of additional spatial derivatives. This includes the coefficient derivatives
required in the energy argument.
\end{theorem}
\begin{proof}
We divide the proof into four steps.

\emph{Step 1. The lower bootstrap and the flux bound.}
Fix $N,N_1,k_0,s_{\rm dy},s_*$ in the order specified in
\eqref{eq:positive-regularity-hierarchy}. In particular
$s_*\ge k_0+L_{\rm coeff}$ supplies the spatial and time coefficient estimates
in Propositions~\ref{prop:raw-state-first-jet}--
\ref{prop:actual-taylor-coefficient} and
Lemma~\ref{lem:normalized-taylor-variation} at the fixed coefficient
index $k_0$.
Impose \eqref{eq:physical-bootstrap-neighbourhood} with $b=s_*$ and
$M_b=M_*$.  The norms in its first line are controlled by
$a^{-1}E_{s_*}+\mathcal O_{s_*}+|R-R_*|$ at the chosen lower
energy level.
First choose a polynomial weight $(1+Y)^{M_*}$ whose degree exceeds
those of the reference inner and outer forcing bounds at this
level, the first mode forcing and the centre tail.  Include one extra
power to absorb the factor $q^{-1}$.  The reference geometry and the finite elliptic estimates determine
these degrees.  The bootstrap terms involving the perturbation variables are
homogeneous errors with integrable coefficients, so the forcing degree
remains unchanged.
The energy identity gives a two-sided bound for $q$ on this
neighborhood using only the physical bootstrap.  To derive it,
write $A=|\Gamma_-|+|\Gamma_+|$ and
$D_b=\pi L(Y+\mathfrak c)$.  The geometric bootstrap and the
spatial capacity estimate give
\begin{equation*}
N:=E_*-\sigma A=E_c+O(a+K(1+Y)^{M_*}a^2),
\qquad c(1+Y)\le D_b\le C(1+Y).
\end{equation*}
Apply the uniform gradient estimate
\eqref{eq:global-gradient-v} to the zero-flux potential with inner
trace $\alpha\cos\theta+f$ and outer trace $g$.  Squaring the bound
and using the uniform cylindrical weight gives
\begin{align}
0\le K_0&\le C\bigl(|\alpha|^2+\|f\|_{H^1}^2+\|g\|_{H^1}^2\bigr)
\label{eq:zero-flux-dirichlet-bound}\\
&\le C\bigl(\delta_1^2+K^2(1+Y)^{2M_*}a^4\bigr)q^2
+CK^2(1+Y)^{2M_*}a^4,\notag
\end{align}
with a constant uniform in the small radius.  Since
$N=D_bq^2+K_0$, nonnegativity first gives $q^2\le C/(1+Y)$.
Increasing $Y_0$ so that the last additive term is smaller than
$E_c/4$ gives $q^2\ge c/(1+Y)$.  Since $q(Y_N)>0$,
continuity fixes the positive branch throughout the bootstrap interval.
Thus $q\asymp(1+Y)^{-1/2}$, using only the physical trace bootstrap.
The determinant in the exact kinematic formulas is positive there,
so $t_Y=\Theta a^2/q>0$.  The terminal data satisfy these bounds strictly.

\emph{Step 2. The first mode and centre motion.}
Choose $Y_0$ so that the fixed lower physical norm is at most $a^{3/2}$.
Let $\delta_b$ denote the sum of the normalized
physical high traces and geometric perturbations in the fixed low norms
used here. The factors $q^{\pm1}$ are included in its polynomial weight.
The preceding calculation gives, for a polynomial $\mathcal P^0$ fixed on the
unit low norm ball,
\begin{equation*}
q^{-1}|\mathfrak f_1|\le \mathcal P^0(Y)(a+\delta_b),\qquad
\delta_b\le a^{3/2}.
\end{equation*}
Indeed the terms quadratic in $\beta=\alpha/q$ at the circular cylindrical reference geometry
have zero first Fourier coefficient. Every other term has a reference
geometric factor $O(\mathcal P^0a)$ or a physical high trace factor
$O(\mathcal P^0\delta_b)$.  The curvature term has size at most
$\mathcal P^0(a^2+a\delta_b)$ after division by $q$.
Thus $\mathcal P^0$ and the resulting first mode and centre bounds are independent
of the bootstrap weight and constant.  The exact backward integral for
$\alpha_Y-\alpha=\mathfrak f_1$, $\alpha(Y_N)=0$, therefore improves the first mode
bound to $|\alpha|+|\alpha_Y|\le Cq\mathcal P(Y)a$, with a constant independent of $K$.
The centre equation then gives $|R_Y|\le C\mathcal P(Y)a^2$.
Integrating from the terminal value gives the same bound for $R-R_*$.

\emph{Step 3. Energy estimates and closure of the bootstrap.}
To track the dependence of the source estimates on the bootstrap
parameters, keep $a,R,\beta$ as parameters and apply
Lemma~\ref{lem:finite-loss-dyadic} to the complete reconstructed vector
field, as in Step~2 of Theorem~\ref{thm:actual-principal-reduction}.
Retain its value and exact differential at zero perturbation, its
paralinearized local symbol difference, and its quadratic remainder.
The value and differential at this reference state have polynomial bounds
determined by the reference geometry, $q\asymp(1+Y)^{-1/2}$ and the
first mode bound just obtained. Apply Lemma~\ref{lem:frozen-admissibility} with this
bootstrap-independent choice of $\mathcal P_{\rm ref}$. The bound
$\delta_b\le a^{3/2}$ places every frozen state and Taylor segment
in its admissible domain. That lemma and
Lemmas~\ref{lem:reconstructed-kernel-bounds}--\ref{lem:cross-kernel-bounds}
give polynomial constants independent of $K$ and $M_*$. Keeping the
low factor $\delta_b$ explicit and denoting the two quadratic source
norms by $\mathcal Q_-$ and $\mathcal Q_+$, the inner and outer reductions give
\begin{equation*} 
\mathcal Q_-\le \mathcal P_s^0\delta_b(E_s+a\mathcal O_s),\qquad
\mathcal Q_+\le \mathcal P_s^0t_Y\delta_b(a^{-1}E_s+\mathcal O_s). 
\end{equation*}
Here $\mathcal P_s^0=\mathcal P_s^0(Y)$ is a polynomial weight fixed before the
bootstrap constants.  Since
$\delta_b\le a^{3/2}$ and $t_Y\le \mathcal P_s^0a^2$, these terms are bounded by
an integrable multiple of the same energy plus, respectively,
$\mathcal P_s^0a^2\mathcal O_s$ and $\mathcal P_s^0aE_s$, after enlarging $\mathcal P_s^0$.
They contribute no independent forcing. The finite-rank scalar and
Schur terms have already been included in this decomposition:
\eqref{eq:finite-rank-kernel-difference} and the summation following
\eqref{eq:scalar-gauge-schur-outer-bound} place their nonlinear kernel
differences in these quadratic bounds. The common gauge is removed
by the original mean projections. Their exact reference kernels
satisfy \eqref{eq:scalar-gauge-schur-inner-bound}--
\eqref{eq:scalar-gauge-schur-outer-bound}, and the cross-boundary terms
satisfy \eqref{eq:complete-cross-source-bound}. Reference forcing therefore remains
$\mathcal P_s^0a^3$ and $\mathcal P_s^0a^2$ in the two rows.
Use $\mathcal X_s^+\le C_s\mathcal O_s$ in
\eqref{eq:actual-inner-energy} and combine it with
\eqref{eq:actual-outer-energy}.  Lemma~\ref{lem:zero-coupled-terminal-budget}, applied to $E_s+a\mathcal O_s$,
gives \eqref{eq:one-state-closed-estimates}.  The terminal estimate is uniform in the radius and applies at every fixed
higher index.

At the fixed level $s_*$, choose $K$ larger than twice the resulting
forcing constant.  Then enlarge $Y_0$ so that, after integration, the
coefficients depending on the bootstrap constant contribute less than the remaining
factor of two.  The trace reconstruction and the centre and first mode bounds improve
every bootstrap inequality.  Continuity closes the fixed lower
bootstrap on the maximal common interval.

At each higher fixed $s$, the closed lower bounds control all
coefficient seminorms.  Integrating the same two energy inequalities
gives \eqref{eq:one-state-closed-estimates} with the same $Y_0$.
Lower indices follow by Sobolev monotonicity.  Reconstruction gives
\eqref{eq:complete-state-closed-estimates}.

\emph{Step 4. Continuation to the common initial parameter.}
If the left endpoint exceeded $Y_0$, the radius would remain at least
$a(Y_N)>0$, all continuation norms would be bounded, and the charts
and boundary separation would remain nondegenerate.  Local capillary
Euler theory would extend this solution with finite terminal value,
contradicting maximality.
The derivative bounds for the original variables follow from
Proposition~\ref{prop:raw-state-first-jet}.
\end{proof}

\section{Terminal approximation and singular limit}
\label{sec:terminal-existence}

Using the uniform estimates of Theorem~\ref{thm:one-state-closed}, we let
the terminal cavities shrink to obtain a canonical collapsing trajectory
with positive surface tension.  We derive its curvature, velocity, energy
and pressure asymptotics, describe the terminal material map and the
cylindrical limit, and construct a weak filling continuation with its
exact energy defect.  The limiting arguments also apply to the
zero-surface-tension trajectories constructed in
Section~\ref{sec:zero-surface-tension}.

\subsection{Curvature asymptotics and compactness}

We establish curvature estimates for circular tubes and show that they
are stable under the relative graph perturbations in
\eqref{eq:asymptotically-tubular-final}.  We use the spectral norm
\eqref{eq:spectral-H32-definition}.

\begin{proposition}
\label{prop:correct-tube-stability}
\begin{enumerate}[label=(\roman*),leftmargin=*]
\item Let $\Sigma_{a,\eta}$ be the surface of revolution generated by
$r=a(1+\eta(\theta))$ about a circle whose radius $R$ lies in a fixed compact
subset of $(0,\infty)$.  For all sufficiently small $a>0$ and
$\|\eta\|_{H^{7/2}}$, every axisymmetric scalar $f$, identified with
its pullback to the meridional angle $\theta$, satisfies
\begin{equation*}
\|f\|_{H^{3/2}_{\rm sp}(\Sigma_{a,\eta})}
\le Ca^{-1}\|f\|_{H^{3/2}(\T)}.
\end{equation*}
If $R_a\to R_*>0$, $\eta_a\to0$ in $H^{7/2}$, and
$g_a\to g_*$ in $H^{3/2}(\T)$, then
\begin{equation}\label{eq:low-mode-fractional-limit}
a^2\|g_a\|_{H^{3/2}_{\rm sp}(\Sigma_{a,\eta_a})}^2
\longrightarrow
2\pi R_*\big\||D_\theta|^{3/2}g_*\big\|_{L^2(\T)}^2.
\end{equation}
\item Let $\Sigma_{a,\eta_a}$ be smooth centred axisymmetric relative graphs
over circles of radii $R_a\to R_*>0$, and put $L_*=2\pi R_*$.
Assume that $\sup_a\|\eta_a\|_{H^{7/2}}$ is sufficiently small.
If $\|\eta_a\|_{H^2}=o(1)$, then, for either orientation of the normal,
\begin{equation*}
\|\mathcal H_{a,\eta_a}\|_{L^2(\Sigma_{a,\eta_a})}
\sim(2\pi L_*)^{1/2}a^{-1/2}.
\end{equation*}
If moreover $\|\eta_a\|_{H^{7/2}}=o(a)$, then
\begin{equation}\label{eq:tube-H32-H}
\|\mathcal H_{a,\eta_a}\|_{H^{3/2}_{\rm sp}(\Sigma_{a,\eta_a})}^2
\sim\frac{2\pi^2}{R_*}a^{-2}.
\end{equation}
In particular, these conclusions hold for the exact circular tubes.
\end{enumerate}
\end{proposition}
\begin{proof}
We divide the proof into two steps.

\emph{Step 1. Fractional scaling.}
Put $j=((1+\eta)^2+\eta_\theta^2)^{1/2}$ and
$\rho=R+a(1+\eta)\cos\theta$.  After integration in the azimuthal variable, the surface measure
becomes $a\mathfrak m\,d\theta$, where $\mathfrak m=2\pi\rho j$.
The axisymmetric subspace reduces $\Delta_\Sigma$ because rotations
commute with this operator.  On that subspace, the map
$Uf=(a\mathfrak m)^{1/2}f$ is unitary from
$L^2(\T,a\mathfrak m\,d\theta)$ to $L^2(\T,d\theta)$ and conjugates
$a^2(1-\Delta_\Sigma)$ to
\begin{equation*}
\mathsf B\varphi=a^2\varphi-\mathfrak m^{-1/2}\partial_\theta
\left(\frac{\mathfrak m}{j^2}\partial_\theta(\mathfrak m^{-1/2}\varphi)\right).
\end{equation*}
This is a nonnegative self-adjoint operator on $L^2(\T)$, with form
domain $H^1(\T)$.  Its operator domain is $H^2(\T)$: the leading
coefficient is positive and has bounded first derivative, and the
coefficient of order zero lies in $L^2$.  The weak equation and
$H^1(\T)\subset L^\infty(\T)$ therefore give $\varphi''\in L^2$.
Since $\mathfrak m,j$ and their inverses are uniformly bounded in $H^{5/2}$,
expansion of the displayed expression gives
\begin{equation*}
\|(\mathsf B+\partial_\theta^2)\varphi\|_2
\le C\bigl(a+\|\eta\|_{H^{7/2}}\bigr)\|\varphi\|_{H^2}.
\end{equation*}
The coefficients of $\varphi''$ and $\varphi'$ converge to those of
$-\partial_\theta^2$ in $L^\infty$, and that of $\varphi$ converges in $L^2$.
Absorbing this perturbation in the elliptic estimate for $1-\partial_\theta^2$ proves
\begin{equation*}
\|(1+\mathsf B)\varphi\|_2\asymp\|\varphi\|_{H^2}
\end{equation*}
uniformly.  Interpolation of the positive self-adjoint Hilbert scale
between $L^2$ and $D(1+\mathsf B)=H^2$ consequently gives
\begin{equation}\label{eq:uniform-fractional-upper}
\|\mathsf B^{3/4}\varphi\|_2
\le\|(1+\mathsf B)^{3/4}\varphi\|_2\le C\|\varphi\|_{H^{3/2}}.
\end{equation}
Since multiplication by $\mathfrak m^{1/2}$ is uniformly bounded on $H^{3/2}$,
the identity
\begin{equation*}
\|f\|_{H^{3/2}_{\rm sp}}
=a^{-3/2}\|\mathsf B^{3/4}Uf\|_2
=a^{-1}\|\mathsf B^{3/4}(\mathfrak m^{1/2}f)\|_2
\end{equation*}
proves the upper bound.

For the fractional limit, let $\mathsf B_a$ denote the operator
$\mathsf B$ for the parameters $(a,R_a,\eta_a)$, and set
$e_m=(2\pi)^{-1/2}e^{im\theta}$.
The coefficient estimate gives
$\|(\mathsf B_a-m^2)e_m\|_2\to0$ for each fixed $m\in\mathbb Z$.
If $\mu_{a,m}$ is the spectral probability measure of $\mathsf B_a$
associated with $e_m$, the scalar inequality
$|\lambda^{3/4}-\nu^{3/4}|\le|\lambda-\nu|^{3/4}$ for
$\lambda,\nu\ge0$ and H\"older's inequality give
\begin{equation*}
\|\mathsf B_a^{3/4}e_m-|m|^{3/2}e_m\|_2^2
\le\int|\lambda-m^2|^{3/2}\,d\mu_{a,m}
\le\|(\mathsf B_a-m^2)e_m\|_2^{3/2}\longrightarrow0.
\end{equation*}
Hence the fractional powers converge strongly on every finite Fourier
sum, including its constant mode.  The uniform bound \eqref{eq:uniform-fractional-upper} and
Fourier approximation extend this to every vector in $H^{3/2}$,
and to varying vectors converging in that space.
Since $\mathfrak m_a\to2\pi R_*$ in $H^{5/2}$, apply this conclusion
to $\mathfrak m_a^{1/2}g_a\to(2\pi R_*)^{1/2}g_*$ in $H^{3/2}$.
The preceding norm identity proves \eqref{eq:low-mode-fractional-limit}.

\emph{Step 2. Curvature asymptotics.}
Choose the normal pointing away from the centre circle.  Formula
\eqref{eq:explicit-dimensionless-curvature} gives
$a\mathcal H_{a,\eta}=\mathcal J_{a/R}(\eta)$.
Its denominators are uniformly bounded away from zero for the stated class of graphs.
Thus $\eta_a\to0$ in $H^2$ implies
$a\mathcal H_{a,\eta_a}\to1$ in $L^2(\T)$, while
$\mathfrak m_a\to2\pi R_*$ uniformly.  Consequently
\begin{equation*}
a\|\mathcal H_{a,\eta_a}\|_{L^2(\Sigma_{a,\eta_a})}^2
=\int_\T \mathfrak m_a|a\mathcal H_{a,\eta_a}|^2\,d\theta
\longrightarrow4\pi^2R_*=2\pi L_*.
\end{equation*}

For the fractional estimate, compare with the exact circular curvature
at the same radius $R_a$, pulled back to the perturbed surface:
\begin{equation}\label{eq:exact-tube-H}
\mathcal H_{a,\eta_a}=\frac1a+g_a+e_a,
\qquad g_a=\frac{\cos\theta}{R_a+a\cos\theta},
\qquad e_a=\mathcal H_{a,\eta_a}-\mathcal H_{a,0}.
\end{equation}
The explicit curvature formula and the one-dimensional Sobolev product
estimates show that $\mathcal J_{a/R}:H^{7/2}\to H^{3/2}$ is uniformly locally
Lipschitz.  Therefore
\begin{equation*}
\|ae_a\|_{H^{3/2}}
\le C\|\eta_a\|_{H^{7/2}}=o(a).
\end{equation*}
The bound in \textup{(i)}, applied to $e_a$, yields
\begin{equation*}
\|e_a\|_{H^{3/2}_{\rm sp}(\Sigma_{a,\eta_a})}=o(a^{-1}).
\end{equation*}
Also $g_a\to R_*^{-1}\cos\theta$ in $H^{3/2}(\T)$.
Since $|D_\theta|^{3/2}\cos\theta=\cos\theta$ and
$\int_\T\cos^2\theta\,d\theta=\pi$, \eqref{eq:low-mode-fractional-limit} gives
\begin{equation*}
a^2\|g_a\|_{H^{3/2}_{\rm sp}(\Sigma_{a,\eta_a})}^2
\longrightarrow\frac{2\pi^2}{R_*}.
\end{equation*}
Finally, constants are eigenfunctions of $-\Delta_\Sigma$ with eigenvalue zero, so
\begin{equation*}
\|1/a\|_{H^{3/2}_{\rm sp}(\Sigma_{a,\eta_a})}
=a^{-1}|\Sigma_{a,\eta_a}|^{1/2}=O(a^{-1/2}).
\end{equation*}
Applying the triangle inequality and its reverse to
\eqref{eq:exact-tube-H} proves \eqref{eq:tube-H32-H}.
Reversing the normal changes only the sign of the curvature.
\end{proof}

The following compactness statement holds for every fixed $\sigma\ge0$.
All approximations use the common gauge \eqref{eq:common-Y-gauge},
so their inner radii agree at each $Y$: $a(Y)=\ell e^{-Y}$.

\begin{proposition}
\label{prop:compactness-without-relative-energy}
Fix $\sigma\ge0$, $E_c>0$ and
$E_*=E_c+\sigma|\Gamma_+^*|$.
Let $Y_N\to\infty$.  Suppose the Euler solutions on $[Y_0,Y_N]$
with prescribed terminal data conserve
$E_*$, have positive flux, a common positive lower bound for
$E_*-E_0$, and share the same $Y_0$, centering and separation constants.
Assume the physical trace bounds and the bounds for the first harmonics
\eqref{eq:first-mode-physical-bootstrap}--\eqref{eq:first-mode-amplitude-closed}
and the decay estimates for the centre and outer boundary hold at every
fixed spatial regularity,
uniformly in $N$, with constants allowed to depend on that regularity.
Then a diagonal subsequence converges on each compact $Y$ interval to a
smooth Euler trajectory.  The limit satisfies the same physical trace and geometric bounds,
including $\eta=O_{H^k}((1+Y)^{M_k}a^2)$ for every fixed $k$, and hence
has the collapse law \eqref{eq:final-collapse-law} and the curvature
asymptotics \eqref{eq:L2-curvature-final}--\eqref{eq:H32-curvature-final}.
\end{proposition}

\begin{proof}
Fix $Y_1<\infty$ and pull the domains back to a common reference
domain using the boundary charts on $[Y_0,Y_1]$.
Since $a\ge\ell e^{-Y_1}>0$, the physical and scaled Sobolev norms are
equivalent, with constants depending on $Y_1$.
The state coordinates are $(\eta,R,X_+,\alpha,f,g)$. The zero-flux condition
determines the inner constant, and the exact energy identity determines
$q$.

The spatial bounds and Proposition~\ref{prop:raw-state-first-jet}
bound the first $Y$ derivatives in each fixed lower Sobolev space.
Interpolation gives equicontinuity: a uniform $H^{k+1}$ bound and a
uniform $H^{k-2}$ bound for the time derivative yield an $H^k$ modulus
$C|Y-Y'|^{1/3}$.
Compact Sobolev embedding and Arzel\`a--Ascoli then give subsequential
strong convergence in $C([Y_0,Y_1],H^k)$.  A diagonal extraction over
integer $Y_1$ and $k$ produces compatible limits on all compact intervals.

The capacity and the zero-flux extension $\Phi_0$ are continuous under
this convergence on the compact interval, where the domains remain
nondegenerate.  To see this, subtract their elliptic equations after
pullback and apply the fixed domain estimate.  The coefficients and
traces converge strongly, while ellipticity remains uniform.  Hence $D_b,K_0,A$ converge.  The fixed energy and positive energy gap
then give convergence of the positive root
$q=((E_*-\sigma A-K_0)/D_b)^{1/2}$, followed by convergence of the full
harmonic potential.
The algebraic modulation formula gives convergence of $t_Y$ and the
geometric right-hand side.  Curvature and the Bernoulli quadratic term
converge by the Sobolev product and trace estimates.

Passing to the limit in the integral boundary equations in the common
potential gauge, using strong convergence of their right-hand sides,
shows that the limit is differentiable in $Y$ and solves the exact
boundary system.  Lemma~\ref{lem:exact-boundary-reduction} gives
the Euler equations in the fluid.  On $[Y_0,Y_1]$ one has $0<c_{Y_1}\le t_Y\le C_{Y_1}$.  The common separation and chart bounds pass to the limit,
so the limiting domains are embedded and nondegenerate there.

Finally, the pointwise estimates pass to the limit by strong convergence
and, at the highest regularity index under consideration, by lower semicontinuity.  Their constants are
independent of $Y_1$, so the same bounds hold for every finite $Y$.
The decay estimates for the centre and outer boundary identify the prescribed terminal geometries.
The exact energy identity, bounded capacity correction and the decay estimates
for the centre and outer boundary give $q^2Y\to E_c/(\pi L_*)$.
The algebraic time parametrization gives $\Theta\to1$.  Integrating
$t_Y=\Theta a^2/q$ then gives a finite terminal physical time and the
collapse law \eqref{eq:final-collapse-law}.  Proposition~\ref{prop:correct-tube-stability}
gives \eqref{eq:L2-curvature-final} and \eqref{eq:H32-curvature-final}.  Smoothness follows from the diagonal
bounds at every fixed spatial index and the exact evolution equations.
\end{proof}

\subsection{Energy concentration and sharp velocity estimates}
\label{subsec:energy-concentration}

\noindent\emph{Common estimates for the limiting analysis.}
For the limiting trajectory of
Proposition~\ref{prop:compactness-without-relative-energy}, write
$Y=\log(\ell/a)$, $L=2\pi R$ and $\widetilde v=\mathbf1_{\Omega_t}v$.
The physical boundary estimates, the outer trace estimate and the energy
identity yield, for every fixed spatial regularity index $k$, a polynomial
$\mathcal P_k$ such that, for all large $Y$,
\begin{equation}\label{eq:concentration-input-tails}
\begin{gathered}
\|\eta\|_{H^k}+|R-R_*|+|R_Y|
+\|X_+-X_+^*\|_{H^k}\le\mathcal P_k(Y)a^2,\\
|\alpha|\le q\mathcal P_k(Y)a,\qquad
\|f\|_{H^k}\le q\mathcal P_k(Y)a^2,\qquad
\|g\|_{H^k}\le\mathcal P_k(Y)a^2,\\
\Lambda:=\frac{2\pi L}{\operatorname{Cap}}=Y+O(1),\qquad
q^2Y\longrightarrow\frac{E_c}{\pi L_*}>0,\qquad
t_Y=\Theta a^2/q,\quad |\Theta-1|\le\mathcal P_k(Y)a^2.
\end{gathered}
\end{equation}
The outer norms use the prescribed fixed charts.  For the canonical
trajectory at positive surface tension, these estimates follow from
Theorem~\ref{thm:one-state-closed},
Lemma~\ref{lem:capacity-single-layer-schur}, and
Proposition~\ref{prop:compactness-without-relative-energy}.
Below, $\mathcal P$ denotes a polynomial depending on the fixed regularity indices.

Throughout this subsection and the next two, we use
\eqref{eq:concentration-input-tails}, a global velocity potential,
and the boundary condition $p=\sigma\mathcal H$ for a fixed
$\sigma\ge0$.  The local elliptic estimates of
Section~\ref{sec:scaled-DN} also give the scaled collar bound
\eqref{eq:scaled-inner-collar-velocity} from the physical boundary
estimates.  The conclusions hold for every trajectory satisfying
these assumptions.

Choose a fixed $r_0<R_*/4$ so small that, for all sufficiently late times,
$\overline{N_{r_0}(\gamma(t))}$ lies strictly inside the region enclosed
by $\Gamma_+(t)$.  Set $b(\theta)=a(1+\eta(\theta))$ and
$\Lambda_0=\log(r_0/a)$.  We parametrize the centre circles by
\begin{equation*}
\gamma(t,\phi)=(R(t)\cos\phi,R(t)\sin\phi,0),\qquad
\gamma_*(\phi)=(R_*\cos\phi,R_*\sin\phi,0).
\end{equation*}
For $0\le\phi,\theta<2\pi$ the tube coordinates are
\begin{equation}\label{eq:concentration-tube-coordinates}
\begin{aligned}
X_t(\phi,r,\theta)
&=\big((R+r\cos\theta)\cos\phi,
(R+r\cos\theta)\sin\phi,r\sin\theta\big),\\
e_r(\phi,\theta)&=(\cos\theta\cos\phi,\cos\theta\sin\phi,\sin\theta),
\qquad ds_t=R\,d\phi,\\
dx&=r\left(1+\frac rR\cos\theta\right)dr\,d\theta\,ds_t.
\end{aligned}
\end{equation}
The remaining unit vectors are $e_\theta=\partial_\theta e_r$ and
$e_\phi=(-\sin\phi,\cos\phi,0)$. In these coordinates the fluid
inside the tube is $b(\theta)<r<r_0$.  Define a vector field on $\R^3$ by
\begin{equation*}
w_t(X_t(\phi,r,\theta))=-\frac q r e_r
\quad\hbox{if }b(\theta)<r<r_0,
\qquad w_t=0\quad\hbox{elsewhere}.
\end{equation*}

\begin{lemma}
\label{lem:global-radial-energy-approximation}
For the constructed trajectory and all sufficiently large $Y$,
\begin{equation}\label{eq:global-radial-energy-error}
\|\widetilde v(t)-w_t\|_{L^2(\R^3)}\le C Y^{-1/2}.
\end{equation}
\end{lemma}
\begin{proof}
The exact decomposition in \eqref{eq:exact-energy-components} is
\begin{equation*}
v=q\nabla W+\nabla\Phi_0,\qquad
W=\Lambda u_c,\qquad
\int_{\Omega_t}|\nabla W|^2\,dx=2\pi L\Lambda.
\end{equation*}
Since $\|\nabla\Phi_0\|_2^2=2K_0$,
\eqref{eq:zero-flux-dirichlet-bound} and
\eqref{eq:concentration-input-tails} give
$\|\nabla\Phi_0\|_2\le q\mathcal P(Y)a=o(q)$.

To compare the condenser potential with a radial potential, define
\begin{equation*}
\Lambda_b(\theta)=\log\frac{r_0}{b(\theta)}
=\Lambda_0-\log(1+\eta(\theta)).
\end{equation*}
Define the $H^1$ comparison potential on the fluid by
\begin{equation*}
W_c(X_t(\phi,r,\theta))
=\Lambda\frac{\log(r_0/r)}{\Lambda_b(\theta)}
\quad(b(\theta)<r<r_0),
\qquad W_c=0\quad(r\ge r_0).
\end{equation*}
This potential has the same traces as $W$: $\Lambda$ on the
inner boundary and zero on the outer boundary.  Continuity across $r=r_0$ gives $W_c\in H^1(\Omega_t)$.
Thus $W_c-W\in H^1_0(\Omega_t)$, and harmonicity of $W$ yields
\begin{equation*}
\|\nabla(W_c-W)\|_2^2
=\int_{\Omega_t}|\nabla W_c|^2\,dx-2\pi L\Lambda.
\end{equation*}

Write $E_r(W_c)$ and $E_\theta(W_c)$ for the radial and angular parts
of the Dirichlet energy of this potential.  The exact volume element gives
\begin{align}
E_r(W_c)
&=L\Lambda^2\int_0^{2\pi}
\left(\frac1\Lambda_b+
\frac{(r_0-b)\cos\theta}{R \Lambda_b^2}\right)d\theta
=\frac{2\pi L\Lambda^2}{\Lambda_0}+O(1),
\label{eq:comparison-radial-energy}\\
E_\theta(W_c)
&\le C\Lambda^2\int_0^{2\pi}
\frac{|\partial_\theta \Lambda_b|^2}{\Lambda_b^4}
\int_b^{r_0}\log^2(r_0/r)\,\frac{dr}{r}\,d\theta
\le C Y\|\eta_\theta\|_{L^2}^2=o(1).
\label{eq:comparison-angular-energy}
\end{align}
Here $\Lambda_b\asymp \Lambda_0\asymp Y$, $\Lambda-\Lambda_0=O(1)$,
and $\Lambda_b-\Lambda_0=-\log(1+\eta)=O(\|\eta\|_{L^\infty})$.
The radial Jacobian correction in
\eqref{eq:comparison-radial-energy} is bounded independently of $Y$,
while the inner integral in \eqref{eq:comparison-angular-energy}
equals $\Lambda_b^3/3$.  Therefore
\begin{equation*}
0\le\|\nabla(W_c-W)\|_2^2
=2\pi L\Lambda
\frac{\Lambda-\Lambda_0}{\Lambda_0}+O(1)\le C.
\end{equation*}

The radial derivative of $W_c$ is $-\Lambda/(r\Lambda_b)$, so the same calculation gives
\begin{equation*}
\|\nabla W_c-w_t/q\|_{L^2(\Omega_t)}^2
\le C\int_0^{2\pi}
\frac{(\Lambda-\Lambda_b)^2}{\Lambda_b}\,d\theta
+CY\|\eta_\theta\|_{L^2}^2
\le \frac C Y+CY\|\eta_\theta\|_{L^2}^2.
\end{equation*}
The exact velocity decomposition and these estimates prove
\eqref{eq:global-radial-energy-error} after zero extension.
\end{proof}

\begin{proposition}
\label{prop:energy-concentration-limits}

The energy and velocity limits
\eqref{eq:main-energy-concentration}, \eqref{eq:main-logarithmic-energy-law}
and \eqref{eq:main-velocity-limits}
hold, including the vanishing energy in $N_{Ca(t)}(\gamma(t))$ for
fixed $C>1$ and the obstruction to strong $L^2$ continuation.
For every fixed $1\le p<2$,
\begin{equation}\label{eq:subcritical-Lp-rate}
\|\widetilde v(t)\|_{L^p(\R^3)}\le C_pY^{-1/2}
\quad\hbox{for all sufficiently large }Y.
\end{equation}
\end{proposition}
\begin{proof}
We divide the proof into two steps.

\emph{Step 1. Energy concentration.}
The conserved energy and
Lemma~\ref{lem:global-radial-energy-approximation} bound
$\|\widetilde v\|_2+\|w_t\|_2$ uniformly.  Consequently
\begin{equation}\label{eq:energy-density-L1-comparison}
\big\||\widetilde v|^2-|w_t|^2\big\|_{L^1}
\le C\|\widetilde v-w_t\|_{L^2}
\longrightarrow0.
\end{equation}
This $L^1$ estimate also controls energy integrals over measurable sets
that may depend on time.

For $\max_\theta b(\theta)<h\le r_0$, direct integration of
\eqref{eq:concentration-tube-coordinates} gives
\begin{equation*} 
\frac12\int_{N_h(\gamma(t))}|w_t|^2\,dx
=\pi Lq^2\log\frac h a
-\frac{Lq^2}{2}\int_0^{2\pi}\log(1+\eta)\,d\theta -\frac{Lq^2}{2R}\int_0^{2\pi}b(\theta)\cos\theta\,d\theta. 
\end{equation*}
The last two terms tend to zero uniformly over this range of $h$.
For $h=r_\lambda$ the first logarithm equals $(1-\lambda)\Lambda_0$,
and $q^2\Lambda_0\to E_c/(\pi L_*)$.
For $h=Ca$ it equals $\log C$.  Both choices exceed
$\max b$ for all sufficiently late times.  Equation~\eqref{eq:energy-density-L1-comparison} proves
\eqref{eq:main-logarithmic-energy-law} and the $Ca$-tube limit.  Taking $h=r_0$
also shows that the total kinetic energy tends to $E_c$.

To determine the spatial energy distribution, let $\varphi\in C_c(\R^3)$.
For fixed $0<\varepsilon<r_0$, the model energy in
$\varepsilon<r<r_0$ is at most
$Cq^2\log(r_0/\varepsilon)$ and tends to zero.  On $r<\varepsilon$,
replace $\varphi(X_t(\phi,r,\theta))$ by
$\varphi(\gamma(t,\phi))$ using uniform continuity.  The resulting
error is bounded by a constant times the modulus of continuity of
$\varphi$ at $\varepsilon$.  Moreover the Jacobian correction
$r\cos\theta/R$ contributes $O(q^2)$, and replacing the lower
radial endpoint $b(\theta)$ by $a$ contributes at most
$Cq^2\|\log(1+\eta)\|_{L^\infty}$.  Thus, first letting $t\uparrow T$
and then $\varepsilon\downarrow0$, we obtain
\begin{equation*}
\frac12\int\varphi|w_t|^2\,dx
-\pi q^2 \Lambda_0\int_{\gamma(t)}\varphi\,ds_t\longrightarrow0.
\end{equation*}
The convergence of $R$ and $q^2\Lambda_0$ proves
\eqref{eq:main-energy-concentration}.

\emph{Step 2. Velocity convergence.}
By convergence of the outer boundary, the fluid domains and the
supports of $w_t$ lie in a fixed bounded set $K$ for all sufficiently
late times.  For $1\le p<2$,
\begin{equation*}
\|w_t\|_{L^p}^p
\le C_pq^p\int_0^{r_0}r^{1-p}\,dr\le C_pq^p,
\qquad
\|\widetilde v-w_t\|_{L^p}
\le |K|^{1/p-1/2}\|\widetilde v-w_t\|_{L^2}.
\end{equation*}
This proves \eqref{eq:subcritical-Lp-rate}.  Strong $L^1$ convergence
to zero and the uniform $L^2$ bound imply weak $L^2$ convergence to
zero, by testing first against bounded compactly supported vector
fields and then using their density in $L^2$.  The kinetic energy limit gives
$\|\widetilde v(t)\|_2^2\to2E_c>0$.
A strong $L^2$ subsequential limit would equal the weak limit zero,
contradicting this positive limiting norm.
\end{proof}

The radial comparison also yields pointwise bounds for the velocity
and its gradient.

\begin{lemma}
\label{lem:velocity-away-from-core}
For every fixed $\delta>0$ and nonnegative integer $k$, there is a
constant $C_{\delta,k}$ such that, for all sufficiently late times,
\begin{equation}\label{eq:velocity-away-from-core}
\max_{0\le j\le k}\ \sup_{x\in\overline{\Omega_t}
\setminus N_\delta(\gamma_*)}
|D_x^jv(t,x)|\le C_{\delta,k}q(t).
\end{equation}
Spatial derivatives on the outer boundary are understood as traces from
the fluid.  Moreover,
\begin{equation}\label{eq:integrable-away-core-velocity}
\int_t^Tq(t')\,dt'
=\int_{Y(t)}^\infty\Theta(Y')a(Y')^2\,dY'
=\frac{a(t)^2}{2}(1+o(1)).
\end{equation}
\end{lemma}
\begin{proof}
Since $q\asymp Y^{-1/2}$, the energy comparison
\eqref{eq:global-radial-energy-error} gives
$\|\widetilde v-w_t\|_2\le Cq$.  The centre circles converge, so the
explicit field $w_t$ satisfies $|w_t|\le C_\delta q$ outside
$N_{\delta/2}(\gamma_*)$ for all sufficiently large $Y$.  All supports
lie in a fixed bounded set.  Consequently
\begin{equation*}
\|v\|_{L^2(\Omega_t\setminus N_{\delta/2}(\gamma_*))}
\le C_\delta q.
\end{equation*}
The Cartesian components of $v$ are harmonic, so local interior
estimates prove \eqref{eq:velocity-away-from-core} away from the outer
boundary.

Smooth convergence of the outer boundary provides charts of fixed size
with uniformly bounded derivatives at every fixed order.  The
normal velocity estimate \eqref{eq:outer-velocity-first-jet}, together
with the decay estimate for the outer potential trace $g$ at sufficiently
high regularity, bounds the
full Cartesian Dirichlet trace of $v$ in every fixed $C^m$ norm by
$C_mq$: its normal component is $B_+$, and its tangential component is
$\nabla_{\Gamma_+}g$.  Local boundary estimates for the harmonic
Cartesian components, using the trace bounds and
$\|v\|_{L^2(\Omega_t\setminus N_{\delta/2}(\gamma_*))}\le C_\delta q$
on slightly larger patches, give
\eqref{eq:velocity-away-from-core} up to $\Gamma_+(t)$.  These patches stay a fixed distance from the inner component.
A finite covering therefore gives constants independent of the
shrinking cavity.
Finally $q\,dt=\Theta a^2\,dY$, $a=\ell e^{-Y}$ and
$\Theta\to1$ give \eqref{eq:integrable-away-core-velocity}.
\end{proof}

The gradient is the
Cartesian Jacobian, $(\nabla v)_{ij}=\partial_jv_i$, and the tangential
derivative of the full velocity trace on
$\Gamma=\Gamma_-\cup\Gamma_+$ is
$\nabla_\Gamma v=(\nabla v)(I-n\otimes n)$.

\begin{lemma}
\label{lem:sharp-global-velocity-gradient}
There are constants $C,M>0$ such that, with
$\epsilon(Y)=C(1+Y)^Ma(Y)$ and $Y$ sufficiently large,
\begin{equation}\label{eq:sharp-global-velocity-gradient}
\left|\frac aq\|v\|_{L^\infty(\Omega_t)}-1\right|
+\left|\frac{a^2}q\|\nabla v\|_{L^\infty(\Omega_t)}-1\right|
+|\Theta-1|\le\epsilon(Y),
\qquad \int^{\infty}\epsilon(Y)\,dY<\infty.
\end{equation}
The normal velocity satisfies \eqref{eq:hly-normal-velocity-divergence},
and the tangential gradient satisfies
\begin{equation}
\|\nabla_\Gamma v\|_{L^\infty(\Gamma)}
\sim\|\nabla v\|_{L^\infty(\Omega_t)}
\sim\frac q{a^2}\sim\frac1{2(T-t)}.
\label{eq:hly-two-gradient-rates}
\end{equation}
The relative error in the comparison of the tangential gradient with $q/a^2$
is also bounded by $\epsilon(Y)$.  Both gradient integrals from a fixed
late time $t_0$ to $t(Y)$ equal $Y+O(1)$.  In particular,
\eqref{eq:main-gradient-blowup-rate} holds.
\end{lemma}
\begin{proof}
The state estimates at every fixed regularity pass to the limiting trajectory by
Proposition~\ref{prop:compactness-without-relative-energy}.  Hence the
scaled $C^1$ collar estimate \eqref{eq:scaled-inner-collar-velocity}
applies to this trajectory, and
$\eta=O_{C^1}(\mathcal P(Y)a^2)$.
In the orthonormal tube frame $(e_r,e_\theta,e_\phi)$, with
$\rho=R+r\cos\theta$, the radial comparison field has the exact
Cartesian gradient
\begin{equation}\label{eq:radial-cartesian-gradient}
\nabla\left(-\frac q r e_r\right)
=\frac q{r^2}(e_r\otimes e_r-e_\theta\otimes e_\theta)
-\frac{q\cos\theta}{r\rho}\,e_\phi\otimes e_\phi.
\end{equation}
Here $q$ is fixed when taking spatial derivatives.  On the inner
boundary $r=a(1+\eta)$, $\rho$ is uniformly positive and $r<\rho$ for
large $Y$.  The operator norm of the matrix in
\eqref{eq:radial-cartesian-gradient} is therefore exactly $q/r^2$.
The collar estimate now gives, uniformly on $\Gamma_-$,
\begin{equation*}
\frac a q|v|=1+O(\mathcal P(Y)a),\qquad
\frac{a^2}q\|\nabla v\|
=1+O(\mathcal P(Y)a).
\end{equation*}

Since $\Gamma_+$ stays a fixed distance from the core,
Lemma~\ref{lem:velocity-away-from-core} gives
$|v|+\|\nabla v\|\le Cq$ there.

Each Cartesian component of $v$ and of $\nabla v$ is harmonic in the
fluid.  For fixed Cartesian unit vectors $e,f$, apply the scalar
maximum principle to $e\cdot v$ and $e\cdot(\nabla v)f$, and then take
the supremum over these vectors.  This proves
\begin{equation*}
\|v\|_{L^\infty(\Omega_t)}
\le\sup_{\partial\Omega_t}|v|,\qquad
\|\nabla v\|_{L^\infty(\Omega_t)}
\le\sup_{\partial\Omega_t}\|\nabla v\|.
\end{equation*}
Continuity up to the boundary gives the matching lower bounds on
$\Gamma_-$.  The boundary estimates prove the velocity and gradient
bounds in \eqref{eq:sharp-global-velocity-gradient}.

The inner normal and meridional unit tangent satisfy
$n_-=-e_r+O(\mathcal P(Y)a^2)$ and
$\tau_-=e_\theta+O(\mathcal P(Y)a^2)$.  The same collar estimate gives
\begin{equation*}
v\cdot n_-=\frac qa(1+O(\mathcal P(Y)a)),\qquad
(\nabla v)\tau_-=-\frac q{a^2}e_\theta
+O\!\left(\frac q{a^2}\mathcal P(Y)a\right).
\end{equation*}
Together with $|\Gamma_-|=2\pi L_*a(1+o(1))$, the first identity
proves \eqref{eq:hly-normal-velocity-divergence}.  The second gives
the lower bound for the tangential gradient.  The upper bound follows from
$\|\nabla_\Gamma v\|\le\|\nabla v\|$.
The bound for $\Theta$ follows from \eqref{eq:concentration-input-tails}.
Increasing $C,M$ gives a common error bound $\epsilon(Y)$, whose integrability follows from
$a=\ell e^{-Y}$.  Multiplying either gradient estimate by
$t_Y=\Theta a^2/q$ gives $1+O(\epsilon(Y))$, so its time integral
is $Y+O(1)$.  Finally, \eqref{eq:final-collapse-law} and
\eqref{eq:concentration-input-tails} give $(T-t)q/a^2\to1/2$, proving
\eqref{eq:hly-two-gradient-rates} and \eqref{eq:main-gradient-blowup-rate}.
\end{proof}

\subsection{The terminal material map and critical regularity}

The velocity bounds give a uniform limit of the particle trajectories.
We identify its fibers and relate the terminal distance from the
collapse circle to the initial distance from the inner boundary.

\begin{proposition}
\label{prop:terminal-lagrangian-map}
Let $\mathcal X$ be the material flow of the constructed solution, and
let $\Omega_*$ be the open region enclosed by $\Gamma_+^*$, including the
limiting core circle.  Then $\mathcal X(t,\cdot)$ converges uniformly
on $\overline{\Omega_0}$ to a continuous map
$\mathcal X_T:\overline{\Omega_0}\to\overline{\Omega_*}$, with
the sharp rate \eqref{eq:main-lagrangian-terminal-rate} and the
upper modulus bound $\omega_T(\delta)\le C\sqrt\delta$ for all
sufficiently small $\delta>0$.
If $\phi(\xi)$ denotes the cylindrical azimuth of the initial position
$\xi$ on the inner boundary, then
\begin{equation}\label{eq:terminal-meridian-collapse}
\mathcal X_T(\xi)
=\bigl(R_*\cos\phi(\xi),R_*\sin\phi(\xi),0\bigr),
\qquad \xi\in\Gamma_-(0).
\end{equation}
\end{proposition}
\begin{proof}
We divide the proof into two steps.

\emph{Step 1. Uniform convergence and the sharp rate.}
For each $t<T$, smoothness and the kinematic boundary condition make
the material flow a diffeomorphism from $\overline{\Omega_0}$ to
$\overline{\Omega_t}$.  Write $\mathcal X_t=\mathcal X(t,\cdot)$.

Lemma~\ref{lem:sharp-global-velocity-gradient} and the exact time parametrization give
\begin{equation*}
\int_{t(Y)}^T\|v(t')\|_\infty\,dt'
=\int_Y^\infty a(Y')(1+O(\epsilon(Y')))\,dY'
=a(Y)(1+o(1)).
\end{equation*}
Integrating the material ODE proves uniform convergence and
\begin{equation}\label{eq:terminal-flow-upper-error}
\|\mathcal X_T-\mathcal X_t\|_{C^0(\overline{\Omega_0})}
\le a(t)(1+o(1)).
\end{equation}
Convergence of the outer boundary places the image in $\overline{\Omega_*}$.
The velocity associated with the single-valued axisymmetric potential
is meridional, so each fluid particle retains its initial cylindrical azimuth.  Particles
on the inner boundary remain at meridional distance $a(1+\eta)$ from the centre
circle.  Thus $R\to R_*$ and $a\to0$ prove
\eqref{eq:terminal-meridian-collapse}.  Moreover,
$|R-R_*|=O(\mathcal P(Y)a^2)=o(a)$ and $\|\eta\|_\infty=o(1)$
give $|\mathcal X_T(\xi)-\mathcal X_t(\xi)|=a(t)(1+o(1))$
uniformly on $\Gamma_-(0)$.  This supplies the matching lower bound
in \eqref{eq:main-lagrangian-terminal-rate}.

\emph{Step 2. The upper H\"older bound.}
The integrable error in \eqref{eq:sharp-global-velocity-gradient}
and the exact time parametrization give
\begin{equation*}
\int_0^{t(Y)}\|\nabla v(t')\|_{L^\infty(\Omega_{t'})}\,dt'\le Y+C.
\end{equation*}
Gr\"onwall applied to the material variational equation yields
\begin{equation}\label{eq:material-flow-derivative-bound}
\|D_\xi\mathcal X_t\|_{L^\infty(\Omega_0)}
\le Ce^Y\le C/a(t).
\end{equation}
Since $\Omega_0$ is a fixed smooth bounded connected domain, any two
points $\xi,\zeta\in\Omega_0$ can be joined inside $\Omega_0$ by a path
of length at most $C|\xi-\zeta|$.  Integrating the derivative bound
along such a path and using the uniform terminal error gives
\begin{equation*}
|\mathcal X_T(\xi)-\mathcal X_T(\zeta)|
\le Ca(t)+C a(t)^{-1}|\xi-\zeta|.
\end{equation*}
The estimate extends to initial positions on the boundary by continuity.  Choosing $a(t)=\sqrt\delta$
proves $\omega_T(\delta)\le C\sqrt\delta$.
\end{proof}

\begin{proposition}
\label{prop:exact-terminal-topology}
The terminal map has the exact fibers and induces the volume preserving
interior diffeomorphism, outer boundary diffeomorphism and quotient
specified in
\eqref{eq:main-terminal-exact-fibers}--\eqref{eq:main-terminal-quotient}.
Moreover, $\mathcal X_t\to\mathcal X_T$ in $C^\infty$ on compact
subsets of $\Omega_0$ and in smooth boundary charts near $\Gamma_+(0)$.
\end{proposition}
\begin{proof}
We divide the proof into four steps.

\emph{Step 1. Interior particles avoid the core at the terminal time.}
The rescaled meridional boundary is close to the unit circle in $C^2$
and has an embedded normal collar of uniformly positive width.  In three dimensions,
a point and its nearest boundary point have the same azimuth, since the cylindrical radii
are positive.  Together with $\rho\ge\rho_*>0$ and separation of the
outer boundary, this gives a smooth nearest point
projection in a fluid collar of width $ca(t)$, with $c>0$ fixed.

Suppose a particle initially at $\xi\in\Omega_0$ had terminal position
$x_*\in\gamma_*$.  Put $x(t)=\mathcal X_t(\xi)$ and
$r(t)=\operatorname{dist}(x(t),\gamma(t))$.  The sharp uniform
displacement estimate \eqref{eq:main-lagrangian-terminal-rate} and
$|R-R_*|=o(a)$ give
\begin{equation*}
r(t)\le |x(t)-x_*|+\operatorname{dist}(x_*,\gamma(t))
\le a(t)(1+o(1)).
\end{equation*}
Since the material diffeomorphism keeps $x(t)$ in the fluid for $t<T$,
$r(t)\ge a(t)(1-\|\eta\|_\infty)$ at late times.  Its distance from
$\Gamma_-(t)$ thus satisfies $0<d(t)=o(a(t))$, placing the trajectory
in the collar of width $ca(t)$.

Let $\pi(t)$ be its nearest point on $\Gamma_-(t)$, and let $\nu(t)$
be the unit normal there pointing out of the cavity and into the
fluid, so $x(t)=\pi(t)+d(t)\nu(t)$.  Differentiating the distance and
using the kinematic normal velocity gives the exact identity
\begin{equation}\label{eq:moving-inner-distance-identity} 
d'(t)=\bigl(v(t,x(t))-v(t,\pi(t))\bigr)\cdot\nu(t)=\int_0^{d(t)}\nu(t)^{\mathsf T}
\nabla v(t,\pi(t)+s\nu(t))\nu(t)\,ds. 
\end{equation}
Here the footpoint moves with the boundary normal velocity and a tangential
sliding component, and $\nu\cdot\nu'=0$.

Uniformly on this segment, $r/a=1+o(1)$, while the normal pointing out of the cavity
differs from the radial direction by $o(1)$.  The scaled
$C^1$ estimate \eqref{eq:scaled-inner-collar-velocity} and the exact
matrix \eqref{eq:radial-cartesian-gradient} therefore yield
\begin{equation*}
\nu^{\mathsf T}\nabla v(t,\pi+s\nu)\nu
\ge \frac{q(t)}{2a(t)^2},\qquad 0\le s\le d(t),
\end{equation*}
for all sufficiently late times.  The positive sign comes from the radial eigenvalue of
$\nabla(-q e_r/r)$, which is $q/r^2$.
Equation~\eqref{eq:moving-inner-distance-identity} implies
$d'(t)\ge q(t)d(t)/(2a(t)^2)>0$.  Thus $d(t)$ cannot tend to zero,
contradicting $d(t)=o(a(t))$.

Labels on the outer boundary end on $\Gamma_+^*$, which is disjoint from
$\gamma_*$.  Together with \eqref{eq:terminal-meridian-collapse}, this
proves the exact core preimage and its meridional fibers.

\emph{Step 2. Interior particles remain inside the limiting outer boundary.}
If an interior trajectory ended on $\Gamma_+^*$, it would eventually
lie in the fluid collar of fixed width supplied by smooth convergence
of the outer boundary.  Its positive distance $d_+(t)$ from $\Gamma_+(t)$ obeys,
by \eqref{eq:moving-inner-distance-identity} and
Lemma~\ref{lem:velocity-away-from-core},
\begin{equation*}
|d_+'(t)|\le Cq(t)d_+(t),\qquad
d_+(t)\ge d_+(t_0)
\exp\left(-C\int_{t_0}^tq(t')\,dt'\right)>c_{t_0}>0.
\end{equation*}
Integrability of $q$ therefore excludes this possibility.

\emph{Step 3. Smoothness and injectivity away from the core.}
Fix $x_*\in\overline{\Omega_*}\setminus\gamma_*$ and a neighborhood
$U$ separated from the core, chosen inside $\Omega_*$ when $x_*$
is interior.  Lemma~\ref{lem:velocity-away-from-core} and the uniformly
smooth outer charts give, for every integer $m\ge1$, an ambient extension
$\widehat v$ agreeing with $v$ on a smaller moving fluid neighborhood
and satisfying
\begin{equation*}
\|\widehat v(t)\|_{C^m(\mathbb R^3)}\le C_mq(t).
\end{equation*}
A cutoff suffices in interior neighborhoods.  Near the outer boundary,
we flatten the boundary, extend from the half space to the required
finite order, and use a partition of unity.  Since $q\in L^1(0,T)$, the
extended ODE starting at a sufficiently late $t_0$ has a terminal $C^m$
diffeomorphic flow, whose inverse is given by the backward ODE.

By uniform terminal convergence, every trajectory ending in a smaller
neighborhood of $x_*$ remains where the extension agrees with $v$ for
all sufficiently late times.  Uniqueness for the ODE then identifies
the physical and extended trajectories.  Applying this argument to a
finite covering of the image of $K\Subset\Omega_0$ and to the outer
boundary proves local $C^m$ convergence and smooth local
invertibility of the terminal map.  The extensions for different
$m$ give the same map, so the conclusion holds in $C^\infty$.
In the interior, incompressibility passes to the limit:
\begin{equation*}
\det D_\xi\mathcal X_T
=\lim_{t\uparrow T}\det D_\xi\mathcal X_t=1.
\end{equation*}
If two particles had the same terminal point away from the core, their
late trajectories would lie in one such extension neighborhood.
Injectivity of the terminal extended flow implies that their positions at
$t_0$ coincide. The physical diffeomorphism from time zero to $t_0$ then implies that the
initial positions coincide.  This proves injectivity both in the interior and on
the outer boundary.

\emph{Step 4. Surjectivity and the quotient.}
For any $x\in\overline{\Omega_*}$, choose
$x_j\in\overline{\Omega_{t_j}}$ with $t_j\uparrow T$ and $x_j\to x$.
Such points exist by convergence of the outer boundary and shrinkage
of the cavity.  Compactness gives a convergent subsequence of
$\xi_j=\mathcal X_{t_j}^{-1}(x_j)$. Uniform flow convergence then
gives $\mathcal X_T(\lim\xi_j)=x$.  Thus $\mathcal X_T$ is onto.
The initial inner boundary, interior and outer boundary map into the
disjoint sets $\gamma_*$, $\Omega_*\setminus\gamma_*$ and
$\Gamma_+^*$, respectively.  Each restriction is consequently onto
the corresponding set.  Smooth local invertibility, injectivity and
$\det D\mathcal X_T=1$ prove the diffeomorphism assertions in
\eqref{eq:main-terminal-exact-fibers}.
The nonsingleton fibers are precisely the inner meridians.  The induced
continuous bijection from the compact quotient to the Hausdorff
space $\overline{\Omega_*}$ is a homeomorphism, proving
\eqref{eq:main-terminal-quotient}.
\end{proof}

We next determine the normal form and critical regularity of the terminal
map, beginning with velocity estimates between the inner and macroscopic scales.

\begin{lemma}\label{lem:all-scale-material-velocity}
After decreasing a fixed tubular radius $r_0>0$, for every integer
$j\ge0$ there is $C_j$ such that, at all sufficiently late times,
\begin{equation}\label{eq:all-scale-material-velocity}
|D_x^jv(t,x)|\le C_j\frac{q(t)}{r^{j+1}},
\qquad r=\operatorname{dist}(x,\gamma(t))<r_0,
\quad x\in\overline{\Omega_t}.
\end{equation}
Write the meridional velocity components as $v_\rho,v_z$.
For the rescaled meridional variable $z=(z_1,z_2)\in\R^2$, define
\begin{equation*}
\mathcal V_Y(z)=\frac{a(Y)}{q(Y)}
(v_\rho,v_z)\bigl(t(Y),R(Y)+a(Y)z_1,a(Y)z_2\bigr).
\end{equation*}
Then, for every compact $K\Subset\{z\in\R^2:|z|>1\}$ and
integer $j\ge0$,
\begin{equation}\label{eq:fixed-annulus-material-profile}
\mathcal V_Y\longrightarrow-\frac z{|z|^2}
\quad\hbox{in }C^j(K).
\end{equation}
\end{lemma}
\begin{proof}
Fix a small $\delta>0$ within the collar of
\eqref{eq:scaled-inner-collar-velocity}.  That estimate proves
\eqref{eq:all-scale-material-velocity} for $r\le(1+\delta)a$.
For $(1+\delta/2)a\le r<r_0$, a meridional disk $B_{cr}$ about
the point lies in the fluid if $c>0$ is sufficiently small, depending
only on $\delta$.  Shrinking $r_0$ keeps these disks away from the
outer boundary.  The graph bound $\eta=o_{C^1}(1)$ makes these
choices uniform.

Since $q\asymp Y^{-1/2}$, the global radial comparison
\eqref{eq:global-radial-energy-error} has error at most $Cq$.
Axisymmetry and the uniform positive lower and upper bounds on $\rho$
convert this to the corresponding meridional $L^2$ estimate.  On
$B_{cr}$ the comparison field has size $O(q/r)$ and hence meridional
$L^2$ norm at most $Cq$.  Consequently
\begin{equation*}
\|(v_\rho,v_z)\|_{L^2(B_{cr})}\le Cq.
\end{equation*}
The exact component equations are
\begin{equation*}
(\partial_\rho^2+\rho^{-1}\partial_\rho+\partial_z^2
-\rho^{-2})v_\rho=0,
\qquad
(\partial_\rho^2+\rho^{-1}\partial_\rho+\partial_z^2)v_z=0.
\end{equation*}
Rescaling the disk by $r$ gives uniformly elliptic equations with
uniformly bounded smooth coefficients.  Interior estimates thus give
$|\partial_{\rho,z}^jv|\le C_jqr^{-j-1}$.  Cartesian derivatives
introduce only additional bounded factors of $\rho^{-1}$, proving
\eqref{eq:all-scale-material-velocity}.

These estimates make $\mathcal V_Y$ precompact in every fixed
$C^j$ norm on compact subsets of $|z|>1$.  Since the coefficients of
first and zeroth order in the rescaled equations vanish in the limit,
every subsequential limit has harmonic components in the connected
exterior of the unit disk.  The collar estimate
identifies this limit with $-z/|z|^2$ on an open annulus adjoining that disk.
Unique continuation for harmonic functions identifies it on the whole
exterior.  A diagonal extraction and uniqueness of the limit prove
\eqref{eq:fixed-annulus-material-profile} for the full trajectory.
\end{proof}

Parametrize the initial inner boundary positively around each
meridian by
\begin{equation*}
\zeta(\theta,\phi)
=(\rho_0(\theta)\cos\phi,\rho_0(\theta)\sin\phi,z_0(\theta)),
\qquad
j_0(\theta)=|(\rho_0',z_0')|>0.
\end{equation*}
Let $\nu_0$ be the unit normal pointing from the cavity into the fluid.  In a fixed
normal collar on the fluid side put
\begin{equation}\label{eq:initial-material-normal-coordinates}
\Xi(\theta,\phi,d)=\zeta(\theta,\phi)+d\nu_0(\theta,\phi),
\qquad 0\le d<d_0.
\end{equation}
Thus $d=\operatorname{dist}(\Xi,\Gamma_-(0))$.  We retain the notation
$s_1(A)\ge s_2(A)\ge s_3(A)$ for the Euclidean singular values
of a $3\times3$ matrix $A$.

\begin{proposition}\label{prop:critical-terminal-deformation}
There is an orientation-preserving $C^\infty$ circle diffeomorphism
$\vartheta_*$ such that, with
\begin{equation}\label{eq:material-square-root-coefficient}
\kappa_*(\theta)
=\frac{2\rho_0(\theta)j_0(\theta)}{R_*\vartheta_*'(\theta)}>0,
\end{equation}
the terminal map satisfies
\begin{equation}\label{eq:terminal-square-root-normal-form}
\mathcal X_T(\Xi(\theta,\phi,d))
=\gamma_*(\phi)
+\sqrt{\kappa_*(\theta)d}\,
e_r(\phi,\vartheta_*(\theta))+o(\sqrt d),
\end{equation}
uniformly in $\theta,\phi$ as $d\downarrow0$.

The estimates for the singular values
\eqref{eq:main-anisotropic-material-deformation} and the strong
Sobolev convergence and endpoint regularity
\eqref{eq:main-critical-material-regularity} hold.  In addition,
\begin{equation}\label{eq:terminal-critical-Sobolev}
\sup_{t<T}\|D\mathcal X_t\|_{L^{2,\infty}(\Omega_0)}<\infty,
\qquad
\|D\mathcal X_t\|_{L^2(\Omega_0)}^2\asymp1+Y(t)
\quad(t\uparrow T).
\end{equation}

The inverse $\mathcal X_T^{-1}:\Omega_*\setminus\gamma_*\to\Omega_0$,
extended as a bounded function across $\gamma_*$, belongs to
$W^{1,p}(\Omega_*)$ for every
$1\le p<2$, and
\begin{equation}\label{eq:inverse-critical-Sobolev}
D\mathcal X_T^{-1}\in L^{2,\infty}(\Omega_*)\setminus L^2(\Omega_*).
\end{equation}
The inverse has no continuous extension across the core circle.
The terminal map also satisfies the matching lower modulus bound in
\eqref{eq:main-optimal-lagrangian-modulus}, so its H\"older exponent
$1/2$ is optimal.
\end{proposition}

\begin{proof}
We divide the proof into five steps.

\emph{Step 1. Transport equations in the inner collar.}
Use $Y$ as time and retain the initial coordinates in
\eqref{eq:initial-material-normal-coordinates}.  Constants may depend
on a fixed sufficiently late starting time $Y_b$ and the earlier smooth
flow.

Write $b(Y,\vartheta)=a(1+\eta(Y,\vartheta))$.  Along a particle,
use polar coordinates $(r,\vartheta)$ centred at $(R(Y),0)$ in
its fixed meridional plane.  The exact time parametrization and
\eqref{eq:scaled-inner-collar-velocity}, together with
$R_Y=O(\mathcal P(Y)a^2)$, give
\begin{equation*}
r_Y=-\frac{\Theta a^2}{r}+\varepsilon_r(Y,r,\vartheta),
\qquad
\vartheta_Y=\varepsilon_\vartheta(Y,r,\vartheta),
\end{equation*}
where, on a fixed scaled collar and for every fixed $m$,
\begin{equation*}
\|\varepsilon_r\|_{C^m(r/a,\vartheta)}
\le C_m(1+Y)^{M_m}a^2,
\qquad
\|\varepsilon_\vartheta\|_{C^m(r/a,\vartheta)}
\le C_m(1+Y)^{M_m}a.
\end{equation*}
These remainders include the motion of the centre circle.  On the boundary
the kinematic condition reads
\begin{equation*}
b_Y+b_\vartheta\varepsilon_\vartheta(Y,b,\vartheta)
=-\frac{\Theta a^2}{b}+\varepsilon_r(Y,b,\vartheta).
\end{equation*}
Set $\mathscr I=r^2-b(Y,\vartheta)^2$.  Subtraction gives the
exact cancellation
\begin{equation}\label{eq:material-collar-invariant-identity}
\mathscr I_Y
=2\{r\varepsilon_r(Y,r,\vartheta)-b\varepsilon_r(Y,b,\vartheta)\}
-2bb_\vartheta\{\varepsilon_\vartheta(Y,r,\vartheta)
-\varepsilon_\vartheta(Y,b,\vartheta)\}.
\end{equation}
Choose $c_e>0$ sufficiently small that
$0\le\mathscr I\le c_ea^2$ lies inside the established collar.
Dividing \eqref{eq:material-collar-invariant-identity} by
$(r-b)(r+b)$ and using divided differences yields
\begin{equation}\label{eq:material-collar-invariant-system}
\mathscr I_Y=\mathscr I\mathcal B(Y,u,\vartheta),
\qquad
\vartheta_Y=\mathcal A(Y,u,\vartheta),
\qquad u=\mathscr I/a^2,
\end{equation}
where
\begin{equation}\label{eq:material-collar-integrable-coefficients}
\|\mathcal A(Y)\|_{C^m([0,c_e]\times\T)}
+\|\mathcal B(Y)\|_{C^m([0,c_e]\times\T)}
\le\epsilon_m(Y):=C_m(1+Y)^{M_m}a(Y),
\qquad \int_{Y_b}^\infty\epsilon_m(Y)\,dY<\infty.
\end{equation}
$b/a=1+O_{C^m}(\mathcal P a^2)$ and
$r/a=((b/a)^2+u)^{1/2}$ have uniformly bounded derivatives on
this fixed set.  The divided difference of $r\varepsilon_r$,
followed by division by $r+b$, is $O_{C^m}(\mathcal P a)$.
The term containing $b_\vartheta$ is smaller.  The same formula for divided
differences gives smoothness at $u=0$.

The earlier smooth flow maps a sufficiently small initial
normal collar into this collar at $Y_b$.  Its smoothness and the
positive normal Jacobian give
\begin{equation}\label{eq:initial-collar-invariant-factor}
\mathscr I(Y_b,\theta,d)=d\,h_b(\theta,d),
\qquad h_b>0,
\end{equation}
with $h_b$ smooth and $0<c\le h_b\le C$.
The earlier meridional flow has positive Jacobian $\rho_0/\rho$
and preserves the boundary orientation.  Since both inner meridians
are radial graphs with their positive angular parametrizations,
the induced angular map has degree one and positive derivative.
On $d=0$, the boundary angular map $\vartheta_b(Y,\theta)$ records the
polar angle of the particle with initial meridional coordinate $\theta$.
It satisfies
$\partial_Y\vartheta_b=\mathcal A(Y,0,\vartheta_b)$.
Equation~\eqref{eq:material-collar-integrable-coefficients} and
successive differentiation of this scalar ODE show convergence
to $\vartheta_*$ in every $C^m$ norm. Moreover
\begin{equation*}
\partial_\theta\vartheta_b(Y,\theta)=\partial_\theta\vartheta_b(Y_b,\theta)
\exp\!\left(\int_{Y_b}^Y
\partial_\vartheta\mathcal A(Y',0,\vartheta_b(Y',\theta))\,dY'\right).
\end{equation*}
The derivative satisfies $0<c\le\partial_\theta\vartheta_b\le C$.
The limit has degree one and strictly positive derivative, so
$\vartheta_*$ is a $C^\infty$ circle diffeomorphism.

\emph{Step 2. Exit from the collar and the terminal normal form.}
Increase $Y_b$ so that $|\mathcal B|\le1$.  Then
$u_Y=(2+\mathcal B)u>0$ for $u>0$.  Each sufficiently small
$d>0$ has a unique exit time $Y_e$ at $u=c_e$.  Before exit,
\eqref{eq:initial-collar-invariant-factor} and the integrable
bound for $\mathcal B$ imply
\begin{equation}\label{eq:material-exit-scale}
\mathscr I(Y,\theta,d)\asymp d,
\qquad a_e^2:=a(Y_e)^2\asymp d.
\end{equation}
In particular $Y_e\to\infty$ uniformly as $d\downarrow0$.
On each fixed $Y$ interval, the particle converges to the boundary
particle with the same initial meridional coordinate as $d\downarrow0$.  Applying dominated convergence to
the integral equations in
\eqref{eq:material-collar-invariant-system}, with the integral tails
bounded by \eqref{eq:material-collar-integrable-coefficients}, gives
\begin{equation}\label{eq:material-exit-invariant-limit}
\frac{\mathscr I(Y_e,\theta,d)}d\longrightarrow\kappa_*(\theta),
\qquad \vartheta(Y_e,\theta,d)\longrightarrow\vartheta_*(\theta),
\end{equation}
uniformly.  Here the first limit is also
$\lim_{Y\to\infty}\partial_d\mathscr I(Y,\theta,0)$:
both equal $h_b(\theta,0)$ times the exponential of the integral
of $\mathcal B(Y,0,\vartheta_b(Y,\theta))$.

To evaluate this limit, put $j_b=(b^2+b_\vartheta^2)^{1/2}$ and
$\rho_b(Y,\theta)=R(Y)+b(Y,\vartheta_b(Y,\theta))
\cos\vartheta_b(Y,\theta)$ along the boundary particle.  Its
surface Jacobian is
\begin{equation*}
J_s(Y,\theta)=\frac{\rho_bj_b\partial_\theta\vartheta_b(Y,\theta)}{\rho_0j_0}.
\end{equation*}
Let $\nu_Y$ be the unit normal pointing from the current cavity into
the fluid.  Since $\det D\mathcal X_t=1$, the normal stretch satisfies
$\nu_Y\cdot D\mathcal X_t\nu_0=J_s^{-1}$.  On the boundary,
$\nabla\mathscr I=2j_b\nu_Y$, and consequently
\begin{equation*}
\partial_d\mathscr I(Y,\theta,0)
=\frac{2\rho_0(\theta)j_0(\theta)}
{\rho_b\partial_\theta\vartheta_b(Y,\theta)}.
\end{equation*}
Taking the limit proves
\eqref{eq:material-square-root-coefficient}.

For each fixed $\Delta>0$, put $\widehat Y=Y-Y_e$ and normalize lengths by $a_e$.
Denote the limiting rescaled radius and polar angle by
$\widehat r$ and $\widehat\vartheta$, respectively.
Lemma~\ref{lem:all-scale-material-velocity},
$\Theta\to1$, and $R_Y=O(\mathcal P a^2)$ give convergence of
the particle ODE on $0\le Y-Y_e\le \Delta$ to
\begin{equation*}
\frac{d\widehat r}{d\widehat Y}=-\frac{e^{-2\widehat Y}}{\widehat r},
\qquad \frac{d\widehat\vartheta}{d\widehat Y}=0,
\qquad \widehat r(0)=\sqrt{1+c_e}.
\end{equation*}
For each fixed $\Delta$, the limiting trajectory stays a positive distance
from the shrinking boundary throughout this interval.  Continuous
dependence for the particle ODE on a fixed annulus therefore gives
\begin{equation*}
\frac{r(Y_e+\Delta)^2}{a_e^2}\longrightarrow c_e+e^{-2\Delta},
\qquad
\vartheta(Y_e+\Delta)-\vartheta(Y_e)\longrightarrow0.
\end{equation*}
The terminal displacement estimate gives
\begin{equation*}
\frac{|\mathcal X_T-\mathcal X_{t(Y_e+\Delta)}|}{a_e}
\le e^{-\Delta}(1+o(1)),
\qquad \frac{|R(Y_e+\Delta)-R_*|}{a_e}\longrightarrow0.
\end{equation*}
First letting $d\downarrow0$ with $\Delta$ fixed and then
$\Delta\to\infty$ proves that the terminal displacement from the
core, divided by $a_e$, tends to
$\sqrt{c_e}\,e_r(\phi,\vartheta_*(\theta))$.  Together with
\eqref{eq:material-exit-invariant-limit} and
$\mathscr I(Y_e)=c_ea_e^2$, this proves
\eqref{eq:terminal-square-root-normal-form}.

Since $|\Xi(\theta,\phi,\delta)-\Xi(\theta,\phi,0)|=\delta$,
the normal form and \eqref{eq:terminal-meridian-collapse} give
$\omega_T(\delta)\ge\sqrt{\kappa_*(\theta)\delta}+o(\sqrt\delta)
\ge c\sqrt\delta$.  With the upper bound in
Proposition~\ref{prop:terminal-lagrangian-map}, this proves
\eqref{eq:main-optimal-lagrangian-modulus} and optimality of the
H\"older exponent $1/2$.

\emph{Step 3. Uniform derivative bounds after exit.}
The normal form and \eqref{eq:material-exit-invariant-limit} imply
\begin{equation*}
\operatorname{dist}(\mathcal X_T(\Xi),\gamma_*)
=\sqrt{c_e}\,a_e(1+o(1))\ge c_{\rm L}a_e
\end{equation*}
for all sufficiently small $d$.  Choose a fixed $\Delta_{\rm e}>0$ large
enough that $e^{-\Delta_{\rm e}}<c_{\rm L}/4$.  The sharp global gradient bound
\eqref{eq:sharp-global-velocity-gradient} and the exact time parametrization give
\begin{equation*}
\int_{t(Y_e)}^{t(Y_e+\Delta_{\rm e})}\|\nabla v(t')\|_\infty\,dt'\le C.
\end{equation*}
For $Y\ge Y_e+\Delta_{\rm e}$, the terminal displacement bound and
$|R(Y)-R_*|=O(\mathcal P(Y)a(Y)^2)$ imply
$r(Y)\ge c_{\rm L}a_e/2$, after decreasing $d_0$.  Equation~\eqref{eq:terminal-flow-upper-error} also gives
$r(Y)\le Ca_e<r_0$, so these trajectories remain in the tube
of Lemma~\ref{lem:all-scale-material-velocity}.  Therefore
\eqref{eq:all-scale-material-velocity} gives
\begin{equation*}
\int_{t(Y_e+\Delta_{\rm e})}^T
|\nabla v(t',\mathcal X_{t'}(\Xi))|\,dt'
\le\frac C{a_e^2}\int_{Y_e+\Delta_{\rm e}}^\infty a(Y)^2\,dY\le C.
\end{equation*}
Let $\mathcal X_{t,t'}$ be the flow from time $t'$ to $t$, so
$\mathcal X_t=\mathcal X_{t,0}$.  Its variational equation and the
equation for the inverse matrix show, along every trajectory leaving the
collar, that
\begin{equation}\label{eq:post-exit-flow-derivative}
\|D\mathcal X_{t,t_e}\|+
\|(D\mathcal X_{t,t_e})^{-1}\|\le C,
\qquad t_e=t(Y_e),\quad t_e\le t\le T.
\end{equation}
For $t=T$, Proposition~\ref{prop:exact-terminal-topology} identifies
these derivatives with the limits as $t\uparrow T$ by smooth convergence
near each initial position in the fluid interior.

\emph{Step 4. Tangential compression and singular values.}
Before exit, differentiate with respect to the initial meridional coordinate $\theta$ at fixed $d$ and set
$J_1=\partial_\theta\log\mathscr I$ and
$J_2=\partial_\theta\vartheta$.  They are uniformly bounded at
$Y_b$ by \eqref{eq:initial-collar-invariant-factor}.  Their
equations are
\begin{equation*}
\partial_Y J_1=u\mathcal B_uJ_1+\mathcal B_\vartheta J_2,
\qquad
\partial_Y J_2=u\mathcal A_uJ_1+\mathcal A_\vartheta J_2.
\end{equation*}
All coefficients are bounded by an integrable multiple of
$\epsilon_1(Y)$.  Thus $|J_1|+|J_2|\le C$ up to exit.
Differentiating $r^2=b(Y,\vartheta)^2+\mathscr I$ gives
\begin{equation*}
\partial_\theta r
=\frac{bb_\vartheta J_2}{r}+\frac{\mathscr I J_1}{2r}.
\end{equation*}
Since $r\asymp a$, $\mathscr I\le c_ea^2$ and
$b_\vartheta=O(\mathcal P a^3)$, this proves
\begin{equation}\label{eq:pre-exit-material-tangent}
|\partial_\theta\mathcal X_{t(Y)}(\Xi)|\le Ca(Y)
\qquad(Y\le Y_e).
\end{equation}

The exit time depends smoothly on the initial coordinates because
$u_Y(Y_e)=(2+\mathcal B)c_e>0$.
Write $x_e(\theta,d)=\mathcal X_{t_e(\theta,d)}(\Xi)$.
At a fixed finite time $t\ge t_e$, the flow identity and
$\partial_{t'}\mathcal X_{t,t'}=-D\mathcal X_{t,t'}\,v(t',\cdot)$
give the exact formula
\begin{equation*} 
\partial_\theta\mathcal X_t(\Xi)
=D\mathcal X_{t,t_e}(x_e)
\bigl(\partial_\theta x_e-v(t_e,x_e)\partial_\theta t_e\bigr)=D\mathcal X_{t,t_e}(x_e)
\left.\partial_\theta\mathcal X_{t(Y)}(\Xi)\right|_{Y=Y_e}, 
\end{equation*}
where the last derivative is taken with $Y$ fixed.  Equations
\eqref{eq:post-exit-flow-derivative} and
\eqref{eq:pre-exit-material-tangent} therefore give
$|\partial_\theta\mathcal X_t(\Xi)|\le Ca_e$ after exit.
Proposition~\ref{prop:exact-terminal-topology} permits passage to $t=T$.

The global bound \eqref{eq:material-flow-derivative-bound} gives
$|D\mathcal X_t|\le C/a(t)$ before exit.  Apply it at $t_e$
and then use \eqref{eq:post-exit-flow-derivative} after exit.
By \eqref{eq:material-exit-scale}, these estimates combine to
\begin{equation*}
|D\mathcal X_t(\Xi)|\le C(a(t)^2+d)^{-1/2},
\qquad
|\partial_\theta\mathcal X_t(\Xi)|
\le C(a(t)^2+d)^{1/2}.
\end{equation*}
The initial tangential vector satisfies $|\partial_\theta\Xi|\asymp1$ and
is orthogonal to $\nu_0$.  Axisymmetry and zero swirl preserve
the azimuthal direction, with stretch
$\rho(\mathcal X_t(\Xi))/\rho(\Xi)\asymp1$.
The derivative therefore splits into the azimuthal scalar stretch and
a two-dimensional meridional matrix.  By incompressibility, the
determinant of the meridional matrix is comparable to one.  Its largest singular value is at most
$C(a^2+d)^{-1/2}$, while its smallest is at most
$C(a^2+d)^{1/2}$ by the tangential estimate.  The determinant
then supplies the matching lower bounds.  Decreasing $d_0$ and
restricting to late times proves
\eqref{eq:main-anisotropic-material-deformation}, including $t=T$.

\emph{Step 5. Sobolev regularity of the map and its inverse.}
The initial collar volume element is a smooth positive multiple
of $d\theta\,d\phi\,dd$.  Outside that collar, all material derivatives remain bounded as
$t\uparrow T$, by interior convergence and the flow constructed on a
fixed outer neighborhood in Proposition~\ref{prop:exact-terminal-topology}.  Hence
\eqref{eq:main-anisotropic-material-deformation} implies
\begin{equation*}
c\lambda^{-2}\le
|\{\xi\in\Omega_0:|D\mathcal X_T(\xi)|>\lambda\}|
\le C\lambda^{-2}
\end{equation*}
for all sufficiently large $\lambda$, and gives the uniform
weak $L^2$ bound for $t<T$.  Also
\begin{equation*}
\int_0^{d_0}\frac{dd}{a(t)^2+d}
=\log\left(1+\frac{d_0}{a(t)^2}\right)\asymp1+Y(t).
\end{equation*}
This proves the endpoint bounds in \eqref{eq:terminal-critical-Sobolev}.  For $p<2$, the bound
$C d^{-p/2}$ is integrable.  Interior smooth convergence,
dominated convergence and the uniform convergence of the maps
prove their strong $W^{1,p}$ convergence in
\eqref{eq:main-critical-material-regularity}.

Compactness and the description of the core preimage imply that the initial positions
of particles ending near $\gamma_*$ lie in the initial collar.
The normal form \eqref{eq:terminal-square-root-normal-form}, positivity of $\kappa_*$ and
invertibility of $\vartheta_*$ give
$d(\mathcal X_T^{-1}(x))\asymp r^2$, where
$r=\operatorname{dist}(x,\gamma_*)$.  For
$x_r=\gamma_*(\phi)+r e_r(\phi,\vartheta)$ they also give
\begin{equation*}
\pi_0(\mathcal X_T^{-1}(x_r))
\longrightarrow\zeta(\vartheta_*^{-1}(\vartheta),\phi),
\end{equation*}
where $\pi_0$ is the initial normal projection.
Taking reciprocals in \eqref{eq:main-anisotropic-material-deformation}
therefore yields
\begin{equation*}
s_1(D\mathcal X_T^{-1})\asymp r^{-1},\qquad
s_2(D\mathcal X_T^{-1})\asymp1,\qquad
s_3(D\mathcal X_T^{-1})\asymp r.
\end{equation*}
The volume element in tubular coordinates is comparable to
$r\,dr\,d\vartheta\,d\phi$, and the derivatives remain bounded away
from the core.  These bounds prove \eqref{eq:inverse-critical-Sobolev}
and integrability for $p<2$.
To identify the weak derivatives across the core, integrate by parts
against a compactly supported test function on
$\Omega_*\setminus\overline{N_\epsilon(\gamma_*)}$.
The new boundary term is $O(\epsilon)$, since $\mathcal X_T^{-1}$ is
bounded and $|\partial N_\epsilon(\gamma_*)|=O(\epsilon)$.
The volume terms converge by the $L^1$ derivative bound established above.
Letting $\epsilon\downarrow0$ proves the distributional identity
on $\Omega_*$.  The distinct angular limits of the normal projection
exclude continuity at every core point.
\end{proof}

\subsection{Pressure profiles and the cylindrical tangent}
\label{subsec:inner-pressure-tangent}

We next resolve the fields at two spatial scales.  The Dirichlet Green
potential gives the pressure profile on fixed compact subsets away from
the circle.  Magnification at the cavity scale gives the cylindrical
spacetime limit of Proposition~\ref{prop:universal-cylindrical-tangent}.
Figure~\ref{fig:cylindrical-zoom} illustrates the rescaling and its
limiting flow.

Let $D_t$ be the region enclosed by $\Gamma_+(t)$, with the cavity
filled in.  Then $D_t$ converges smoothly to $\Omega_*$ and contains a
fixed neighborhood of $\gamma_*$.  Recall the normalized flux
\begin{equation*}
q(t)=\frac1{2\pi L(t)}\int_{\Gamma_-(t)}v\cdot n\,dS,
\qquad L(t)=2\pi R(t),
\end{equation*}
and the capacity $\operatorname{Cap}(t)=\int_{\Omega_t}|\nabla u_c(t)|^2\,dx$,
where $u_c(t)$ is harmonic with traces $1$ and $0$ on the inner and outer
boundaries, respectively.  Set
\begin{equation}\label{eq:terminal-pressure-normalization}
\Lambda(t)=\frac{2\pi L(t)}{\operatorname{Cap}(t)},\qquad
S(t)=\frac{q(t)^2}{2a(t)^2\Lambda(t)},\qquad
U_*(x)=2\pi\int_{\gamma_*}G_{\Omega_*}(x,y)\,d\mathcal H^1(y),
\end{equation}
where the Dirichlet Green function is normalized by
$-\Delta_xG_{\Omega_*}(x,y)=\delta_y$ and has zero boundary trace.
In particular
\begin{equation}\label{eq:terminal-green-source}
-\Delta U_*=2\pi\mathcal H^1\!\restriction_{\gamma_*}
\quad\hbox{in }\Omega_*,\qquad
U_*|_{\Gamma_+^*}=0,\qquad
U_*(x)>0\quad(x\in\Omega_*\setminus\gamma_*).
\end{equation}
Near the core,
$U_*(x)=-\log\operatorname{dist}(x,\gamma_*)+O(1)$.

\begin{lemma}
\label{lem:terminal-condenser-profile}
Write $W_t=\Lambda(t)u_c(t)$ on $\Omega_t$.  For every
$K\Subset\Omega_*\setminus\gamma_*$ and nonnegative integer $k$,
\begin{equation}\label{eq:terminal-condenser-local}
W_t\longrightarrow U_*\quad\hbox{in }C^k(K).
\end{equation}
Each such compact set lies in $\Omega_t$ for all sufficiently late
times.  The condenser also satisfies the uniform bound
\begin{equation}\label{eq:terminal-condenser-L2-bound}
\|W_t\|_{L^2(\Omega_t)}\le C.
\end{equation}
\end{lemma}
\begin{proof}
Extend $W_t$ by its inner trace $\Lambda(t)$ throughout the cavity,
and denote this continuous function on $D_t$ by $\overline W_t$.
Since $n_-$ is the outward normal of the fluid, pointing into the
cavity, integration by parts gives
\begin{equation*}
-\Delta\overline W_t=\mu_t\quad\hbox{in }D_t,\qquad
d\mu_t=\partial_{n_-}W_t\,dS_{\Gamma_-(t)},\qquad
\mu_t\ge0,\qquad \mu_t(D_t)=2\pi L(t).
\end{equation*}
The condenser potential attains its maximum only on the inner boundary. The Hopf lemma gives
$\partial_{n_-}W_t>0$.  Green's identity gives
$\int_{\Gamma_-}\partial_{n_-}u_c\,dS
=\operatorname{Cap}$, proving the mass formula.  The outer trace
of $\overline W_t$ is zero.  Uniqueness of the Dirichlet problem
therefore gives the exact representation
\begin{equation}\label{eq:filled-green-representation}
\overline W_t(x)=\int G_{D_t}(x,y)\,d\mu_t(y).
\end{equation}

The measures $\mu_t$ have bounded mass and support in a fixed compact
set.  Every subsequential weak limit is rotation invariant, supported
on $\gamma_*$, and has mass $2\pi L_*$.  The unique measure with these
properties is uniform arclength measure with density $2\pi$, so
\begin{equation}\label{eq:terminal-condenser-source-limit}
\mu_t\rightharpoonup
2\pi\mathcal H^1\!\restriction_{\gamma_*}.
\end{equation}

For poles $y$ in a fixed neighborhood of $\gamma_*$, write
\begin{equation*}
G_{D_t}(x,y)=\frac1{4\pi|x-y|}-h_t(x,y),
\end{equation*}
where $h_t(\cdot,y)$ is harmonic with the Newton kernel as boundary data.
The poles remain uniformly separated from the outer boundary.  Pulling
back by smooth diffeomorphisms $\Omega_*\to D_t$ that converge to the
identity and fix any prescribed compact interior set gives smoothly
convergent elliptic coefficients and boundary data, uniformly in $y$.
Elliptic estimates and uniqueness therefore give convergence of $h_t$
on compact sets for every fixed order of differentiation in $x$, together
with uniform bounds for $h_t,\nabla_xh_t$ up to the outer boundary.
For $K\Subset\Omega_*\setminus\gamma_*$ choose the pole neighborhood
disjoint from $K$.  The resulting uniform kernel convergence and
\eqref{eq:terminal-condenser-source-limit}, inserted in
\eqref{eq:filled-green-representation}, prove
\eqref{eq:terminal-condenser-local} and
\eqref{eq:terminal-green-source}.  Integrating the Newton kernel along
the curve gives $-\log r+O(1)$, with coefficient one under the chosen
$2\pi$ normalization.  The regular part of the Green function is bounded.

For the $L^2$ bound, choose a fixed ball $B$ containing all $D_t$.
Positivity of the Dirichlet Green function and the maximum principle give
$0\le G_{D_t}(x,y)\le(4\pi|x-y|)^{-1}$.  Hence
$\sup_{y\in D_t}\|G_{D_t}(\cdot,y)\|_{L^2(D_t)}\le C_B$.
Minkowski's inequality in \eqref{eq:filled-green-representation} and
the bounded mass of $\mu_t$ prove \eqref{eq:terminal-condenser-L2-bound}.
\end{proof}

\begin{proposition}
\label{prop:terminal-pressure-profile}
For every compact $K\Subset\Omega_*\setminus\gamma_*$ and fixed
integer $k\ge0$, the local Green limits are
\begin{equation}\label{eq:main-local-green-limits}
\frac{v(t)}{q(t)}\longrightarrow\nabla U_*,
\qquad \frac{p(t)}{S(t)}\longrightarrow U_*
\quad\text{in }C^k(K).
\end{equation}
The pressure also satisfies \eqref{eq:main-pressure-time-profile}.
In particular, $p(t,x)\to+\infty$ at every fixed
$x\in\Omega_*\setminus\gamma_*$.  Moreover,
\begin{equation}\label{eq:terminal-pressure-L2-bound}
\|p(t)\|_{L^2(\Omega_t)}\le CS(t).
\end{equation}
\end{proposition}
\begin{proof}
The exact decomposition is
$v=q\nabla W_t+\nabla\Phi_0$.  By
\eqref{eq:zero-flux-dirichlet-bound} and the terminal decay estimates,
$\|\nabla\Phi_0\|_2=o(q)$.  The Cartesian components of the zero-flux
velocity are harmonic, so local interior estimates on slightly larger
fixed compact sets give $q^{-1}\nabla\Phi_0\to0$ in every $C^k(K)$.
Lemma~\ref{lem:terminal-condenser-profile} proves the velocity limit.

For the pressure, use the harmonic decomposition from
Proposition~\ref{prop:actual-taylor-coefficient}:
\begin{equation}\label{eq:pressure-profile-exact-decomposition}
p+\tfrac12|v|^2
=D u_c+\mathsf R,\qquad
D=\frac{q^2}{2a^2},\qquad \Delta\mathsf R=0.
\end{equation}
The trace bounds for the remainder and the maximum principle imply,
for a fixed polynomial $\mathcal P$,
\begin{equation*}
\|\mathsf R\|_{L^\infty(\Omega_t)}
\le C D\mathcal P(Y)a+C(1+q^2).
\end{equation*}
Since $S=D/\Lambda$, $\Lambda\asymp Y$ and $q^2\asymp Y^{-1}$,
\begin{equation}\label{eq:pressure-profile-remainder-small}
\left\|\frac{\mathsf R}{S}\right\|_\infty
\le C\Lambda\mathcal P(Y)a
+C(1+q^2)\frac{\Lambda}{D}\longrightarrow0.
\end{equation}
Harmonic interior estimates give convergence of this normalized
remainder in every $C^k(K)$.  Dividing
\eqref{eq:pressure-profile-exact-decomposition} by $S$ gives
\begin{equation*}
\frac{p}{S}
=W_t+\frac{\mathsf R}{S}-\frac{|v|^2}{2S}.
\end{equation*}
The local velocity estimates imply
$\||v|^2/S\|_{C^k(K)}\le C_{K,k}a^2\Lambda\to0$.
This proves the pressure limit in
\eqref{eq:main-local-green-limits}.

For the global $L^2$ estimate, the energy bound and
Lemma~\ref{lem:sharp-global-velocity-gradient} imply
\begin{equation*}
\left\|\frac{|v|^2}{S}\right\|_2
\le S^{-1}\|v\|_\infty\|v\|_2
\le C\frac{a\Lambda}{q}\longrightarrow0.
\end{equation*}
Together with \eqref{eq:pressure-profile-remainder-small}, the uniform
volume bound, and \eqref{eq:terminal-condenser-L2-bound}, this gives
\eqref{eq:terminal-pressure-L2-bound}.

Equations~\eqref{eq:final-collapse-law} and
\eqref{eq:concentration-input-tails} give
\begin{equation*}
T-t=\frac{a^2}{2q}(1+o(1)),\qquad
|\log(T-t)|=2Y+O(\log Y),\qquad
q\sqrt Y\longrightarrow
\sqrt{\frac{E_c}{\pi L_*}}.
\end{equation*}
Substitution yields
\begin{equation*}
(T-t)|\log(T-t)|^{3/2}S(t)
\longrightarrow\sqrt{\frac{E_c}{2\pi L_*}},
\qquad S(t)\longrightarrow\infty.
\end{equation*}
This proves \eqref{eq:main-pressure-time-profile} and the
pointwise pressure divergence.
\end{proof}

The outer pressure trace
$p|_{\Gamma_+}=\sigma\mathcal H_+$ remains bounded.
We now specify the rescaling in
Proposition~\ref{prop:universal-cylindrical-tangent} and prove its
cylindrical spacetime limit.

Fix $x_*\in\gamma_*$ and an orthogonal map $Q$ sending the third
coordinate vector $e_3=(0,0,1)$ to a
unit tangent of $\gamma_*$ at $x_*$.  Let $t_j\uparrow T$, and put
$a_j=a(t_j)$ and $\tau_j=T-t_j$.  Use the coordinates
\begin{equation*}
t=t_j+\tau\tau_j,\qquad x=x_*+a_jQz,
\qquad z=(z_1,z_2,z_3)\in\R^3.
\end{equation*}
Thus $\tau=0$ corresponds to $t_j$, $\tau=1$ to collapse, and the
$z_3$ axis is the tangent line to $\gamma_*$ at $x_*$.
The rescaled fluid domain and fields are
\begin{equation*}
\begin{split}
\Omega_j(\tau)&=\{z:x_*+a_jQz\in\Omega_{t_j+\tau\tau_j}\},\\
u_j(\tau,z)&=\frac{\tau_j}{a_j}Q^{\mathsf T}
v(t_j+\tau\tau_j,x_*+a_jQz),\\
\pi_j(\tau,z)&=\frac{\tau_j^2}{a_j^2}
p(t_j+\tau\tau_j,x_*+a_jQz).
\end{split}
\end{equation*}
Write $z_\perp=(z_1,z_2,0)$ and $r=|z_\perp|$, the distance from
the $z_3$ axis.  For $\tau<1$, the limiting cavity is the cylinder
$r<\sqrt{1-\tau}$, and the limiting fluid occupies its exterior.
Set
\begin{equation*}
\mathcal D_\infty
=\{(\tau,z):\tau<1,\ r>\sqrt{1-\tau}\}.
\end{equation*}
In these variables, the conclusions of the proposition read as follows.
\begin{enumerate}[label=\textup{(\roman*)},leftmargin=*,itemsep=4pt]
\item The rescaled inner boundaries converge locally to
$r=\sqrt{1-\tau}$, uniformly on compact time intervals in $(-\infty,1)$
and in every fixed spatial graph norm.  Every compact subset of
$\mathcal D_\infty$ is contained in
$\{(\tau,z):z\in\Omega_j(\tau)\}$ for all sufficiently large $j$.

\item On each compact subset of $\mathcal D_\infty$, the fields and
all spatial derivatives of any fixed order converge uniformly to
the profiles in \eqref{eq:universal-cylindrical-tangent}.
\end{enumerate}
The same limiting profiles are obtained at every $x_*$ and along
every sequence $t_j\uparrow T$.  They solve the incompressible Euler
equations with inward radial velocity, zero velocity along the cylinder
axis, and zero boundary pressure.  The cylinder boundary is transported
by this velocity, since
\begin{equation*}
u_\infty\cdot\frac{z_\perp}{r}
\bigg|_{r=\sqrt{1-\tau}}
=-\frac1{2\sqrt{1-\tau}}
=\frac{d}{d\tau}\sqrt{1-\tau}.
\end{equation*}

\begin{proof}[Proof of Proposition~\ref{prop:universal-cylindrical-tangent}]
Under \eqref{eq:concentration-input-tails} and
\eqref{eq:scaled-inner-collar-velocity}, put $D=q^2/(2a^2)$.
For every fixed $1<s_-<s_+<\infty$, on
$\mathbb T_\phi\times[s_-,s_+]\times\mathbb T_\theta$
the tube coordinates satisfy
\begin{equation}\label{eq:inner-pressure-leading-profile}
\frac aqv\bigl(t,X_t(\phi,as,\theta)\bigr)
\longrightarrow-\frac{e_r}{s},\qquad
\frac pD\bigl(t,X_t(\phi,as,\theta)\bigr)
\longrightarrow1-\frac1{s^2}
\end{equation}
in every fixed spatial $C^k$ norm.
Equation~\eqref{eq:fixed-annulus-material-profile} gives the velocity limit.
Axisymmetry gives smooth dependence of the frame on the azimuthal variable.
For the pressure, it suffices to estimate $u_c$ on the same scaled
annuli.  The bounds on the normal derivative of the condenser potential
in the proof of Proposition~\ref{prop:actual-taylor-coefficient} give
\begin{equation*}
|\nabla u_c|\le C/(a\Lambda)\quad\hbox{on }\Gamma_-\cup\Gamma_+,
\end{equation*}
since its traces are constant.  The maximum principle for Cartesian
directional derivatives extends this bound to the fluid.  Integrating
along radial segments from $r=a(1+\eta(\theta))$ to $r=as$ therefore gives
$|u_c-1|\le C/\Lambda$ on each fixed scaled annulus.
Interior elliptic estimates for the rescaled meridional operator on a
slightly larger annulus give $C^k$ convergence to one.

Since $D=\Lambda S$ and $\Lambda\to\infty$,
\eqref{eq:pressure-profile-remainder-small} gives
$\|\mathsf R/D\|_\infty\to0$.  Scaled interior elliptic estimates
for the harmonic remainder give this convergence in every fixed
$C^k$ norm on the annulus.
Finally,
\begin{equation*}
\frac pD=u_c+\frac{\mathsf R}{D}
-\left|\frac aqv\right|^2,
\end{equation*}
which proves the pressure limit.

Put $q_*=(E_c/(\pi L_*))^{1/2}$.  The energy asymptotic
\eqref{eq:concentration-input-tails} and \eqref{eq:final-collapse-law} give
\begin{equation*}
q\sqrt Y\longrightarrow q_*,\qquad
T-t=\frac{a^2}{2q}(1+o(1)).
\end{equation*}
For $\tau$ in a compact subset of $(-\infty,1)$, set
$t'=t_j+\tau\tau_j$.  The ratio
$(T-t')/(T-t_j)=1-\tau$ stays in a fixed compact subset of
$(0,\infty)$.  Taking logarithms of
$T-t=(\ell^2/(2q_*))e^{-2Y}\sqrt Y(1+o(1))$
first gives a uniform bound on $Y(t')-Y(t_j)$ and then shows that this
difference equals $-\tfrac12\log(1-\tau)+o(1)$.  It follows that
\begin{equation}\label{eq:cylindrical-tangent-scale-ratios}
\frac{a(t')}{a_j}\longrightarrow\sqrt{1-\tau},\qquad
\frac{q(t')}{q(t_j)}\longrightarrow1,\qquad
\frac{\tau_jq(t_j)}{a_j^2}\longrightarrow\frac12,
\end{equation}
uniformly on these time intervals.

The bound $|R-R_*|\le\mathcal P(Y)a^2$ and the graph estimates at every fixed regularity
show that, in the coordinates $x=x_*+a_jQz$, the
instantaneous core approaches its tangent line and the inner graph
approaches the stated cylinder.  This holds with every fixed
spatial derivative, uniformly on compact time intervals.  The outer boundary leaves every bounded set in these coordinates.
Write $r_{t'}(x)=\operatorname{dist}(x,\gamma(t'))$ for the physical
tube radius and $e_{r,t'}(x)=\nabla_x r_{t'}(x)$ for its radial unit vector.
On a compact subset of $\mathcal D_\infty$, the normalized radius
\begin{equation*}
s_j(\tau,z)=\frac{r_{t'}(x_*+a_jQz)}{a(t')}
\end{equation*}
therefore stays in a fixed interval $1<s_-\le s_j\le s_+<\infty$ and satisfies
\begin{equation*}
s_j\longrightarrow\frac{|z_\perp|}{\sqrt{1-\tau}},\qquad
Q^{\mathsf T}e_{r,t'}(x_*+a_jQz)
\longrightarrow\frac{z_\perp}{|z_\perp|},
\end{equation*}
with all fixed spatial derivatives.  Combining
\eqref{eq:inner-pressure-leading-profile} and
\eqref{eq:cylindrical-tangent-scale-ratios} proves
\eqref{eq:universal-cylindrical-tangent}.  The coefficient $1/8$ in the pressure profile follows from
$(\tau_jq(t_j)/a_j^2)^2\to1/4$ and the factor $1/2$ in $D$.
The convergence is in spatial $C^k$ norms, uniformly on compact
$\tau$ intervals.

To verify the limiting equations, write $r=|z_\perp|$ and
$e_r=z_\perp/r$ in the rescaled coordinates.  The velocity
$-e_r/(2r)$ is divergence free and irrotational, and satisfies
\begin{equation*}
\partial_\tau u_\infty=0,\qquad
(u_\infty\cdot\nabla)u_\infty=-\frac{e_r}{4r^3}
=-\nabla\pi_\infty.
\end{equation*}
At $r=\sqrt{1-\tau}$ its radial velocity equals
$\partial_\tau\sqrt{1-\tau}$ and its pressure is zero.
This verifies the limiting equations and their free boundary conditions.
\end{proof}

\subsection{Weak continuation and completion of the singularity theorem}
\label{subsec:weak-continuation}

Figure~\ref{fig:filling-energy} shows the filling continuations and
their energy jump.  We formulate the mass, momentum and energy balances,
establish a concatenation criterion, and apply it to the canonical
trajectory with positive surface tension.

We use $\chi_t$, $A_t$, $n_t$ and the energy and flux measures introduced
in \eqref{eq:weak-capillary-energy-flux}.  For almost every time $t$, write
$v(t)=v(t,\cdot)$ and $p(t)=p(t,\cdot)$, and set
\begin{equation*}
m_t=\chi_tv(t),\qquad
P_t=I-n_t\otimes n_t,
\end{equation*}
where $P_t$ is the tangential projection.  The distributional mass
and momentum equations on $\mathcal I\times\R^3$ are
\begin{equation}\label{eq:weak-capillary-equations} 
\partial_t\chi+\operatorname{div}_x m=0,\qquad
\partial_t m+\operatorname{div}_x(\chi v\otimes v)
+\nabla_x(\chi p)=\sigma\operatorname{div}_x(P_tA_t\,dt). 
\end{equation}
Here $\mathcal I$ is an open time interval, and the surface measure on the
right is integrated in time against $dt$.  Thus, for
$\varphi\in C_c^\infty(\mathcal I\times\R^3)$ and
$\psi\in C_c^\infty(\mathcal I\times\R^3,\R^3)$, these equations mean
\begin{equation}\label{eq:weak-capillary-tests}
\begin{aligned}
&\int_{\mathcal I}\!\int_{\R^3}
(\chi\partial_t\varphi+m\cdot\nabla\varphi)\,dx\,dt=0,\\
&\int_{\mathcal I}\!\int_{\R^3}
\bigl(m\cdot\partial_t\psi+
\chi(v\otimes v):\nabla\psi+
\chi p\operatorname{div}\psi\bigr)\,dx\,dt
-\sigma\int_{\mathcal I}\!\int_{\partial\Omega_t}
P_t:\nabla\psi\,dA_t\,dt=0.
\end{aligned}
\end{equation}
Incompressibility holds in the liquid for almost every time, and the
first equation describes transport of the interface between liquid and vacuum.

The local balance of total energy with defect $\mathfrak D$ is
\begin{equation}\label{eq:weak-capillary-energy-balance}
\partial_t\mathcal E+\operatorname{div}_x\mathcal F=-\mathfrak D,
\end{equation}
or, equivalently,
\begin{equation}\label{eq:weak-capillary-energy-test}
\int_{\mathcal I}\langle\mathcal E_t,\partial_t\varphi(t)\rangle\,dt
+\int_{\mathcal I}\langle\mathcal F_t,\nabla\varphi(t)\rangle\,dt
=\int_{\mathcal I\times\R^3}\varphi\,d\mathfrak D.
\end{equation}
If $\mathfrak D$ is a nonnegative Radon measure, this balance gives the local inequality for total energy.
For tests whose support reaches an initial endpoint, the initial mass,
momentum and energy pairings are added to the left sides
of \eqref{eq:weak-capillary-tests} and
\eqref{eq:weak-capillary-energy-test}.

On a closed smooth free surface,
$\operatorname{div}_x(P_tA_t)=-\mathcal H n_tA_t$.
The equations in the liquid, the kinematic boundary condition, and
$p=\sigma\mathcal H$ therefore give
\eqref{eq:weak-capillary-equations}.  Set $v_\tau=P_tv$.  The bulk energy identity and surface transport
formula give
\begin{equation*}
\partial_t(\chi|v|^2/2)+\operatorname{div}_x[\chi(|v|^2/2+p)v]
=-\sigma\mathcal H(v\cdot n_t) A_t,
\qquad
\partial_t A_t+\operatorname{div}_x(vA_t)
=(\operatorname{div}_{\partial\Omega_t}v)A_t.
\end{equation*}
These identities are distributional in space and time, with each
spatial surface measure integrated against $dt$.  Since
\begin{equation*}
\operatorname{div}_{\partial\Omega_t}v
=\operatorname{div}_{\partial\Omega_t}v_\tau+\mathcal H(v\cdot n_t),
\qquad
\operatorname{div}_x(v_\tau A_t)
=(\operatorname{div}_{\partial\Omega_t}v_\tau)A_t,
\end{equation*}
the surface identity becomes
\begin{equation*}
\partial_t(\sigma A_t)+\operatorname{div}_x(\sigma(v\cdot n_t)n_tA_t)
=\sigma\mathcal H(v\cdot n_t) A_t.
\end{equation*}
Adding the bulk and surface identities proves
\eqref{eq:weak-capillary-energy-balance} with $\mathfrak D=0$ on every classical
time interval.  The surface identity accounts for the normal surface
flux in \eqref{eq:weak-capillary-energy-flux}.

\begin{proposition} 
\label{prop:weak-capillary-continuation} 
\begin{enumerate}[label=(\roman*),leftmargin=*]
\item Let two classical capillary Euler branches be defined on opposite
sides of $T$.  Assume uniform kinetic energy and area bounds near $T$,
a common bounded set containing their domains there, and integrability
of the pressure and the full energy flux up to $T$ on both sides:
for some $\delta>0$,
\begin{equation*}
\int_{(T-\delta,T)\cup(T,T+\delta)}
\bigl(\|\chi_t p(t)\|_{L^1(\R^3)}
+|\mathcal F_t|(\R^3)\bigr)\,dt<\infty.
\end{equation*}
Here the flux is defined by \eqref{eq:weak-capillary-energy-flux}.
If the two branches have the same $L^1$ trace of their characteristic
functions at $T$ and the same distributional trace of their momenta, their
concatenation satisfies \eqref{eq:weak-capillary-equations} through $T$.
If their energy measures have weak-* limits
$\mathcal E_{T-}$ and $\mathcal E_{T+}$, then
\begin{equation}\label{eq:continuation-energy-jump}
\partial_t\mathcal E+\operatorname{div}_x\mathcal F
=\delta_T\otimes(\mathcal E_{T+}-\mathcal E_{T-}).
\end{equation}
\item Fix $\sigma>0$.  There is $\delta>0$ for which the canonical trajectory
constructed using Theorem~\ref{thm:one-state-closed} and
Proposition~\ref{prop:compactness-without-relative-energy} extends to
$[0,T+\delta)$ in the distributional formulation
\eqref{eq:weak-capillary-equations}.  The extension satisfies
\begin{equation}\label{eq:weak-capillary-state-continuity}
\chi\in C([0,T+\delta),L^1(\R^3)),\qquad
m\in C_{\mathrm w}([0,T+\delta),L^2(\R^3)),\qquad
(\chi_T,m_T)=(\mathbf1_{\Omega_*},0).
\end{equation}
Here $C_{\mathrm w}$ denotes continuity in the weak $L^2$ topology.
For $t\ge T$, the extension is the classical capillary Euler solution
starting from the filled domain $\Omega_*$ with zero initial velocity.  Its energy and
flux measures in \eqref{eq:weak-capillary-energy-flux} have locally
finite total variation in spacetime, and its exact local defect in total energy is
\begin{equation}\label{eq:weak-capillary-exact-defect}
\mathfrak D=\delta_T\otimes\mu_*,\qquad
\mu_*:=\frac{E_c}{L_*}
\mathcal H^1\!\restriction_{\gamma_*}.
\end{equation}
Thus the total energy is $E_c+\sigma|\Gamma_+^*|$ for $t<T$
and $\sigma|\Gamma_+^*|$ for $t\ge T$.
\end{enumerate}
\end{proposition}
\begin{proof}
We divide the proof into three steps.

\emph{Step 1. Concatenation across the singular time.}
Apply the classical mass, momentum and energy identities on the
intervals ending at $T-h$ and starting at $T+h$, sum them against
smooth spacetime tests, and let $h\downarrow0$.
The uniform bounds on mass, the $L^1$ norm of momentum, and energy
allow us to replace the endpoint tests by their values at $T$.  The common mass and
momentum traces cancel.  The omitted stress integrals vanish by the
kinetic energy and area bounds and pressure integrability.  Since the
stresses enter through spatial derivatives, integrability in time suffices.
This proves \eqref{eq:weak-capillary-equations}.
For energy, the contribution from the omitted interval vanishes by
integrability of the full flux, leaving
\begin{equation*}
\int\langle\mathcal E_t,\partial_t\varphi(t)\rangle\,dt
+\int\langle\mathcal F_t,\nabla_x\varphi(t)\rangle\,dt
=\langle\mathcal E_{T-}-\mathcal E_{T+},\varphi(T)\rangle,
\end{equation*}
which is \eqref{eq:continuation-energy-jump}.

\emph{Step 2. Integrability up to the singular time.}
Smooth convergence of the outer boundary places the fluid
domains in a fixed bounded set for all sufficiently late times.  The inner graph estimates in
\eqref{eq:concentration-input-tails} give
$|\Gamma_-(t)|=O(a(t))$ and cavity volume $O(a(t)^2)$.  In particular,
the areas $A_t(\R^3)$ and fluid volumes are uniformly bounded, and
\begin{equation}\label{eq:weak-capillary-geometric-traces}
\chi_t\longrightarrow\mathbf1_{\Omega_*}\quad\hbox{in }L^1(\R^3),
\qquad A_t\overset{*}{\rightharpoonup}
A_*:=\mathcal H^2\!\restriction_{\Gamma_+^*}
\quad(t\uparrow T).
\end{equation}
Proposition~\ref{prop:energy-concentration-limits} bounds
$\|v(t)\|_{L^2(\Omega_t)}$ uniformly.  Lemma~\ref{lem:sharp-global-velocity-gradient}
and \eqref{eq:terminal-pressure-L2-bound} give
\begin{equation*}
\|v(t)\|_{L^\infty(\Omega_t)}\le Cq/a,
\qquad \|p(t)\|_{L^2(\Omega_t)}\le CS.
\end{equation*}
By \eqref{eq:concentration-input-tails}, the exact time parametrization
$dt=\Theta a^2q^{-1}\,dY$, and
\eqref{eq:terminal-pressure-normalization}, for all sufficiently late
$t_1<T$,
\begin{equation}\label{eq:weak-capillary-integrable-scales}
\int_{t_1}^TS(t)\,dt
=\frac12\int_{Y(t_1)}^\infty
\frac{\Theta(Y)q(Y)}{\Lambda(Y)}\,dY
\le C\int_{Y(t_1)}^\infty Y^{-3/2}\,dY<\infty.
\end{equation}
The proof of Proposition~\ref{prop:terminal-lagrangian-map} gives
$\int_{t_1}^T\|v(t)\|_\infty\,dt<\infty$.
Consequently, the mass flux,
convective stress, pressure extended by zero and surface stress in
\eqref{eq:weak-capillary-equations} are
integrable in time as finite spatial measures.  For the energy flux, the
uniform kinetic energy and area bounds give
\begin{equation}\label{eq:weak-capillary-flux-integrability}
\begin{aligned}
\int_{t_1}^T\!\int_{\Omega_t}|v|^3\,dx\,dt
&\le C\int_{t_1}^T\|v(t)\|_{L^\infty}\,dt<\infty,\\
\int_{t_1}^T\!\int_{\Omega_t}|pv|\,dx\,dt
&\le\int_{t_1}^T\|p(t)\|_{L^2}\|v(t)\|_{L^2}\,dt
\le C\int_{t_1}^TS(t)\,dt<\infty,\\
\int_{t_1}^T\!\int_{\partial\Omega_t}|v\cdot n_t|\,dA_t\,dt
&\le C\int_{t_1}^T\|v(t)\|_{L^\infty}\,dt<\infty.
\end{aligned}
\end{equation}
Classical regularity gives the corresponding bounds on earlier compact
time intervals.

\emph{Step 3. Filling the cavity and computing the energy defect.}
Applying the local theory for positive surface tension in
Subsection~\ref{subsec:canonical-terminal-data} to the smooth bounded
filled domain $\Omega_*=N_\ell(\gamma_*)$ at rest gives a classical
solution on $[T,T+\delta)$ with sufficient Sobolev regularity for the
mass, momentum and energy identities
\eqref{eq:weak-capillary-equations} and
\eqref{eq:weak-capillary-energy-balance}.
The traces from the right are
\begin{equation*}
\chi_{T+}=\mathbf1_{\Omega_*},\qquad
m_{T+}=0,\qquad A_{T+}=A_*,\qquad
\mathcal E_{T+}=\sigma A_*.
\end{equation*}
The outgoing momentum converges strongly in $L^2$.  On the incoming side,
Proposition~\ref{prop:energy-concentration-limits} and
\eqref{eq:weak-capillary-geometric-traces} give
\begin{equation*}
\chi_{T-}=\mathbf1_{\Omega_*},\qquad
m_{T-}=0,\qquad A_{T-}=A_*,\qquad
\mathcal E_{T-}=\mu_*+\sigma A_*.
\end{equation*}
The incoming momentum convergence is strong in $L^1$ and weak in
$L^2$.  The energy and surface traces are weak-* limits of finite Radon
measures.

The bounds \eqref{eq:weak-capillary-integrable-scales}--\eqref{eq:weak-capillary-flux-integrability}
and outgoing classical regularity verify the hypotheses of part (i).
The common mass and momentum traces give the continuation and
\eqref{eq:weak-capillary-state-continuity}.  The energy traces give
\begin{equation*}
\partial_t\mathcal E+\operatorname{div}_x\mathcal F
=-\delta_T\otimes\mu_*.
\end{equation*}
Thus \eqref{eq:weak-capillary-exact-defect} holds with $\mu_*\ge0$.
Both surface traces equal $A_*$.  Classical energy conservation on each
side then gives the stated total energies and the loss $E_c$.
\end{proof}

\begin{remark}
The outgoing capillary solution in part (ii) has initial acceleration $-\nabla p$, where $p$
is harmonic in $\Omega_*$ with the nonconstant boundary trace
$\sigma\mathcal H_{\Gamma_+^*}$.
\end{remark}

\begin{proof}[Proof of Theorem~\ref{thm:toroidal-curvature-singularity}]
Apply Theorem~\ref{thm:one-state-closed} to the canonical terminal data
with energy gap $E_c$ and $Y_N\to\infty$.
Proposition~\ref{prop:compactness-without-relative-energy} gives a smooth
Euler trajectory, embeddedness and separation on compact time intervals,
and the collapse and curvature asymptotics.  Normalize $t(Y_0)=0$.
Inverting \eqref{eq:final-collapse-law}, with
$|\log(T-t)|\sim2Y$, gives the stated factor $2^{3/4}$.

Proposition~\ref{prop:energy-concentration-limits} gives the energy and velocity
limits and the obstruction to strong $L^2$ continuation.
Lemma~\ref{lem:sharp-global-velocity-gradient} gives the normal velocity
asymptotic and the gradient rate.
Proposition~\ref{prop:terminal-pressure-profile} gives the Green and pressure profiles.
Proposition~\ref{prop:weak-capillary-continuation} gives the weak
continuation and its exact energy defect.

Suppose a strong continuation has finitely many boundary charts
$X_i(t)$ continuous in $H^4$ on $[T-\tau,T+\tau]$, with uniformly
nondegenerate surface metrics.  In each chart, write
$g_{\alpha\beta}=\partial_\alpha X_i\cdot\partial_\beta X_i$
for the induced metric and $g^{\alpha\beta}$ for its inverse coefficients.
Then $\mathcal H=-g^{\alpha\beta}\partial_{\alpha\beta}X_i\cdot n$.
Continuity on the compact time interval bounds the chart norms.
Metric nondegeneracy and $H^4$ control give bounds for $g^{-1}$ and the
surface Jacobian in $L^\infty$, and for $\partial^2X_i$ in $L^2$.
A finite partition of unity therefore gives
\begin{equation}\label{eq:H4-controls-curvature}
\sup_{|t-T|\le\tau}\|\mathcal H(t)\|_{L^2(\Gamma(t))}<\infty
\end{equation}
on each boundary component.  This contradicts
\eqref{eq:L2-curvature-final}, proving maximality in the stated strong class.
\end{proof}

\section{Cavity collapse at zero surface tension}
\label{sec:zero-surface-tension}

We construct the canonical trajectories at $\sigma=0$ in the
axisymmetric, reflection symmetric class with a global velocity potential.
The spatial elliptic estimates, scalar reconstruction and first mode
identities remain valid.  We use Taylor energies in place of the
capillary energies and estimate the full Taylor coefficients and their
first $Y$ derivatives.  The coupled inner and outer estimates allow us to pass to
the singular limit and prove Theorem~\ref{thm:zero-surface-tension-construction}.
We then complete the proof of Theorem~\ref{thm:terminal-material-deformation}
for each fixed $\sigma\ge0$.

Recall that $F=f/q$ and $p_q=q_Y/q$.  Fix a lower Sobolev index $b_0$ larger than
$\max\{s_{\rm dy}+2,k_0+d_f+16\}$, where $k_0$ controls the fixed
coefficient seminorms and $d_f$ bounds the derivative losses
in the first two elliptic variations, the first $Y$ derivatives and the scalar
reconstruction.  Increase $b_0$ to include the additional inner regularity
required by Lemma~\ref{lem:zero-separated-asymmetric}.  On the neighbourhood
used for the estimates, assume
\begin{equation*}
\|\eta\|_{H^{b_0}}+\|F\|_{H^{b_0}}
+\|h_+\|_{H^{b_0+1}}+\|g\|_{H^{b_0+1/2}}+|R-R_*|
\le K(1+Y)^{M_b}a^2,
\end{equation*}
$|\alpha|\le q\mathcal P a$, $q\asymp(1+Y)^{-1/2}$, and
$p_q=-1/(2(Y+\mathfrak c))+O(\mathcal P a)$.
We impose the bootstrap only at this fixed lower index and allow the
energy index to be chosen arbitrarily.

\subsection{Taylor coefficients and common state variables}

Assume the macroscopic geometry conditions \eqref{eq:macro-geometry-assumptions} and \eqref{eq:outer-distance-assumption}.  Write
\begin{equation*}
\Lambda=Y+\mathfrak c=\frac{2\pi L}{\operatorname{Cap}},\qquad
D=\frac{q^2}{2a^2},\qquad S=\frac D\Lambda.
\end{equation*}
At each coefficient index $k$, assume the physical bootstrap
and the hypotheses on first $Y$ derivatives in
Lemma~\ref{lem:normalized-taylor-variation}, with enough additional
spatial derivatives for the first mode and scalar estimates.
These bounds follow from the original equations.  The number of additional
derivatives is fixed before the energy index is chosen.

\begin{lemma} 
\label{lem:zero-outer-taylor}
On the neighbourhood described above, for sufficiently large $Y_0$
independent of the terminal parameter $Y_N$, the coefficient
\begin{equation*}
K_+=\frac{\mathfrak a_+}{S},\qquad
\mathfrak a_+=-\partial_{n_+}p,
\end{equation*}
satisfies, in the fixed outer pullback,
\begin{equation}\label{eq:zero-outer-taylor-jets}
K_+\ge \underline\kappa>0,\qquad
\|K_+\|_{C^k}\le C_k,\qquad
\|(K_+)_Y\|_{C^k}\le\mathcal P_k(Y)a.
\end{equation}
Here $\underline\kappa$ depends on the fixed terminal geometry.  Moreover,
\begin{equation}\label{eq:zero-outer-S-growth}
\frac{S_Y}{S}
=2p_q+2-\frac{1+\mathfrak c_Y}{\Lambda}
=2-\frac2\Lambda+O(\mathcal P(Y)a)\ge1,
\qquad
t_Y\sqrt S=\frac{\Theta a}{\sqrt{2\Lambda}}.
\end{equation}
\end{lemma}
\begin{proof}
We divide the proof into two steps.

\emph{Step 1. Normalized pressure reconstruction.}
Use the exact harmonic pressure reconstruction of
Proposition~\ref{prop:actual-taylor-coefficient}, with zero pressure
trace on both components.  Write
\begin{equation*}
p+\tfrac12|v|^2=D u_c+\mathsf R,
\qquad r=\mathsf R/D,
\qquad \kappa_{\rm cap}=-\Lambda\partial_{n_+}u_c.
\end{equation*}
Here $r_\pm$ are the traces of $r$ on $\Gamma_\pm$, pulled back to
boundary coordinates.  The normalized collar velocity estimates and their first $Y$
derivatives give
\begin{equation*}
\|r_-\|_{H^{k+1}}+\|(r_-)_Y\|_{H^{k+1}}
\le\mathcal P_k a,
\qquad
\|r_+\|_{H^{k+4}}+\|(r_+)_Y\|_{H^{k+4}}
\le\mathcal P_k a^2.
\end{equation*}
The derivatives are taken after material pullback.  Normalization
by $D$ leaves only $p_q$ and first derivatives supplied by the original equations.
The differentiated estimates on a fixed collar and the Schur estimates of Lemmas~\ref{lem:normalized-taylor-variation}
and~\ref{lem:zero-charge-outer-derivative} give
\begin{equation}\label{eq:zero-outer-taylor-error}
K_+=\kappa_{\rm cap}-\Lambda\partial_{n_+}r
+\frac{\Lambda}{2D}\partial_{n_+}|v|^2,
\qquad
\|K_+-\kappa_{\rm cap}\|_{C^k}
+\|\partial_Y(K_+-\kappa_{\rm cap})\|_{C^k}
\le\mathcal P_k a.
\end{equation}
\emph{Step 2. Positivity and time variation.}
The same normalization controls the condenser time derivative at
the outer boundary.  The potential
$\Lambda u_c$ has the filled domain representation
\eqref{eq:filled-green-representation}.  Its positive source has
total mass $2\pi L$.  In the inner pullback, the source density and its
first normalized variation are given by the layer system in
Lemma~\ref{lem:capacity-single-layer-schur}.  Its mass derivative is
$2\pi L_Y=O(\mathcal P a^2)$. The remaining density variation is
$O(\mathcal P a)$ in a fixed negative trace norm, and the support
velocity is $O(a)$ in that pullback.  Pair the density variation and support velocity with the smooth
outer normal derivative of the filled Green kernel.  Outer boundary
and centre derivatives have size $O(\mathcal P a^2)$.  It follows that
\begin{equation*}
\|(\kappa_{\rm cap})_Y\|_{C^k}\le\mathcal P_k a,
\qquad
\kappa_{\rm cap}\longrightarrow-\partial_{n_*}U_*
\quad\hbox{smoothly on }\Gamma_+^*.
\end{equation*}
The positive Green potential $U_*$ vanishes on the outer boundary.
The Hopf lemma and compactness of this fixed boundary give
$\min_{\Gamma_+^*}(-\partial_{n_*}U_*)>0$.  Equation
\eqref{eq:zero-outer-taylor-error} proves
\eqref{eq:zero-outer-taylor-jets}.  Finally, use
$p_q=-1/(2\Lambda)+O(\mathcal P a)$ and
$\mathfrak c_Y=O(\mathcal P a)$ to differentiate $S$ and obtain
\eqref{eq:zero-outer-S-growth}.
\end{proof}

\noindent\textbf{The outer gauge and a uniformly coercive DN form.}
For a prescribed outer trace $\varphi$, minimize the weighted Dirichlet
energy over functions with this trace and an arbitrary constant inner trace.
In the fixed outer angular coordinate, put $\mu_+=\rho_+j_+$, up to a
fixed normalization.  Let $\mathsf G_+^{\rm zc}$ be the
resulting ordinary normal DN operator and set
$\mathcal A_+^{\rm zc}=\mu_+\mathsf G_+^{\rm zc}$.  We work in the fixed
angular measure $d\theta$.
Variation of the free inner constant imposes zero
inner flux.  The outer flux consequently has zero total integral.
The operator $\mathcal A_+^{\rm zc}$ is self-adjoint and nonnegative,
with kernel consisting of the constants.  Uniformly in $a$,
\begin{equation}\label{eq:zero-outer-dn-coercivity}
c\|\varphi\|_{H^{1/2}/\mathbb R}^2
\le \langle\mathcal A_+^{\rm zc}\varphi,\varphi\rangle_{L^2(d\theta)}
\le C\|\varphi\|_{H^{1/2}/\mathbb R}^2.
\end{equation}
Extending the function into the hole by its constant trace gives an
admissible competitor in the filled domain with exactly the same
Dirichlet energy.  The filled domain DN form therefore gives the lower bound.  For the
upper bound, start with the filled harmonic extension and correct
its nonconstant inner trace in a collar of radius comparable to $a$.
Interior regularity bounds that trace by $O(a)$ times the outer
$H^{1/2}$ norm. Scale invariance of the meridional energy bounds the
correction.  The filled domain spectral gap modulo constants is
uniform on the fixed macroscopic geometry class.

Choose the representative of the good trace with zero ordinary
$d\theta$ mean and write $P_+=I-P_0$.  This fixed projection
commutes with $\langle D_\theta\rangle^r$ for every $r$.
The flux operator satisfies $\mathcal A_+^{\rm zc}1=0$ and takes values
in the space of functions with zero ordinary mean.  Thus the kernel of
the DN form is fixed.  The density remains
uniformly positive and bounded, and $\|(\mu_+)_Y\|_{C^k}\le\mathcal P_k a^2$.
The positive Taylor form defined below controls the mean of the outer height.
The rank-one difference between
an unreduced outer DN block and $\mathsf G_+^{\rm zc}$ is retained
as an explicit smooth kernel term.

The first $Y$ derivative of the DN form satisfies a relative bound
in the $H^{1/2}/\mathbb R$ form norm, with coefficient
\begin{equation}\label{eq:zero-outer-dn-time-form}
\epsilon_{\rm geo}(Y)
=C(1+Y)^{-2}+\mathcal P(Y)a,
\qquad \epsilon_{\rm geo}\in L^1([Y_0,\infty)).
\end{equation}
To prove this bound, fix a smooth outer trace $\varphi$ modulo constants
and denote its zero-flux harmonic extension by $u$.  Let $W$ be the
harmonic extension of the same trace to the filled domain.  Interior
elliptic estimates in a fixed neighbourhood of the centre give
\begin{equation*}
\|\nabla W\|_{C^1}\le C\|\varphi\|_{H^{1/2}/\mathbb R}.
\end{equation*}
The field $u-W$ has zero outer trace, zero total inner flux, and
nonconstant inner trace $-P_{\ne0}(W\circ X_-)$.  The flux of $W$
through the inner curve is zero by its equation in the filled hole.
Since $X_-=c+a(1+\eta)e_r$, the uniform zero-flux energy estimate gives
\begin{equation*}
\|\nabla(u-W)\|_{L^2}^2
\le Ca^2\|\varphi\|_{H^{1/2}/\mathbb R}^2,
\qquad
\int_{\mathcal C_a}\rho|\nabla u|^2
\le Ca^2\|\varphi\|_{H^{1/2}/\mathbb R}^2,
\end{equation*}
where $\mathcal C_a$ is an inner collar of thickness comparable to $a$.
The second bound also uses the area $O(a^2)$ of this meridional collar.
All constants depend only on fixed lower geometric bounds.

At any fixed $Y$, choose a deformation field at fixed material boundary
coordinates and split it into the relative inner motion, supported in
$\mathcal C_a$, and a macroscopic field.  The relations $a_Y=-a$,
$\eta_Y=O(\mathcal P a^2)$, $c_Y=O(\mathcal P a^2)$ and the outer
geometric velocity bounds give a deformation field $V$ with
$|DV|\le C\mathbf1_{\mathcal C_a}+\mathcal P a^2$ and
$|V|\le Ca\mathbf1_{\mathcal C_a}+\mathcal P a^2$.
Let $\Psi$ be the deformation map generated by $V$.  At the chosen
time $\Psi=I$, and differentiation of
$A=(\rho\circ\Psi)\det(D\Psi)D\Psi^{-1}D\Psi^{-T}$ gives
\begin{equation*}
A_Y=(V\cdot\nabla\rho+\rho\operatorname{div}V)I
-\rho(DV+DV^{\mathsf T}).
\end{equation*}
Consequently 
\begin{equation*}
|A_Y|\le C\mathbf1_{\mathcal C_a}+\mathcal P(Y)a^2.
\end{equation*}
A varying scalar normalization of the Dirichlet form contributes only
$O(\mathcal P a^2)A$.  Since the extension minimizes the energy and its
inner constant is free, we have
\begin{equation*}
\partial_Y\langle\mathcal A_+^{\rm zc}\varphi,\varphi\rangle
=\int A_Y\nabla u\cdot\nabla u.
\end{equation*}
Indeed, the derivative of the extension has zero outer trace and a
constant inner trace, so its energy pairing with $u$ vanishes by
harmonicity and zero inner flux.  This is the first Schur identity
\eqref{eq:schur-first-differential}: the derivatives of the minimizing
scalar cancel, including all inverse capacity terms.
Combining the preceding bounds with
\eqref{eq:zero-outer-dn-coercivity} proves
\begin{equation*}
\bigl|\partial_Y\langle\mathcal A_+^{\rm zc}\varphi,\varphi\rangle\bigr|
\le \mathcal P(Y)a^2\|\varphi\|_{H^{1/2}/\mathbb R}^2.
\end{equation*}
By density, the estimate extends from smooth traces to $H^{1/2}/\mathbb R$,
yielding \eqref{eq:zero-outer-dn-time-form}.

\noindent\textbf{Good variables and positive forms.}
Put $\mathfrak t_-=a^3(-\partial_{n_-}p)/q^2$.  The normalized pressure
calculation of Proposition~\ref{prop:actual-taylor-coefficient} at
$\sigma=0$ gives, in fixed coefficient norms,
\begin{equation*}
\mathfrak t_-=1-\frac1{2(Y+\mathfrak c)}+O_{C^k}(\mathcal P a),
\qquad d_-=1+O_{C^k}(\mathcal P a).
\end{equation*}
In particular $\mathfrak t_-\ge1/2$ for large fixed $Y_0$.
Lemma~\ref{lem:normalized-taylor-variation} gives
\begin{equation}\label{eq:zero-inner-taylor-jet}
\|(\mathfrak t_-)_Y\|_{C^k}
\le C_k\bigl((1+Y)^{-2}+\mathcal P a\bigr).
\end{equation}
At $\sigma=0$, the harmonic function $p+|v|^2/2$ has boundary data
$|v|^2/2$.  Its differentiated normalized reconstruction uses the first
state derivatives obtained from the original equations and $q_Y/q$.
The other coefficient estimates are
\begin{equation*}
|\Theta-1|+|\Theta_Y|+\|j-1\|_{C^k}+\|j_Y\|_{C^k}
\le\mathcal P a^2,\qquad
\|\mathcal V\|_{C^k}+\|\ell_0\|_{C^k}
+\|(\ell_0)_Y\|_{C^k}\le\mathcal P a .
\end{equation*}
The lower index $b_0$ includes all coefficient indices used here.

Use $P=P_{\ge2}$, $J^s=\langle D_\theta\rangle^s$, and the polar
coefficients $d,e,j,b$ of Proposition~\ref{prop:polar-full-cancellation}.
Define the inner variables and the operators retained in the equations by
\begin{equation}\label{eq:zero-inner-good-unknown}
\binom{\mathfrak h}{\mathfrak\pi}
=P\mathcal W(\mathsf S)P\binom{a\eta}{aF},\qquad
\mathsf S=\sqrt w\begin{pmatrix}d/j&0\\-b&-1\end{pmatrix},
\qquad w=(\rho/R)j,
\end{equation}
\begin{equation}\label{eq:zero-inner-retained-operators}
\mathsf L_0=\Theta P\mathcal W(x+\ell_0)P,\qquad
\mathsf M_0=\Theta P\mathcal W(\mathfrak t_-)P,\qquad
x=\varrho(\xi)/j .
\end{equation}
The real symbols, the fixed low-frequency completions of
Lemma~\ref{lem:finite-bony-weyl}, and the small fixed lower coefficient
bounds give
$\langle\mathsf M_0 h,h\rangle\asymp\|h\|_2^2$ and
$\langle\mathsf L_0\pi,\pi\rangle\asymp\|\pi\|_{H^{1/2}}^2$
on $P L^2$.  Define
\begin{equation}\label{eq:zero-common-inner-form}
E_s^2=\langle\mathsf M_0J^s\mathfrak h,J^s\mathfrak h\rangle
+\langle\mathsf L_0J^s\mathfrak\pi,J^s\mathfrak\pi\rangle.
\end{equation}
Thus
\begin{equation}\label{eq:zero-inner-energy}
E_s\asymp\|\mathfrak h\|_{H^s}
+\|\mathfrak\pi\|_{H^{s+1/2}}.
\end{equation}
At the circular cylindrical reference state,
$\mathsf S=\left(\begin{smallmatrix}1&0\\-1&-1\end{smallmatrix}\right)$.
Under the fixed lower bootstrap, the compressed operator in
\eqref{eq:zero-inner-good-unknown} is a small $L^2$ perturbation of this
invertible constant matrix.  Lemma~\ref{lem:common-conjugacy} therefore
gives, with the same $Y_0$ at every $s$,
\begin{equation}\label{eq:zero-inner-reconstruction}
\|\eta\|_{H^s}+\|F\|_{H^s}\le C_sa^{-1}E_s .
\end{equation}
To relate variations in the fixed normal chart to physical normal
variations, write $X_+=X_+^*+h_+n_+^*$ and set
\begin{equation*}
c_+^{\rm ch}=n_+\cdot n_+^*,\qquad b_+^{\rm ch}=v_+\cdot n_+^* .
\end{equation*}
For a variation in this chart the physical normal displacement and
Eulerian potential trace are
\begin{equation*}
H_+=c_+^{\rm ch}\,\delta h_+,\qquad
U_+=\delta g-b_+^{\rm ch}\,\delta h_+ .
\end{equation*}
The Eulerian trace is unchanged when the tangential chart variation is
removed from the normal variational identity.  This gives the nonlinear variables
\begin{equation}\label{eq:zero-outer-chart-good-unknown}
\mathfrak h_+=\mathcal W(c_+^{\rm ch})h_+,\qquad
\mathfrak u_+=P_+\bigl(g-\mathcal W(b_+^{\rm ch})h_+\bigr).
\end{equation}
The negative-order composition terms remain in the sources.
Thus the shear coefficient is the velocity component along the
reference normal $n_+^*$.
The positive coefficient $c_+^{\rm ch}$ is close to one, and
$b_+^{\rm ch}=O(q)$ in the fixed lower norms.  The Taylor entry retains all zeroth-order terms arising from the
time derivative of this chart shear.

The operator $\mathcal A_+$ is the sum of the exact canonical
zero-flux DN operator at radius $a$, including its complete
reference kernel, and the self-adjoint perturbation of the local symbol.
The latter is formed with the fixed cutoff and compressed by $P_+$ on both sides.
The perturbation is small in the $H^{1/2}\to H^{-1/2}$ form norm
by the lower geometric bootstrap.  Consequently
$\mathcal A_+1=0$ and \eqref{eq:zero-outer-dn-coercivity}
holds for $\mathcal A_+$ on the fixed space $P_+L^2$.
Lemma~\ref{lem:zero-separated-asymmetric} estimates the state-dependent
separated kernels with fixed lower inner norms.

For $r=s+1$, set $H_r=J^r\mathfrak h_+$, $U_r=J^r\mathfrak u_+$ and define
\begin{equation}\label{eq:zero-common-outer-form}
(\mathcal O_{s+1}^0)^2
=S\int_\T\mu_+K_+|H_r|^2\,d\theta
+\langle\mathcal A_+U_r,U_r\rangle_{L^2(d\theta)}.
\end{equation}
The positivity of $K_+$ and \eqref{eq:zero-outer-dn-coercivity} imply
\begin{equation*}
\mathcal O_{s+1}^0\asymp
\sqrt S\,\|\mathfrak h_+\|_{H^{s+1}}
+\|\mathfrak u_+\|_{H^{s+3/2}}.
\end{equation*}
The fixed lower coefficient bounds give
$c_+^{\rm ch}=1+O(\mathcal P a^2)$ in the symbol seminorms needed for
$L^2$ boundedness.  Thus $\mathcal W(c_+^{\rm ch})$ is a small $L^2$
perturbation of the identity, with a uniformly bounded inverse for
one fixed $Y_0$.  Equation~\eqref{eq:common-elliptic-reconstruction}
then gives its inverse bounds at every higher Sobolev index.
Recover $g$ from the good trace and $\mathcal W(b_+^{\rm ch})h_+$ in the
fixed zero-mean gauge.  Composition errors are handled by
Lemma~\ref{lem:common-conjugacy}.
The operator bounds use fixed lower seminorms of $b_+^{\rm ch}=O(q)$.
The composition errors require high norms only of the outer geometry,
with at most one factor at high regularity.  We obtain
\begin{equation}\label{eq:zero-outer-trace-recovery}
\|h_+\|_{H^{s+1}}\le C_s S^{-1/2}\mathcal O_{s+1}^0,
\qquad
\|g\|_{H^{s+1/2}}\le\mathcal P_s(Y)\mathcal O_{s+1}^0.
\end{equation}
Here $qS^{-1/2}=\sqrt2\,a\sqrt\Lambda$, so the shear contribution
to the recovered physical trace has at most a polynomial weight.
The two forms and the trace bounds use only the fixed lower
bootstrap and the spatial elliptic estimates.
All four transformed variables vanish at the canonical terminal data.

\subsection{The outer and inner energy estimates}

\begin{proposition} 

On the fixed lower bootstrap neighbourhood, for every $s\ge b_0+2$
the variables \eqref{eq:zero-outer-chart-good-unknown} and the form
\eqref{eq:zero-common-outer-form} satisfy
\begin{equation}\label{eq:zero-outer-energy}
-\partial_Y\mathcal O_{s+1}^0
\le C_s\varepsilon_s\mathcal O_{s+1}^0
+\mathcal P_s(Y)(a+E_s),
\end{equation}
where one may take $\varepsilon_s=C_s\varepsilon_*$ with the common
integrable function
$\varepsilon_*=(1+Y)^{-3/2}+a^{1/8}$.  Polynomial factors multiplying the positive powers of $a$ are
absorbed in $C_s$, with the same $Y_0$ for all $s$.
The inner separated source, including the common potential gauge,
obeys
\begin{equation}\label{eq:zero-outer-to-inner}
\mathcal N_s^0(F_{-+},G_{-+})
\le C_s\varepsilon_s E_s
+\mathcal P_s a^2\mathcal O_{s+1}^0,
\qquad
\mathcal N_s^0(F,G)=\|F\|_{H^s}+\|G\|_{H^{s+1/2}}.
\end{equation}
\end{proposition}
\begin{proof}
We divide the proof into six steps.

\emph{Step 1. Dyadic reduction in the outer variables.}
Write the two outer source components in the good variables as
$(F_+,G_+)$ and use the norm
\begin{equation*}
\|(F_+,G_+)\|_{\mathcal Y_s}
:=\sqrt S\,\|F_+\|_{H^{s+1}}
+\|G_+\|_{H^{s+3/2}}.
\end{equation*}
This source norm is obtained by pairing the two equations with
\eqref{eq:zero-common-outer-form}.

During the outer spatial dyadic expansion, keep $a,R,\beta$
and the normalized inner variables $(\eta,F)$ fixed, and reconstruct
$q,\Theta,\kappa,p_q$ at each outer state.  Lemma~\ref{lem:frozen-admissibility}
provides an admissible domain for the frozen states and Taylor segments.
The reconstructed scalars are frozen only in the local partial
differential; their chain rule terms
\eqref{eq:complete-scalar-chain-rule} remain in the kernel contributions.
Decompose the outer harmonic field into the filled domain extension, a
flux-carrying field, and the separated correction of
Lemma~\ref{lem:zero-separated-asymmetric}.  The filled domain
Dirichlet--Neumann map has the usual local outer symbol.  The hole
correction and the inner trace input have the asymmetric tame bounds
of that lemma.  For the scalar chain rule terms, the scalar input
functionals use fixed lower norms, while their outer range functions
are estimated using the same separated representation, as in Steps~2--3
below.  Thus the first two differentials in this outer expansion lose
a fixed number of derivatives only in the outer variables, with the
inner factors measured at the fixed index $b_0$.

Set $\nu=s+1/2$ and $Z=(h_+,g)$. Scalar powers of $q$ may be absorbed
in the polynomial constants.  Equation~\eqref{eq:zero-outer-trace-recovery}
gives $\|Z\|_{H^\nu}\le\mathcal P_s\mathcal O_{s+1}^0$.
Apply the dyadic argument of Lemma~\ref{lem:finite-loss-dyadic} at
index $\nu$, with the inner variables held fixed.  Its remainder is
in $H^{\nu+2}=H^{s+5/2}$, which contains both required source spaces.
To determine the additional regularity, let $d$ cover all losses in the
first two outer differentials, and let $r_0$ be their fixed lower index.  If $Z_j$ is an outer dyadic block, $L=\|Z\|_{H^b}$,
$H=\|Z\|_{H^\nu}$ and $\delta=b-r_0$, the sum of the mixed quadratic terms is bounded by
\begin{equation*}
C_s\mathcal P_s\sum_{j\ge0}
2^{(d+2-\delta)j}L\|Z_j\|_{H^\nu}
\le C_s\mathcal P_sLH
\qquad(\delta>d+3).
\end{equation*}
The term evaluated at the frozen state has the summable factor
$2^{(d+2-2\delta)j}L^2H$.  The symbol and separated kernel terms are
summed using the additional fixed regularity specified in
Lemma~\ref{lem:finite-loss-dyadic}.  Choose $b<b_0$ above all these
fixed thresholds.  Lemma~\ref{lem:zero-separated-asymmetric} places
no $j$-dependent derivatives on the inner factors: their norms are
fixed $H^{b_0}$ norms throughout this summation.

\emph{Step 2. Retention of the full Taylor coefficient.}
At fixed geometry, the separated trace-to-trace kernel is smoothing.
Its shape differential includes local terms from changes in the evaluation
point and the normal.  In particular the full DN
shape derivative contains $-\mathsf G_{++}(B_+H_+)$ and the tangential
transport terms.  Before applying any cutoff, combine these terms with the exact normal
good unknown and its actual time derivative, using the full DN map.  Equation~\eqref{eq:exact-normal-variational-system} identifies the
retained local DN and Taylor blocks.  In the corresponding half-density
variables, these are represented by $\mathsf L_+$ and $SK_+$.
The outer reduction treats only the remaining separated trace
kernels as smoothing terms.
The common transport, fixed angular gauge, and all real
zeroth-order terms in the lower-left entry are retained as in the normal
conjugacy.
For the pressure coefficients this separation follows directly from
$p+|v|^2/2=D u_c+\mathsf R$.  For its normalized inner datum, put $V_-=(a/q)v|_{\Gamma_-}$,
so that $r_-=|V_-|^2-1$.  Apply
Proposition~\ref{prop:full-C2-flux-DN} at a fixed index above $1/2$
to $V_-$ and its first two outer variations.  The additional regularity included in
$b_0$ controls the required traces and tangential derivatives.
For outer variations $Z_i$, put $\nu_i=\|Z_i\|_{H^{b_0}}$.
The product rule, including
$D^2|V_-|^2[Z_1,Z_2]=2DV_-[Z_1]\cdot DV_-[Z_2]
+2V_-\cdot D^2V_-[Z_1,Z_2]$, gives
\begin{equation*}
\|Dr_-[Z_i]\|_{H^{1/2}}\le\mathcal P\nu_i,
\qquad
\|D^2r_-[Z_1,Z_2]\|_{H^{1/2}}
\le\mathcal P\nu_1\nu_2.
\end{equation*}
Compose these fixed index estimates with
Lemma~\ref{lem:zero-separated-asymmetric}.  In the chain rule, high output regularity falls only on the outer graph,
while all factors from $r_-$ remain in the fixed inner trace norm.  The outer datum is a local quadratic expression in
the outer velocity.  Its local shape and normal variation terms are
included in \eqref{eq:exact-normal-variational-system} before the dyadic expansion in the outer variables only.  Thus a coefficient measured at high
outer regularity requires high regularity only of the outer state.

The dyadic quadratic remainder includes the difference between the
separated kernel and its reference value.  This reference value is
evaluated at zero outer state, with the current inner parameters held fixed.
Lemma~\ref{lem:zero-separated-asymmetric} controls its high output
norms through the fixed lower inner norms.

\emph{Step 3. Source estimates.}
The physical inner good trace has size $(q/a)E_s$ at the fixed input
index.  Its separated outer kinematic coefficient is
$t_Y(q/a)\mathsf G_{+-}$.  Multiplication by the height weight gives
\begin{equation*}
\sqrt S\,t_Y\frac qa
=\frac{\Theta q}{\sqrt{2\Lambda}},
\end{equation*}
which is bounded by a polynomial in $Y$.  This gives the $\mathcal P_s E_s$ term at any fixed outer
output index.  The inputs to the clock and the other scalar functionals may have size $\mathcal P E_s/a$.  They couple to an outer kinematic reference field of size $\mathcal P a^2$,
so the weighted contribution is at most $\mathcal P E_s$.  The corresponding terms in the dynamic equations
carry $t_Y$ and satisfy smaller bounds.  Fixing the zero-mean gauge in
the outer angular coordinate determines
the common Bernoulli constant before the change of unknowns.

The reference outer normal velocity is $O(q)$ and its reference
Bernoulli field is $O(q^2)$.  Hence
\begin{equation*}
\sqrt S\,t_Y q=\frac{\Theta aq}{\sqrt{2\Lambda}}
=\mathcal P(Y)a,\qquad
t_Yq^2=\Theta a^2q=\mathcal P(Y)a^2.
\end{equation*}
These yield the reference term in
\eqref{eq:zero-outer-asymmetric-source}.
The quadratic terms in the original outer equations carry $t_Y$.
The largest factor introduced by the source norm is the height weight
$\sqrt S=\mathcal P/a$, so the resulting prefactor is at most $\mathcal P a$.  Under the lower bootstrap
bounds, the normalized
outer trace is $O(\mathcal P a)$, while the inner high trace and the
outer graph are $O(\mathcal P a^2)$.  The quadratic estimate is
therefore bounded by
\begin{equation*}
\mathcal P a\,\mathcal P a
\bigl(a^{-1}E_s+\mathcal O_{s+1}^0\bigr)
\le \mathcal P a E_s+\mathcal P a^2\mathcal O_{s+1}^0.
\end{equation*}
All coefficient $Y$ derivatives are evaluated from the original
equations at the current state.  We choose the lower index to include
the required additional derivatives before fixing $s$.  Under the exact conjugacy, the Taylor term is incorporated into $SK_+$.  The remaining diagonal and negative-order
terms have integrable bounds from the fixed lower coefficient estimates
and \eqref{eq:zero-outer-dn-time-form}.  Thus the source bound is
\begin{equation}\label{eq:zero-outer-asymmetric-source}
C_s\mathcal{P}_s(Y)\bigl(a+E_s\bigr)
+C_s\varepsilon_s(Y)\mathcal{O}_{s+1}^0,
\qquad \varepsilon_s\in L^1([Y_0,\infty)).
\end{equation}

\emph{Step 4. The differentiated outer system.}
The preceding reduction gives common transport, an order-one DN
entry and the full Taylor entry $-SK_+$, including the chart shear
derivative.  The remaining off-diagonal entries have order at most
$-1$, and the sources obey \eqref{eq:zero-outer-asymmetric-source}.

For the energy calculation, set $r=s+1$,
$H_r=\langle D_\theta\rangle^r\mathfrak h_+$, and
$U_r=\langle D_\theta\rangle^r\mathfrak u_+$.
In particular $P_+U_r=U_r$.  Paralinearizing the outer equations and
commuting derivatives gives the retained block
\begin{equation*} 
(\partial_Y+T_+)H_r
=t_Y\mu_+^{-1}\mathcal A_+U_r+F_r,\qquad
(\partial_Y+T_+)U_r
=-t_YS\,P_+(K_+H_r)+G_r. 
\end{equation*}
Here $T_+$ denotes the tangential transport part of the fixed cutoff
paralinearization, with principal symbol $i\xi\mathcal V_+$, where 
\begin{equation*}
\mathcal V_+=\frac{(t_Yv_+-X_{+,Y})\cdot\tau_+}{j_+},
\qquad j_+=|\partial_\theta X_+|.
\end{equation*} 
In the second row this transport is compressed by the fixed $P_+$.
Its projection commutators are fixed smooth kernels.
At the differentiated energy level, the cutoff gives
\begin{equation}\label{eq:zero-fixed-target-multiplication}
\|cz-\mathcal W(c)z\|_{H^{\nu+1}}
\le C\|c\|_{C^4}\|z\|_{H^\nu},\qquad
\nu\in\{-\tfrac12,0\}.
\end{equation}
The Fourier matrix of the difference is
$[1-\chi(j-k,(j+k)/2)]\widehat c(j-k)$.
Its support satisfies
$\langle j\rangle+\langle k\rangle\le C\langle j-k\rangle$.
For these two values of $\nu$, the weighted row and column sums are
bounded by $C\sum_\ell\langle\ell\rangle|\widehat c(\ell)|
\le C\|c\|_{H^2}$, so the Schur test proves the estimate.
The fixed low frequency completion obeys the same bound.
Apply \eqref{eq:zero-fixed-target-multiplication} with
$(c,z,\nu)=(K_+,H_r,0)$ and
$(c,z,\nu)=(\mu_+^{-1},\mathcal A_+U_r,-1/2)$.
The resulting errors belong to $H^1\subset H^{1/2}$ and
$H^{1/2}\subset L^2$, respectively, and involve only the fixed $C^4$
coefficient norms.  The smooth fixed reference coefficients may also be
retained exactly.
The errors in the two rows contribute the coefficient
$C_r t_Y\sqrt S$ to the energy estimate and are included in $(F_r,G_r)$.
The height weight $S\mu_+K_+$ is positive.
The form derivatives satisfy \eqref{eq:zero-outer-taylor-jets}
and \eqref{eq:zero-outer-dn-time-form}, using the same fixed cutoff
as the nonlinear reduction.

\emph{Step 5. The outer energy estimate.}
Differentiate the form \eqref{eq:zero-common-outer-form}.
The Hamiltonian cross terms cancel
exactly, since
\begin{equation*}
\langle\mu_+K_+H_r,\mu_+^{-1}\mathcal A_+U_r\rangle
=\langle K_+H_r,\mathcal A_+U_r\rangle
=\langle P_+(K_+H_r),\mathcal A_+U_r\rangle .
\end{equation*}
The last equality uses $\int\mathcal A_+U_r\,d\theta=0$.
It holds at every differentiated level because $P_+$ is fixed.
The term $S_Y\int\mu_+K_+|H_r|^2$ has a favourable
sign in the backward estimate.  The other weight derivatives have
integrable form bounds.  Commutation with the order-one DN entry
costs at most $r$ derivatives of $\mathfrak u_+$. Commutation with
the order-zero Taylor entry costs at most $r-1$ derivatives of
$\mathfrak h_+$.  The half-derivative in the DN energy therefore controls both
commutators, with coefficient at most
$C_r t_Y\sqrt S$.  We integrate the common transport terms by parts in the weighted form.
The derivatives of the transport coefficient and the resulting derivatives
of the density are $O(\mathcal P a^2)$.
The remaining terms of order $-1$ from height to momentum satisfy the
same energy bound after multiplication by $t_YS$, with coefficient $t_Y\sqrt S$.
Together with the nonlinear and separated source bound
\eqref{eq:zero-outer-asymmetric-source}, these estimates give
\eqref{eq:zero-outer-energy} after taking the square root,
regularized at zero.

\emph{Step 6. Coupling to the inner energy.}
The outer input to the exact inner variational kinematic
row is $\mathsf G_{-+}U_+$.
Lemma~\ref{lem:cross-block-small-hole} bounds its output in high Sobolev
norms on the circular reference boundary.  We estimate its graph correction separately.  An inner factor bounded
by $\mathcal P a^2$ in the lower norms multiplies the outer input at high
regularity.  An inner factor at high regularity, multiplied by the outer
state in a lower norm, contributes to $\varepsilon_sE_s$ by the same
fixed cutoff tame estimate.  The outer shape differential is the separated term
$-\mathsf G_{-+}(B_+H_+)$.
Separation permits a fixed low input norm, and
\eqref{eq:zero-outer-trace-recovery} supplies that norm.
The Taylor norm of the outer height therefore couples to the inner energy
at index $s$.

It remains to estimate the common outer gauge in the Taylor energy norm.
Its variation may contain a smooth scalar functional of
$\mathfrak a_+H_+=SK_+H_+$, measured with the Taylor weight $S$.  Its contribution to the scaled inner
momentum is bounded by
\begin{equation*}
\frac{a t_Y}{q}\,S\|H_+\|_{L^2}
\le C\frac{a}{\Lambda}S^{-1/2}\mathcal O_{s+1}^0
=C\frac{a^2}{q\sqrt\Lambda}\mathcal O_{s+1}^0
\le C a^2\mathcal O_{s+1}^0.
\end{equation*}
Fixed numerical factors are absorbed in $C$.
The remaining variations of the gauge, charge, and time parametrization are smooth finite-rank
functionals of the same traces.  Their coefficients have polynomial
bounds by the scalar Schur formulas in
Propositions~\ref{prop:schur-differentiation} and~\ref{prop:clock-centre-finite-rank}.  On the
outer boundary the gauge is eliminated by the fixed $P_+$.
This proves \eqref{eq:zero-outer-to-inner}.
\end{proof}

The inner Taylor estimate uses the outer coupling just established.

\begin{lemma} 

Let $(F_*,G_*)$ denote the sources in the reduced inner system
\eqref{eq:zero-inner-reduced-system}.  On the fixed lower bootstrap
neighbourhood, for every $s\ge b_0+2$, the variables and energy in
\eqref{eq:zero-inner-good-unknown}--\eqref{eq:zero-inner-energy}
satisfy the source and backward energy estimates
\begin{equation}\label{eq:zero-inner-source}
\|F_*\|_{H^s}+\|G_*\|_{H^{s+1/2}}
\le C_s\varepsilon_*(Y)E_s+
C_s\mathcal P_s(Y)(a^3+a^2\mathcal O_{s+1}^0),
\end{equation}
\begin{equation}\label{eq:zero-inner-budget}
-\partial_YE_s\le(1+C_s\varepsilon_*(Y))E_s+
C_s\mathcal P_s(Y)(a^3+a^2\mathcal O_{s+1}^0).
\end{equation}
Here $\varepsilon_*=(1+Y)^{-3/2}+a^{1/8}$, as above, and the constants
depend on the fixed lower bootstrap.
\end{lemma}

\begin{proof}
We divide the proof into three steps.

\emph{Step 1. The reduced inner system.}
Set $\sigma=0$ in Proposition~\ref{prop:polar-full-cancellation}.
With $W=\sqrt w$ and $\delta=\mathcal D\log W$, the homogeneous
block in half-density variables is
\begin{equation*} 
\mathcal D h=\mathsf L_{\rm full}\pi+(d_-+\delta)h,\qquad
\mathcal D\pi=-\Theta\mathfrak t_-h-(1+p_q-\delta)\pi,\qquad
\mathsf L_{\rm full}=\Theta W(a\mathsf G_{--})W^{-1}. 
\end{equation*}
The operator $\mathsf L_{\rm full}$ is self-adjoint in the fixed
space of half densities.  Replacing $\mathcal V\partial_\theta$ by its
skew completion changes only the diagonal coefficients, adding
$\mathcal V_\theta/2$ to both.  All the displayed diagonal corrections to
$1$ and $-(1+p_q)$ are $O(\mathcal P a)$.

Use $U=(\eta,F,h_+,g)$, with $a,R,\beta$ as parameters, and
reconstruct $q,\Theta,\kappa,p_q$ as in
Lemma~\ref{lem:scalar-noncircular}.  Apply
Lemmas~\ref{lem:finite-tame-dn-symbol},
\ref{lem:finite-loss-dyadic} and~\ref{lem:common-conjugacy}
to the complete reconstructed equations
\eqref{eq:polar-h-exact} and~\eqref{eq:polar-F-exact}
on the domain of Lemma~\ref{lem:frozen-admissibility}, with $\sigma=0$.
The reconstructed scalars are held fixed only when computing the
local partial differential; their chain rule contributions
\eqref{eq:complete-scalar-chain-rule} are retained among the kernel terms.

The local differentials contain the same DN variations and kinetic
products, with the curvature term deleted.  The conjugacy remainder
has order $-1$.  Its coefficient derivative contains the terms in
\eqref{eq:zero-inner-pressure-cancellation}, so the lower-left entry
retains the full multiplication operator $-\Theta\mathfrak t_-$.

Use the operators in \eqref{eq:zero-inner-retained-operators}
and set $T=iP\mathcal W(\mathcal V\xi)P$.  The nonlinear equations
then take the form
\begin{equation}\label{eq:zero-inner-reduced-system} 
(\partial_Y+T)\mathfrak h=\mathfrak h+\mathsf L_0\mathfrak\pi+F_*,\qquad
(\partial_Y+T)\mathfrak\pi=-\mathsf M_0\mathfrak h-(1+p_q)\mathfrak\pi+G_* . 
\end{equation}
The projection in the original equations removes the potential constant
before the change to half densities.

\emph{Step 2. Source estimates.}
The local remainders have the following maximal orders:
\begin{equation*}
\begin{array}{c|c|c}
\text{entry}&\text{order}&\text{input norm}\\ \hline
\mathfrak h\longmapsto F_*&0&H^s\\
\mathfrak\pi\longmapsto G_*&0&H^{s+1/2}\\
\mathfrak\pi\longmapsto F_*&0&H^{s+1/2}\\
\mathfrak h\longmapsto G_*&-1&H^s
\end{array}
\end{equation*}
The coefficient bounds for $\mathfrak t_-,d_-,\Theta,j,\mathcal V,\ell_0$
and the localization estimate \eqref{eq:kernel-small-localization}
give integrable bounds for the lower coefficient seminorms of these remainders.
In particular, the last row maps $H^s$ into $H^{s+1}$, which
contains the required $H^{s+1/2}$ momentum source space.  An
order-zero upper-right remainder is admissible because the
kinematic source requires only $H^s$.

For the smooth kernel, constant trace and finite projection
corrections, use \eqref{eq:global-kernel-tame-operator} and the
kernel decomposition of Lemma~\ref{lem:reconstructed-kernel-bounds}.
The scalar chain rule terms are retained before applying the dyadic
summation.  The two source components are measured in $H^s$ and
$H^{s+1/2}$, respectively, as in \eqref{eq:zero-inner-source}.
The inner-to-inner reference kernels remain explicit linear terms.  Their
$\mathcal P_s a^{1/4}$ operator bounds contribute
$C_s\varepsilon_*E_s$ to the sources.  For the state-dependent parts,
the summation with frozen kernels in Lemma~\ref{lem:finite-loss-dyadic}
places high output derivatives on at most one factor of the quadratic
remainder.  All actual time derivatives, including those of the
parameters fixed in the spatial expansion and of the reconstructed
scalars, are retained when differentiating $\mathsf S$ and the
retained operators.

Set $U=(\eta,F,h_+,g)$.  As in the outer dyadic argument,
we use the physical outer trace $g$ and obtain
\begin{equation*}
a\|\mathcal R_{\rm hh}\|_{H^{s+2}}
\le C_s\mathcal P_s(Y)\|U\|_{H^{s_{\rm dy}}}
(E_s+a\mathcal O_{s+1}^0).
\end{equation*}
Here $s_{\rm dy}<b_0$ is fixed.  The estimate follows from
\eqref{eq:finite-loss-hh-bound}, the algebraic $H^s$ reconstruction
\eqref{eq:zero-inner-reconstruction}, and the bound for the original
outer trace with one additional outer derivative.  The inner height
is measured in $H^s$.  Since the lower bootstrap factor is
$O(\mathcal P a^2)$, this remainder is bounded as in
\eqref{eq:zero-inner-source}.

The forcing evaluated at zero high mode state is $O(\mathcal P_s a^3)$.
Indeed, the high Fourier modes of the normalized condenser profile,
first mode leakage, and
quadratic first mode terms are $O(\mathcal P_s a^2)$, by
Lemma~\ref{lem:weighted-layer-expansion} and
\eqref{eq:exact-kinetic-low-source}, and the scaled equations carry
the additional factor $a$.  The curvature source is absent.
For the separated outer input, the spatial kernel estimate of
Lemma~\ref{lem:cross-block-small-hole}, which is independent of surface tension,
supplies the reference factor $\mathcal P_s a^2$.
The graph, pressure and common gauge terms satisfy
\eqref{eq:zero-outer-to-inner}, giving
$C_s\varepsilon_*E_s+
C_s\mathcal P_s a^2\mathcal O_{s+1}^0$.
These estimates prove \eqref{eq:zero-inner-source}.

\emph{Step 3. The inner energy estimate.}
Differentiate the form \eqref{eq:zero-common-inner-form}, using
$J^s=\langle D_\theta\rangle^s$.
Use the fixed low-frequency completions of
Lemma~\ref{lem:finite-bony-weyl}.  The coefficient $Y$ derivatives and the
transport commutators of $\mathsf L_0,\mathsf M_0$ have the corresponding form
bounds with coefficient $C\varepsilon_*$. The assertion for
$\mathsf M_0$ uses \eqref{eq:zero-inner-taylor-jet}.

The commutators $[J^s,\mathsf L_0]$ and $[J^s,\mathsf M_0]$ have orders
at most $s$ and $s-1$, respectively.  They map $H^{s+1/2}$
to $L^2$ and $H^s$ to $H^{1/2}$, the required source spaces.
Their constant-coefficient reference parts vanish.  The order-$s$
commutator with $T$ is bounded in both component norms by an
integrable coefficient determined by the size of $\mathcal V$.

Set $h_s=J^s\mathfrak h$ and $\pi_s=J^s\mathfrak\pi$.
The Hamiltonian cross terms cancel.  After estimating the coefficient and transport terms, the remaining contribution to the negative derivative of the
quadratic energy is
\begin{equation*}
\langle(-(\mathsf M_0)_Y-2\mathsf M_0)h_s,h_s\rangle
+\langle(-(\mathsf L_0)_Y+2(1+p_q)\mathsf L_0)\pi_s,\pi_s\rangle .
\end{equation*}
It is at most $(2+C_s\varepsilon_*)E_s^2$, since
$p_q\le\mathcal P a$.  The source pairing is bounded by
$C_sE_s(\|F_*\|_{H^s}+\|G_*\|_{H^{s+1/2}})$.
Taking the square root, with regularization at zero, gives
\eqref{eq:zero-inner-budget}.
\end{proof}

\subsection{The terminal construction at zero surface tension}

Use the canonical terminal geometry, common potential gauge and
centred radius $a=\ell e^{-Y}$ of
Subsection~\ref{subsec:canonical-terminal-data}.
The energy estimates above couple $E_s$ to $\mathcal O_{s+1}^0$.
We apply Lemma~\ref{lem:zero-coupled-terminal-budget} to this pair in the
form \eqref{eq:zero-coupled-terminal-input}--\eqref{eq:zero-coupled-terminal-output}.

\begin{theorem} 
\label{thm:zero-one-state-closed}
Choose the canonical domains and condenser potentials in
\eqref{eq:aN-def}--\eqref{eq:cap-terminal}.  Fix $E_c>0$
and prescribe
\begin{equation*}
v_N^*=c_N\nabla u_N,\qquad
\frac12c_N^2\operatorname{Cap}_N=E_c,\qquad c_N>0.
\end{equation*}
There is a common $Y_0$, independent of $N$, such that the Euler
equations have smooth solutions on $[Y_0,Y_N]$ with these data at $Y_N$.
For every fixed sufficiently large $s$,
\begin{equation}\label{eq:zero-energy-tails}
E_s\le\mathcal P_s a^3,\qquad
\mathcal O_{s+1}^0\le\mathcal P_s a.
\end{equation}
For every fixed $k$, their physical coordinates satisfy
\begin{equation}\label{eq:zero-physical-tails}
\begin{aligned}
\|\eta\|_{H^k}+q^{-1}\|f\|_{H^k}
+\|h_+\|_{H^k}+q^{-1}\|g\|_{H^k}
+|R-R_*|&\le\mathcal P_k a^2,\\
q^{-1}(|\alpha|+|\alpha_Y|)&\le\mathcal P_k a,\qquad
|R_Y|\le\mathcal P_k a^2.
\end{aligned}
\end{equation}
All constants and polynomial degrees are independent of $N$.
The first $Y$ derivatives satisfy the corresponding bounds with the
additional spatial derivatives required by
Proposition~\ref{prop:raw-state-first-jet}.
\end{theorem}
\begin{proof}
We divide the proof into five steps.

\emph{Step 1. Local solutions.}
At $\sigma=0$, Lemma~\ref{lem:terminal-data-admissible} gives smooth
backward branches satisfying both Taylor sign conditions and preserving
the symmetries and the global velocity potential.

\emph{Step 2. The lower bootstrap and scalar estimates.}
Set $b=b_0$, increasing $b_0$ once to cover the coefficient seminorms,
the admissibility index $b_{\rm ad}$, the lower index in the dyadic
estimates, the first mode estimate,
and the additional derivatives required for the first pressure variation.  Bootstrap the physical
trace bounds at this index by $K_b(1+Y)^{M_b}a^2$, together with
$|\alpha|\le\delta_1q$, admissible separated geometry and the
positive kinematic determinant.  The outer coordinate norms
required at $b$ include the fixed extra derivatives appearing in
the coefficient estimates.  All these additional regularity requirements are fixed before an
energy index is chosen.

At $\sigma=0$, the exact energy is
\begin{equation*}
E_c=D_bq^2+K_0,
\qquad D_b=\pi L(Y+\mathfrak c).
\end{equation*}
The geometric capacity bound and \eqref{eq:zero-flux-dirichlet-bound}
give $D_b\asymp1+Y$ and the required upper bound for $K_0$.
This gives $q\asymp(1+Y)^{-1/2}$ before any energy estimate at higher
regularity is applied.  The positive branch and the exact modulation
give $t_Y=\Theta a^2/q>0$.
Proposition~\ref{prop:first-mode-backward-amplitude} applies with
its curvature term removed.  Its exact kinetic cancellation yields
$|\alpha|+|\alpha_Y|\le q\mathcal P a$ and
$|R-R_*|+|R_Y|\le\mathcal P a^2$.
As in Step~2 of the proof of Theorem~\ref{thm:one-state-closed},
enlarge $Y_0$ so that the lower physical perturbation norm is at most
$a^{3/2}$. This fixes $\mathcal P$ independently of $K_b,M_b$.
The instantaneous formula \eqref{eq:p-normalized-explicit}, with
$\sigma A$ deleted, and the first pressure variation give the
scalar and coefficient bounds required by the zero-surface-tension
inner and outer energy estimates.  This step uses only the bootstrap at the fixed lower index and the
scalar reconstruction.

\emph{Step 3. Energy estimates above the bootstrap index.}
Choose an energy index $r>b$ once and for all, large enough to
reconstruct the outer good unknown at the lower index, apply one outer
DN operator, and use the original Bernoulli equation at all indices
required in the bootstrap.  It suffices to take $r\ge b+d+2$,
where $d$ is the maximum fixed loss in these operations.
The estimates with one factor at the highest regularity use smallness
at index $b$ and therefore apply at $s=r$ on the bootstrap interval.  The energy $\mathcal O_{r+1}^0$ controls the high outer potential norm.

Keep $a,R,\beta$ as parameters and apply the reference, linear and
quadratic decomposition to the complete reconstructed vector field,
as in Step~3 of the proof of
Theorem~\ref{thm:one-state-closed}, using
\eqref{eq:frozen-admissible-domain} with $\sigma=0$ and enlarging
$Y_0$ so that $\delta_b\le a^{3/2}$. Its
constants depend on $b,r$ and the reference geometry, but not on
$K_b,M_b$.  Let $P=\mathcal P_r^0\ge1$ dominate the reference forcing
and separated couplings, including the first mode bounds from Step~2.
Its degree is therefore fixed before the bootstrap weight.
Write $\delta_b$ for the lower physical and geometric norm.
The quadratic source bounds are
\begin{equation*}
\mathcal Q_-\le P\delta_b(E_r+a\mathcal O_{r+1}^0),\qquad
\mathcal Q_+\le P\delta_b(a^{-1}E_r+\mathcal O_{r+1}^0).
\end{equation*}
The second bound includes the outer height weight $\sqrt S=\mathcal P/a$
and the factor $t_Y$.  Choose $Y_0$ finally so that
$K_b(1+Y)^{M_b}a^2\le a^{3/2}$, hence $\delta_b\le a^{3/2}$.
The diagonal coefficients are then integrable. The cross terms are
bounded by $Pa^2\mathcal O_{r+1}^0$ and $PE_r$, respectively.
The finite-rank scalar and gauge terms are split in the same way
before taking norms: their reference values contribute forcing, and
their differences retain a perturbation factor.
Thus the only independent forcing is $Pa^3$ in the inner row and
$Pa$ in the outer row.

Combining \eqref{eq:zero-inner-budget} and
\eqref{eq:zero-outer-energy} gives
\eqref{eq:zero-coupled-terminal-input} for
$(E,O)=(E_r,\mathcal O_{r+1}^0)$ with $P=\mathcal P_r^0$.
The coupling of the two different regularity indices is justified by
Lemma~\ref{lem:zero-separated-asymmetric} and
\eqref{eq:zero-outer-asymmetric-source}.
Both transformed states vanish at the canonical terminal data.
Lemma~\ref{lem:zero-coupled-terminal-budget} therefore gives
\eqref{eq:zero-energy-tails} at $r$.
The inner reconstruction \eqref{eq:zero-inner-reconstruction}
gives $\|\eta\|_{H^r}+\|F\|_{H^r}\le\mathcal P a^2$.
Since
\begin{equation*}
S^{1/2}=\frac{q}{\sqrt2\,a\sqrt\Lambda}
\asymp\frac1{a(1+Y)},
\end{equation*}
the outer energy gives $\|h_+\|_{H^{r+1}}\le\mathcal P a^2$ and
$\|\mathfrak u_+\|_{H^{r+3/2}}\le\mathcal P a$.
The trace recovery \eqref{eq:zero-outer-trace-recovery} gives
$\|g\|_{H^k}\le\mathcal P a$ at the lower indices $k$ needed next.

\emph{Step 4. Recovery of the outer potential and closure of the bootstrap.}
The elliptic reconstruction at the fixed outer scale uses the condenser
velocity of size $O(q)$, the inner decay estimates just proved, and the bound
$\|g\|_{H^k}\le\mathcal P a$ at the required indices.  It gives
$v|_{\Gamma_+}=O(\mathcal P q)$, using the additional regularity $r\ge b+d+2$.  Here polynomial
multiples of $a$ are bounded by polynomial multiples of $q$.
Also $X_{+,Y}=O(\mathcal P a^2)$.
In the common mean-zero outer gauge, the nonconstant projection
of the original Bernoulli equation is
\begin{equation}\label{eq:zero-outer-raw-recovery}
g_Y=P_{\ne0}\left\{
t_Y\left(\tfrac12B_+^2-\tfrac12|v_{\tau,+}|^2\right)
+v_{\tau,+}(X_{+,Y}\cdot\tau_+)\right\}.
\end{equation}
Thus at every low index required in the bootstrap,
$\|g_Y\|\le\mathcal P q a^2$.
Since $g(Y_N)=0$, backward integration and
$q(s)/q(Y)\le C\sqrt{(1+Y)/(1+s)}$ give
\begin{equation*}
q(Y)^{-1}\|g(Y)\|\le\mathcal P(Y)a(Y)^2.
\end{equation*}
The energy estimate at $r$ supplies the extra derivative in
\eqref{eq:zero-outer-raw-recovery}, independently of the lower
bootstrap bound being improved.

Now choose $M_b$ above the degree $2\deg \mathcal P_r^0$ from the
weighted terminal estimate and above the finitely many resulting
degrees in the reconstruction, centre estimate and outer
recovery, including their factors of $q^{-1}$.
These degrees have already been determined independently of
$M_b$.  When defining the resulting forcing constant $C_b$, bound the integrated
homogeneous coefficients in the coupled estimate by a fixed constant, say one.  Choose $K_b>2C_b$, and finally
increase $Y_0$ so that the fixed ball conditions,
$K_b(1+Y)^{M_b}a^2\le a^{3/2}$, and the chosen bound on all
integrated homogeneous coefficients hold.
All bootstrap dependent contributions contain a positive power of $a$
and therefore decay exponentially, so this final choice is possible.
The estimates then improve the entire low bootstrap with the fixed
degree $M_b$ and constant $K_b$.

\emph{Step 5. Higher regularity and continuation.}
After the low bootstrap is closed, apply the same pair of energy
estimates at every higher fixed $s$.  Their coefficients depend
only on the established bounds at the lower index.  The weighted terminal lemma applies with arbitrary fixed polynomial
constants at higher regularity and the same $Y_0$, proving
\eqref{eq:zero-energy-tails} at these indices.  For any prescribed outer trace index, apply
\eqref{eq:zero-outer-raw-recovery} using the energy with the required fixed number of additional
derivatives.  This proves
\eqref{eq:zero-physical-tails} at every fixed index on the same
interval.  Proposition~\ref{prop:raw-state-first-jet} applies with the
curvature term deleted and supplies the first $Y$ derivatives.

Proposition~\ref{prop:actual-taylor-coefficient} and
Lemma~\ref{lem:zero-outer-taylor} give
$\mathfrak a_-\ge c q^2/a^3$ and
$\mathfrak a_+\ge\underline\kappa S>0$ on the closed bootstrap.
These lower bounds are uniform on each compact $Y$ interval.
If the left endpoint of a solution with finite terminal parameter
were greater than $Y_0$, its positive cavity radius, separated smooth
charts, Sobolev bounds and Taylor lower bounds would permit continuation
by the local zero-surface-tension theory.  The solution therefore extends to $Y_0$.
\end{proof}

\begin{proof}
We prove Theorem~\ref{thm:zero-surface-tension-construction}.
Apply Proposition~\ref{prop:compactness-without-relative-energy}
with $\sigma=0$ to the branches of
Theorem~\ref{thm:zero-one-state-closed}.  It gives a smooth limiting trajectory satisfying the geometric and
potential estimates and $q^2Y\to E_c/(\pi L_*)$, with strictly increasing
physical time and a finite collapse time.  Thus all hypotheses in
\eqref{eq:concentration-input-tails} hold.
The global potential and the same local layer representation give
\eqref{eq:scaled-inner-collar-velocity}.  Since $p=0$ on both
boundaries, the harmonic pressure identity
\eqref{eq:harmonic-pressure-reconstruction} holds with $\sigma=0$.
Thus the common arguments in
Subsections~\ref{subsec:energy-concentration}--\ref{subsec:inner-pressure-tangent}
give the energy and velocity limits, exact gradient rate, Green and
pressure profiles, and cylindrical tangent.
Proposition~\ref{prop:correct-tube-stability} gives both curvature
asymptotics.  The chart estimate \eqref{eq:H4-controls-curvature}
excludes the stated strong continuation.

For the filling continuation, the incoming bounds $\|v\|_\infty\le Cq/a$ and
$\|p\|_2\le CS$ imply the pressure and bulk flux integrability in
\eqref{eq:weak-capillary-integrable-scales}--\eqref{eq:weak-capillary-flux-integrability}.
The static branch \eqref{eq:zero-static-filling} has the same mass
and momentum traces and zero outgoing energy.  Applying
Proposition~\ref{prop:weak-capillary-continuation}(i) with $\sigma=0$
gives the exact defect \eqref{eq:main-weak-energy-defect}.
\end{proof}

\begin{proof}
We prove Theorem~\ref{thm:terminal-material-deformation}.
Both constructions satisfy the common hypotheses in
\eqref{eq:concentration-input-tails} and the scaled collar bound
\eqref{eq:scaled-inner-collar-velocity}.  Propositions
\ref{prop:terminal-lagrangian-map}, \ref{prop:exact-terminal-topology},
and~\ref{prop:critical-terminal-deformation} therefore give all
assertions for each fixed $\sigma\ge0$.
\end{proof}

\section{Construction with prescribed outer geometry}
\label{sec:prescribed-outer-geometry}

We prove Theorem~\ref{thm:prescribed-outer-geometry} by adapting the
canonical construction to the prescribed outer geometry.  We first
introduce the centred coordinates and the corresponding geometric
estimates.  The proof then applies the preceding elliptic and energy
estimates while retaining both centre coordinates and both first
potential harmonics.

\subsection{Centred coordinates and geometric estimates}

Fix $D_*$, $c_*$, $\sigma$ and $E_c$ as in
Theorem~\ref{thm:prescribed-outer-geometry}.
We use the following parametrizations and establish the estimates below
for the solution in that theorem.

Let $a(t)$ denote the mean radius of the inner meridional section
in the centred parametrization below.  Fix a sufficiently small local
reference length $\ell>0$, and use the logarithmic scale variable
$Y=\log(\ell/a)$, so that $a(Y)=\ell e^{-Y}$ and $Y\to\infty$ as
$t\uparrow T$.  The inner meridian is parametrized by
\begin{equation*}
X_-(Y,\theta)=c(Y)+a(Y)(1+\eta(Y,\theta))e_r(\theta),
\qquad c=(R,Z),\qquad P_0\eta=P_1\eta=0,
\end{equation*}
where $c(Y)=(R(Y),Z(Y))$ is the centre in the $(\rho,z)$ plane.
The angle $\theta\in\T$ is measured at $c$ from the direction of
increasing $\rho$, and $e_r(\theta)=(\cos\theta,\sin\theta)$ is the
corresponding unit radial vector.  Thus the distance from $c$ to
$X_-(Y,\theta)$ is $a(1+\eta)$: the dimensionless function $\eta$
describes the shape, and $a\eta$ is the signed radial displacement
from the reference circle of radius $a$.  The operators $P_0$ and
$P_1$ are the $L^2(\T)$ orthogonal projections onto constants and
onto $\operatorname{span}\{\cos\theta,\sin\theta\}$, respectively.
The conditions $P_0\eta=P_1\eta=0$ make $a$ the angular mean radius
and remove the first harmonics from the shape variable, separating
the radius and centre from the remaining deformation.

Let $X_+^*:\T\to\partial D_*$ be a fixed smooth parametrization
and $n_+^*$ its outward unit normal.  In the corresponding normal
graph chart, the moving outer meridian is
\begin{equation*}
X_+(Y,s)=X_+^*(s)+h_+(Y,s)n_+^*(s),\qquad s\in\T,
\end{equation*}
where $h_+$ is its signed normal displacement from the prescribed
terminal curve.  Here $s$ parametrizes the outer curve independently of
the inner polar angle $\theta$.  For every fixed $k\ge0$, there is
a polynomial $\mathcal P_k$ with nonnegative coefficients,
depending on the fixed data and on $k$, such that
\begin{equation}\label{eq:prescribed-geometry-tails}
\|\eta\|_{H^k}+|c-c_*|+|c_Y|
+\|X_+-X_+^*\|_{H^k}\le \mathcal P_k(Y)a^2.
\end{equation}
Here $c_Y=dc/dY$, and the Sobolev norms are taken on $\T$.
Thus the relative inner deformation, centre displacement and velocity
in the $Y$ coordinate, and outer displacement all tend to zero as
$Y\to\infty$.

Figure~\ref{fig:prescribed-coordinates} shows the centred inner chart
and the normal graph over the prescribed outer meridian.

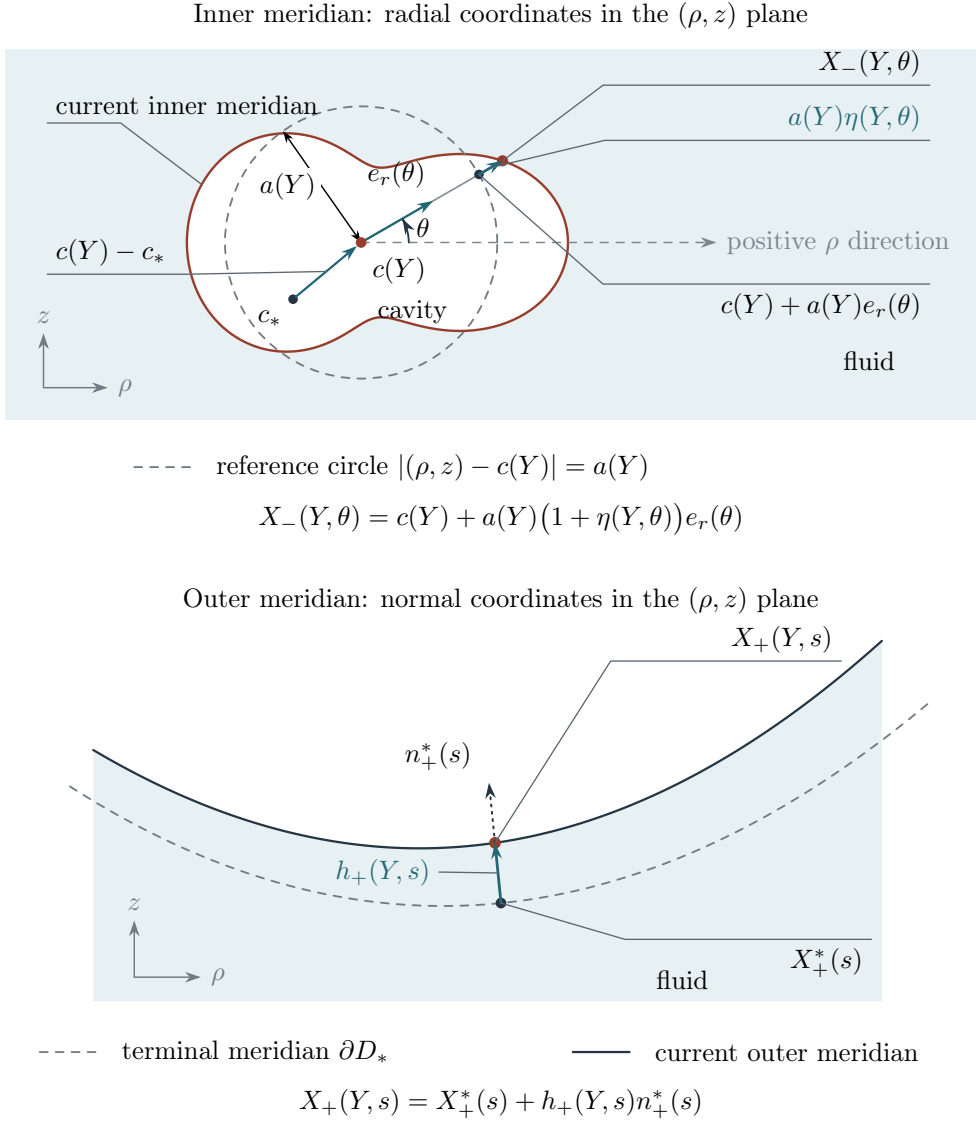
\begin{figure}[!htb]
\centering
\begin{tikzpicture}[cbpfigure]
\path[use as bounding box] (0,0) rectangle (15.7,7.30);
\node at (7.85,7.00) {Inner meridian: radial coordinates in the $(\rho,z)$ plane};
\begin{scope}[shift={(6.00,4.00)}]
\fill[cbpFluid] (-4.70,-2.35) rectangle (8.20,2.55);
\filldraw[fill=white,cbpinner,domain=0:360,samples=200,variable=\ang]
plot ({1.80*(1+.40*cos(2*\ang)+.12*cos(3*\ang))*cos(\ang)},
{1.80*(1+.40*cos(2*\ang)+.12*cos(3*\ang))*sin(\ang)})--cycle;
\draw[dashed,cbpInk!65,line width=.7pt] (0,0) circle (1.80);
\draw[dashed,->,cbpInk!60] (0,0)--(4.70,0)
node[right] {positive $\rho$ direction};
\draw[cbpInk!55,line width=.65pt] (0,0)--({1.80*cos(30)},{1.80*sin(30)});
\draw[cbpflow] (0,0)--({1.10*cos(30)},{1.10*sin(30)});
\node[anchor=south east,inner sep=2pt] at (.93,.64) {$e_r(\theta)$};
\draw[->,cbpInk,line width=.7pt,domain=0:30,samples=25,variable=\ang]
plot ({.63*cos(\ang)},{.63*sin(\ang)});
\node at (.82,.20) {$\theta$};
\draw[cbpmeasure] (0,0)--({1.80*cos(125)},{1.80*sin(125)})
node[midway,fill=white,inner sep=2pt,anchor=east] {$a(Y)$};
\fill[cbpInk] ({1.80*cos(30)},{1.80*sin(30)}) circle (1.8pt);
\fill[cbpCore] ({2.16*cos(30)},{2.16*sin(30)}) circle (2.1pt);
\draw[cbpflow,line width=1.05pt]
({1.80*cos(30)},{1.80*sin(30)})--({2.16*cos(30)},{2.16*sin(30)});
\draw[cbpleader] ({2.16*cos(30)},{2.16*sin(30)})--(3.20,2.08)--(7.52,2.08);
\node[anchor=south east,inner sep=3pt] at (7.52,2.08) {$X_-(Y,\theta)$};
\draw[cbpleader] ({1.98*cos(30)},{1.98*sin(30)})--(3.22,1.35)--(7.52,1.35);
\node[anchor=south east,inner sep=3pt,cbpFlow] at (7.52,1.35) {$a(Y)\eta(Y,\theta)$};
\draw[cbpleader] ({1.80*cos(30)},{1.80*sin(30)})--(3.20,-.55)--(7.52,-.55);
\node[anchor=north east,inner sep=3pt] at (7.52,-.55) {$c(Y)+a(Y)e_r(\theta)$};
\fill[cbpCore] (0,0) circle (2pt);
\node[anchor=north west,inner sep=4pt] at (0,-.04) {$c(Y)$};
\fill[cbpInk] (-.90,-.75) circle (1.8pt);
\node[anchor=north east,inner sep=4pt] at (-.90,-.75) {$c_*$};
\draw[cbpflow] (-.86,-.715)--(-.055,-.046);
\draw[cbpleader] (-.46,-.383)--(-2.30,-.42)--(-4.15,-.42);
\node[anchor=south west,inner sep=3pt] at (-4.15,-.42) {$c(Y)-c_*$};
\node at (.68,-.91) {cavity};
\node at (6.72,-1.57) {fluid};
\draw[cbpleader]
({1.80*(1+.40*cos(320)+.12*cos(480))*cos(160)},
{1.80*(1+.40*cos(320)+.12*cos(480))*sin(160)})--(-3.18,1.58)--(-4.15,1.58);
\node[anchor=south west,inner sep=3pt] at (-4.15,1.58) {current inner meridian};
\draw[->,cbpInk!65] (-4.20,-1.92)--(-3.35,-1.92) node[right] {$\rho$};
\draw[->,cbpInk!65] (-4.20,-1.92)--(-4.20,-1.20) node[above] {$z$};
\end{scope}
\draw[dashed,cbpInk!65,line width=.7pt] (3.00,1.03)--(3.75,1.03);
\node[anchor=west] at (3.95,1.03) {reference circle $|(\rho,z)-c(Y)|=a(Y)$};
\node at (7.85,.35) {$X_-(Y,\theta)=c(Y)+a(Y)\bigl(1+\eta(Y,\theta)\bigr)e_r(\theta)$};
\end{tikzpicture}

\vspace{4mm}
\begin{tikzpicture}[cbpfigure]
\path[use as bounding box] (0,0) rectangle (15.7,7.30);
\node at (7.85,7.00) {Outer meridian: normal coordinates in the $(\rho,z)$ plane};
\begin{scope}[shift={(7.85,0)}]
\path[fill=cbpFluid]
plot[domain=-5.7:5.7,samples=160,variable=\s]
({\s-(.80+.04*\s)*(.10+.13*\s)/sqrt(1+(.10+.13*\s)^2)},
{3.00+.10*\s+.065*(\s)^2+(.80+.04*\s)/sqrt(1+(.10+.13*\s)^2)})
--({5.7-(.80+.04*5.7)*(.10+.13*5.7)/sqrt(1+(.10+.13*5.7)^2)},1.70)
--({-5.7-(.80-.04*5.7)*(.10-.13*5.7)/sqrt(1+(.10-.13*5.7)^2)},1.70)--cycle;
\draw[cbpwall,line width=.85pt,domain=-5.7:5.7,samples=160,variable=\s]
plot ({\s-(.80+.04*\s)*(.10+.13*\s)/sqrt(1+(.10+.13*\s)^2)},
{3.00+.10*\s+.065*(\s)^2+(.80+.04*\s)/sqrt(1+(.10+.13*\s)^2)});
\draw[dashed,cbpInk!65,line width=.7pt,domain=-5.7:5.7,samples=160,variable=\s]
plot (\s,{3.00+.10*\s+.065*(\s)^2});
\fill[cbpInk] (0,3.00) circle (2pt);
\fill[cbpCore] ({-.80*.10/sqrt(1+.10^2)},
{3.00+.80/sqrt(1+.10^2)}) circle (2.2pt);
\draw[cbpflow,line width=1.05pt]
(0,3.00)--({-.80*.10/sqrt(1+.10^2)},{3.00+.80/sqrt(1+.10^2)});
\draw[dotted,->,cbpInk,line width=.75pt]
({-.80*.10/sqrt(1+.10^2)},{3.00+.80/sqrt(1+.10^2)})
--({-1.60*.10/sqrt(1+.10^2)},{3.00+1.60/sqrt(1+.10^2)});
\node[anchor=south east,inner sep=4pt] at (-.24,4.60) {$n_+^*(s)$};
\draw[cbpleader]
({-.80*.10/sqrt(1+.10^2)},{3.00+.80/sqrt(1+.10^2)})--(1.45,6.20)--(4.50,6.20);
\node[anchor=south east,inner sep=3pt] at (4.50,6.20) {$X_+(Y,s)$};
\draw[cbpleader] (0,3.00)--(1.64,2.52)--(4.90,2.52);
\node[anchor=north east,inner sep=3pt] at (4.90,2.52) {$X_+^*(s)$};
\draw[cbpleader] ({-.40*.10/sqrt(1+.10^2)},
{3.00+.40/sqrt(1+.10^2)})--(-.82,3.41);
\node[anchor=east,inner sep=3pt,cbpFlow] at (-.82,3.41) {$h_+(Y,s)$};
\node at (2.40,1.99) {fluid};
\draw[->,cbpInk!65] (-4.85,2.02)--(-3.97,2.02) node[right] {$\rho$};
\draw[->,cbpInk!65] (-4.85,2.02)--(-4.85,2.77) node[above] {$z$};
\end{scope}
\draw[dashed,cbpInk!65,line width=.7pt] (1.75,1.03)--(2.50,1.03);
\node[anchor=west] at (2.70,1.03) {terminal meridian $\partial D_*$};
\draw[cbpwall,line width=.85pt] (8.80,1.03)--(9.55,1.03);
\node[anchor=west] at (9.75,1.03) {current outer meridian};
\node at (7.85,.35) {$X_+(Y,s)=X_+^*(s)+h_+(Y,s)n_+^*(s)$};
\end{tikzpicture}
\caption{Local coordinates for
Theorem~\ref{thm:prescribed-outer-geometry}.
In the upper panel, $\theta$ is measured about the moving centre
$c(Y)$ from the positive $\rho$ direction.  The dashed circle has
radius $a(Y)$. The signed radial displacement from it to the inner
meridian is $a(Y)\eta(Y,\theta)$.  The arrow from $c_*$ to $c(Y)$ is the centre
displacement.  In the lower panel, the corresponding outer boundary
points differ by $h_+(Y,s)n_+^*(s)$, where $n_+^*(s)$ is the outward
unit normal to the terminal meridian.  The displayed radial and normal displacements are
positive. A negative displacement points in the opposite direction.}
\label{fig:prescribed-coordinates}
\end{figure}

\subsection{Proof of the theorem for prescribed geometry}

\begin{proof}[Proof of Theorem~\ref{thm:prescribed-outer-geometry}]
We divide the proof into five steps.

\emph{Step 1. Elliptic estimates for the prescribed geometry.}
Put $\rho_*=\tfrac12\min_{\overline D_*}\rho>0$ and
$d_*=\tfrac12\operatorname{dist}(c_*,\partial D_*)>0$.  Choose the local reference length
$\ell$ sufficiently small that $4\ell<\min(R_*,d_*)$, and set
$r_0=\ell$.  Choose fixed, sufficiently small neighbourhoods of $\partial D_*$ in the
normal graph topology and of $c_*$ so that
\eqref{eq:macro-geometry-assumptions}--\eqref{eq:outer-distance-assumption} hold.
The constants depend on the fixed pair $D_*,c_*$.

Lemma~\ref{lem:axisymmetric-green} applies to this filled domain.
Its regular kernel and the geometric variations of that kernel satisfy
the local bounds used in Lemma~\ref{lem:weighted-layer-expansion}.
The coefficient $\rho$ is invariant under translation in $z$, so
the layer estimates also apply when both centre coordinates vary.
The same nonconstant single-layer inverse and scalar Schur
complement give Lemma~\ref{lem:capacity-single-layer-schur}, with
the Robin coefficient depending on $(R,Z,\Gamma_+)$.
Proposition~\ref{prop:charge-free-localization} applies to the present
macroscopic domain.  Thus the capacity bounds, the estimates for separated
boundaries, and the estimates with fixed derivative loss in
Section~\ref{sec:scaled-DN} hold with the stated powers of $a$.

\emph{Step 2. Centre modulation and the first potential harmonics.}
Retain both centre coordinates and both first harmonic coefficients
of the potential.  For a scalar $f$ write
$f_0=(2\pi)^{-1}\int_\T f\,d\theta$ and
$f_1=\pi^{-1}\int_\T f e_r\,d\theta\in\R^2$.
For vector-valued functions, take means componentwise.  Set
\begin{equation*}
d=1+\eta,\quad e=\eta_\theta,\quad j=(d^2+e^2)^{1/2},
\quad C=e_r-(e/d)e_\theta,\quad b_\eta=(j/d)B_-.
\end{equation*}
The exact kinematic equation is
$a(\eta_Y-d)+c_Y\cdot C=-t_Yb_\eta$.  Define
$M_{ij}=\pi^{-1}\int(e_r)_iC_j$.  Taking the mean and first harmonic projections of the kinematic equation gives
\begin{equation}\label{eq:general-centre-modulation}
\begin{pmatrix}(b_\eta)_0&C_0^\top\\ (b_\eta)_1&M\end{pmatrix}
\begin{pmatrix}t_Y\\c_Y\end{pmatrix}
=\begin{pmatrix}a\\0\end{pmatrix},\qquad
t_Y=\frac a{\Delta_c},\quad
c_Y=-\frac{aM^{-1}(b_\eta)_1}{\Delta_c},\quad
\Delta_c=(b_\eta)_0-C_0\cdot M^{-1}(b_\eta)_1.
\end{equation}
Here $M=I+O(\eta)$ and $C_0=O(\eta^2)$, the latter because
$P_1\eta=0$.  In normalized variables, the matrix is a small perturbation
of the matrix at the circular cylindrical reference state, and $\Delta_c>0$ under the physical bootstrap.

Use the zero-flux decomposition
$\Phi_{0,-}=c_-+\alpha\cdot e_r+f$, with
$P_0f=P_1f=0$ and $\alpha\in\R^2$.  Put
\begin{equation*}
n=aB_-/q,\quad \upsilon=av_{\tau,-}/q,\quad
\boldsymbol\kappa=(\kappa_1,\kappa_2)=c_Y/a,\quad
\beta=\alpha/q,\quad\Theta=qt_Y/a^2.
\end{equation*}
The complete dimensionless kinetic term is
\begin{equation*}
\mathcal K=\frac\Theta2(n^2-\upsilon^2)-\Theta\frac{e}{d}n\upsilon
+\frac jd\upsilon(\boldsymbol\kappa\cdot e_\theta).
\end{equation*}
At the circular cylindrical reference state,
$n=1+\beta\cdot e_r$, $\upsilon=\beta\cdot e_\theta$ and
$\boldsymbol\kappa=-\Theta\beta$.  Hence
\begin{equation*}
\mathcal K=\Theta\left\{\tfrac12(1+\beta\cdot e_r)^2
-\tfrac32(\beta\cdot e_\theta)^2\right\},
\qquad \mathcal K_1=\Theta\beta.
\end{equation*}
The corresponding second harmonic is
\begin{equation*}
P_{\ge2}\mathcal K
=\Theta\bigl((\beta_1^2-\beta_2^2)\cos2\theta
+2\beta_1\beta_2\sin2\theta\bigr).
\end{equation*}
Here and below $P_{\ge2}=I-P_0-P_1$ removes both first harmonics.
With $B_0=\partial_{n_-}\Phi_0$ as before, inversion of the layer
operator on $\operatorname{span}\{\cos\theta,\sin\theta\}$ gives
\begin{align*}
\|aB_0-\alpha\cdot e_r\|_{H^k}
&\le\mathcal P_k\bigl((a+\|\eta\|_{H^{k+2}})|\alpha|
+\|f\|_{H^{k+1}}+\|g\|_{H^{k+4}}\bigr),\\
\|av_{\tau,-}-\alpha\cdot e_\theta\|_{H^k}
&\le C_k\bigl(\|\eta\|_{H^{k+2}}|\alpha|
+\|f\|_{H^{k+1}}\bigr).
\end{align*}
Indeed, the circular nonzero DN block acts as the identity on this
two-dimensional space, and the same exact scalar flux elimination
applies.  The curvature formula depends on $R$ and the graph but
not on $Z$. Its first vector coefficient satisfies
$|\mathcal H_{-,1}|\le C_k(1+a^{-1}\|\eta\|_{H^{k+2}})$.
Substituting these estimates into the exact kinetic function, under
the physical trace bootstrap and $|\beta|\le\delta_1$, gives
\begin{equation*}
\alpha_Y-\alpha=\mathfrak f_1,\qquad |\mathfrak f_1|\le q\mathcal P a.
\end{equation*}
The terminal identity
$\alpha(Y)=-\int_Y^{Y_N}e^{Y-s}\mathfrak f_1(s)\,ds$ holds componentwise.
Thus $|\alpha|+|\alpha_Y|\le q\mathcal P a$ when $\alpha(Y_N)=0$.
Put $u_1=\pi^{-1}\int_{\mathbb T}n e_r\,d\theta$ and
$v_1=\pi^{-1}\int_{\mathbb T}\upsilon e_\theta\,d\theta$, and write
$n=n_0+u_1\cdot e_r+N$ and $\upsilon=v_1\cdot e_\theta+V$.
After subtracting the product of the displayed low modes, each remaining
term contains $N$, $V$, $e/d$, or $j/d-1$.  The high mode projection
of the terms involving only low modes is quadratic in
$(u_1,v_1,\boldsymbol\kappa)$ and is therefore $O_{H^k}(\mathcal P a^2)$ by
the first mode estimate.  This also controls the mixed product
$\beta_1\beta_2$ and proves the $O(\mathcal P a^3)$ forcing
bound in the scaled inner equations for general first harmonics.
Substituting these estimates into
\eqref{eq:general-centre-modulation} and integrating from
$c(Y_N)=c_*$ gives $|c_Y|+|c-c_*|\le\mathcal P a^2$.
The same exact energy reconstruction applies.  Its geometric derivatives
now include $\partial_Z$, with the same kernel bounds at the fixed scale.
Hence $q\asymp(1+Y)^{-1/2}$, $q_Y/q=-1/(2\Lambda)+O(\mathcal P a)$,
and $\Theta=1+O(\mathcal P a^2)$.

\emph{Step 3. The normal variational system.}
For the principal system, put
\begin{equation*}
b_g=\frac{dn-e\upsilon}{j},\qquad z=\frac{en+d\upsilon}{j},\qquad
\kappa_\tau=\boldsymbol\kappa\cdot e_\theta,\qquad
\mathcal V=\frac{\Theta z-\kappa_\tau}{d},\qquad
\mathcal D=\partial_Y+\mathcal V\partial_\theta.
\end{equation*}
The kinematic equation gives
$\mathcal D X_-=\Theta a(-b_ge_r+ze_\theta)=t_Yv$.
Thus the substitution from radial to normal variations in
Proposition~\ref{prop:polar-full-cancellation} applies with this
transport and retains the full Taylor multiplication operator.
Normalize the modulation matrix by setting
\begin{equation*}
N_c=(j/d)n,\qquad v_c=M^{-1}(N_c)_1,\qquad
\Delta=(N_c)_0-C_0\cdot v_c.
\end{equation*}
Then $\Theta=\Delta^{-1}$ and $\boldsymbol\kappa=-\Theta v_c$.
For a variation at fixed $a,c$, put $u=e/d$.  The differentials are
\begin{align*}
\dot u&=\dot\eta_\theta/d-e\dot\eta/d^2,
&\dot C&=-\dot u\,e_\theta,
&\dot N_c&=(j/d)\dot n+(u/(j/d))n\dot u,\\
\dot M_{ij}&=\pi^{-1}\int(e_r)_i\dot C_j,
&\dot v_c&=M^{-1}(\dot N_c)_1-M^{-1}\dot Mv_c,\\
\dot\Delta&=(\dot N_c)_0-\dot C_0\cdot v_c-C_0\cdot\dot v_c,
&\dot\Theta&=-\Theta^2\dot\Delta,
&\dot{\boldsymbol\kappa}&=-\dot\Theta v_c-\Theta\dot v_c.
\end{align*}
Both $M^{-1}$ and $\Delta^{-1}$ are uniformly bounded under the
bootstrap at the fixed lower index.  Angular integration by parts and
the DN variation formula express these derivatives as finite sums of
scalar functionals of the variations, with the fixed derivative loss
and polynomial constants of the spatial estimates.
The resulting additional terms in the scaled high mode equations are
\begin{align*}
\dot h_Y\big|_{\rm m}
&=-aP_{\ge2}(N_c\dot\Theta+C\cdot\dot{\boldsymbol\kappa}),\\
a\dot{(f_Y/q)}\big|_{\rm m}
&=aP_{\ge2}\left[
\left(\tfrac12(n^2-\upsilon^2)-un\upsilon
-B_q\mathscr C\right)\dot\Theta
+(j/d)\upsilon(\dot{\boldsymbol\kappa}\cdot e_\theta)\right].
\end{align*}
The reconstruction estimates and the centering constraints imply,
in each fixed coefficient norm,
\begin{equation*}
P_{\ge2}N_c=O(\mathcal P a^2),\qquad
P_{\ge2}C=O(\mathcal P a^2),\qquad
P_{\ge2}\left(\tfrac12(n^2-\upsilon^2)-un\upsilon
-B_q\mathscr C\right)=O(\mathcal P a^2),
\end{equation*}
whereas $P_{\ge2}((j/d)\upsilon e_\theta)=O(\mathcal P a)$.
We express the kernels of these finite-rank terms in the variables
$U=(\eta,F,h_+,G)$, where $F=f/q$ and $G=g/q$, with $a,c,q,\beta$ held fixed.
Let $\boldsymbol\Psi=(\lambda+\beta\cdot e_r+F,G)$ be the trace
of $\Phi/q$.  With $e_-=(1,0)$ and $w=(\rho_-/R)j$, its constant
$\lambda$ is determined by
\begin{equation*}
(\mathfrak G\boldsymbol\Psi)_-=wn,
\qquad \langle e_-,\mathfrak G\boldsymbol\Psi\rangle=2\pi.
\end{equation*}
Pairings use $d\theta$ on each boundary.  For
$\chi\in\{(2\pi)^{-1},\pi^{-1}\cos\theta,\pi^{-1}\sin\theta\}$, put
\begin{equation*}
t_\chi=\frac{R\chi}{\rho_-d},\qquad
d_c=\langle e_-,\mathfrak Ge_-\rangle,\qquad
r_\chi=\frac{\langle e_-,\mathfrak G(t_\chi,0)\rangle}{d_c},
\qquad Z_\chi=\mathfrak G\bigl((t_\chi,0)-r_\chi e_-\bigr).
\end{equation*}
Thus $\langle e_-,Z_\chi\rangle=0$ and $d_c\asymp(1+Y)^{-1}$.
Write $\bar b_+=q^{-1}b_+^{\rm ch}=q^{-1}v|_{\Gamma_+}\cdot n_+^*$, where $n_+^*$
is the fixed normal defining the outer graph.  The moment derivative is
\begin{align}
D\!\left(\int\chi N_c\,d\theta\right)[\dot U]
&=\langle B_\chi^\eta,\dot\eta\rangle
+\langle (Z_\chi)_-,\dot F\rangle
-\langle\bar b_+(Z_\chi)_+,\dot h_+\rangle
+\langle(Z_\chi)_+,\dot G\rangle,\notag\\
B_\chi^\eta
&=b_g(Z_\chi)_--\frac{\rho_-}{R}z\,\partial_\theta t_\chi
-(\mathfrak G\boldsymbol\Psi)_-t_\chi
\left(\frac1d+\frac{a\cos\theta}{\rho_-}\right).
\label{eq:general-moment-kernel}
\end{align}
Indeed, $N_c=R(\mathfrak G\boldsymbol\Psi)_-/(\rho_-d)$.
The radial inner variation has normal component
$H_-=-a(d/j)\dot\eta$ and tangential component
$a(e/j)\dot\eta$.  Hadamard's formula, including the tangential
pullback of the flux density, gives its geometric contribution as
\begin{equation*}
D\mathfrak G[\dot\eta]\boldsymbol\Psi
=\mathfrak G(b_g\dot\eta,0)
+\left(\partial_\theta\bigl((\rho_-/R)z\dot\eta\bigr),0\right).
\end{equation*}
By self-adjointness, we move $\mathfrak G$ onto the test trace and
integrate the second term by parts.  The differentiated flux constraint
eliminates $\dot\lambda$ after replacing the test trace
with $(t_\chi,0)-r_\chi e_-$, with $r_\chi$ held fixed.
The derivative
$\dot t_\chi=-t_\chi(d^{-1}+a\cos\theta/\rho_-)\dot\eta$
supplies the last term in \eqref{eq:general-moment-kernel}.
The test trace vanishes on the outer boundary, so the normal and
tangential contributions combine into
$-q^{-1}(v\cdot n_+^*)\dot h_+(Z_\chi)_+$.
The remaining moments have the local kernels
\begin{equation*}
D\!\left(\int\chi C_i\,d\theta\right)[\dot\eta]
=\int\frac{\partial_\theta(\chi(e_\theta)_i)}d\,
\dot\eta\,d\theta,\qquad i=1,2.
\end{equation*}
This follows by integrating the $\dot\eta_\theta$ term once.
The two terms containing $e\dot\eta/d^2$ cancel.
Substitution into the displayed differentials of $M,v_c,\Delta$
gives row kernels $B_\Theta,B_{\kappa_i}$ such that
$\dot\Theta=\langle B_\Theta,\dot U\rangle$ and
$\dot\kappa_i=\langle B_{\kappa_i},\dot U\rangle$.

Set $P=P_{\ge2}$ and
\begin{equation*}
A_\Theta=\binom{-PN_c}
{P\left(\tfrac12(n^2-\upsilon^2)-un\upsilon-B_q\mathscr C\right)},
\qquad
A_i=\binom{-PC_i}{P\left((j/d)\upsilon(e_\theta)_i\right)}.
\end{equation*}
The two scaled rows above have the exact modulation contribution
$a\mathsf K(U)\dot U$, where
\begin{equation*}
\mathsf K(U)=A_\Theta\otimes B_\Theta+
\sum_{i=1}^2A_i\otimes B_{\kappa_i},
\qquad (A\otimes B)\dot U=A\langle B,\dot U\rangle.
\end{equation*}
Fix $N_1$.  The explicit kernels, the flux estimates and the scalar
denominators determine a fixed index $b_{N_1}$ such that, on a bounded
ball at this lower index,
\begin{equation*}
\sum_\nu\|B_\nu(U)\|_{H^{N_1}}\le\mathcal P(Y),\qquad
\sum_\nu\|B_\nu(U)-B_\nu(0)\|_{H^{N_1}}
\le\mathcal P(Y)\|U\|_{H^{b_{N_1}}}.
\end{equation*}
Here $\nu\in\{\Theta,1,2\}$ indexes the three terms in $\mathsf K$.
We write $B_i=B_{\kappa_i}$ for $i=1,2$.  Each DN operator
acts on one of the displayed test traces, so these estimates use only
spatial derivatives of the state up to a fixed order.  In the inner Hadamard pairing, the normal
displacement, arclength density and two physical derivatives have
scales $a,a,a^{-2}$, respectively.  Their product is scale invariant.
The factor $d_c^{-1}$ contributes only polynomial factors in $Y$.
For $r\ge b_{N_1}$, the estimates at high output regularity have a fixed
loss $d_{\rm out}$ independent of $r$:
\begin{equation*}
\sum_\nu\|A_\nu(U)-A_\nu(0)\|_{H^r}
\le C_r\mathcal P_r(Y)\|U\|_{H^{r+d_{\rm out}}},\qquad
\sum_\nu\|A_\nu(0)\|_{H^r}
\le C_r\mathcal P_r(Y)a.
\end{equation*}
The reference state keeps the current parameters and prescribed
outer geometry.  Its bound follows from $|\beta|\le\mathcal P a$ and the
circular condenser expansion.  The identity
\begin{equation*}
A(U)\otimes B(U)-A(0)\otimes B(0)
=(A(U)-A(0))\otimes B(U)+A(0)\otimes(B(U)-B(0))
\end{equation*}
therefore gives, for a fixed $d_{N_1}$ independent of $r$,
\begin{equation*}
\|\mathsf K(U)-\mathsf K(0)\|_{H^{-N_1}\to H^r}
\le C_r\mathcal P_r(Y)\|U\|_{H^{r+d_{N_1}}},\qquad
\|\mathsf K(0)\|_{H^{-N_1}\to H^r}\le C_r\mathcal P_r(Y)a.
\end{equation*}
Only the range functions $A_\nu$ are measured at the output index $r$.
For input $(a\dot\eta,a\dot F,\dot h_+,\dot g)$ at fixed $q$,
the prefactor $a$ replaces a row kernel
$(B_\eta,B_F,B_h,B_G)$ by
$(B_\eta,B_F,aB_h,(a/q)B_G)$.
The preceding kernel formulas use normalized trace coordinates.
For the spatial dyadic expansion, now write $U=(\eta,F,h_+,g)$,
keep $a,c,\beta$ fixed, and reconstruct $q$ from the energy equation.
The coordinate conversion gives
\begin{equation*}
\dot G=q^{-1}\dot g-G\,\dot q/q,\qquad \dot\beta=0.
\end{equation*}
The energy equation is quadratic in $q$, with positive root
\eqref{eq:frozen-positive-root} and implicit derivative denominator
$2q(D_b+D_{ii})+D_{ig}\asymp q(1+Y)$.
Differentiating this equation and \eqref{eq:p-normalized-explicit},
and moving tangential and DN derivatives onto the test fields,
gives the scalar input kernels in these coordinates.  Their additional
inner range functions include $(0,2\Theta B_qP\mathscr C)$ for
$\dot q/q$ and $(0,-F)$ for $\dot p_q$; the outer trace conversion
retains its factor $G$.  As in
Lemma~\ref{lem:reconstructed-kernel-bounds}, these kernels have the
same fixed-loss difference bounds, with polynomial factors from
$q^{\pm1}$. Enlarge $b_{N_1}$ and $d_{N_1}$ once to cover these
scalar kernels and the trace conversion.
In the dyadic sum below, $\mathsf K$ denotes the sum
of the modulation and scalar chain rule kernels in the unscaled physical
coordinates $U$, before multiplication by the row factor $a$.
For a physical row kernel
$\widehat B=(\widehat B_\eta,\widehat B_F,\widehat B_h,\widehat B_g)$,
this multiplication and the input
$(a\dot\eta,a\dot F,\dot h_+,\dot g)$ give the row
\begin{equation*}
(\widehat B_\eta,\widehat B_F,a\widehat B_h,a\widehat B_g).
\end{equation*}
All factors from $G=g/q$, including $q^{-1}$, are already included
in $\widehat B$.
Fixing $\beta$ in the spatial expansion leaves its actual time
variation unchanged:
\begin{equation*}
\beta_Y=q^{-1}\alpha_Y-p_q\beta.
\end{equation*}
This derivative remains in the differentiated coefficients.

Use \eqref{eq:frozen-admissible-domain} in these physical coordinates,
with $c=(R,Z)$ and $\beta\in\mathbb R^2$ as parameters, increasing
$b_{\rm ad}$ to control both first harmonic moments.
This fixed choice precedes the final choice of the lower bootstrap
and energy indices.
The inner Dirichlet energy is a nonnegative quadratic form in
$(\beta,F)$, so the energy-root argument of
Lemma~\ref{lem:frozen-admissibility} applies with the same energy
and Cauchy--Schwarz bounds.
Moreover, \eqref{eq:general-centre-modulation} gives
$M=I+O(\|\eta\|_{H^{b_{\rm ad}}})$.
The circular flux expansion and the moment estimates give
\begin{equation*}
\Delta=1+O(\mathcal P(Y)a)
+O(\mathcal P(Y)\|U\|_{H^{b_{\rm ad}}}).
\end{equation*}
Thus $M$ is invertible and $\Delta>1/2$ throughout this domain
for sufficiently large $Y$.  Its Fourier cutoffs and Taylor segments
are admissible by the same convexity argument as in
Lemma~\ref{lem:frozen-admissibility}.

On a frozen dyadic block, for $b\ge b_{N_1}$,
\begin{equation*}
\|[\mathsf K(S_{j-K}U)-\mathsf K(0)]U_j\|_{H^{s+2}}
\le C_s\mathcal P_s(Y)
2^{(2+d_{N_1})j}\|U\|_{H^s}
2^{-(b+N_1)j}\|U\|_{H^b}.
\end{equation*}
Choose $b+N_1>2+d_{N_1}$ while satisfying the other fixed lower index
requirements.  Summation gives the quadratic remainder of
Lemma~\ref{lem:finite-loss-dyadic}. The reference kernel remains an
integrable linear source.  Multiplication by the row factor $a$
and the preceding scale conversion give the source bounds of
Theorem~\ref{thm:actual-principal-reduction} for $\sigma>0$.
For $\sigma=0$, the two source components are measured in $H^s$
and $H^{s+1/2}$ and combined with
\eqref{eq:zero-outer-to-inner} and \eqref{eq:zero-inner-source}.

Differentiation in $Z$ acts only on the regular Green kernels at the
fixed scale, since the cylindrical weight and the inner curvature
formula are independent of $Z$.
The corresponding time derivative carries the factor $Z_Y=O(\mathcal P a^2)$
and therefore satisfies the same coefficient bounds as the other
centre terms.

\emph{Step 4. Energy and trace estimates.}
We apply the inner calculus on the full real space
$P_{\ge2}H^s(\mathbb T)$.  The Fourier proofs of the product,
composition, dyadic and commutator estimates sum over all signed
frequencies and thus apply on this space.  The principal
symbol is $x=|\xi|/j$. Ellipticity on the retained modes follows from
$|m|\ge2$ for both sine and cosine modes, where $m\in\mathbb Z$
is the Fourier index.  The change of
unknowns has the same constant limiting matrix and the same
small perturbation in the lower norm.  The preceding explicit
modulation calculation accounts for the additional finite-rank
terms.  Curvature and pressure are handled by their geometric
variation formulas on the full graph space.  The estimates for separated boundaries in
Lemma~\ref{lem:zero-separated-asymmetric} apply to arbitrary $c$ and
arbitrary inner graphs.  In particular, the inner inputs remain at the
fixed lower index when the outer output is estimated at index $s+1$.
Thus the second centre coordinate and the sine harmonics preserve the
source orders and the regularity assigned to each boundary.

For $\sigma>0$, the outer argument at a fixed scale in
Proposition~\ref{prop:actual-outer-principal-energy} uses ellipticity.
The finitely many nonpositive low modes are completed by fixed
finite-rank corrections, which are returned to the source with the
integrable factor $t_Y$.  Consequently the principal reduction and
Lemma~\ref{lem:zero-coupled-terminal-budget} give
$E_s\le\mathcal P_sa^3$ and $\mathcal O_s\le\mathcal P_sa^2$.
For $\sigma=0$, the Hopf bound for the filled Green potential in
Lemma~\ref{lem:zero-outer-taylor} and the filled domain spectral gap in
\eqref{eq:zero-outer-dn-coercivity} apply to this smooth connected
fixed geometry.  With the inner and outer energies paired as in
Theorem~\ref{thm:zero-one-state-closed},
Lemma~\ref{lem:zero-coupled-terminal-budget} gives
$E_s\le\mathcal P_sa^3$ and $\mathcal O_{s+1}^0\le\mathcal P_sa$.
The original outer Bernoulli equation
\eqref{eq:zero-outer-raw-recovery} then recovers
$q^{-1}\|g\|_{H^k}\le\mathcal P_ka^2$ with the required fixed number
of additional spatial derivatives.
Choose the bootstrap parameters in the order used in
Theorem~\ref{thm:zero-one-state-closed}: first the lower index and
the number of additional derivatives, then the reference forcing
polynomial, and finally the bootstrap constants and $Y_0$.
The lower index estimates and the terminal inequalities give bounds
at every higher index on the same interval.

\emph{Step 5. Terminal data and the limiting trajectory.}
Let $Y_N\to\infty$ and put $a_N=\ell e^{-Y_N}$.
At $Y_N$, take
$\Omega_N=\Omega_*\setminus\overline{N_{a_N}(\gamma_*)}$.
Let $u_N$ be its harmonic condenser, with traces $1$ on the inner
boundary and $0$ on the outer boundary, and put
$\operatorname{Cap}_N=\int_{\Omega_N}|\nabla u_N|^2\,dx$.
Prescribe $v_N=c_N\nabla u_N$, where
\begin{equation*}
\tfrac12c_N^2\operatorname{Cap}_N
+\sigma|\partial N_{a_N}(\gamma_*)|=E_c.
\end{equation*}
Because $|\partial N_{a_N}(\gamma_*)|=2\pi L_*a_N$,
this determines a positive $c_N$ for all sufficiently large $N$,
for every prescribed $E_c>0$ and fixed $\sigma\ge0$.
The local existence argument in the proof of
Lemma~\ref{lem:terminal-data-admissible} applies to these smooth condenser
data and preserves axisymmetry and the global velocity potential.  At $\sigma=0$, their pressure solves
$-\Delta p_N=\sum_{i,j=1}^3(\partial_i(v_N)_j)^2$ with zero boundary trace.
The positive inner flux gives $Dv_N\not\equiv0$, so the maximum
principle and Hopf's lemma give the strict Taylor sign on both
boundary components.
The starting value $Y_0$ may depend on these parameters and the
prescribed fixed geometry.  The energy and trace bounds give
continuation for each finite terminal parameter, as in Theorems~\ref{thm:one-state-closed}
and~\ref{thm:zero-one-state-closed}.
The conserved energy is $E_*=E_c+\sigma|\partial\Omega_*|$.
The uniform trace and geometric bounds give
$A\to|\partial\Omega_*|$ and $K_0\to0$ as $Y\to\infty$, so
$E_*-E_0\ge E_c/2$ after increasing $Y_0$, uniformly in the terminal
parameter.  On each compact $Y$ interval, the kinematic and Bernoulli
equations together with \eqref{eq:general-centre-modulation} bound the
first $Y$ derivatives at every fixed lower Sobolev index.
Apply the compactness argument of
Proposition~\ref{prop:compactness-without-relative-energy} to the state
$(\eta,c,X_+,\alpha,f,g)$, with $c=(R,Z)$ and $\alpha\in\R^2$.
The same chart and elliptic convergence arguments apply, and the
modulation variables converge by \eqref{eq:general-centre-modulation}.
Thus a diagonal subsequence gives a smooth limit for each fixed
$\sigma\ge0$.
The resulting decay estimates imply the stated smooth geometric limits.
The exact capacity energy identity yields
$q^2Y\to E_c/(\pi L_*)$ and
$T-t\sim a^2/(2q)$.  Proposition~\ref{prop:correct-tube-stability}
gives the curvature laws.  Finally, apply the condenser comparison in
Lemma~\ref{lem:global-radial-energy-approximation} in tubular coordinates
centred at $c(Y)=(R(Y),Z(Y))$.  At each fixed $Y$, translation in $z$
preserves the cylindrical weight, so the same local comparison applies.
The complementary region has bounded fixed scale coefficients and
vanishing kinetic energy.  Together, these facts give
\eqref{eq:main-energy-concentration} with the asserted mass.
\end{proof}
\subsection*{Acknowledgment}
S. Yang was supported by the National Natural Science Foundation of China under Grant No. 12601417, the Beijing Natural Science Foundation under Grant
No. 1264051, the Postdoctoral Fellowship Program (Grade C) of China
Postdoctoral Science Foundation under Grant No. GZC20261671, the Beijing
Postdoctoral Research Activity Funding Program under Grant No. 2026-ZZ-28,
and the Funding Program for Newly Recruited Young Teachers of Beijing
University of Technology. T. Luo was supported by a General Research Fund of the Research Grants
Council of Hong Kong under Grant No. 11313025.

\medskip

\noindent\textbf{Data Availability} Data sharing is not applicable to this article as no datasets were generated or analyzed during the current study.

\medskip

\noindent\textbf{Declarations}
\medskip

\noindent\textbf{Conflict of interest} The authors declare that there is no conflict of interest.

\end{document}